\documentclass[a4paper,twoside,11pt]{article}

\usepackage{a4wide, fancyhdr, amsmath, amssymb, mathtools, yfonts}
\usepackage{mathrsfs}
\usepackage{graphicx}
\usepackage{tikz}
\usepackage[all]{xy}
\usepackage[utf8]{inputenc}
\usepackage{amsthm}
\usepackage[english]{babel}
\usepackage{chngcntr}
\usepackage{ifthen}
\usepackage{calc}
\usepackage{authblk}
\usepackage{hyperref}
\numberwithin{equation}{section}

\newcommand{\Z}{\mathbb{Z}}
\newcommand{\Q}{\mathbb{Q}}
\newcommand{\R}{\mathbb{R}}

\newcommand\FF{\mathbb{F}}

\newcommand\Gal{\mathrm{Gal}}

\newtheorem{lemma}{Lemma}[section]
\newtheorem{theorem}[lemma]{Theorem}
\newtheorem{prop}[lemma]{Proposition}
\newtheorem{corollary}[lemma]{Corollary}
\newtheorem{mydef}[lemma]{Definition}

\newtheorem{remark}[lemma]{Remark}

\title{\vspace{-\baselineskip}\sffamily\bfseries On Stevenhagen's conjecture}
\author[1]{Peter Koymans\thanks{Department of Mathematics, Ann Arbor, MI 48109, USA, koymans@umich.edu}}
\author[2]{Carlo Pagano\thanks{School of Mathematics and Statistics, University Place, Glasgow G12 8SQ, UK, carlein90@gmail.com}}
\affil[1]{University of Michigan}
\affil[2]{University of Glasgow}
\date{\today}

\begin{document}
\maketitle

\begin{abstract}
We generalize a classical reciprocity law due to R\'edei \cite{Redei} using our recently developed description of the $2$-torsion of class groups of multiquadratic fields \cite{KP2}. This result is then used to prove a variety of new reflection principles for class groups, one of which involves a symbol similar to the spin symbol as defined in work of Friedlander, Iwaniec, Mazur and Rubin \cite{FIMR}. We combine these reflection principles with Smith's techniques \cite{Smith} to prove Stevenhagen's conjecture \cite{Stevenhagen} on the solubility of the negative Pell equation.
\end{abstract}

\section{Introduction}
Integral points on conics are a classical topic of study going back to at least the ancient Greeks. For fixed squarefree $d > 0$, the equation
\[
x^2 - dy^2 = 1 \text{ to be solved in } x, y \in \Z
\]
is known as Pell's equation and features in Archimedes' cattle problem. This equation was extensively studied by the Indian mathematicians Brahmagupta and Bhaskara II, who gave an algorithm to find a non-trivial solution of the equation for a given value of $d$.

Unbeknownst of the work of these Indian mathematicians, its study began in Europe with Pierre de Fermat who challenged several English mathematicians to solve it for various values of $d$ with $d = 61$ being particularly notorious (the smallest non-trivial solution being $x = 1766319049$ and $y = 226153980$). Inspired by Fermat's inquiries, William Brouncker and John Wallis also found an algorithm to solve Pell's equation. Euler then misattributed their work to John Pell and named the equation in his honor.

In this paper we shall focus on the equation
\begin{align}
\label{eNegPell}
x^2 - dy^2 = -1 \text{ to be solved in } x, y \in \Z,
\end{align}
which is known as the negative Pell equation. Unlike Pell's equation, it is not true that there is always a solution to equation (\ref{eNegPell}). Indeed, for equation (\ref{eNegPell}) to be soluble it must certainly be soluble with $x, y \in \Q$. The Hasse--Minkowski theorem then shows that equation (\ref{eNegPell}) is soluble with $x, y \in \Q$ if and only if $p \mid d$ implies $p \equiv 1, 2 \bmod 4$. 

Inspired by this, we define $\mathcal{D}$ to be the set of squarefree integers $d > 0$ such that $p \mid d$ implies $p \equiv 1, 2 \bmod 4$ and we define $\mathcal{D}^-$ to be the set of squarefree integers $d$ for which equation (\ref{eNegPell}) is soluble. Dirichlet \cite{Dirichlet} proved that every prime number $p \equiv 1 \bmod 4$ lies in $\mathcal{D}^-$. In particular, $\mathcal{D}^-$ is an infinite set. It is well-known, see \cite[Lemma 1]{FK1}, that $d \in \mathcal{D}^-$ if and only if the $2$-Sylow subgroup of the ordinary class group of $\Q(\sqrt{d})$ coincides with the $2$-Sylow subgroup of the narrow class group. We define $\mathcal{D}(X)$ and $\mathcal{D}^-(X)$ to be those $d$ that further satisfy $d < X$. We are interested in the quantity
\[
\lim_{X \rightarrow \infty} \frac{|\mathcal{D}^-(X)|}{|\mathcal{D}(X)|}.
\]
Nagell \cite{Nagell} conjectured that the limit exists and lies in the interval $(0, 1)$. Stevenhagen developed a heuristical model to predict the value of the above limit\footnote{Strictly speaking, this is not true, since Stevenhagen orders the real quadratic fields by discriminants instead of radicands. However, both the heuristical model and our proofs work for both orderings.}. Based on this model, he \cite{Stevenhagen} conjectured that
\[
\lim_{X \rightarrow \infty} \frac{|\mathcal{D}^-(X)|}{|\mathcal{D}(X)|} = 1 - \alpha, \quad \alpha = \prod_{j \text{ odd}} (1 - 2^{-j}) \approx 0.41942.
\]
Cremona--Odoni \cite{CO} studied the analogous problem in case the number of prime divisors is equal to a fixed integer $t \geq 1$, which is traditionally viewed as an easier problem\footnote{In retrospect, this might in fact be the harder problem. At the time of writing the authors are unable to prove the analogue of Stevenhagen's conjecture when the number of prime divisors is fixed.}. Their main theorem shows that, in this setting, the $\liminf$ is bounded below by some number $\lambda_t$ satisfying $\lambda_t \geq \alpha$. Their method is similar to the one used by Gerth \cite{Gerth} to deal with the $4$-rank of class groups. Blomer \cite{Blomer} showed that 
\[
\mathcal{D}^- (X) \gg \frac{X}{(\log X)^{0.62}},
\]
which is a small logarithmic power off (namely $0.12$) from the correct order of magnitude. Fouvry and Kl\"uners \cite{FK1} were the first to make substantial progress towards Stevenhagen's conjecture. They showed that
\[
\alpha \leq \liminf_{X \rightarrow \infty} \frac{|\mathcal{D}^-(X)|}{|\mathcal{D}(X)|} \leq \limsup_{X \rightarrow \infty} \frac{|\mathcal{D}^-(X)|}{|\mathcal{D}(X)|} \leq \frac{2}{3}.
\]
One important feature of their work is that the number of prime factors is no longer treated as fixed. The lower bound was further improved by Fouvry and Kl\"uners \cite{FK2}, and later by joint work of Chan--Koymans--Milovic--Pagano \cite{CKMP}, who extended \cite{Smith8} to the family $\mathcal{D}$. These works proceeded by studying the $4$-torsion and $8$-torsion of the narrow class group for which the algebraic foundations were laid by R\'edei and Reichardt \cite{Redei8, Redei, RR}. Recently, the upper bound was further improved by the authors \cite{KP3} with the main tool being a generalization of the R\'edei reciprocity law. In this paper we shall prove Stevenhagen's conjecture.

\begin{theorem}
\label{tPell}
We have
\[
\lim_{X \rightarrow \infty} \frac{|\mathcal{D}^-(X)|}{|\mathcal{D}(X)|} = 1 - \alpha.
\]
\end{theorem}

Using an argument similar to \cite[Corollary 7.2]{Smith}, one can use Theorem \ref{tRank} to give an explicit error term for the rate of convergence in Theorem \ref{tPell}. The implied constant is effectively computable.

We remark that the Galois module structure of the unit group of a real quadratic field is precisely determined by the solubility of the negative Pell equation. From this standpoint, Theorem \ref{tPell} can be interpreted as giving an asymptotic formula for the occurrence of a certain Galois module structure, namely $\Z \oplus \Z/2\Z$ with the action that sends $(1,0)$ to $(-1, 1)$ and fixes $(0, 1)$. It is also worth mentioning that some interesting conjectures have recently been proposed for the Galois module structure of the unit group of odd degree abelian extensions of $\Q$, see \cite{BVV}.

Since classical techniques in analytic number theory give an asymptotic formula for $|\mathcal{D}(X)|$, Theorem \ref{tPell} implies an asymptotic formula for $|\mathcal{D}^-(X)|$. From now on we will write $\text{Cl}(K)$ for the narrow class group of a number field $K$. Our proof proceeds by studying the distribution of $2\text{Cl}(\Q(\sqrt{d}))[2^\infty]$ for $d \in \mathcal{D}$. In outstanding work, Smith \cite{Smith} recently proved Gerth's \cite{Gerth} extension fo the Cohen--Lenstra heuristics \cite{CL}, that is
\begin{align}
\label{eCL}
\frac{|\{K \text{ im. quadr.} : |D_K| \leq X, 2\text{Cl}(K)[2^\infty] \cong A\}|}{|\{K \text{ im. quadr.} : |D_K| \leq X\}|} = \frac{\prod_{i = 1}^\infty (1 - 2^{-i})}{|\text{Aut}(A)|}
\end{align}
for any finite, abelian $2$-group $A$. Here $D_K$ denotes the discriminant of $K$. Earlier, Fouvry and Kl\"uners \cite{FK3, FK4} found the distribution of $2\text{Cl}(K)[4]$ based on work of Heath-Brown on $2$-Selmer groups \cite{HB}. Using Smith's ideas, the authors \cite{KP1} established the analogue of equation (\ref{eCL}) for cyclic fields of odd prime degree conditional on GRH. This generalizes earlier work of Klys \cite{Klys} who dealt with the analogous results of Fouvry and Kl\"uners \cite{FK3, FK4} in this setting. Smith recently announced that he has obtained an unconditional generalization of the results in \cite{KP1} to arbitrary base fields. In the process, he also obtained an excellent error term by developing a flexible version of the large sieve.

To prove Theorem \ref{tPell}, it is natural to use the ideas introduced by Smith in \cite{Smith}. There are four obstacles that one faces in such an approach. To explain the issues, we write $\langle b, \chi_a \rangle_{k, \Q(\sqrt{d})}$ for the pairing $\text{Art}_{k, \Q(\sqrt{d})}$ (see Subsection \ref{ssArtinpairing} for the definition) between the character $\chi_a: G_\Q\rightarrow \FF_2$, corresponding to the field $\Q(\sqrt{a})$, and the unique ideal of $\Q(\sqrt{d})$ of norm $b$. Let
\[
C = \{d_0\} \times \{p_{1, 1}, p_{1, 2}\} \times \dots \times \{p_{k, 1}, p_{k, 2}\}
\]
be a product space, thought of as squarefree integers by multiplying out the coordinates. Given divisors $a$ and $b$ of $d_0$, Smith's main algebraic result is a reflection principle of the shape
\begin{multline}
\label{eSmithRef}
\sum_{d \in C} \langle b\pi_k(d), \chi_a + \chi_{\pi_1(d)} \rangle_{k - 1, \Q(\sqrt{d})} = \text{Frob}_{K_{p_{1, 1}, p_{1, 2}, \dots, p_{k - 1, 1}, p_{k - 1, 2}}/\Q}(p_{k, 1}) + \\ 
\text{Frob}_{K_{p_{1, 1}, p_{1, 2}, \dots, p_{k - 1, 1}, p_{k - 1, 2}}/\Q}(p_{k, 2}),
\end{multline}
where $\pi_i$ denotes the natural projection map, $K_{p_{1, 1}, p_{1, 2}, \dots, p_{k - 1, 1}, p_{k - 1, 2}}$ is an explicit number field depending on $p_{1, 1}, p_{1, 2}, \dots, p_{k - 1, 1}, p_{k - 1, 2}$ and $\text{Frob}(p_{k, i})$ lands in
\[
Z(\Gal(K_{p_{1, 1}, p_{1, 2}, \dots, p_{k - 1, 1}, p_{k - 1, 2}}/\Q)) \cong \FF_2
\]
with $Z(G)$ denoting the center of a group $G$. The fields $K_{p_{1, 1}, p_{1, 2}, \dots, p_{k - 1, 1}, p_{k - 1, 2}}$ are extensively studied in \cite{KP2}. To apply such a reflection principle one needs to find appropriate choices for $a$ and $b$, which is not always possible in our setting. Indeed, a first problem is that the solubility of negative Pell is equivalent to $(\sqrt{d})$ being trivial in the narrow class group. This means that we have to compute the Artin pairing with $d$, which is not covered by equation (\ref{eSmithRef}). In this case one gets
\[
\sum_{d \in C} \langle d, \chi_a + \chi_{\pi_1(d)} \rangle_{k, \Q(\sqrt{d})} = \text{Frob}_{K_{p_{1, 1}, p_{1, 2}, \dots, p_{k, 1}, p_{k, 2}}/\Q}(\infty).
\]
This is problematic, since Smith shows equidistribution of the right-hand side and then uses ingenious combinatorics to deduce from this equidistribution of the left-hand side. To prove equidistribution of the right-hand side, we extend a reciprocity law first presented in \cite{KP3}. This reciprocity law is a generalization of the classical quadratic reciprocity law and R\'edei's reciprocity law \cite{Redei} (see Stevenhagen's work for an extensive treatment \cite{StevenhagenRedei} or Corsman's thesis \cite{Corsman}). Equidistribution of the right-hand side is then a consequence of the Chebotarev density theorem.

To make matters worse, the first Artin pairing $\text{Art}_{1, \Q(\sqrt{d})}$ is symmetric for $d \in \mathcal{D}$. Therefore we would like to have a reflection principle of the shape
\[
\sum_{d \in C} \langle a, \chi_a \rangle_{k + 1, \Q(\sqrt{d})} = \dots
\]
We first show in Theorem \ref{auxiliary result for the self-pairing} that the right-hand side resembles a spin symbol, see \cite{FIMR} for the precise definition of spin symbols. Then we will show in Theorem \ref{involution spin are 0} that this particular spin symbol is trivial. In principle this is possible by adapting the argument in \cite[Section 12]{FIMR} together with our description of the $2$-torsion of class groups of multiquadratic fields \cite{KP2}. We have however opted for a different argument based on Massey symbols and Hilbert reciprocity. 

At first sight, it may seem rather problematic that the right-hand side is identically zero. However, we use Theorem \ref{involution spin are 0} and a simple trick to compute
\[
\sum_{d \in C} \langle a\pi_1(d) \cdot \ldots \cdot \pi_{k - 1}(d), \chi_a + \chi_{\pi_1(d)} + \dots + \chi_{\pi_{k - 1}(d)} \rangle_{k, \Q(\sqrt{d})}.
\]
This does give a reflection principle of the desired shape after an application of Theorem \ref{Redei reciprocity}, which is a strengthening of the reciprocity law from \cite{KP3}. 

A final obstacle comes from the fact that the cube $C$ is one dimension smaller than the cubes appearing in \cite{Smith}. This fact rules out the usual strategy of proving reflection principles, introduced in \cite{Smith}, by imposing physical equalities between certain $1$-cochains. To overcome this issue, we introduce the formalism of profitable triples (see Section \ref{profitability}), a new method to prove reflection principles based on our novel reciprocity law Theorem \ref{Redei reciprocity}.

Once these obstacles are overcome, the rest of the proof is a straightforward adaptation of Smith's work. Although we shall not prove it in this paper, it is not hard to extend our techniques to obtain the distribution of $2\text{Cl}(\Q(\sqrt{d}))[2^\infty]$ in the family $\mathcal{D}$.

There are several open problems where Smith's method faces similar issues because of symmetry properties of a certain pairing:

\begin{itemize}
\item show $100\%$ non-vanishing of $L(\frac{1}{2}, \chi)$ over function fields\footnote{We thank Mark Shusterman for pointing this out to the authors. We hope to study this problem further in future joint work with Mark Shusterman.};
\item remove the assumption on rational $4$-torsion points in \cite[Theorem 1.1]{Smith};
\item show that Greenberg's conjecture holds for the cyclotomic $\Z_2$-extension for $100\%$ of the real quadratic fields ordered by discriminant;
\item study of the Galois module structure of the unit group for biquadratic (or more general abelian) extensions ordered by discriminant or conductor. A conjectural framework for odd degree abelian extensions is developed in \cite{BVV};
\item the strong form of Malle's conjecture for nilpotent extensions;
\item the distribution of unramified $G$-extensions, where $G$ is a $2$-group. Heuristics and function field results are available for odd $p$-groups in \cite{BBH, BW, LWZ, WW};
\item deal with the distribution of $2$-parts of class groups in the family $\FF_q(t, \sqrt{D})$ with $q \equiv 1 \bmod 4$;
\item obtain the distribution of $\text{Cl}(K)[2^\infty]$ as $K$ varies over cyclic degree $4$ extensions;
\item find the distribution of $\text{Cl}(K)[2^\infty]$ as $K$ varies over quadratic extensions over $\Q(i)$ or if $K$ varies over biquadratic extensions containing $\Q(i)$.
\end{itemize}

\noindent It might be possible to extend our reflection principles to some of the above settings. This could then be combined with Smith's method. In some of the cases listed above there are additional roots of unity. In this case it is known that also some of the higher Artin pairings obey certain symmetry properties, see the work of Morgan--Smith \cite{MS} and Lipnowski--Sawin--Tsimerman \cite{LST}. It would be interesting to see the interplay between Smith's method, the techniques presented here and the additional symmetry properties.

Finally, we will mention some other results related to the Cohen--Lenstra heuristics. There is the classical work of Davenport--Heilbronn \cite{DH} on the first moment of $\text{Cl}(K)[3]$ that predates the work of Cohen and Lenstra. The work of Davenport--Heilbronn was extended by Datskovsky--Wright \cite{DW} to general number fields and to a large class of $2$-extensions by Lemke Oliver--Wang--Wood \cite{LWW}, while \cite{BST} and \cite{TT} independently gave a secondary main term over the rational number field. There is also work of Bhargava \cite{Bhargava1, Bhargava2} on the number of $S_4$ and $S_5$-extensions. Over function fields the situation is much better understood thanks to the work of Ellenberg--Venkatesh--Westerland \cite{EVW}.

\subsection{Notation and conventions}
Throughout the paper we shall make use of the following notations.
\begin{itemize}
\item We use $\subseteq$ for inclusions and $\subset$ for strict inclusions.
\item If $n \in \Z_{\geq 0}$, we define $[n] := \{1, \dots, n\}$.
\item If $X = X_1 \times \dots \times X_r$ is a product space, we write $\pi_i: X \rightarrow X_i$ for the natural projection map.
\item For a product space $X = X_1 \times \dots \times X_r$ and $S \subseteq [r]$, we define
\[
\text{Cube}(X, S) = \prod_{i \in S} X_i^2 \times \prod_{i \in [r] - S} X_i.
\]
For $T \subseteq [r]$, we let $\pi_T$ be the natural projection map from $\text{Cube}(X, S)$ to
\[
\prod_{i \in T \cap S} X_i^2 \times \prod_{i \in T \cap ([r] - S)} X_i.
\]
We write $\text{pr}_1$ and $\text{pr}_2$ for the two natural projection maps from $X_i^2$ to $X_i$.
\item Let $X = X_1 \times \dots \times X_r$, $S \subseteq [r]$ and $\bar{x} \in \text{Cube}(X, S)$. If $T \subseteq S$, then $\bar{x}(T)$ is the subset of those $\bar{y} \in \text{Cube}(X, T)$ such that
\[
\pi_i(\bar{y}) \in \{\text{pr}_1(\pi_i(\bar{x})), \text{pr}_2(\pi_i(\bar{x}))\} \text{ for all } i \in S - T, \quad \pi_{[r] - (S - T)}(\bar{x}) = \pi_{[r] - (S - T)}(\bar{y}).
\]
\item If $K$ is a field, we write $G_K$ for its absolute Galois group.
\item If $K$ is a global field, we write $\Omega_K$ for its set of places.
\item For $v \in \Omega_K$, we write $K_v$ for the completion of $K$ at $v$. If $v$ is finite, then we write $\mathbb{F}_v$ for the residue field.
\item We say that a finite place is odd if it does not lie above $(2)$.
\item If $K$ is a local field, we write $K^{\text{unr}}$ for the maximal unramified extension of $K$.
\item We will sometimes use $\cdot$ and $\prod$ to denote the compositum of fields.
\item If $a \in \Q^\ast$, then $\chi_a: G_\Q \rightarrow \FF_2$ is the quadratic character of $G_\Q$ associated to $\Q(\sqrt{a})$. More generally, if $K$ is a field of characteristic $0$ and $\gamma \in K^\ast$, we will write $\chi_\gamma$ for the associated quadratic character.
\item We will often implicitly view $\FF_2$ as a $G_\Q$-module with the discrete topology and the trivial action. 
\item We write $\text{Cl}(K)$ for the narrow class group of a number field $K$.
\item If $G$ is a profinite group, $X$ is a discrete topological space and $\phi: G \rightarrow X$ is a continuous map, we define $N(\phi)$ to be the largest normal open subgroup through which $\phi$ factors; it is an exercise to show that there exists at least one normal open subgroup through which $\phi$ factors, so that the above definition makes sense. In case $G$ is a Galois group, we denote by $L(\phi)$ the field extension corresponding to $N(\phi)$, and call it \emph{the field of definition}.
\item Let $p$ be a prime. Let $K/\Q$ be a quadratic extension in which $p$ ramifies. Then we write $\text{Up}_{K/\Q}(p)$ for the unique place of $K$ above $p$. More generally, if $a$ is a positive, squarefree integer composed of primes $p$ ramifying in $K/\Q$, we define $\text{Up}_{K/\Q}(a)$ to be the unique integral ideal with norm $a$.
\item If $A$ is a set of quadratic characters from $G_\Q$ to $\FF_2$, we define $\Q(A)$ to be the fixed field of $\bigcap_{\chi \in A} \text{ker}(\chi)$.
\item We call an integer $n \in \Z$ squarefree if $p \mid n$ implies $p^2 \nmid n$. In particular, squarefree integers can be of any sign.
\end{itemize}

\subsection*{Acknowledgements}
The authors wish to thank the Max Planck Institute for Mathematics in Bonn for its great work conditions and an inspiring atmosphere. The majority of this work was done while the authors held postdoctoral positions at the Max Planck. We thank Alexander Smith for invaluable discussions when the authors met him in London, 2018. Without these discussions this paper would most likely not have come to be. In particular, we thank Alexander for several discussions that inspired the authors to prove Theorem \ref{involution spin are 0}. 

We are grateful to Andrew Granville for generously providing financial support which allowed us to meet Alexander Smith in London. We also thank Alex Bartel, Stephanie Chan, Wei Ho, Jeffrey Lagarias, Hendrik Lenstra, Djordjo Milovic, Adam Morgan, Mark Shusterman, Efthymios Sofos and Peter Stevenhagen for useful discussions. The second named author gratefully acknowledges financial support through EPSRC Fellowship EP/P019188/1, ``Cohen--Lenstra heuristics, Brauer relations, and low-dimensional manifolds''.

\section{Preliminaries}
The goal of this section is to fix a number of choices used in the rest of this paper and to introduce raw cocycles, expansion maps and Artin pairings. Fix once and for all a separable closure $\mathbb{Q}^{\text{sep}}$ of $\mathbb{Q}$. We also fix, for each place $v$ of $\mathbb{Q}$, a separable closure $\mathbb{Q}_v^{\text{sep}}$ of $\mathbb{Q}_v$. We further choose an embedding
$$
i_v:\mathbb{Q}^{\text{sep}} \to \mathbb{Q}_v^{\text{sep}}.
$$
Such an embedding induces a continuous, injective homomorphism
$$
i_v^{*}: G_{\mathbb{Q}_v} \to G_{\mathbb{Q}}.  
$$
From now on all our number fields are implicitly taken inside our fixed separable closure $\mathbb{Q}^{\text{sep}}$.

We denote by $\mathbb{Q}^{\text{pro-}2}$ the union of all finite Galois subextensions of $\mathbb{Q}^{\text{sep}}/\mathbb{Q}$ whose degree is a power of $2$ and we write
$$
\mathcal{G}_{\mathbb{Q}}^{\text{pro-}2} := \text{Gal}(\mathbb{Q}^{\text{pro-}2}/\mathbb{Q})
$$
for the corresponding Galois group. Let $\pi_{\text{pro-}2}:G_{\mathbb{Q}} \to \mathcal{G}_{\mathbb{Q}}^{\text{pro-}2}$ be the natural projection map. 

For each finite place $v$ of $\Q$, we let
\[
I_v := \text{ker}(G_{\mathbb{Q}_v} \to G_{\mathbb{F}_v}) = \Gal(\Q_v^{\text{sep}}/\Q_v^{\text{unr}}).
\]
be the inertia group at $v$. We denote by $\mathcal{P}$ the collection of odd places of $\Q$. Let us recall the following basic fact.

\begin{prop} 
\label{inertias procyclic}
We have
$$
I_v(2) := \pi_{\textup{pro-}2} \circ i_v^{*}(I_v) \cong_{\textup{top.gr.}} \mathbb{Z}_2
$$
for each $v \in \mathcal{P}$.
\end{prop}

\begin{proof}
Since $v$ is odd, it follows from local class field theory\footnote{Alternatively, one can use an elementary argument based on \cite[Proposition A.5]{KP1}.} that the maximal pro-$2$-extension $L$ of $\Q_v$ equals
\begin{align}
\label{eLocalGalois}
\bigcup_{k = 1}^\infty \Q_v\left(\zeta_{2^k}, \sqrt[2^k]{\pi}\right),
\end{align}
where $\pi$ is a uniformizer at $v$. Let $I_v^{\text{pro-}2}$ be the maximal pro-$2$-quotient of $I_v$. Equation (\ref{eLocalGalois}) implies that $I_v^{\text{pro-2}} \cong_{\textup{top.gr.}} \Z_2$, since
\[
\bigcup_{k = 1}^\infty \Q_v\left(\zeta_{2^k}\right)
\]
is an unramified extension of $\Q_v$ containing all roots of unity $\zeta_{2^k}$. The continuous homomorphism $\pi_{\text{pro-}2} \circ i_v^*: I_v \rightarrow \mathcal{G}_{\mathbb{Q}}^{\text{pro-}2}$ induces a continuous homomorphism $f: I_v^{\text{pro-}2} \rightarrow \mathcal{G}_{\mathbb{Q}}^{\text{pro-}2}$ with the same image as $\pi_{\text{pro-}2} \circ i_v^*$. Using equation (\ref{eLocalGalois}) one more time, we see that $f$ is injective. Indeed, take any $a \in \Q$ with $v(a) = 1$. Since $\zeta_2 \in \Q$, the field
\[
K := \bigcup_{k = 1}^\infty \Q\left(\zeta_{2^k}, \sqrt[2^k]{a}\right)
\]
is a pro-$2$ extension of $\Q$. Furthermore, $\Q_v \cdot i_v(K) = L$. This completes the proof.
\end{proof}

Thanks to Proposition \ref{inertias procyclic} there exists a topological generator of $I_v(2)$ for each odd place $v$. We make a choice $\sigma_v$ of such a generator for the remainder of this paper. Write $I_2(2) := \pi_{\textup{pro-}2} \circ i_2^{*}(I_2)$ and pick $\sigma_2(1), \sigma_2(2) \in I_2(2)$ such that
\[
\chi_{-1}(\sigma_2(1)) = 1, \quad \chi_{-1}(\sigma_2(2)) = 0, \quad \chi_2(\sigma_2(1)) = 0, \quad \chi_2(\sigma_2(2)) = 1.
\]
This is possible, since the extension $\Q(\zeta_8)/\Q$ is totally ramified at $(2)$.

\begin{prop} 
\label{prop:minimal set of gen}
The set
$$
\mathfrak{G} := \{\sigma_v : v \in \mathcal{P}\} \cup \{\sigma_2(1), \sigma_2(2)\} 
$$
is a minimal set of topological generators of $\mathcal{G}_{\Q}^{\textup{pro-}2}$. 
\end{prop}

\begin{proof}
We will first argue that $\mathfrak{G}$ is a set of topological generators. Note that a subset $S$ of a profinite group $\mathcal{G}$ topologically generates $\mathcal{G}$ if and only if $\pi(S)$ generates $G$ in every continuous finite quotient $\pi: \mathcal{G} \rightarrow G$. So let $K/\Q$ be a finite $2$-extension. We must show that (the images of) $\mathfrak{G}$ generates $G = \Gal(K/\Q)$. 

But since $G$ is a finite group, we know that a subset $S$ generates if and only if it generates modulo the Frattini subgroup $\Phi(G)$. Furthermore, $\Phi(G) = G^2 [G, G]$ for a finite $2$-group $G$. Hence it suffices to show that $\mathfrak{G}$ generates $\Gal(L/\Q)$, where $L/\Q$ is an arbitrary multiquadratic extension. This is in turn equivalent to $\mathfrak{G}$ generating $\Gal(M/\Q)$ for every quadratic extension $M$ of $\Q$, which is readily verified.

Finally, if we remove any element of $\mathfrak{G}$, it is not hard to find a quadratic extension $M/\Q$ in which $\mathfrak{G}$ does not generate. This shows the minimality claim.
\end{proof}

From now on we denote by $\mathfrak{G}$ this fixed choice of topological generators of $\mathcal{G}_{\Q}^{\text{pro-}2}$. We shall often invoke the following basic fact.

\begin{prop} 
\label{ramification read off by inertia}
Let $L/\Q$ be a finite Galois subextension of $\Q^{\textup{pro-}2}/\Q$.

$(a)$ Let $v \in \mathcal{P}$. Then $v$ is unramified in $L/\Q$ if and only if the image of $\sigma_v$ in $\textup{Gal}(L/\Q)$ is the identity element. More precisely, the ramification index of $v$ in $L/\Q$ equals 
\[
|\langle \textup{proj}_{\mathcal{G}_{\Q}^{\textup{pro-}2} \to \textup{Gal}(L/\Q)}(\sigma_v) \rangle|.
\] 
$(b)$ If $(2)$ is unramified in $L/\Q$, then both $\sigma_2(1)$ and $\sigma_2(2)$ have trivial image in $\textup{Gal}(L/\Q)$. More precisely, the ramification index of $(2)$ in $L/\Q$ is divisible by
$$
|\langle \textup{proj}_{\mathcal{G}_{\Q}^{\textup{pro-}2} \to \textup{Gal}(L/\Q)}(\sigma_2(1)), \textup{proj}_{\mathcal{G}_{\Q}^{\textup{pro-}2} \to \textup{Gal}(L/\Q)}(\sigma_2(2)) \rangle|.
$$
\end{prop}

\begin{proof}
We recall that, in general, the ramification index of $v$ in $L/\Q$ equals
$$
|(\textup{proj}_{G_\Q \to \textup{Gal}(L/\Q)} \circ i_v^\ast)(I_v)|.
$$
The projection map from $G_\Q$ to $\textup{Gal}(L/\Q)$ factors through $\mathcal{G}_\Q^{\text{pro-}2}$. Therefore part $(a)$ is an immediate consequence of the definition of $\sigma_v$. To prove part $(b)$ it suffices to observe that
$$
\langle \textup{proj}_{\mathcal{G}_{\Q}^{\textup{pro-}2} \to \textup{Gal}(L/\Q)}(\sigma_2(1)), \textup{proj}_{\mathcal{G}_{\Q}^{\textup{pro-}2} \to \textup{Gal}(L/\Q)}(\sigma_2(2)) \rangle
$$
is a subgroup of 
$$
(\textup{proj}_{G_\Q \to \textup{Gal}(L/\Q)} \circ i_2^\ast)(I_2).
$$
Hence the desired divisibility. 
\end{proof}

We write $\Gamma_{\FF_2}(\Q) := \text{Hom}_{\text{top.gr.}}(G_\Q, \FF_2)$ for the set of quadratic characters of $\mathbb{Q}$. We now recall some notation first introduced in \cite{Smith}. Define
\[
N := \frac{\mathbb{Q}_2}{\mathbb{Z}_2},
\]
which we endow with the discrete topology and view as a $G_\Q$-module with trivial action. For each $\chi \in \Gamma_{\FF_2}(\Q)$, we denote by
$$
N(\chi)
$$
the $G_\Q$-module given by the topological abelian group $N$ with the continuous action of $G_\Q$ defined by the formula
$$
\sigma \cdot_{\chi} n := (-1)^{\chi(\sigma)} \cdot n
$$
for each $\sigma \in G_\Q$ and $n \in N$. We now recall the basic material on raw cocycles and expansion maps that will be used throughout this paper. 

\subsection{Raw cocycles} 
\label{characters and cocycles}
For a given $\chi \in \Gamma_{\mathbb{F}_2}(\Q)$, we denote by
$$
\text{Cocy}(G_\Q, N(\chi))
$$
the group of continuous $1$-cocycles from $G_\Q$ to the $G_\Q$-module $N(\chi)$ defined above. In case $x$ is a squarefree integer, we will also use the notation $N(x) := N(\chi_x)$. We notice that the group $\text{Cocy}(G_\Q, N(\chi))$ is a subgroup of the group of continuous $1$-cochains $\text{Map}_{\text{cont}}(G_\Q, N)$. As such all these groups of $1$-cocycles live in the common ambient group $\text{Map}_{\text{cont}}(G_\Q, N)$, which does not depend on $\chi$. This simple observation plays a crucial role in the key Proposition \ref{key calculation of cocycles}. For a general continuous $1$-cochain
$$
\psi: G_\Q \to N
$$
and a squarefree integer $x$, we define for each $\sigma, \tau \in G_\Q$ the coboundary
$$
d_x(\psi)(\sigma, \tau) := -\psi(\sigma\tau) + (-1)^{\chi_x(\sigma)} \cdot \psi(\tau) + \psi(\sigma).
$$
By definition $\psi$ is in $\text{Cocy}(G_\Q, N(\chi_x))$ if and only if $d_x(\psi)$ is identically zero.

\begin{mydef}
\label{semidirect products}
Whenever we consider the group $N \rtimes \mathbb{F}_2$, the implicit action of $\mathbb{F}_2$ on $N$ is by $-\textup{id}$. The same applies for $N[2^s] \rtimes \mathbb{F}_2$ for $s \in \Z_{\geq 0}$.
\end{mydef}

We will now demonstrate that elements $\psi$ of $\text{Cocy}(G_\Q, N(\chi))$ are essentially the same as certain homomorphisms from $G_\Q$ to $N \rtimes \mathbb{F}_2$.

\begin{prop} 
\label{cocycles and semi direct products} 
Let $\psi \in \textup{Cocy}(G_\Q, N(\chi))$. Then the assignment 
$$
G_\Q \to N(\chi) \rtimes \mathbb{F}_2,
$$
given by 
$$
\sigma \mapsto (\psi(\sigma), \chi(\sigma)),
$$
is a continuous group homomorphism. Conversely, given a continuous $1$-cochain $\psi: G_\Q \rightarrow N$ such that
$$
\sigma \mapsto (\psi(\sigma), \chi(\sigma))
$$
is a group homomorphism from $G_\Q$ to $N(\chi) \rtimes \mathbb{F}_2$, we have that $\psi \in \textup{Cocy}(G_\Q, N(\chi))$. The restriction of $\psi$ to $G_{\Q(\chi)}$ is a continuous group homomorphism.

In case $\psi(G_{\Q(\chi)}) \not \subseteq N[2]$, then the field of definition of $\psi$ contains $\Q(\chi)$. 
\end{prop}

\begin{proof}
The first statement is straightforward. Since $N(\chi)$ is a trivial $G_{\Q(\chi)}$-module, the restriction of $\psi$ to $G_{\Q(\chi)}$ is a $1$-cocycle with the trivial action, and hence is a homomorphism. We now prove the last claim. By assumption there exists $\tau \in G_{\Q(\chi)}$ such that $2 \cdot \psi(\tau) \neq 0$. Let $H$ be the largest normal open subgroup of $G_\Q$ through which $\psi$ factors and let $\sigma \in H$. Then we have
$$
\psi(\sigma) = \psi(\text{id}) = 0,
$$
since a $1$-cocycle vanishes on the identity element. We now deduce that 
$$
\psi(\tau) = \psi(\sigma\tau) = (-1)^{\chi(\sigma)} \cdot \psi(\tau) + \psi(\sigma) = (-1)^{\chi(\sigma)} \cdot \psi(\tau).
$$
Since $2 \cdot \psi(\tau) \neq 0$, we obtain the equality $\chi(\sigma) = 0$, which shows that $\text{ker}(\chi) \supseteq H$ as desired. 
\end{proof}

Let $\psi \in \text{Cocy}(G_\Q, N(\chi))$. Since $\psi$ is continuous, $G_\Q$ is compact and $N$ is discrete, the image of $\psi$ is finite. Then it follows from Proposition \ref{cocycles and semi direct products} that the restriction of $\psi$ gives an element of $(G_{\Q(\chi)}^{\text{ab}})^{\vee}[2^\infty]$. We write 
\[
\textup{Cocy}_{\textup{unr}}(G_\Q, N(\chi))
\]
for the inverse image of $\text{Cl}(\Q(\chi))^\vee[2^\infty]$ under the above restriction map. The relevance of such $1$-cocycles in this work comes from the fact that the Hilbert class field of a quadratic field is always a generalized dihedral extension over $\Q$. This is formalized in the following classical proposition.

\begin{prop} 
\label{unramified characters are always cocycles}
Let $\chi \in \Gamma_{\mathbb{F}_2}(\Q)$ and let $s \in \Z_{\geq 0}$. Then the natural restriction map
$$
\textup{Cocy}_{\textup{unr}}(G_\Q, N(\chi))[2^s] \to \textup{Cl}(\Q(\chi))^\vee[2^s]
$$
is a split surjection of abelian groups. More precisely, the kernel is isomorphic to $\mathbb{Z}/2^s\mathbb{Z}$. Furthermore,
\[
\psi \in 2 \cdot \textup{Cocy}_{\textup{unr}}(G_\Q, N(\chi))[2^{s + 1}] \Longleftrightarrow \psi|_{G_{\Q(\chi)}} \in 2 \cdot \textup{Cl}(\Q(\chi))^{\vee}[2^{s + 1}]
\]
for every $\psi \in \textup{Cocy}_{\textup{unr}}(G_\Q, N(\chi))[2^s]$.
\end{prop}

\begin{proof}
Recall that $\text{Gal}(\Q(\chi)/\Q)$ acts on $\text{Cl}(\Q(\chi))^{\vee}$ by $-\text{id}$. This implies that all subgroups of $\text{Cl}(\Q(\chi))^{\vee}$ are $G_\Q$-invariant, and hence define an extension that is Galois over $\Q$. Let us first show that the natural restriction map is surjective. Then, thanks to Proposition \ref{cocycles and semi direct products}, it suffices to show that the natural surjection
\[
\text{Gal}(L/\Q) \twoheadrightarrow \text{Gal}(\Q(\chi)/\Q)
\]
splits as a semidirect product for all cyclic $2$-power extensions $L$ of $\Q(\chi)$ unramified at all finite places.

Since $L/\Q(\chi)$ is unramified at all finite places, it follows that $L/\Q$ has ramification index at most $2$ at all finite places. Therefore, thanks to Proposition \ref{ramification read off by inertia}, every element $\sigma \in \mathfrak{G}$ is sent to an involution under the natural quotient map $\pi: \mathcal{G}_{\Q}^{\text{pro-}2} \rightarrow \Gal(L/\Q)$. Since $\mathfrak{G}$ is a set of topological generators of $\mathcal{G}_{\Q}^{\text{pro-}2}$ (see Proposition \ref{prop:minimal set of gen}), there exists $\sigma \in \mathfrak{G}$ such that $\chi(\sigma) = 1$. Then $\pi(\sigma)$ is a non-trivial involution of $\text{Gal}(L/\Q)$ that projects to the generator of $\text{Gal}(\Q(\chi)/\Q)$. This yields the desired splitting.

Furthermore, the above argument also shows that the surjection 
\[
\text{Gal}(H_{2^\infty}(\Q(\chi))/\Q) \to \text{Gal}(\Q(\chi)/\Q)
\]
splits as a semidirect product, where $H_{2^\infty}(\Q(\chi))$ is the largest abelian $2$-power extension of $\Q(\chi)$ unramified at all finite places. Fix an element $\sigma \in \text{Gal}(H_{2^\infty}(\Q(\chi))/\Q)$ projecting to a generator of $\text{Gal}(\Q(\chi)/\Q)$. Then, a simple calculation, using the semidirect product structure of $\text{Gal}(H_{2^\infty}(\Q(\chi))/\Q)$, shows that the elements of 
\[
\text{ker}\left(\textup{Cocy}_{\textup{unr}}(G_\Q, N(\chi))[2^s] \to \textup{Cl}(\Q(\chi))^{\vee}[2^s]\right)
\]
are entirely determined by their value on $\sigma$. Conversely, the elements in the kernel can take any desired value on $\sigma$. This identifies the kernel with $N(\chi)[2^s]$, which also proves that the surjection $\textup{Cocy}_{\textup{unr}}(G_\Q, N(\chi))[2^s] \to \textup{Cl}(\Q(\chi))^{\vee}[2^s]$ is split, since $\mathbb{Z}/2^s\mathbb{Z}$ is injective as a $\mathbb{Z}/2^s\mathbb{Z}$-module. 

The final conclusion is now a trivial consequences of such a splitting applied both to $s$ and to $s + 1$.  
\end{proof}

\begin{remark}
\label{It suffices to lift with 1-cocycles} 
If $(2)$ is unramified in the extension $L(\psi)/\Q$ or if $(2)$ has residue field degree $1$ in the extension $L(\psi)/\Q$, then the following stronger conclusion holds. For each cocycle $\psi \in \textup{Cocy}_{\textup{unr}}(G_\Q, N(\chi))[2^s]$, we have that 
\[
\psi \in 2 \cdot \textup{Cocy}(G_\Q, N(\chi))[2^{s + 1}] \Longleftrightarrow \psi|_{G_{\Q(\chi)}} \in 2 \cdot \textup{Cl}(\Q(\chi))^{\vee}[2^{s + 1}]. 
\]
A proof of this stronger claim is given during the proof of Theorem \ref{main thm: minimal triples} and used again in the proof of Theorem \ref{main thm: governing triples}. 
\end{remark}

We now give the definition of raw cocycles, which is repeatedly used throughout this text.

\begin{mydef} 
\label{def: raw cocycles}
Let $\chi$ be a non-trivial element of $\Gamma_{\mathbb{F}_2}(\Q)$ and let $s \in \mathbb{Z}_{\geq 0}$. We say that a sequence
$$
(\psi_i)_{0 \leq i \leq s}
$$ 
with $\psi_i \in \textup{Cocy}_{\textup{unr}}(G_\Q, N(\chi))[2^i]$ is a \emph{raw cocycle} in case $2 \cdot \psi_{i + 1} = \psi_i$ for each $0 \leq i \leq s - 1$. If we want to stress the dependence on $N(\chi)$, we say that $\psi_i$ is a raw cocycle for $N(\chi)$.
\end{mydef} 

To give a raw cocycle as in Definition \ref{def: raw cocycles} is clearly the same as giving a cocycle $\psi_s \in \text{Cocy}_{\text{unr}}(G_\Q, N(\chi))$ with $2^s \cdot \psi_s = 0$. We shall sometimes, by abuse of terminology, say that $\psi_s$ is a raw cocycle, under this identification (and with the $\psi_i$ automatically defined by the equation $\psi_i := 2^{s - i} \psi_s$). Furthermore, whenever the index $j$ of $\psi_j$ is negative, we use the convention $\psi_j = 0$, which conveniently preserves the equation $2 \cdot \psi_{i + 1} = \psi_i$.
 
We now present an important combinatorial calculation. This result is taken from \cite[page 17]{Smith}. We give a detailed proof for the sake of completeness. Let $s \in \Z_{>0}$ and let $(p_i(1), p_i(2))_{i \in [s]}$ be $2s$ distinct primes, all of them coprime to the squarefree integer $d_0$. Let now
$$
C := \{p_1(1), p_1(2)\} \times \ldots \times \{p_s(1), p_s(2)\} \times \{d_0\}.
$$
We will often identify a point $x \in C$ with the squarefree integer $d_0 \cdot \prod_{i \in [s]} \pi_i(x)$. Given a finite collection of squarefree numbers $H$, we denote by $\chi_H$ the continuous $1$-cochain from $G_\Q$ to $\FF_2$ given by
$$
\sigma \mapsto \prod_{x \in H} \chi_x(\sigma),
$$
where the product is the usual multiplication in $\mathbb{F}_2$. By abuse of notation we also view $\chi_H$ as an indicator function valued in $\{0,1\}$. For each subset $T$ of $[s]$, we write $C_T$ for the subset of $x \in C$ such that $\pi_i(x) = p_i(2)$ for all $i \in T$.

\begin{prop} 
\label{key calculation of cocycles}
Let $j \in \Z_{>0}$ and let $C$ be as above. Suppose that we are given a raw cocycle $\psi_j(x)$ for $N(x)$ for each $x \in C$. Let $x_0 \in C$. Then
$$
d_{x_0}\left(\sum_{x \in C} \psi_j(x)\right)(\sigma, \tau) = \sum_{\varnothing \neq T \subseteq [s]} \chi_{\{p_i(1)p_i(2) : i \in T\}}(\sigma) \cdot (-1)^{|T| + 1 + \chi_{x_0}(\sigma)} \cdot \left( \sum_{x \in C_T} \psi_{j - |T|}(x)(\tau) \right).
$$
\end{prop}

\begin{remark}
We emphasize that here and later in the paper the sum
\[
\sum_{x \in C} \psi_j(x)
\]
takes place in the space $\textup{Map}_{\textup{cont}}(G_\Q, N)$.
\end{remark}

\begin{proof}
Let $x \in C$. For each $\sigma, \tau \in G_\Q$ we compute
\begin{align*}
d_{x_0}(\psi_j(x))(\sigma, \tau) &= -\psi_j(x)(\sigma\tau) + (-1)^{\chi_{x_0}(\sigma)}\cdot \psi_j(x)(\tau) + \psi_j(x)(\sigma) \\
&= -\left((-1)^{\chi_x(\sigma)} \cdot \psi_j(x)(\tau) + \psi_j(x)(\sigma)\right) + (-1)^{\chi_{x_0}(\sigma)}\cdot \psi_j(x)(\tau) + \psi_j(x)(\sigma) \\
&= (-1)^{\chi_{x_0}(\sigma)}\cdot(-(-1)^{\chi_{xx_0}(\sigma)} + 1)\cdot \psi_j(x)(\tau),
\end{align*}
where the second equality follows from the definition of $\text{Cocy}(G_\Q, N(x))$.

Observe that the function from $G_\Q$ to $\mathbb{Z}_2 = \text{End}_{\mathbb{Z}_2\text{-mod}}(N)$ given by $\sigma \mapsto -(-1)^{\chi_{xx_0}(\sigma)} + 1$ factors through $\text{Gal}(\Q(\{\sqrt{p_i(1)p_i(2)} : i \in [s]\})/\Q)$. The $\mathbb{Z}_2$-module of functions from the finite Galois group $\text{Gal}(\Q(\{\sqrt{p_i(1)p_i(2)} : i \in [s]\})/\Q)$ to $\mathbb{Z}_2$ is free of rank $2^s$, and a basis of functions is provided by $\{\chi_{\{p_i(1)p_i(2) : i \in T\}}\}_{T \subseteq [s]}$. One could proceed now by explicitly writing down the indicator functions of elements of $\text{Gal}(\Q(\{\sqrt{p_i(1)p_i(2)} : i \in [s]\})/\Q)$ in terms of this basis and use this to expand the function $\sigma \mapsto -(-1)^{\chi_{xx_0}(\sigma)} + 1$ in terms of this basis. However, in \cite[page 17]{Smith} one finds a shortcut to quickly obtain this expansion, which we explain next.   

We identify the projection of $\sigma$ in $\text{Gal}(\Q(\{\sqrt{p_i(1)p_i(2)} : i \in [s]\})/\Q)$ with the largest $V \subseteq [s]$ such that $\chi_{\{p_i(1)p_i(2) : i \in V\}}(\sigma) = 1$. We denote this subset by $T_\sigma$. Similarly, we identify each $x \in C$ with the subset of $[s]$, denoted by $T_x$, consisting of those $i \in [s]$ such that $\pi_i(x) = p_i(2)$. We now rewrite  
\begin{align*}
-(-1)^{\chi_{xx_0}(\sigma)} + 1 &= (-1)\cdot (-1)^{|T_x \cap T_\sigma|} + 1 \\ 
&= (-1) \cdot (1-2)^{|T_x \cap T_\sigma|} + 1 \\
&= \sum_{\varnothing \neq U \subseteq T_x \cap T_\sigma}(-1)^{|U| + 1}2^{|U|}. 
\end{align*}
Hence we obtain
\begin{align*}
d_{x_0}(\psi_j(x))(\sigma, \tau) &= (-1)^{\chi_{x_0}(\sigma)} \cdot \left(\sum_{\varnothing \neq U \subseteq T_x \cap T_\sigma}(-1)^{|U| + 1}2^{|U|}\right) \cdot \psi_j(x)(\tau) \\
&= (-1)^{\chi_{x_0}(\sigma)} \cdot \left(\sum_{\varnothing \neq U \subseteq T_x \cap T_\sigma}(-1)^{|U| + 1}\psi_{j - |U|}(x)(\tau)\right).
\end{align*}
We now let $x$ vary in $C$ and invoke the linearity of $d_{x_0}$. This gives
$$
d_{x_0}\left(\sum_{x \in C}\psi_j(x)\right)(\sigma, \tau) = \sum_{x \in C}(-1)^{\chi_{x_0}(\sigma)} \cdot \left(\sum_{\varnothing \neq U \subseteq T_x \cap T_\sigma}(-1)^{|U| + 1}\psi_{j - |U|}(x)(\tau)\right).
$$
We exchange the order of summation and observe that for a given non-empty $U \subseteq [s]$ we get a contribution only from those pairs $(x, \sigma)$ such that $U \subseteq T_x$ and $U \subseteq T_\sigma$. The first containment can be rewritten as $x \in C_U$. The second containment can be rewritten as $\chi_{\{p_i(1)p_i(2) : i \in U\}}(\sigma) = 1$, since this equality means, by definition of $T_\sigma$, exactly that $U \subseteq T_\sigma$. Hence we get
$$
d_{x_0}\left(\sum_{x \in C}\psi_j(x)\right)(\sigma, \tau) = \sum_{\varnothing \neq U \subseteq [s]} \chi_{\{p_i(1)p_i(2) : i \in U\}}(\sigma) \cdot (-1)^{|U| + 1 + \chi_{x_0}(\sigma)} \cdot \left(\sum_{x \in C_U} \psi_{j - |U|}(x)(\tau)\right)
$$
as desired. 
\end{proof}

We will often use the following version of Proposition \ref{key calculation of cocycles}.

\begin{corollary}
Let $j \in \Z_{>0}$, let $C$ be as above and let $x_0 \in C$. Suppose that we are given a raw cocycle $\psi_j(x)$ for $N(x)$ for each $x \in C - \{x_0\}$. Then 
$$
d_{x_0}\left(\sum_{\substack{x \in C \\ x \neq x_0}}\psi_j(x)\right)(\sigma, \tau) = \sum_{\varnothing \neq T \subseteq [s]} \chi_{\{p_i(1)p_i(2) : i \in T\}}(\sigma) \cdot (-1)^{|T| + 1 + \chi_{x_0}(\sigma)} \cdot \left( \sum_{x \in C_T} \psi_{j - |T|}(x)(\tau) \right).
$$
\end{corollary}

\begin{proof}
Define $\psi_j(x_0) = 0$ and apply Proposition \ref{key calculation of cocycles}.
\end{proof}

\subsection{Expansion maps} 
\label{def normalized maps}
We start by recalling a definition from \cite[Definition 3.21]{KP2}. We write $\FF_2^A$ for the free $\FF_2$-vector space on the set $A$.

\begin{mydef}
Let $X \subseteq \Gamma_{\FF_2}(\Q)$ be a linearly independent finite set and let $\chi_0 \in X$. An expansion map with support $X$ and pointer $\chi_0$ is a continuous group homomorphism
\[
\psi : G_\Q \rightarrow \mathbb{F}_2[\FF_2^{X - \{\chi_0\}}] \rtimes \mathbb{F}_2^{X - \{\chi_0\}}
\]
such that 
\begin{itemize}
\item $\pi_\chi \circ \psi = \chi$ for every $\chi \in X - \{\chi_0\}$, where $\pi_\chi : \mathbb{F}_2[\FF_2^{X - \{\chi_0\}}] \rtimes \mathbb{F}_2^{X - \{\chi_0\}} \rightarrow \FF_2$ is the natural projection;
\item $\pi \circ \psi = \chi_0$, where $\pi : \mathbb{F}_2[\FF_2^{X - \{\chi_0\}}] \rtimes \mathbb{F}_2^{X - \{\chi_0\}} \rightarrow \FF_2$ is the unique non-trivial character that sends the subgroup $\{0\} \rtimes \FF_2^{X - \{\chi_0\}}$ to $0$.
\end{itemize}
\end{mydef}

An expansion map is automatically surjective, since it surjects by construction on the quotient of $\mathbb{F}_2[\FF_2^{X - \{\chi_0\}}] \rtimes \mathbb{F}_2^{X - \{\chi_0\}}$ by its Frattini subgroup. Although we shall not make use of it, we mention that
\[
\mathbb{F}_2[\FF_2^{X - \{\chi_0\}}] \rtimes \mathbb{F}_2^{X - \{\chi_0\}} \cong \FF_2 \wr \mathbb{F}_2^{X - \{\chi_0\}},
\]
where $\wr$ is the wreath product. We will now give an alternative characterization of expansion maps based on the material in \cite[Section 3.3]{KP2}. Consider the isomorphism
\[
\FF_2[\FF_2^{X - \{\chi_0\}}] \cong \FF_2[\{t_x : x \in X - \{\chi_0\}\}]/(\{t_x^2 : x \in X - \{\chi_0\}\})
\]
obtained by sending $t_x$ to $1 \cdot \text{id} + 1 \cdot e_x$, where $e_x \in \FF_2^{X - \{\chi_0\}}$ is the vector that is $1$ exactly on the $x$-th coordinate. Observe that the collection of squarefree monomials $t_Y := \prod_{y \in Y} t_y$, as $Y$ varies through the subsets of $X - \{\chi_0\}$, give a basis of the $\FF_2$-vector space
\[
\FF_2[\{t_x : x \in X - \{\chi_0\}\}]/(\{t_x^2 : x \in X - \{\chi_0\}\}).
\]
Hence projection on the monomials $t_Y$ gives rise to continuous $1$-cochains
\[
\phi_Y(\psi) : G_\Q \rightarrow \FF_2
\]
for every $Y \subseteq X - \{\chi_0\}$. These $1$-cochains allow us to reconstruct $\psi$ using the formula
\begin{equation}
\label{eReconstruct}
\psi(g) = \left(\sum_{Y \subseteq X - \{\chi_0\}} \phi_Y(\psi)(g) t_Y, (\chi(g))_{\chi \in X - \{\chi_0\}} \right).
\end{equation}
Let $d: \text{Map}(G_\Q^k, \FF_2) \rightarrow \text{Map}(G_\Q^{k + 1}, \FF_2)$ be the operator that sends a map $\phi \in \text{Map}(G_\Q^k, \FF_2)$ to
\[
(d\phi)(g_1, \dots, g_{k + 1}) = \phi(g_2, \dots, g_{k + 1}) + \phi(g_1, \dots, g_k) + \sum_{i = 1}^k \phi(g_1, \dots, g_{i - 1}, g_i g_{i + 1}, g_{i + 2}, \dots, g_{k + 1}).
\]
From equation (\ref{eReconstruct}) and the composition law for the semidirect product we deduce that
\begin{equation}
\label{eSmith22}
(d\phi_Y(\psi))(g_1, g_2) = \sum_{\varnothing \subsetneq S \subseteq Y} \chi_S(g_1) \phi_{Y - S}(\psi)(g_2),
\end{equation}
where $\chi_S := \prod_{\chi \in S} \chi$. Equation (\ref{eSmith22}) is just \cite[equation (2.2)]{Smith}. Conversely, if we have a system of maps $(\phi_Y)_{Y \subseteq X - \{\chi_0\}}$ satisfying equation (\ref{eSmith22}) and $\phi_\varnothing = \chi_0$, then we obtain an expansion map $\psi$ with support $X$ and pointer $\chi_0$ using equation (\ref{eReconstruct}).

We now show that for each finite linearly independent set $X \subseteq \Gamma_{\FF_2}(\Q)$ and $\chi_0 \in X$, there exists at most one expansion map, denoted $\psi(\mathfrak{G})$, with support $X$ and pointer $\chi_0$ such that $\phi_{Y}(\psi(\mathfrak{G}))(\sigma) = 0$ for all non-empty subsets $Y \subseteq X - \{\chi_0\}$ and all $\sigma \in \mathfrak{G}$. Indeed, let $\psi$ be such an expansion map. We claim that $\psi(\sigma)$ is determined for every $\sigma \in \mathfrak{G}$. But this follows from equation (\ref{eReconstruct}), our assumption $\phi_Y(\psi)(\sigma) = 0$ and the fact that $\chi(\sigma)$ is prescribed for every $\chi \in \Gamma_{\FF_2}(\Q)$ and $\sigma \in \mathfrak{G}$. Since $\mathfrak{G}$ is a set of topological generators of $\mathcal{G}_{\Q}^{\text{pro-}2}$ by Proposition \ref{prop:minimal set of gen}, the homomorphism $\psi$ is completely determined by its values on $\mathfrak{G}$. This yields the desired uniqueness.

If there exists such an expansion map $\psi(\mathfrak{G})$, we write
$$
\phi_{S; d}(\mathfrak{G})
$$
for the $1$-cochain $\phi_{X - \{\chi_0\}}(\psi(\mathfrak{G}))$, where $S$ is the set of squarefree integers corresponding to the characters in $X - \{\chi_0\}$ and $d$ is the squarefree integer corresponding to $\chi_0$. In this way, whenever we have a non-negative integer $s$, a set of squarefree integers $S$ of cardinality $s$ and a squarefree integer $a_{s + 1}$ such that $S$ together with $a_{s + 1}$ spans a $s + 1$-dimensional space in $\frac{\Q^{*}}{\Q^{*2}}$, we have at most one $1$-cochain
$$
\phi_{S; a_{s + 1}}(\mathfrak{G})
$$
satisfying the above properties. In case we have one, we will say that $\phi_{S; a_{s + 1}}(\mathfrak{G})$ exists. We shall refer to such $1$-cochains as \emph{normalized expansion maps}. 

\begin{lemma}
\label{lAdditivity}
Let $s \in \Z_{\geq 1}$. Let $S$ be a set of squarefree integers of cardinality $s$ and let $x_1, x_2, x_3$ also be squarefree integers with $x_1x_2 = x_3$ in $\Q^\ast/\Q^{\ast 2}$. Suppose that $\phi_{S; x_j}(\mathfrak{G})$ exists for $j \in [2]$. Then $\phi_{S; x_3}(\mathfrak{G})$ exists and furthermore
\[
\phi_{S; x_3}(\mathfrak{G}) = \phi_{S; x_1}(\mathfrak{G}) + \phi_{S; x_2}(\mathfrak{G}).
\]
Now suppose that $A = \{a_1, \dots, a_s\}$, $B = \{b_1, \dots, b_s\}$ and $C = \{c_1, \dots, c_s\}$ are sets of squarefree integers such that there exists $j \in [s]$ with
\[
a_i = b_i = c_i \textup{ for all } i \in [s] - \{j\} \quad \textup{ and } \quad a_jb_j = c_j \textup{ in } \frac{\Q^{*}}{\Q^{*2}}.
\]
Let $x$ be a squarefree integer. Suppose that $\phi_{A; x}(\mathfrak{G})$ and $\phi_{B; x}(\mathfrak{G})$ exist. Then $\phi_{C; x}(\mathfrak{G})$ exists and 
\[
\phi_{C; x}(\mathfrak{G}) = \phi_{A; x}(\mathfrak{G}) + \phi_{B; x}(\mathfrak{G}).
\]
\end{lemma}

\begin{proof}
Let us prove the first part, the second part being similar. We proceed by induction on $s$. First suppose that $s = 1$ and consider
\[
\phi := \phi_{a_1; x_1}(\mathfrak{G}) + \phi_{a_1; x_2}(\mathfrak{G}).
\]
It follows from equation (\ref{eSmith22}) that 
\[
d\phi(\sigma, \tau) = \chi_{a_1}(\sigma) \cdot \chi_{x_1x_2}(\tau) = \chi_{a_1}(\sigma) \cdot \chi_{x_3}(\tau)
\]
Since $\phi$ vanishes at all $\sigma \in \mathfrak{G}$, we conclude that $\phi_{a_1; x_3}(\mathfrak{G})$ exists. Now suppose that $s > 1$. Again we consider
\[
\phi := \phi_{S; x_1}(\mathfrak{G}) + \phi_{S; x_2}(\mathfrak{G}).
\]
It follows from equation (\ref{eSmith22}) and the induction hypothesis that 
\[
d\phi(\sigma, \tau) = \sum_{\varnothing \subsetneq T \subseteq S} \chi_T(\sigma) \cdot \phi_{S - T; x_3}(\mathfrak{G})(\tau).
\]
Furthermore, $\phi$ vanishes at all $\sigma \in \mathfrak{G}$. Therefore $\phi_{S; x_3}(\mathfrak{G})$ exists.
\end{proof}

The following proposition examines the field of definition of $\phi_{S; d}$ in terms of $\phi_{T; d}$ for $T \subset S$. The proof is straightforward group theory and we leave it to the reader.

\begin{prop} 
\label{prop: field of def of exp maps}
Let $X \subseteq \Gamma_{\mathbb{F}_2}(\Q)$ be a linearly independent finite set of cardinality at least $2$ and let $\chi_0 \in X$. Let
$$
\psi: G_\Q \to \mathbb{F}_2[\mathbb{F}_2^{X - \{\chi_0\}}] \rtimes \mathbb{F}_2^{X - \{\chi_0\}}
$$
be an expansion map with support $X$ and pointer $\chi_0$. Then $L(\phi_{X - \{\chi_0\}}(\psi))$ coincides with the field ${(\Q^{\textup{sep}})}^{\textup{ker}(\psi)}$. This field is a multiquadratic extension of $\Q(X - \{\chi_0\})$ given by the quadratic characters
$$
\{\phi_Y(\psi)|_{G_{\Q(X - \{\chi_0\})}} : Y \subseteq X - \{\chi_0\}\}.
$$
Writing
$$
M(\psi) := \Q(X) \cdot \prod_{Y \subsetneq X - \{\chi_0\}}L(\phi_Y(\psi)),
$$
we have that $L(\phi_{X - \{\chi_0\}}(\psi))/M(\psi)$ is a central $\mathbb{F}_2$-extension. Furthermore, 
\[
\textup{Gal}(L(\phi_{X - \{\chi_0\}}(\psi))/M(\psi))
\]
equals the center of $\textup{Gal}(L(\phi_{X - \{\chi_0\}}(\psi))/\Q)$. 
\end{prop}

Following the notation of Proposition \ref{prop: field of def of exp maps}, we will also use the notation
$$
M(\phi_{X - \{\chi_0\}}(\psi)) := M(\psi)
$$
in the rest of the paper. In the next proposition we give a sense of the ramification behavior in the field of definition of a normalized expansion map. 

\begin{prop} 
\label{normalized expansions are unramified}
Let $n$ be a positive integer. Suppose that $a_1, \ldots, a_{n + 1}$ are squarefree, pairwise coprime and not equal to $1$. Suppose that $a_i \equiv 1 \bmod 8$ for each $i \in [n]$ and that $\phi_{\{a_1, \dots, a_n\}; a_{n + 1}}(\mathfrak{G})$ exists. Then the extension
$$
L(\phi_{\{a_1, \dots, a_n\}; a_{n + 1}}(\mathfrak{G}))/\Q(\{\sqrt{a_i} : i \in [n]\})
$$
is a multiquadratic extension containing $\Q(\sqrt{a_{n + 1}})$. Furthermore, the multiquadratic extension
$$
L(\phi_{\{a_1, \dots, a_n\}; a_{n + 1}}(\mathfrak{G}))/\Q(\{\sqrt{a_i} : i \in [n + 1]\})
$$
is unramified at all finite places.
\end{prop} 

\begin{proof}
We will abbreviate $A = \{a_1, \dots, a_n\}$. Proposition \ref{prop: field of def of exp maps} implies that the extension $L(\phi_{A; a_{n + 1}}(\mathfrak{G}))/\Q(\{\sqrt{a_i} : i \in [n]\})$ is a multiquadratic extension containing the field $\Q(\sqrt{a_{n + 1}})$. We now focus on the ramification claims. 

Let us first prove this for a place $v \in \mathcal{P}$ coprime to all the $a_i$. Then combining equation (\ref{eReconstruct}) with the definition of normalized expansion maps, we see that $\sigma_v$ is sent to the identity element of $\Gal(L(\phi_{A; a_{n + 1}}(\mathfrak{G}))/\Q)$. Hence the extension $L(\phi_{A; a_{n + 1}}(\mathfrak{G}))/\Q$ is unramified at $v$ in view of Proposition \ref{ramification read off by inertia}. 

Next, since $a_i \equiv 1 \bmod 8$ for all $i \in [n]$, it follows from equation (\ref{eSmith22}), combined with the definition of normalized expansion maps, that the map
$$
\phi_{\{a_i : i \in T\}; a_{n + 1}}(\mathfrak{G}) \circ i_2^{*}
$$ 
is an unramified quadratic character of $G_{\Q_2}$ for each subset $\varnothing \subset T \subseteq [n]$, . Hence it follows from equation (\ref{eReconstruct}) that the homomorphism $\psi_{A; a_{n + 1}}(\mathfrak{G}) \circ i_2^{*}$ factors through the kernel of $(\chi_5, \chi_{a_{n + 1}})$. As such, the extension $L(\phi_{A; a_{n + 1}}(\mathfrak{G}))/\Q(\{\sqrt{a_i} : i \in [n + 1]\})$ is unramified at any place above $(2)$. 

Finally, suppose that $v \in \mathcal{P}$ divides one of the $a_i$. Observe that, by the coprimality conditions, there exists exactly one $i_0 \in [n]$ such that $v \mid a_{i_0}$. It follows from the definition of normalized expansion maps, combined with equation (\ref{eReconstruct}), that $\psi_{A; a_{n + 1}}(\mathfrak{G})(\sigma_v)$ is an involution. Therefore Proposition \ref{ramification read off by inertia} implies that the ramification index of $v$ in $L(\phi_{A; a_{n + 1}}(\mathfrak{G}))/\Q$ equals $2$. Since the ramification index of $v$ in $\Q(\{\sqrt{a_i} : i \in [n + 1]\})/\Q$ is also $2$, we conclude that any place of $\Q(\{\sqrt{a_i} : i \in [n + 1]\})$ above $v$ is unramified in 
$$
L(\phi_{A; a_{n + 1}}(\mathfrak{G}))/\Q(\{\sqrt{a_i} : i \in [n + 1]\}).
$$
This ends the proof. 
\end{proof}

The following criterion gives an inductive procedure for creating expansion maps, see also \cite[Proposition 2.1]{Smith}.

\begin{prop}
\label{pCreatePhi}
Let $a_1, \dots, a_{s + 1}$ be odd, squarefree integers that are pairwise coprime. Assume that $a_1, \dots, a_s > 1$ and $a_{s + 1} \neq 1$. Further suppose that 
\begin{align}
\label{eStronglyConsistent}
\left(\frac{a_j}{p}\right) = 1
\end{align}
for all distinct $i, j \in [s + 1]$ and all primes $p$ dividing $a_i$. Then $\phi_{\{a_i : i \in [s]\}; a_{s + 1}}(\mathfrak{G})$ exists if $\phi_{\{a_i : i \in [s] - \{j\}\}; a_{s + 1}}(\mathfrak{G})$ exists for all $j \in [s]$ and every prime divisor $p$ of $a_j$ splits completely in $L(\phi_{\{a_i : i \in [s] - \{j\}\}; a_{s + 1}}(\mathfrak{G}))$.
\end{prop}

\begin{remark}
The converse of Proposition \ref{pCreatePhi} is also true, but we shall not need or prove it here.
\end{remark}

\begin{proof}
We claim that there exists a continuous $1$-cochain $\phi: G_\Q \rightarrow \FF_2$ such that
\[
d\phi(\sigma, \tau) = \sum_{\varnothing \subset S \subseteq [s]} \left(\prod_{j \in S} \chi_{a_j}(\sigma)\right) \cdot \phi_{\{a_i : i \in [s] - S\}; a_{s + 1}}(\mathfrak{G})(\tau) =: \theta(\sigma, \tau).
\]
Once such a $\phi$ exists, we certainly have that $\phi(\text{id}) = 0$. Since $\phi$ is continuous, this implies that $\phi(\sigma) = 0$ except for possibly finitely many $\sigma \in \mathfrak{G}$. Then the proposition follows by twisting $\phi$ by the unique quadratic character $\chi: G_\Q \rightarrow \FF_2$ such that $\phi(\sigma) = \chi(\sigma)$ for all $\sigma \in \mathfrak{G}$.

To prove the claim, we first show that $\theta$ is a $2$-cocycle. If $\psi$ satisfies $\psi(\sigma, \tau) = \phi_1(\sigma) \cdot \phi_2(\tau)$ with $\phi_1, \phi_2: G_\Q \rightarrow \FF_2$, then we have the formula
\[
(d\psi)(\sigma, \tau, \mu) = \phi_1(\sigma) \cdot d\phi_2(\tau, \mu) + d\phi_1(\sigma, \tau) \cdot \phi_2(\mu).
\]
We combine this with the equations
\[
d\left(\prod_{j \in T} \chi_{a_j}\right)(\sigma, \tau) = \sum_{\varnothing \subset S \subset T} \left(\prod_{j \in S} \chi_{a_j}(\sigma)\right) \cdot \left(\prod_{j \in T - S} \chi_{a_j}(\tau)\right)
\]
for all $T \subseteq [s]$ and
\[
d\phi_{\{a_i : i \in T\}; a_{s + 1}}(\mathfrak{G})(\sigma, \tau) = \sum_{\varnothing \subset S \subseteq T} \left(\prod_{j \in S} \chi_{a_j}(\sigma)\right) \cdot \phi_{\{a_i : i \in T - S\}; a_{s + 1}}(\mathfrak{G})(\tau)
\]
for all strict subsets $T$ of $[s]$ to deduce that $\theta$ is indeed a $2$-cocycle. Hence the claim is equivalent to the class of $\theta$ vanishing in $H^2(G_\Q, \mathbb{F}_2)$, which is in turn equivalent to the restriction of $\theta$ vanishing in $H^2(G_{\Q_v}, \mathbb{F}_2)$ for all places $v$ of $\Q$ by class field theory.

If $v$ is the infinite place, then $\theta$ becomes the zero map when restricted to $G_\R$ since the $a_i$ are positive for $i \in [s]$. If $v$ is an odd, finite place, then we distinguish two cases. If $v$ does not divide any of the $a_i$, then the restriction of $\theta$ to $G_{\Q_v}$ factors through the maximal unramified extension of $\Q_v$. It follows that $\theta$ is in the image of the inflation map from
\[
H^2(\Gal(\Q_v^{\text{unr}}/\Q_v), \FF_2) = H^2(\hat{\Z}, \FF_2) = 0,
\]
which implies that $\theta$ is trivial at such places $v$. Now suppose that $v$ divides $a_j$ for some $j$. If $j = s + 1$, then equation (\ref{eStronglyConsistent}) implies that $\theta$ becomes the zero map when restricted to $G_{\Q_v}$. Therefore $\theta$ is certainly trivial in $H^2(G_{\Q_v}, \mathbb{F}_2)$. Instead suppose that $j \in [s]$. Now observe that our assumptions imply that $\phi_{\{a_i : i \in [s] - \{j\}\}; a_{s + 1}}(\mathfrak{G})$ is the zero map when restricted to $G_{\Q_v}$. Then it follows from equation (\ref{eStronglyConsistent}) that $\theta$ is also locally trivial at such $v$. We have now dealt with all odd places and the infinite place. Then $\theta$ also vanishes locally at the unique $2$-adic place of $\Q$ by Hilbert reciprocity.
\end{proof}

Finally, we need some additional maps that can be thought of as degenerate expansion maps with two indices. Let $a \in \mathcal{D}$. Then observe that 
$$
\chi_a \cup \chi_a
$$ 
vanishes in $H^2(G_\Q, \mathbb{F}_2)$. There is a unique map
$$
\phi_{a; a}(\mathfrak{G}): G_\Q \rightarrow \mathbb{F}_2
$$
such that
$$
(d\phi_{a; a}(\mathfrak{G}))(\sigma, \tau) = \chi_a(\sigma) \cdot \chi_a(\tau)
$$
and $\phi_{a; a}(\mathfrak{G})$ vanishes at all elements of $\mathfrak{G}$. If $a > 1$, then $L(\phi_{a; a}(\mathfrak{G}))$ is a cyclic degree $4$ extension of $\Q$ with its unique quadratic subextension equal to $\Q(\sqrt{a})$. Indeed, one can write $\mathbb{Z}/4\mathbb{Z}$ set-theoretically as $\mathbb{F}_2 \times \mathbb{F}_2$ with the group law given by 
\[
(x, y) * (x', y') = (x + x' + yy', y + y').
\]
Then the $\mathbb{Z}/4\mathbb{Z}$-character corresponding to $\phi_{a; a}(\mathfrak{G})$ is the map 
\[
\sigma \mapsto (\phi_{a; a}(\mathfrak{G})(\sigma), \chi_a(\sigma)).
\]
Take $v \in \mathcal{P}$ and suppose that $v$ is unramified in $\Q(\sqrt{a})$. Then the above map sends $\sigma_v$ to $(0, 0)$, which is the identity. We deduce from Proposition \ref{ramification read off by inertia} that all $v \in \mathcal{P}$ ramifying in $L(\phi_{a; a}(\mathfrak{G}))/\Q$ must divide $a$. In case $v$ divides $a$, then the ramification index of $v$ in $L(\phi_{a; a}(\mathfrak{G}))/\Q$ is $4$ thanks to Proposition \ref{ramification read off by inertia}. The above statements also apply to $v = (2)$ in case $2 \mid a$. 

Suppose now that $2$ does not divide $a$. Observe that $a$ is then either $1$ or $5$ modulo $8$. In case $a$ is $1$ modulo $8$, then $\phi_{a; a}(\mathfrak{G}) \circ i_2^{*}$ becomes a quadratic character, and since it vanishes on $\sigma_2(1), \sigma_2(2)$ it must be an unramified quadratic character. Therefore, in this case, $(2)$ does not ramify in $L(\phi_{a; a}(\mathfrak{G}))/\Q$. 

In case $a$ is $5$ modulo $8$, then there exists $\phi':\text{Gal}(\Q_2^{\text{unr}}/\Q_2) \to \mathbb{F}_2$ with $(d\phi')(\sigma, \tau) = \chi_a(\sigma) \cdot \chi_a(\tau)$. Indeed, this $1$-cochain comes from the $\mathbb{Z}/4\mathbb{Z}$-character $G_{\Q_2} \twoheadrightarrow \text{Gal}(\Q_{16}/\Q_2)$. Hence $\phi_{a; a}(\mathfrak{G}) - \phi'$ is a quadratic character in $\Gamma_{\mathbb{F}_2}(\Q_2)$. However, it vanishes on $\sigma_2(1), \sigma_2(2)$, since each of the two $1$-cochains do. Hence it must be in the span of the unramified quadratic character of $\Q_2$. It follows that $(\phi_{a; a}(\mathfrak{G}), \chi_a)$ is unramified at $(2)$. We summarize our results in the next lemma.

\begin{lemma}
\label{lCyclicDegree4}
Let $a \in \mathcal{D}$ be greater than $1$. Then the extension $L(\phi_{a; a}(\mathfrak{G}))/\Q$ is a $\mathbb{Z}/4\mathbb{Z}$-extension of $\Q$ ramifying exactly at those places where $\Q(\sqrt{a})/\Q$ ramifies. Furthermore, the ramification index of such places equals $4$.
\end{lemma}

\subsection{The Artin pairing}
\label{ssArtinpairing}
Let $A$ be a finite abelian $2$-group and let $s \in \mathbb{Z}_{\geq 1}$. Denote by 
$$
A^{\vee} := \text{Hom}_{\text{ab.gr.}} \left(A, \ \frac{\Q_2}{\mathbb{Z}_2} \right)
$$
the dual group. We define the pairing
$$
\langle -, -\rangle_{\text{Art}_s(A)}: 2^{s - 1} \cdot A[2^s] \times 2^{s - 1} \cdot A^{\vee}[2^s] \to \mathbb{F}_2
$$
by the formula
$$
\langle a, \chi \rangle_{\text{Art}_s(A)} := \psi(a),
$$
where $\psi$ is any element of $A^{\vee}$ with $2^{s - 1} \cdot \psi = \chi$. Observe that changing the choice of $\psi$ amounts to taking an element of the shape $\psi + \psi'$ with $\psi' \in A^\vee[2^{s - 1}]$. This does not affect the resulting pairing, since $a$ is in $2^{s - 1} \cdot A[2^s]$ and thus vanishes when paired, through the canonical duality pairing, with $A^\vee[2^{s - 1}]$. Also observe that the pairing is valued in $\mathbb{F}_2$, as we claimed, since $a \in A[2]$ and $\psi$ is a group homomorphism. Thus we have shown that the pairing does not depend on the choice of $\psi$ and is valued in $\mathbb{F}_2$. 

We claim that the left kernel and the right kernel of $\langle -, -\rangle_{\text{Art}_s(A)}$ actually coincide with, respectively, $2^s \cdot A[2^{s + 1}]$ and $2^s \cdot A^{\vee}[2^{s + 1}]$. Indeed, a moment reflection shows that 
$$
\langle -, - \rangle_{\text{Art}_s(A)} = \langle -, - \rangle_{\text{Art}_{1}(2^{s - 1} \cdot A)}.
$$
In this last equality we are implicitly identifying $(2^{s - 1} \cdot A)^{\vee}$ and $2^{s - 1} \cdot A^{\vee}$, through the standard inclusion of $(2^{s - 1} \cdot A)^{\vee}$ in $A^{\vee}$ induced by the natural surjection $A \twoheadrightarrow 2^{s - 1} \cdot A$ given by multiplication by $2^{s - 1}$. Hence the claim follows from the case $s = 1$, which is an immediate consequence of the duality theory of finite dimensional vector spaces over $\mathbb{F}_2$. 

In case $x$ is a squarefree integer different from $1$ and $A = \text{Cl}(\Q(\sqrt{x}))[2^\infty]$, we will simply write
$$
\langle -, -\rangle_{\text{Art}_s(x)}:=\langle -, -\rangle_{\text{Art}_s(\text{Cl}(\Q(\sqrt{x}))[2^\infty])}.
$$
In what follows we will canonically identify $\text{Cl}(\Q(\sqrt{x}))[2^\infty]$ with the largest quotient of $\text{Cl}(\Q(\sqrt{x}))$ which is a $2$-group: the natural projection map from the former to the latter induces an isomorphism. In this way, via the Artin reciprocity map, $\text{Cl}(\Q(\sqrt{x}))[2^\infty]$ is identified with
$$
\text{Gal}(H_{2^\infty}(\Q(\sqrt{x}))/\Q(\sqrt{x})),
$$
where we recall that $H_{2^\infty}(\Q(\sqrt{x}))$ is the largest extension of $\Q(\sqrt{x})$ inside $\Q^{\text{sep}}$ that is abelian, unramified at all finite places of $\Q(\sqrt{x})$ and of degree a power of $2$. This allows us to reinterpret the Artin pairing $\langle -, -\rangle_{\text{Art}_s(x)}$ as the Artin symbol of a $2$-torsion ideal class in a cyclic degree $2^s$-extension of $\Q(\sqrt{x})$ unramified at all finite places. We will repeatedly use throughout the text this way of computing $\langle -, -\rangle_{\text{Art}_s(x)}$.

\section{Higher R\'edei reciprocity}
\label{sRed}
This section generalizes one of the central results in \cite{KP3}, which is a generalization of the classical R\'edei reciprocity law (in turn a generalization of quadratic reciprocity).

\subsection{Statement of the reciprocity law}
Let $n \in \mathbb{Z}_{\geq 1}$ and let $A \subseteq \Gamma_{\FF_2}(\Q)$ with $|A| = n$. Let $\chi_1, \chi_2$ be two distinct elements of $\Gamma_{\FF_2}(\Q) - A$ and put
$$
A_1 := A \cup \{\chi_1\}, \quad A_2 := A \cup \{\chi_2\}. 
$$ 
Given a finite Galois extension $L/\Q$, we denote by $\text{Ram}(L/\Q)$ the set of places of $\Q$ that ramify in the extension $L/\Q$. Furthermore, for a collection of characters $T \subseteq \Gamma_{\FF_2}(\Q)$, we recall that $\Q(T)$ denotes the corresponding multiquadratic extension of $\Q$. 

Write $\infty$ for the infinite place of $\Q$. We assume that, as $\chi$ varies in $A$, the $n$ sets $\text{Ram}(\Q(\chi)/\Q)$ are non-empty, pairwise disjoint and none of them contains $\infty$ (in other words $\Q(A)/\Q$ is totally real). This forces $A$ to be a linearly independent set of characters over $\FF_2$. Additionally, we assume that $\text{Ram}(\Q(\chi_1)/\Q)$ and $\text{Ram}(\Q(\chi_2)/\Q)$ are non-empty and disjoint from $\cup_{\chi \in A}\text{Ram}(\Q(\chi)/\Q)$, which certainly implies that $\chi_1$ is linearly independent from $A$, and similarly for $\chi_2$.

Suppose that we are given two expansion maps
$$
\psi_1, \psi_2: G_\Q \twoheadrightarrow \FF_2[\FF_2^A] \rtimes \FF_2^A
$$
with supports $A_1, A_2$ and pointers $\chi_1, \chi_2$ respectively. Write $(\phi_{1, B})_{B \subseteq A}, (\phi_{2, B})_{B \subseteq A}$ for the corresponding system of continuous $1$-cochains from $G_\Q$ to $\FF_2$ satisfying $\phi_{1, \varnothing} = \chi_1, \phi_{2, \varnothing} = \chi_2$. Thanks to Proposition \ref{prop: field of def of exp maps}, $L(\psi_1)$ and $L(\psi_2)$ are central $\FF_2$-extensions of respectively $M(\psi_1)$ and $M(\psi_2)$. We will impose some further conditions on $\chi_1$ and $\chi_2$, which we will refer to as the \emph{coprimality constraints} on the pointers. 

First of all, we demand that $\text{Ram}(\Q(\chi_1)/\Q) \cap \text{Ram}(\Q(\chi_2)/\Q) \subseteq \{(2)\}$. Next we demand that $\text{inv}_2(\chi_1 \cup \chi_2) = 0$ and that at least one between $\chi_1$ and $\chi_2$ vanishes on $\sigma_2(2)$. Finally, we require $(2)$ to split completely in $\Q(A)/\Q$. If all these conditions are met, we say that the pointers satisfy the coprimality constraints. 

We give one more definition before stating our reciprocity law. We remark that whenever we have a $4$-tuple as above, then for each subset $T' \subseteq A$ the $1$-cochains $\phi_{T'}(\psi_1) \circ i_2^{*}, \phi_{T'}(\psi_2) \circ i_2^{*}$ are quadratic characters from $G_{\Q_2}$ to $\mathbb{F}_2$, as one can see by combining equation (\ref{eSmith22}) with the fact that $(2)$ splits completely in $\Q(A)/\Q$.

\begin{mydef} 
\label{def: Redei admissible}
Let $(A_1, A_2, \psi_1, \psi_2)$ be a $4$-tuple as above. We call $(A_1, A_2, \psi_1, \psi_2)$ \emph{R\'edei admissible} if the following four conditions hold
\begin{itemize}
\item if $\infty$ splits completely in $\Q(A_1 \cup A_2)/\Q$, then $\infty$ splits completely in $M(\psi_1)M(\psi_2)/\Q$ as well. 
\item whenever $\{i, j\} = \{1,2\}$ and $\infty$ ramifies in $L(\psi_i)/\Q$, then $\infty$ splits completely in $M(\psi_j)/\Q$. 
\item whenever $\{i, j\} = \{1, 2\}$, if a place $w$ of $\Q(A)$, lying above a place $v \in \mathcal{P}$, ramifies in the extension $L(\psi_i)/\Q(A)$, then $v$ splits completely in $M(\psi_j)/ \Q$. Furthermore, we also demand that $v$ is unramified in $L(\psi_j)/\Q$;
\item whenever $L(\psi_1)L(\psi_2)/\Q$ is ramified at $(2)$, then there exist $i, j$ with $\{i, j\} = \{1, 2\}$ such that $\psi_i$ is a normalized expansion map satisfying $\chi_i(\sigma_2(2)) = 0$ and $\phi_T(\psi_i) \circ i_2^\ast = 0$ for each $\varnothing \neq T \subsetneq A$. Furthermore, the characters $\chi_i \circ i_2^\ast$ and $\phi_{T'}(\psi_j) \circ i_2^\ast$ are orthogonal with respect to the local Hilbert pairing at $(2)$ for each $T' \subseteq A$. 
\end{itemize}
We say that $(\chi_1, \chi_2)$ is the \emph{pointer vector} of the $4$-tuple and that $A$ is the \emph{base set} of the $4$-tuple. 
\end{mydef}

Let $(A_1, A_2, \psi_1, \psi_2)$ be a R\'edei admissible $4$-tuple with pointer vector $(\chi_1, \chi_2)$. Then each place $v \in \text{Ram}(\Q(\chi_1)/\Q) \cap \mathcal{P}$ is unramified in $L(\psi_2)/\Q$. Indeed, this follows from the third condition in Definition \ref{def: Redei admissible} combined with the fact that $\text{Ram}(\chi_1)$ is disjoint from $\text{Ram}(\chi)$ for each $\chi \in A$. Consequently, it makes sense to speak of the Artin class $\text{Art}(v, L(\psi_2)/\Q)$ for each place $v \in \text{Ram}(\Q(\chi_1)/\Q) \cap \mathcal{P}$. Furthermore, this Artin symbol lands in $\text{Gal}(L(\psi_2)/M(\psi_2))$, which is the center of $\text{Gal}(L(\psi_2)/\Q)$ (see Proposition \ref{prop: field of def of exp maps}) of size equal to $2$ and hence can uniquely be identified with $\mathbb{F}_2$. Therefore $\text{Art}(v, L(\psi_2)/\Q)$ is a well-defined element of $\FF_2$. Symmetrically, the same holds if we swap the roles of $1$ and $2$.

Next, whenever $\{i, j\} = \{1, 2\}$ and $\infty$ ramifies in $\Q(\chi_i)/\Q$, then $\infty$ splits completely in $M(\psi_j)/\Q$ by the second condition of Definition \ref{def: Redei admissible}. Consequently the Artin symbol $\text{Art}(\infty, L(\psi_j)/\Q)$ is also an element of the center $\text{Gal}(L(\psi_j)/M(\psi_j))$ and hence is a well-defined element of $\mathbb{F}_2$. 

Whenever $(2)$ ramifies in $L(\psi_1)L(\psi_2)/\Q$, let $\{i, j\} = \{1, 2\}$ be as in the last condition of Definition \ref{def: Redei admissible}. Then, in case $\chi_i \circ i_2^\ast = 0$, $(2)$ splits completely in $M(\psi_i)/\Q$ and $(2)$ is unramified in $L(\psi_i)/\Q$. Consequently, the Artin symbol $\text{Art}((2), L(\psi_i)/\Q)$ is also an element of the center $\text{Gal}(L(\psi_i)/M(\psi_i))$ and thus a well-defined element of $\mathbb{F}_2$. 

Instead suppose that $\chi_i \circ i_2^{*} \neq 0$. Then $L(\psi_i)/\Q(\chi_i)$ is unramified at $\text{Up}_{\Q(\chi_i)/\Q}(2)$, thanks to Proposition \ref{normalized expansions are unramified}, and this place splits completely in $M(\psi_i)/\Q(\chi_i)$. Consequently, the Artin symbol $\text{Art}(\text{Up}_{\Q(\chi_i)/\Q}(2), L(\psi_i)/\Q(\chi_i))$ is well-defined and an element of the center $\text{Gal}(L(\psi_i)/M(\psi_i))$ and therefore a well-defined element of $\mathbb{F}_2$. By abuse of notation we will sometimes denote this symbol as $\text{Art}((2), L(\psi_i)/\Q)$. Also, observe that in case $\chi_i \circ i_2^{*} = \chi_5$, then it must be that $\chi_j(\sigma_2(2)) = 0$, thanks to the fact that $\chi_i$ is orthogonal to $\chi_j$ with respect to the local Hilbert pairing at $(2)$. 

Finally, for a quadratic extension $\Q(\sqrt{d})/\Q$, we put 
\[
\widetilde{\text{Ram}}(\Q(\sqrt{d})/\Q) = 
\left\{
\begin{array}{ll}
\text{Ram}(\Q(\sqrt{d})/\Q) \setminus (2) & \mbox{if } d \text{ has even 2-adic valuation} \\
\text{Ram}(\Q(\sqrt{d})/\Q) & \mbox{otherwise.}
\end{array}
\right.
\]
We can now state the reciprocity law.

\begin{theorem} 
\label{Redei reciprocity}
Let $(A_1, A_2, \psi_1, \psi_2)$ be a R\'edei admissible $4$-tuple with pointer vector $(\chi_1, \chi_2)$. Then
$$
\sum_{v \in \widetilde{\emph{Ram}}(\Q(\chi_1)/\Q)} \emph{Art}(v, L(\psi_2)/\Q) = \sum_{v' \in \widetilde{\emph{Ram}}(\Q(\chi_2)/\Q)} \emph{Art}(v', L(\psi_1)/\Q).
$$
\end{theorem}

\subsection{\texorpdfstring{Proof of Theorem \ref{Redei reciprocity}}{Proof of Theorem 3.2}}
Take a R\'edei admissible $4$-tuple $(A_1, A_2, \psi_1, \psi_2)$ with pointer vector $(\chi_1, \chi_2)$ and base set $A$. We derive from equation (\ref{eSmith22}) that the tuple of $1$-cochains $(\phi_{i, B})_{B \subseteq A}$ becomes a tuple of quadratic characters when restricted to $G_{\Q(A)}$ for every $i \in \{1, 2\}$. Furthermore, it follows directly from the definition of an expansion map that the character $\phi_{i, A}$ generates a rank $1$ free module over the ring $\mathbb{F}_2[\text{Gal}(\Q(A)/\Q)]$ for every $i \in \{1, 2\}$, and furthermore the corresponding Galois extension of $\Q$ equals $L(\psi_i)$.

For each $B \subseteq A$ we denote by $\alpha_{i, B} \in \frac{\Q(B)^{*}}{\Q(B)^{*2}}$ the unique element, provided by Kummer theory, corresponding to the restriction of $\phi_{i, B}$ to $G_{\Q(B)}$ (which is a quadratic character). Our next proposition is a generalization of the well-known connection between $D_4$-extensions and solution sets of certain conics, see for instance \cite[Section 5]{StevenhagenRedei}, which also explains how this phenomenon is related to R\'edei symbols.

\begin{prop} 
\label{norming stuff down}
We have for all $i \in \{1, 2\}$ and all $B \subseteq A$ that
$$
N_{\Q(A)/\Q(B)}(\alpha_{i, A}) = \alpha_{i, B}
$$
as elements of $\frac{\Q(B)^{*}}{\Q(B)^{*2}}$.
\end{prop}

\begin{proof}
From the recursive formula (\ref{eSmith22}) we see that it suffices to show that
$$
N_{\Q(A)/\Q(A - \{a\})}(\alpha_{i, A}) = \alpha_{i, A - \{a\}} \ \text{ in } \ \frac{\Q(A - \{a\})^{*}}{\Q(A - \{a\})^{*2}}
$$
for all $a \in A$: the full proposition is then obtained by applying this repeatedly.

By Kummer theory, this is the same as showing that the co-restriction of the character $\phi_{i, A}$ from $G_{\Q(A)}$ to $G_{\Q(A - \{a\})}$ equals the character $\phi_{i, A - \{a\}}$. To this end, let us recall the following basic fact. Let $G_1 \subseteq G_2$ be a continuous inclusion of profinite groups with $[G_2 : G_1] = 2$. If $\chi:G_1 \rightarrow \mathbb{F}_2$ and $\chi':G_2 \rightarrow \mathbb{F}_2$ are two continuous characters, then the co-restriction of $\chi$ to $G_2$ equals $\chi'$ if and only if 
\begin{align}
\label{eCores1}
\chi(\sigma^2) = \chi'(\sigma)
\end{align}
and 
\begin{align}
\label{eCores2}
\chi(\sigma\tau\sigma^{-1}) + \chi(\tau) = \chi'(\tau)
\end{align}
for each $\sigma \in G_2 - G_1$, $\tau \in G_1$. Equation (\ref{eCores2}) only implies that the co-restriction of $\chi$ equals $\chi'$ as characters of the index $2$ subgroup $G_1$. This leaves two possibilities for the character $\chi$ from the larger group $G_2$: the two possibilities constitute a single coset under the subgroup generated by the character $\epsilon: G_2 \twoheadrightarrow \frac{G_2}{G_1} = \mathbb{F}_2$. Said differently, equation (\ref{eCores2}) forces the co-restriction of $\chi$ to be in the set $\{\chi', \chi' + \epsilon\}$. These two cases can be distinguished by means of equation (\ref{eCores1}). 
		
Returning to our setup, we will now explain how equation (\ref{eCores2}) follows from the definition of expansion maps as coordinates of monomials in a semidirect product. Here $G_{\Q(A)}$ plays the role of $G_1$ and $G_{\Q(A - \{a\})}$ plays the role of $G_2$. We claim that
\begin{align}
\label{eCores3}
\phi_{i, A}(\sigma \tau \sigma^{-1}) + \phi_{i, A}(\tau) = \phi_{i, A - \{a\}}(\tau)
\end{align}
for each $\sigma \in G_{\Q(A - \{a\})} - G_{\Q(A)}$ and $\tau \in G_{\Q(A)}$. To prove the claim, observe that $\text{ker}(\psi_i)$ is contained in the kernel of all the quadratic characters $\phi_{i, A},(1+\sigma)\phi_{i, A}, \phi_{i, A - \{a\}}$, since it gives a normal extension of $\Q$ containing the kernel of the first and the third. 

Hence it suffices to check equation (\ref{eCores3}) in the quotient given by $\text{ker}(\psi_i)$ or equivalently in $\text{im}(\psi_i)$. The group $\psi_i(G_{\Q(A)})$ equals $\mathbb{F}_2[\mathbb{F}_2^{A}] \rtimes \{0\}$. The character $\phi_{i, A}$ is the projection on the monomial $t_{A}$, the character $\phi_{i, A - \{a\}}$ is the projection on the monomial $t_{A - \{a\}}$ and $\sigma$ acts as multiplication by $1 + t_a$, since $\sigma \in G_{\Q(A - \{a\})} - G_{\Q(A)}$. This yields the identity
\begin{align*}
\sum_{T \subseteq A} (\phi_{i,T}(\sigma \tau \sigma^{-1}) + \phi_{i,T}(\sigma)) \cdot t_T 
&= (1 + \sigma) \cdot \left(\sum_{T \subseteq A} \phi_{i,T}(\tau) \cdot t_T\right) \\
&= t_a \cdot \left(\sum_{T \subseteq A} \phi_{i,T}(\tau) \cdot t_T\right) = \sum_{T \subseteq A - \{a\}} \phi_{i,T}(\tau) \cdot t_{T \cup \{a\}}
\end{align*}
for all $\tau \in \mathbb{F}_2[\mathbb{F}_2^{A}] \rtimes \{0\}$. Comparing the $t_A$-coefficients gives precisely the desired conclusion. 

It remains to check equation (\ref{eCores1}), which forces the norm relation to hold as characters of the larger group $G_{\Q(A - \{a\})}$. To this end, pick $\sigma \in G_{\Q(A - \{a\})} - G_{\Q(A)}$, and plug in $(\sigma, \sigma)$ in equation (\ref{eSmith22}). The left hand side equals $\phi_{i, A}(\sigma^2)$, which is the quantity we are after, while the right hand side equals $\phi_{i, A - \{a\}}(\sigma)$. This establishes the proposition.
\end{proof}

Our next result is a direct consequence of Proposition \ref{norming stuff down}.
 
\begin{corollary} 
\label{right parity}
Let $i \in \{1, 2\}$.
\begin{enumerate}
\item[(a)] Let $v$ be a finite place of $\mathbb{Q}$ that splits completely in $\Q(A)/\Q$. Then $v \in \widetilde{\emph{Ram}}(\Q(\chi_i)/\Q)$ if and only if
$$
|\{w \in \Omega_{\Q(A)} : w \mid v, w(\alpha_{i, A}) \equiv 1 \bmod 2\}| \equiv 1 \bmod 2.
$$
\item[(b)] We have $\infty \in \widetilde{\emph{Ram}}(\Q(\chi_i)/\Q)$ if and only if
$$
|\{\sigma: \Q(A) \rightarrow \mathbb{R} : \sigma(\alpha_{i, A}) < 0\}| \equiv 1 \bmod 2.
$$
\end{enumerate}
\end{corollary}

\begin{proof}
We shall explain the argument for part $(a)$ and leave part $(b)$, which can be proven similarly, to the reader. Recall that $\phi_{i, \varnothing} = \chi_i$ for $i \in \{1, 2\}$ and that $\Q(\chi_i) = \Q(\sqrt{\alpha_{i, \varnothing}})$. Let us start by observing that
\begin{align}
\label{eDefRamtilde}
v \in \widetilde{\text{Ram}}(\Q(\chi_i)/\Q) \Longleftrightarrow v(\alpha_{i, \varnothing}) \equiv 1 \bmod 2
\end{align}
for all finite places $v$ of $\Q$. Furthermore, Proposition \ref{norming stuff down}, applied to $B := \varnothing$, yields
$$
N_{\Q(A)/\Q}(\alpha_{i, A}) = \alpha_{i, \varnothing}.
$$
Since $v$ splits completely in $\Q(A)$ by assumption, this shows that
\begin{align}
\label{valuation relation}
\sum_{\substack{w \in \Omega_{\Q(A)} \\ w \mid v}} w(\alpha_{i, A}) \equiv v(\alpha_{i, \varnothing}) \bmod 2.
\end{align}
The corollary is now a consequence of equations (\ref{eDefRamtilde}) and (\ref{valuation relation}).
\end{proof}

We will now prove a general lemma about local fields. If $K$ is a local field, we denote by $(-, -)_K$ the Hilbert pairing on $\frac{K^{*}}{K^{*2}}$. For a place $v$ of a number field $L$ we will write, by abuse of notation, $(-, -)_v$ for the pairing $(-, -)_{L_v}$.

\begin{lemma} 
\label{Hilbert symbols at 2}
Let $p$ be a finite place of $\Q$ and let $K$ be a finite extension of $\Q_p$. Let $\alpha$ be the unique class of $\frac{K^{*}}{K^{*2}}$ corresponding to the unramified quadratic extension $K(\sqrt{\alpha})$. Then the linear functional
$$
(\alpha, -)_K: \frac{K^{*}}{K^{*2}} \rightarrow \mathbb{F}_2
$$
equals the reduction of $v_K(-)$ modulo $2$.
\end{lemma}

\begin{proof}
Let $\pi$ be a uniformizer of $K$. The local Artin map $\theta_K$ sends $\pi$ to the generator of $\text{Gal}(K(\sqrt{\alpha})/K)$, which implies that the Hilbert symbol $(\alpha, \pi)_K$ is non-trivial. Since this holds for all uniformizers, the lemma follows. 
\end{proof}

We will make use of one final ingredient. If $w$ is a real place of a number field $L$ (i.e. a place corresponding to an embedding $\sigma:L \rightarrow \mathbb{R}$), we put 
\[
w(\alpha) = 
\left\{
\begin{array}{ll}
0 & \mbox{if } \sigma(\alpha) > 0 \\
1 & \mbox{otherwise.}
\end{array}
\right.
\]
Let $v$ be a place of $\Q$. Suppose that that there exists a place $w \in \Omega_{\Q(A)}$ dividing $v$ such that $w$ ramifies in $L(\psi_1)/\Q(A)$. In that case we remark that $\text{Art}(v, L(\psi_2)/\Q)$ is a well-defined element of $\FF_2$ (see also the remarks preceding Theorem \ref{Redei reciprocity}), and similarly if we swap the roles of $1$ and $2$.

\begin{prop} 
\label{Hilbert symbols calculation}
The following statements hold.
\begin{enumerate}
\item[(a.1)] Let $v$ be a finite rational place and let $w \in \Omega_{\Q(A)}$ be a place lying above $v$. Suppose that $w$ is unramified in $L(\psi_1)\cdot L(\psi_2)/\Q(A)$. Then 
\[
(\alpha_{1, A}, \alpha_{2, A})_w = 0.
\]
\item[(a.2)] Suppose that $\infty \not \in \emph{Ram}(\Q(\chi_1)/\Q) \cup \emph{Ram}(\Q(\chi_2)/\Q)$. Then the value of $(\alpha_{1, A}, \alpha_{2, A})_w$ is the same for all places $w \in \Omega_{\Q(A)}$ lying above $\infty$.

\item[(b)] Suppose that $(2)$ ramifies in $L(\psi_1)L(\psi_2)/\Q$. Let $\{i, j\} = \{1, 2\}$ be as in the fourth point of Definition \ref{def: Redei admissible}. Then
\[
(\alpha_{1, A}, \alpha_{2, A})_w = w(\alpha_{j, A}) \cdot \emph{Art}((2), L(\psi_i)/\Q)
\]
for every $w \in \Omega_{\Q(A)}$ lying above $(2)$, where the product is taken in $\mathbb{F}_2$.
\item[(c)] Let $\{1, 2\} = \{i, j\}$. Let $v$ be a rational place not equal to $(2)$ and let $w \in \Omega_{\Q(A)}$ be a place lying above $v$. Suppose that $w$ ramifies in $L(\psi_i)/\Q(A)$. Then 
\[
(\alpha_{1, A}, \alpha_{2, A})_w = w(\alpha_{i, A}) \cdot \emph{Art}(v, L(\psi_j)/\Q),
\]
where the product is taken in $\mathbb{F}_2$.
\end{enumerate}
\end{prop}

\begin{proof}[Proof of Proposition \ref{Hilbert symbols calculation} part (a.1)]
Let $v$ and $w$ be as in the statement. Our assumptions imply that the extensions $L(\psi_1)/\Q(A)$ and $L(\psi_2)/\Q(A)$ are unramified at all places of $\Omega_{\Q(A)}$ above $v$. Now, since the fields $L(\psi_1), L(\psi_2)$ are respectively equal to the Galois closure (over $\Q$) of the quadratic extensions $\Q(A)(\sqrt{\alpha_{1, A}})/\Q(A), \Q(A)(\sqrt{\alpha_{2, A}})/\Q(A)$, it follows that the classes of $\alpha_{1, A}, \alpha_{2, A}$ are unramified classes in $\frac{\Q(A)_w^{*}}{\Q(A)_w^{*2}}$. But, thanks to Lemma \ref{Hilbert symbols at 2}, the Hilbert symbol between two unramified classes is trivial.
\end{proof}
		
\begin{proof}[Proof of Proposition \ref{Hilbert symbols calculation} part (a.2)]
By definition $\infty$ splits completely in $\Q(A)$. Since $\infty \not \in \text{Ram}(\Q(\chi_1)/\Q) \cup \text{Ram}(\Q(\chi_2)/\Q)$, we deduce from Definition \ref{def: Redei admissible} that $\infty$ splits completely in $M(\psi_1)M(\psi_2)$. 

Observe that $\alpha_{1, A}$ and $\alpha_{2, A}$ are $G_\Q$-invariants of respectively $\frac{M(\psi_1)^{*}}{M(\psi_1)^{*2}}$ and $\frac{M(\psi_2)^{*}}{M(\psi_2)^{*2}}$. Hence the conjugates of $\alpha_{i, A}$ are equal to $\alpha_{i, A}$ times a square in $M(\psi_i)$. In particular, $\sigma(\alpha_{i, A})$ equals $\sigma'(\alpha_{i, A})$ times the square of a real number for all real embeddings $\sigma, \sigma': \Q(A) \rightarrow \mathbb{R}$. Therefore $\sigma(\alpha_{i, A})$ has constantly the same sign as we vary $\sigma: \Q(A) \rightarrow \mathbb{R}$ over all real embeddings. Since the Hilbert symbol in the local field $\mathbb{R}$ is entirely determined by the sign of its two entries, the result follows.
\end{proof} 

\begin{proof}[Proof of Proposition \ref{Hilbert symbols calculation} part (b)]
The inclusion $i_2: \Q^{\text{sep}} \to \Q_2^{\text{sep}}$ provides us with a unique place $\mathfrak{t} \in \Omega_{\Q(A)}$ above $(2)$, namely the place $\mathfrak{t}$ corresponding to the absolute value obtained by composing $i_2$ with the canonical absolute value of $\Q_2^{\text{sep}}$. Recalling that $\text{Gal}(\Q(A)/\Q)$ acts transitively (and in this case freely) on the $2^n$ places of $\Omega_{\Q(A)}$ above $(2)$, it suffices to take any element $\sigma$ of $\text{Gal}(\Q(A)/\Q)$, and show the desired conclusion for
$$
(\alpha_{i, A}, \alpha_{j, A})_{\sigma(\mathfrak{t})} = (\sigma^{-1}(\alpha_{i, A}), \sigma^{-1}(\alpha_{j, A}))_{\mathfrak{t}}.
$$
We now apply Proposition \ref{norming stuff down} and rewrite the right hand side as
\begin{align}
\label{eHilbert2}
(\alpha_{i, A}\gamma_i, \alpha_{j, A}\gamma_j)_{\mathfrak{t}} = (\alpha_{i, A}, \alpha_{j, A}\gamma_j)_{\mathfrak{t}}(\gamma_i, \alpha_{j, A}\gamma_j)_{\mathfrak{t}},
\end{align}
where $\gamma_h$ belongs to $\langle \{\alpha_{h,T}\}_{T \subsetneq A} \rangle$ for $h \in \{i, j\}$.  We know that $\phi_T(\psi_i) \circ i_2^{*}$ is the trivial character, whenever $\varnothing \neq T \subsetneq A$, by the fourth condition of Definition \ref{def: Redei admissible}. It follows that $\gamma_i$ is in the span of $\alpha_{i, \varnothing}$ locally at $\mathfrak{t}$. Since $\alpha_{i, \varnothing}$ pairs trivially with all the $\alpha_{j, T}$ by assumption, equation (\ref{eHilbert2}) becomes
$$
(\alpha_{i, A}, \alpha_{j, A}\gamma_j)_{\mathfrak{t}} = (\alpha_{i, A}, \sigma^{-1}(\alpha_{j, A}))_{\mathfrak{t}}.
$$
Note that the character $\alpha_{i, A}$ is \emph{unramified} locally at $\mathfrak{t}$, since $\psi_i$ is a \emph{normalized} expansion map. Therefore we conclude, by means of Lemma \ref{Hilbert symbols at 2}, that this Hilbert symbol equals $\mathfrak{t}(\sigma^{-1}(\alpha_{j, A})) = \sigma(\mathfrak{t})(\alpha_{j, A})$ in case $\alpha_{i, A}$ is the unique non-trivial unramified quadratic class at $\mathfrak{t}$, and equals $0$ in case $\alpha_{i, A}$ is trivial. But now observe that the triviality of the unramified character $\alpha_{i, A}$ locally at $\mathfrak{t}$ precisely coincides with the triviality of the symbol $\text{Art}((2), L(\psi_i)/\Q)$. Hence the result is
$$
\sigma(\mathfrak{t})(\alpha_{j, A}) \cdot \text{Art}((2), L(\psi_i)/\Q)
$$
as desired. 
\end{proof}

\begin{proof}[Proof of Proposition \ref{Hilbert symbols calculation} part (c)]
Let us firstly suppose that $v$ is a finite place (hence in $\mathcal{P}$) and let $w \in \Omega_{\Q(A)}$ be above $v$. Observe that $L(\psi_j)/\Q(A)$ is unramified at $w$ by definition.

Furthermore, $\alpha_{j, A}$ is a $G_\Q$-invariant class in $\frac{M(\psi_j)^{*}}{M(\psi_j)^{*2}}$. It follows that $\sigma(\alpha_{j, A})$ always lands in the same unramified class $c$ of $\frac{\Q_v^{*}}{\Q_v^{*2}}$ as we vary over the $2^n$ embeddings of $\Q(A)$ into $\Q_v$. Recalling that 
\[
L(\psi_j) = M(\psi_j)\left(\sqrt{\alpha_{j, A}} \right)
\]
and that $v$ splits completely in $M(\psi_j)/\Q$ by assumption, we see that $c$ is trivial if and only if $\text{Art}(v, L(\psi_j)/\Q)$ is trivial. Lemma \ref{Hilbert symbols at 2} shows that 
\[
(\alpha_{i, A}, \alpha_{j, A})_w = 
\left\{
\begin{array}{ll}
w(\alpha_{i, A}) & \mbox{if } \text{Art}(v, L(\psi_j)/\Q) \text { is non-trivial}  \\
0 & \mbox{if } \text{Art}(v, L(\psi_j)/\Q) \text{ is trivial}.
\end{array}
\right.
\]
This proves part $(c)$ in case $v$ is finite. 
		
It remains to treat the case where $v$ equals the infinite place $\infty$ of $\Q$. Let $w \in \Omega_{\Q(A)}$ be above $v$. If $w(\alpha_{i, A}) = 0$, then the proposition is correct, because the Hilbert symbol $(-, -)_{\R}$ vanishes in case one of the two entries is positive. Finally, suppose that $w(\alpha_{i, A}) = 1$. Recalling one more time that $\alpha_{j, A}$ is a $G_\Q$-invariant class in $\frac{M(\psi_j)^{*}}{M(\psi_j)^{*2}}$, we see that the sign of $\sigma(\alpha_{j, A})$ does not depend on the embedding $\sigma: \Q(A) \rightarrow \mathbb{R}$. Furthermore, since $\infty$ splits completely in $M(\psi_j)/\Q$, it follows from $L(\psi_j) = M(\psi_j)(\sqrt{\alpha_{j, A}})$ that this sign is positive if and only if $\text{Art}(\infty, L(\psi_j)/\Q)$ is trivial.
\end{proof}

We will now prove the main result of this section.

\begin{proof}[Proof of Theorem \ref{Redei reciprocity}]
Hilbert's reciprocity law yields
\begin{align}
\label{HR}
\sum_{w \in \Omega_{\Q(A)}} (\alpha_{1, A}, \alpha_{2, A})_w = 0.
\end{align}
Suppose that $w \in \Omega_{\Q(A)}$ does not ramify in the extension $L(\psi_1) L(\psi_2)/\Q(A)$. If $w$ is a finite place, then we obtain that $(\alpha_{1, A}, \alpha_{2, A})_w = 0$ thanks to Proposition \ref{Hilbert symbols calculation} part $(a.1)$. If $w$ is an infinite place, then we certainly have
\[
\infty \not \in \widetilde{\text{Ram}}(\Q(\chi_1)/\Q) \cup \widetilde{\text{Ram}}(\Q(\chi_2)/\Q).
\]
We deduce from Proposition \ref{Hilbert symbols calculation} part $(a.2)$ that the total contribution coming from all the places of $\Omega_{\Q(A)}$ above $\infty$ equals $2^n$ times the same number, which is $0$ since $n \geq 1$.

Now suppose that $w$ ramifies in $L(\psi_1) L(\psi_2)/\Q(A)$. Let us first treat the case where $w$ ramifies in $L(\psi_1)/\Q(A)$ and write $v$ for the place of $\Q$ below $w$. Then thanks to Proposition \ref{Hilbert symbols calculation}, part $(b)$ and part $(c)$, we obtain that
$$
(\alpha_{1, A}, \alpha_{2, A})_w = w(\alpha_{1, A}) \cdot \text{Art}(v, L(\psi_2)/\Q)).
$$  
By assumption $v$ splits completely in the extension $M(\psi_2)/\Q$, so it certainly splits completely in the extension $\Q(A)/\Q$. Corollary \ref{right parity} shows that
$$
\sum_{\substack{w \in \Omega_{\Q(A)} \\ w \mid v}} (\alpha_{1, A}, \alpha_{2, A})_w = 
\left\{
\begin{array}{ll}
\text{Art}(v, L(\psi_2)/\Q)) & \mbox{if } v \in \widetilde{\text{Ram}}(\Q(\chi_1)/\Q) \\
0 & \mbox{otherwise.}
\end{array}
\right.
$$
The same conclusion holds by swapping the roles of $1$ and $2$. After grouping all $w \in \Omega_{\Q(A)}$ lying above the same rational place $v$ together, equation (\ref{HR}) becomes
$$
\sum_{v \in \widetilde{\text{Ram}}(\Q(\chi_1)/\Q)} \text{Art}(v, L(\psi_2)/\Q) + \sum_{v' \in \widetilde{\text{Ram}}(\Q(\chi_2)/\Q)} \text{Art}(v', L(\psi_1)/\Q) = 0,
$$
which can be rewritten as
$$
\sum_{v \in \widetilde{\text{Ram}}(\Q(\chi_1)/\Q)} \text{Art}(v, L(\psi_2)/\Q) = \sum_{v' \in \widetilde{\text{Ram}}(\Q(\chi_2)/\Q)} \text{Art}(v', L(\psi_1)/\Q).
$$
This completes the proof of the theorem.
\end{proof}

\section{Profitable triples}
\label{profitability}
The next two sections provide a novel way to gain class group data on one point of a cube of quadratic fields from class group data available on the other points. We will informally refer to this as a reflection principle. Since the material is of a rather technical nature, we will now provide an overview of what is going to unfold below, and its main differences with \cite[Theorem 2.8]{Smith}. We hope in this way to provide some guiding intuition for the reader. 

The key concept is that of a \emph{profitable triple}. Roughly speaking, with respect to \cite[Theorem 2.8]{Smith}, we have less control over the raw cocycles that we encounter in our reflection principles. However, the rather general form of Theorem \ref{Redei reciprocity} allows us to still obtain a reflection principle, where the relative governing fields are reasonably similar to the expansion maps $\phi_{\{p_i(1)p_i(2) : i \in [s]\}; p_{s + 1}(1)p_{s + 1}(2)}(\mathfrak{G})$ appearing in \cite[Theorem 2.8]{Smith}. 

Another novel feature is that we repeatedly invoke Theorem \ref{Redei reciprocity} during the proof of our reflection principles, where the level of generality of Definition \ref{def: Redei admissible} allows us to have remarkably little knowledge of the raw cocycles and expansion maps involved, but still gives sufficient control over the relevant splitting conditions. In contrast, \cite[Theorem 2.8]{Smith} does not require as input any form of reciprocity law. Although we shall not prove it here, our techniques are able to prove a result eerily similar to \cite[Theorem 2.8]{Smith}, except that the roles of $T(w_a)$ and $T(w_b)$ are interchanged. We remark that in the setup of \cite[Theorem 2.8]{Smith} it is not at all obvious how to achieve such a result.

Our approach might seem surprising at first, since in the reflection principles of Smith \cite{Smith} it is customary to find exact relations among raw cocycles and expansion maps and only in a later step physically plug an Artin symbol in such relations. In contrast, in our setup of profitable triples, we have remarkably loose control on the relations between the raw cocycles and expansion maps. It is a crucial appeal to the reciprocity law of Theorem \ref{Redei reciprocity} that gives the reflection principle. Most of the constraints we put on our raw cocycles are merely to meet the conditions of Theorem \ref{Redei reciprocity} to allow such a reciprocity law to come into play.

Let $s \in \mathbb{Z}_{\geq 1}$ and let
$$
C := \{p_1(1), p_1(2)\} \times \dots \times \{p_s(1), p_s(2)\} \times \{d_0\},
$$
where $(p_i(1), p_i(2))_{i \in [s]}$ are $2s$ distinct, positive, odd prime numbers satisfying $p_i(1)p_i(2)  \equiv 1 \bmod 8$ for each $i \in [s]$, and $\prod_{i \in [s]} p_i(1)p_i(2)$ is coprime to the positive squarefree integer $d_0$. Thanks to this assumption, we can identify each element $((p_i(f(i)))_{i \in [s]}, d_0)$ of $C$ with the squarefree integer $d_0 \cdot \prod_{i = 1}^s p_i(f(i))$, where $f$ is any function from $[s]$ to $[2]$. We denote by
$$
x_0 := d_0 \cdot \prod_{i = 1}^s p_i(1).
$$
For each $i \in [s]$ and each $h \in [2]$ we denote by $C_{p_i(h)}$ the subset of $x \in C$ consisting of those elements satisfying $\pi_i(x) = p_i(h)$. More generally, for each $T \subseteq [s]$ and for each function $f$ from $T$ to $\{1,2\}$, we denote by $C_{(p_i(f(i)))_{i \in T}}$ the subset of $C$ consisting of those $x$ such that $\pi_i(x) = p_i(f(i))$ for each $i \in T$. This notation will be often invoked for the function $f$ constantly equal to $2$, for which we give the following special notation in order to lighten up our formulas. We denote by $C_T$ the subset of $x \in C$ such that $i \in T$ implies $\pi_i(x) = p_i(2)$. 

Let $a$ be a divisor of $d_0$. Suppose that we are given a tuple
$$
(\psi_{s + 1}(x))_{x \in C - \{x_0\}},
$$
where $\psi_{s + 1}(x)$ is a raw cocycle for $N(x)$ for each $x \in C - \{x_0\}$ satisfying
\begin{align}
\label{eLiftchia}
\psi_1(x) = 2^{s} \cdot \psi_{s + 1}(x) = \chi_a. 
\end{align}

\begin{remark} 
\label{a must be positive}
Observe that equation (\ref{eLiftchia}) implies that $a$ is \emph{positive}. Indeed, since $s + 1 \geq 2$, equation (\ref{eLiftchia}) implies that $\chi_a \cup \chi_{-x}$ vanishes in $H^2(G_\Q, \mathbb{F}_2)$ for each $x$ in $C - \{x_0\}$. Since $x$ is positive, taking the Hilbert symbol at the unique infinite place of $\Q$ forces $a > 0$. 
\end{remark}

We next make a crucial definition. 

\begin{mydef} 
\label{def: profitable triple}
Let $(C, \chi_a, (\psi_{s + 1}(x))_{x \in C - \{x_0\}})$ be a triple as above. We call the triple $(C, \chi_a, (\psi_{s + 1}(x))_{x \in C - \{x_0\}})$ \emph{profitable} if for each $i \in [s]$
\begin{itemize}
\item we have
$$
\left(\sum_{x \in C_{T}} \psi_{s + 1-|T|}(x) \right)(\sigma_{p_i(2)}) = 0
$$
for each $T \subseteq [s]$ containing $i$;
\item we have
\begin{align}
\label{eLowerMinimality}
\sum_{x \in C_{p_i(2)}} \psi_{s - 1}(x) = 0;
\end{align}
\item the map $\phi_{\{p_j(1)p_j(2) : j \in [s] - \{i\}\}; p_i(1)p_i(2)}(\mathfrak{G})$ exists;
\item every prime $l$ dividing $d_0$ splits completely in
$$
L(\phi_{\{p_j(1)p_j(2) : j \in [s] - \{i\}\}; p_i(1)p_i(2)}(\mathfrak{G}))/\Q;
$$
\item $(2)$ and $\infty$ split completely in $M(\phi_{\{p_j(1)p_j(2) : j \in [s] - \{i\}\}; p_i(1)p_i(2)}(\mathfrak{G}))/\Q$;
\end{itemize}
\end{mydef}

\begin{remark}
\label{higher codimension}
We claim that the second condition in Definition \ref{def: profitable triple} implies that 
\[
\sum_{x \in C_{T}} \psi_{s - |T|}(x) = 0
\]
for each $\varnothing \neq T \subseteq [s]$. To see this, fix some $j \in [s]$ and apply the operator $d_{x_0}$ to equation (\ref{eLowerMinimality}). Then we get
\[
\sum_{\varnothing \neq T \subseteq [s] - \{i\}} \chi_{\{p_k(1) p_k(2): k \in T\}}(\sigma) \cdot (-1)^{|T| + 1 + \chi_{x_0}(\sigma)} \cdot \left(\sum_{x \in C_T} \psi_{s  - |T|}(x)(\tau)\right)
\]
by Proposition \ref{key calculation of cocycles}. Now evaluate the resulting expression at $(\sigma, \tau)$ satisfying $\chi_{p_j(1)p_j(2)}(\sigma) = 1$ and $\chi_{p_k(1)p_k(2)}(\sigma) = 0$ for all $k \in [s] - \{i, j\}$. This gives
$$
\sum_{x \in C_{\{i, j\}}} \psi_{s  -2}(x) = 0.
$$
Iterating this argument we obtain the desired conclusion.
\end{remark}

We will often make use of the following important observation, which, informally put, says that the various quadratic fields, corresponding to the points of $C$, coincide locally at $(2)$ and at the common ramified primes in the cube.

\begin{prop}
\label{profitable cubes are quad consistent} 
Let $(C, \chi_a, (\psi_{s + 1}(x))_{x \in C - \{x_0\}})$ be a profitable triple. Let $p$ be a prime divisor of $d_0$ or let $p = 2$. Then we have
$$
\Q_p(\sqrt{x}) = \Q_p(\sqrt{x_0})
$$
for each $x \in C$. Moreover, we have for each $i \in [s]$, each $h \in [2]$ and all $x, x' \in C_{p_i(h)}$
$$
\Q_{p_i(h)}(\sqrt{x}) = \Q_{p_i(h)}(\sqrt{x'}).
$$
\end{prop}

\begin{proof}
The case $p = 2$ readily follows from the requirement that $p_i(1)p_i(2) \equiv  1 \bmod 8$ for each $i \in [s]$. Every odd prime divisor $p$ of $d_0$ splits completely in 
\[
L(\phi_{\{p_j(1)p_j(2) : j \in [s] - \{i\}\}; p_i(1)p_i(2)}(\mathfrak{G}))/\Q
\]
by assumption. This certainly implies that $p$ splits completely in $\Q(\sqrt{p_i(1)p_i(2)})$ for each $i \in [s]$, which immediately gives us the desired conclusion. 

Finally, take distinct indices $i, j \in [s]$. Since $s + 1 \geq 2$, the map $\phi_{p_j(1)p_j(2); p_i(1)p_i(2)}(\mathfrak{G})$ exists. This map satisfies
$$
(d\phi_{p_j(1)p_j(2);p_i(1)p_i(2)}(\mathfrak{G}))(\sigma, \tau) = \chi_{p_j(1)p_j(2)}(\sigma) \chi_{p_i(1)p_i(2)}(\tau).
$$
Hence the right hand side is trivial in $H^2(G_\Q, \mathbb{F}_2)$. Therefore it is trivial in $H^2(G_{\Q_{p_i(h)}}, \mathbb{F}_2)$ for all $h \in [2]$. This means that $p_j(1)p_j(2)$ is a square modulo $p_i(h)$, which proves the last assertion. 
\end{proof}

Before we present the main theorem on profitable triples, we will state a useful lemma. When we apply Lemma \ref{lX} later, Proposition \ref{profitable cubes are quad consistent} will guarantee that the decomposition groups of the field $\Q(\{\sqrt{x} : x \in S\})$ are indeed cyclic. 

\begin{lemma}
\label{lX}
Let $S$ be a set of squarefree integers all greater than $1$. Assume that the decomposition group of $(2)$ in the field $K = \Q(\{\sqrt{x} : x \in S\})$ is cyclic. Let $k \geq 1$ be an integer and suppose that we are given raw cocycles $\psi_{k + 1}(x)$ for $N(x)$ for each $x \in S$. Then
\begin{itemize}
\item[(i)] if $p$ does not ramify in $K$, $p$ is unramified in
\[
\prod_{x \in S} L(\psi_{k + 1}(x));
\]
\item[(ii)] suppose that $p$ ramifies in $K$. Then the ramification index of $p$ in
\[
\prod_{x \in S} L(\psi_{k + 1}(x))
\]
equals $2$. Now assume that all decomposition groups of $K/\Q$ are cyclic. Then, if $p$ ramifies in $\Q(\sqrt{x})$ for every $x \in S$, the residue field degree of $p$ in
\[
\prod_{x \in S} L(\psi_k(x))
\]
equals $1$.
\end{itemize}
\end{lemma}

\begin{proof}
Straightforward.
\end{proof}

Let a profitable triple $(C, \chi_a, (\psi_{s + 1}(x))_{x \in C - \{x_0\}})$ and a subset $\varnothing \subseteq T \subsetneq [s]$ be given. To this data we associate the map
\[
\psi_T(C, \chi_a, (\psi_{s + 1}(x))_{x \in C - \{x_0\}}) := \sum_{x \in C_{[s] - T}} \psi_{|T| + 1}(x).
\]
Note that the above definition also makes sense for $T = [s]$ in case we are further given a raw cocycle $\psi_{s + 1}(x_0)$. 

\begin{theorem} 
\label{main theorem on profitable triples}
Let $(C, \chi_a, (\psi_{s + 1}(x))_{x \in C - \{x_0\}})$ be a profitable triple. Then there exists a raw cocycle $\psi_{s + 1}(x_0)$ for $N(x_0)$ such that the following two properties hold
\begin{itemize} 
\item for each $i \in [s]$ we have that
\begin{align}
\label{ePointGainLower}
\sum_{x \in C} \psi_i(x) = 0;
\end{align}
\item the tuple
$$
(\psi_T(C, \chi_a, (\psi_{s + 1}(x))_{x \in C - \{x_0\}}))_{T \subseteq [s]}
$$
is an expansion map $\psi(C, \chi_a, (\psi_{s + 1}(x))_{x \in C - \{x_0\}})$ with support $\{\chi_{p_i(1)p_i(2)} : i \in [s]\} \cup \{\chi_a\}$ and pointer $\chi_a$. Furthermore, if $v \in \mathcal{P}$ ramifies in
$$
L(\psi(C, \chi_a, (\psi_{s + 1}(x))_{x \in C - \{x_0\}}))/\Q(\{\sqrt{p_i(1)p_i(2)} : i \in [s]\}),
$$
then $v$ divides $d_0$. The ramification index of $v \in \Omega_{\Q}$ is at most $2$ in the extension $L(\psi(C, \chi_a, (\psi_{s + 1}(x))_{x \in C - \{x_0\}}))/\Q$. Finally, the $1$-cochain 
$$
\phi_T(\psi(C, \chi_a, (\psi_{s + 1}(x))_{x \in C - \{x_0\}})) \circ i_2^{*} = \psi_T(C, \chi_a, (\psi_{s + 1}(x))_{x \in C - \{x_0\}}) \circ i_2^{*}
$$
is a quadratic character contained in the span of $\{\chi_5, \chi_{x_0}\}$ for each subset $T \subseteq [s]$. If $(2)$ ramifies in $\Q(\sqrt{x_0})/\Q$, then $\phi_T(\psi(C, \chi_a, (\psi_{s + 1}(x))_{x \in C - \{x_0\}})) \circ i_2^{*}$ is a quadratic character contained in the span of $\{\chi_{x_0}\}$ for each $T \subsetneq [s]$.
\end{itemize}
\end{theorem}

\begin{proof}
We claim that
\begin{align}
\label{ePointGain}
\psi_s(x_0) := -\sum_{\substack{x \in C \\ x \neq x_0}} \psi_s(x)
\end{align}
is a raw cocycle for $N(x_0)$ lifting $\chi_a$. Since $2^s - 1$ is odd, we certainly have 
\[
2^{s - 1} \psi_s(x_0) = -\sum_{\substack{x \in C \\ x \neq x_0}} 2^{s - 1} \psi_s(x) = -\sum_{\substack{x \in C \\ x \neq x_0}} \chi_a = \chi_a. 
\]
Furthermore, it follows from Proposition \ref{key calculation of cocycles} that
$$
d_{x_0}(\psi_s(x_0))(\sigma, \tau) = \sum_{\varnothing \neq T \subseteq [s]} \chi_{\{p_i(1)p_i(2) : i \in T\}}(\sigma) \cdot \left(\sum_{x \in C_T} \psi_{s - |T|}(x)(\tau)\right),
$$
which is zero term by term thanks to equation (\ref{eLowerMinimality}) of Definition \ref{def: profitable triple} and Remark \ref{higher codimension}. This, together with Proposition \ref{cocycles and semi direct products}, gives a field extension corresponding to the homomorphism
$$
G_\Q \to N(x_0)[2^s] \rtimes \mathbb{F}_2, \quad \sigma \mapsto \left(\psi_s(x_0)(\sigma), \chi_{x_0}(\sigma)\right),
$$
where the generator of $\mathbb{F}_2$ acts by $-\text{id}$ on $N(x_0)$. Our next goal is to show that the extension $L(\psi_s(x_0))\Q(\sqrt{x_0})/\Q(\sqrt{x_0})$ is unramified at all finite places. By equation (\ref{ePointGain}), we have an inclusion
\begin{align}
\label{eFieldOfDef}
L(\psi_s(x_0)) \subseteq \prod_{\substack{x \in C \\ x \neq x_0}} L(\psi_s(x)).
\end{align}
Then Lemma \ref{lX} implies that $L(\psi_s(x_0))\Q(\sqrt{x_0})/\Q(\sqrt{x_0})$ is unramified at all finite places not lying above some $p_i(2)$.

It remains to show that $L(\psi_s(x_0))\Q(\sqrt{x_0})/\Q(\sqrt{x_0})$ is unramified for any place of $\Q(\sqrt{x_0})$ above $p_i(2)$. Recall from Proposition \ref{cocycles and semi direct products} that the Galois extension $L(\psi_s(x))\Q(\sqrt{x})/\Q$ is given by the homomorphism from $G_\Q$ to $N(x)[2^s] \rtimes \mathbb{F}_2$ defined by the formula
$$
\sigma \mapsto (\psi_s(x)(\sigma), \chi_x(\sigma)).
$$
Notice that $p_i(2)$ is unramified in the extension $L(\psi_s(x))/\Q$ for each $x \in C_{p_i(1)} - \{x_0\}$. Then it follows from Proposition \ref{ramification read off by inertia} that $\psi_s(x)(\sigma_{p_i(2)}) = 0$ for all $x \in C_{p_i(1)} - \{x_0\}$. Hence, the first condition of Definition \ref{def: profitable triple} together with the definition of $\psi_s(x_0)$ implies that
$$
(\psi_s(x_0)(\sigma_{p_i(2)}), \chi_{x_0}(\sigma_{p_i(2)})) = \text{id}. 
$$
Invoking one more time Proposition \ref{cocycles and semi direct products} and Proposition \ref{ramification read off by inertia}, we conclude that $p_i(2)$ is unramified in $L(\psi_s(x_0))/\Q$ as desired. Since we have shown that $L(\psi_s(x_0))/\Q(\sqrt{x_0})$ is unramified at all finite places of $\Q(\sqrt{x_0})$, we conclude that $\psi_s(x_0)$ is a raw cocycle for $N(x_0)$ lifting $\chi_a$, which satisfies equation (\ref{ePointGainLower}) by construction.

In order to prove the existence of $\psi_{s + 1}(x_0)$ with $2\psi_{s + 1}(x_0) = \psi_s(x_0)$ we will instead show the equivalent statement that $\text{Up}_{\Q(\sqrt{x_0})/\Q}(p)$ splits completely in $L(\psi_s(x_0))\Q(\sqrt{x_0})/\Q(\sqrt{x_0})$ for each prime $p$ ramifying in $\Q(\sqrt{x_0})/\Q$ (here we use the material in Subsection \ref{ssArtinpairing} and Proposition \ref{unramified characters are always cocycles}). This is in turn equivalent to
\begin{align}
\label{eSplitLower}
i_p(L(\psi_s(x_0))) \subseteq \Q_p(\sqrt{x_0})
\end{align}
for each prime $p$ ramifying in $\Q(\sqrt{x_0})/\Q$.

Let us begin with the case that $p$ divides $d_0$ (or $p=2$ in case $x_0$ is $3$ modulo $4$). For all such choices of $p$, equation (\ref{eSplitLower}) is a consequence of Lemma \ref{lX} and equation (\ref{eFieldOfDef}). We are now left with proving the sought inclusion for $p_i(1)$ for each $i \in [s]$. First of all we claim that $i_{p_i(1)}(L(\psi_s(x))) \subseteq \Q_{p_i(1)}(\sqrt{x_0})$ for each $x \in C_{p_{i}(1)} - \{x_0\}$. Indeed, the Artin symbol of $\text{Up}_{\Q(\sqrt{x})/\Q}(p_i(1))$ in $\text{Gal}(L(\psi_s(x))\Q(\sqrt{x})/\Q(\sqrt{x}))$ vanishes thanks to the equation $2 \cdot \psi_{s + 1}(x) = \psi_s(x)$ and the claim is a consequence of Proposition \ref{profitable cubes are quad consistent}. Therefore noticing that
$$
L(\psi_s(x_0)) \subseteq \left(\prod_{x \in C_{p_i(1)} - \{x_0\}} L(\psi_s(x))\right) L\left(\sum_{x \in C_{p_i(2)}} \psi_s(x)\right),
$$
we see that it is enough to show that $i_{p_i(1)}(L(\sum_{x \in C_{p_i(2)}}\psi_s(x))) \subseteq \Q_{p_i(1)}$. Hence it suffices to show that $p_i(1)$ splits completely in the extension
$$
L\left(\sum_{x \in C_{p_i(2)}} \psi_s(x)\right)/\Q.
$$
Define $\widetilde{C} := C_{p_i(2)}$. We start by observing that equation (\ref{eLowerMinimality}) combined with Remark \ref{higher codimension} and Proposition \ref{key calculation of cocycles} implies that 
\[
\left(\sum_{x \in \widetilde{C}_{[s] - \{i\} - T}} \psi_{|T| + 1}(x)\right)_{T \subseteq [s] - \{i\}}
\]
is an expansion map, say $E$, with support set equal to $\{\chi_{p_j(1)p_j(2)} : j \in [s] - \{i\}\} \cup \{\chi_a\}$ and pointer $\chi_a$. 

We claim that $p_i(2)$ splits completely in $L(E)$. Let us first show that $p_i(2)$ is unramified in $L(E)$. We obtain, as an immediate consequence of equation (\ref{eReconstruct}) and the first point of Definition \ref{def: profitable triple}, that $\sigma_{p_i(2)}$ is mapped to the identity element via $E$. Then $p_i(2)$ is indeed unramified by Proposition \ref{ramification read off by inertia}. Since $L(E)$ is contained in $\prod_{x \in C_{p_i(2)}} L(\psi_s(x))$, it follows from Lemma \ref{lX} that $p_i(2)$ has residue field degree $1$ in $L(E)$ implying the claim.

We next claim that
\begin{align}
\label{eCubeSwitch}
i_{p_i(1)}\left(M\left(\sum_{x \in \widetilde{C}} \psi_s(x)\right)\right) \subseteq \Q_{p_i(1)} \quad \text{and}  \quad i_{p_i(1)}\left(L\left(\sum_{x \in \widetilde{C}} \psi_s(x)\right)\right) \subseteq \Q_{p_i(1)}^{\text{unr}}.
\end{align}
The second containment in equation (\ref{eCubeSwitch}) is clear, since $p_i(1)$ is unramified in $L(\psi_s(x))$ for every $x \in \widetilde{C}$. Thanks to Proposition \ref{key calculation of cocycles} the first condition is equivalent to $p_i(1)$ splitting completely in
\[
L\left(\sum_{x \in C_T} \psi_{s - 1}(x)\right)
\]
for every $T \subseteq [s]$ of size $2$ containing $i$. Let us first observe that the above field is contained in the compositum of $L(\psi_{s}(x))$ for $x$ satisfying $\pi_i(x) = p_i(2)$. Hence 
\begin{align}
\label{epi1unr}
i_{p_i(1)}\left(L\left(\sum_{x \in C_T} \psi_{s - 1}(x)\right)\right) \subseteq \Q_{p_i(1)}^{\text{unr}}.
\end{align}
Next observe that
$$
\sum_{x \in C_T} \psi_{s - 1}(x) = -\sum_{x \in C_{(p_i(1), p_j(2))}} \psi_{s - 1}(x)
$$
thanks to equation (\ref{eLowerMinimality}), where $T = \{i, j\}$. It follows from Lemma \ref{lX} that $p_i(1)$ has residue field degree $1$ in $L(\sum_{x \in C_{(p_i(1), p_j(2))}} \psi_{s - 1}(x))$. Therefore we have
\begin{align}
\label{epi1ram}
i_{p_i(1)}\left(L\left(\sum_{x \in C_T} \psi_{s - 1}(x)\right)\right) \subseteq \Q_{p_i(1)}(\sqrt{x_0}).
\end{align}
We derive from equations (\ref{epi1unr}) and (\ref{epi1ram}) that
$$
i_{p_i(1)}\left(L\left(\sum_{x \in C_T} \psi_{s - 1}(x)\right)\right) \subseteq \Q_{p_i(1)}(\sqrt{x_0}) \cap \Q_{p_i(1)}^{\text{unr}} = \Q_{p_i(1)}.
$$
This establishes equation (\ref{eCubeSwitch}).

We are now in position to show that the $4$-tuple
\begin{align*}
(\{\chi_{p_j(1)p_j(2)} : j \in [s] - \{i\}\} \cup& \{\chi_a\}, \{\chi_{p_j(1)p_j(2)} : j \in [s] - \{i\}\} \cup \{\chi_{p_i(1)p_i(2)}\}, \\
\left(\sum_{x \in \widetilde{C}_{[s] - \{i\} - T}} \psi_{|T| + 1}(x)\right)&_{T \subseteq [s] - \{i\}}, (\phi_{\{p_j(1)p_j(2) : j \in T\}; p_i(1)p_i(2)}(\mathfrak{G}))_{T \subseteq [s] - \{i\}})
\end{align*}
is R\'edei admissible. From this we will derive, by an application of Theorem \ref{Redei reciprocity}, the desired splitting of $p_i(1)$. Let us start by checking the coprimality conditions among the pointers. Observe that $a$ and $p_i(1)p_i(2)$ do not share any common prime divisor by construction. Furthermore, $a$ is positive, see Remark \ref{a must be positive}, and $p_i(1)p_i(2)$ is obviously positive. Finally, we have $p_i(1)p_i(2) \equiv 1 \bmod 8$ by assumption, hence the desired orthogonality at $2$ between $\chi_{a}$ and $\chi_{p_i(1)p_i(2)}$, with respect to the Hilbert pairing, is evidently satisfied.

We now check R\'edei admissability. To lighten the notational burden, we denote, as in Section \ref{sRed}, by $\psi_1$ the first expansion map and by $\psi_2$ the second expansion map appearing in the $4$-tuple above. 

Let us start by examining the infinite place. It follows from Remark \ref{a must be positive} that $\infty$ splits completely in the field $\Q(\{\sqrt{p_k(1)p_k(2)} : k \in [s]\}, \sqrt{a})$. Hence we need to check that $\infty$ splits completely in $M(\psi_1)M(\psi_2)$. The field $L(\psi_s(x))$ is totally real for each $x \in \widetilde{C}$ thanks to the equation $2 \cdot \psi_{s + 1}(x) = \psi_s(x)$. Therefore their compositum is totally real and thus we conclude that $M(\psi_1)$ is totally real. Regarding $M(\psi_2)$ this is explicitly prescribed in the fifth point of Definition \ref{def: profitable triple}.

Let us now consider the place $(2)$. It follows from the fifth condition of Definition \ref{def: profitable triple} that $(2)$ splits completely in the extension $L(\psi_2)/\Q$. Therefore, in view of Definition \ref{def: Redei admissible}, there is no further condition to check for $\psi_1$. 

It remains to examine the odd places dividing $d_0$ or places of the form $p_j(h)$ for some $j \in [s]$ and $h \in [2]$. Indeed, $L(\psi_1)/\Q$ does not ramify at any other odd place by Lemma \ref{lX}, while Proposition \ref{normalized expansions are unramified} implies that $L(\psi_2)/\Q$ ramifies only at primes of the form $p_j(h)$ for $j \in [s]$ and $h \in [2]$. So let us take an odd prime $p$ dividing $d_0$. Thanks to the fourth condition of Definition \ref{def: profitable triple} $p$ splits completely in $L(\psi_2)/\Q$.

We now examine the place $p_j(h)$ with $j$ different from $i$. Then we claim that each place above $p_j(h)$ is unramified in 
\[
L(\psi_1)L(\psi_2)/\Q(\{\sqrt{p_k(1)p_k(2)} : k \in [s] - \{i\}\}). 
\]
In fact we make the stronger claim that $\sigma_{p_j(h)}$ has order at most $2$ when projected to the Galois groups $\text{Gal}(L(\psi_1)/\Q)$ and $\text{Gal}(L(\psi_2)/\Q)$. This implies the original claim, since the ramification index of $p_j(h)$ in $\Q(\{\sqrt{p_k(1)p_k(2)} : k \in [s] - \{i\}\})/\Q$ already equals $2$, since $j$ is different from $i$. For $\psi_1$, observe that this holds for each individual Galois group $\text{Gal}(L(\psi_s(x))/\Q)$, since the extension $L(\psi_s(x))/\Q(\sqrt{x})$ is unramified for each $x \in \tilde{C}$. For $\psi_2$, observe that normalized expansion maps send every $\sigma \in \mathfrak{G}$ to an involution by construction.

We are now left with the places $p_i(1)$ and $p_i(2)$. But we have already shown above that $L(\psi_1)/\Q$ is unramified at both of them, and even that $p_i(2)$ splits completely therein. Furthermore, we have established, see equation (\ref{eCubeSwitch}), that $p_i(1)$ splits completely in $M(\psi_1)/\Q$. Therefore the $4$-tuple is R\'edei admissible. We now apply Theorem \ref{Redei reciprocity}. The relevant Artin symbols in $L(\psi_2)$ are zero thanks to the fourth point of Definition \ref{def: profitable triple}. We have already shown that $\text{Art}(p_i(2), L(\psi_1)/\Q)$ vanishes. Hence Theorem \ref{Redei reciprocity} yields that $p_i(1)$ splits completely in the field of definition of $\psi_1$. 

Since $i \in [s]$ was arbitrary, we have completed the proof that $\psi_s(x_0)$ pairs trivially with $\text{Cl}(\Q(\sqrt{x_0}))[2]$. Therefore there exists $\psi_{s + 1}(x_0)$ with
$$
2 \cdot \left(\sum_{x \in C} \psi_{s + 1}(x) \right) = 0.
$$
Invoking one more time Proposition \ref{key calculation of cocycles} combined with equation (\ref{eLowerMinimality}) and Remark \ref{higher codimension}, we obtain that 
\[
\psi(C, \chi_a, (\psi_{s + 1}(x))_{x \in C - \{x_0\}})
\]
is an expansion map with support $\{\chi_{p_i(1)p_i(2)} : i \in [s]\} \cup \{\chi_a\}$ and pointer $\chi_a$. 

Let us now prove the final claims concerning the ramification locus of the extension $L(\psi(C, \chi_a, (\psi_{s + 1}(x))_{x \in C - \{x_0\}}))/\Q$. First of all, recall that $(2)$ splits completely in the field $\Q(\{\sqrt{p_i(1)p_i(2)} : i \in [s]\})/\Q$, because of our assumption $p_i(1)p_i(2) \equiv 1 \bmod 8$ for each $i \in [s]$. In view of Proposition \ref{key calculation of cocycles} this also shows that the map $\phi_T(\psi(C, \chi_a, (\psi_{s + 1}(x))_{x \in C - \{x_0\}})) \circ i_2^\ast$ is a character from $G_{\Q_2}$ to $\mathbb{F}_2$. The other claims are now a consequence of Lemma \ref{lX}.
\end{proof}

From now on we will refer to the maps $\psi_{s + 1}(x_0)$ and $\psi(C, \chi_a, (\psi_{s + 1}(x))_{x \in C - \{x_0\}})$ obtained in Theorem \ref{main theorem on profitable triples} as, respectively, the raw cocycle and the expansion map \emph{attached} to the profitable triple $(C, \chi_a, (\psi_{s + 1}(x))_{x \in C - \{x_0\}})$. 

\section{Reflection principles for profitable triples}
\label{sReflection}
We divide our reflection principles for profitable triples in three families. The final family of reflection principles is taken from Smith \cite{Smith}, and does not involve the notion of profitable triples.

\subsection{\texorpdfstring{Reflection principles for $\infty$}{Reflection principles for infinity}}
We say that a squarefree integer $d > 1$ is \emph{special} in case it has no prime divisors congruent to $3$ modulo $4$. From now on we shall use the notation for $s$ and $C$ as given at the beginning of Section \ref{profitability}. We say that $C$ is special in case all its elements are special. Let us start by defining $-1$-profitable triples. If $\phi_{\{a_i : i \in [s]\}; a_{s + 1}}(\mathfrak{G})$ exists, then recall that
\[
M(\phi_{\{a_i : i \in [s]\}; a_{s + 1}}(\mathfrak{G})) = \Q(\{\chi_{a_i} : i \in [s]\}) \prod_{T \subsetneq [s]} L(\phi_{\{a_i : i \in T\}; a_{s + 1}}(\mathfrak{G})).
\]

\begin{mydef} 
\label{def:-1-profitable}
Let $(C, \chi_a, (\psi_{s + 1}(x))_{x \in C - \{x_0\}})$ be a special, profitable triple. We say that $(C, \chi_a, (\psi_{s + 1}(x))_{x \in C - \{x_0\}})$ is a $-1$-profitable triple in case
\begin{itemize}
\item the map $\phi_{\{p_i(1)p_i(2) : i \in [s]\}; -1}(\mathfrak{G})$ exists;
\item every odd prime divisor of $d_0$ splits completely in $M(\phi_{\{p_i(1)p_i(2) : i \in [s]\}; -1}(\mathfrak{G}))/\Q$;
\item $(\sqrt{x}) \in 2^s\textup{Cl}(\mathbb{Q}(\sqrt{x}))[2^{s + 1}]$ for each $x \in C$; 
\item the place $(1+i)$ of $\Q(i)$ splits completely in $M(\phi_{\{p_i(1)p_i(2) : i \in [s]\}; -1}(\mathfrak{G}))/\Q(i)$.
\end{itemize}
\end{mydef}

\noindent Before stating out next theorem, we remind the reader that the splitting of $(\sqrt{x})$ in an unramified extension (at all finite places) of $\Q(\sqrt{x})$ is the same as the splitting behavior of an infinite place. Furthermore, we remind the reader of the conventions about $\text{Art}(2,L(\psi_i)/\Q)$ discussed after Definition \ref{def: Redei admissible}. 

\begin{theorem} 
\label{thm: reflection principle for infinity}
Let $(C, \chi_a, (\psi_{s + 1}(x))_{x \in C - \{x_0\}})$ be a $-1$-profitable triple. Then 
\[
\chi_a \in  2^s\textup{Cl}(\mathbb{Q}(\sqrt{x_0}))^{\vee}[2^{s + 1}]. 
\]
Furthermore, each odd prime number $p \mid a$ splits completely in $M(\phi_{\{p_i(1)p_i(2) : i \in [s]\}; -1}(\mathfrak{G}))/\Q$, the place $(1 + i)$ splits completely in $M(\phi_{\{p_i(1)p_i(2) : i \in [s]\}; -1}(\mathfrak{G}))/\Q(i)$  and
$$
\sum_{x \in C} \langle (\sqrt{x}), \chi_a \rangle_{\textup{Art}_{s + 1}(x)} = \sum_{p \mid a} \phi_{\{p_i(1)p_i(2) : i \in [s]\}; -1}(\mathfrak{G})(\textup{Frob}(p)).
$$
\end{theorem}

\begin{proof}
Since $(C, \chi_a, (\psi_{s + 1}(x))_{x \in C - \{x_0\}})$ is in particular a profitable triple, an application of Theorem \ref{main theorem on profitable triples} yields
\[
\chi_a \in  2^s\textup{Cl}(\mathbb{Q}(\sqrt{x_0}))^{\vee}[2^{s + 1}]. 
\]
It is also clear that $(1 + i)$ splits completely in $M(\phi_{\{p_i(1)p_i(2) : i \in [s]\}; -1}(\mathfrak{G}))/\Q(i)$ and that every odd prime $p$ dividing $a$ splits completely in $M(\phi_{\{p_i(1)p_i(2) : i \in [s]\}; -1}(\mathfrak{G}))/\Q$. Indeed, this is literally the second and fourth condition in Definition \ref{def:-1-profitable}.

It remains to prove the last part of the theorem. Pick a raw cocycle $\psi_{s + 1}(x_0)$ as in Theorem \ref{main theorem on profitable triples}. Let $\psi(C, \chi_a, (\psi_{s + 1}(x))_{x \in C - \{x_0\}})$ be the expansion map attached to the profitable triple. Denote by $\text{Frob}(\infty)$ the generator of $\text{Gal}(\Q _{\infty}^{\text{sep}}/\Q_{\infty})$. We claim that
\begin{align}
\label{eInfinityReflection}
\sum_{x \in C} \langle (\sqrt{x}), \chi_a \rangle_{\textup{Art}_{s + 1}(x)} = \psi_{[s]}(C, \chi_a, (\psi_{s + 1}(x))_{x \in C - \{x_0\}})(i_\infty^\ast(\text{Frob}(\infty))).
\end{align}
Indeed, for each $x \in C$, we can write
$$
\langle (\sqrt{x}), \chi_a \rangle_{\textup{Art}_{s + 1}(x)} = \psi_{s + 1}(x)(i_\infty^\ast(\text{Frob}(\infty))), 
$$
and hence equation (\ref{eInfinityReflection}) follows from the equality
$$
\sum_{x \in C} \psi_{s + 1}(x) = \psi_{[s]}(C, \chi_a, (\psi_{s + 1}(x))_{x \in C - \{x_0\}}).
$$
We now show that the $4$-tuple
\begin{align*}
&(\{\chi_{p_i(1)p_i(2)} : i \in [s]\} \cup \{\chi_a\}, \{\chi_{p_i(1)p_i(2)} : i \in [s]\} \cup \{\chi_{-1}\}, \\
\psi&(C, \chi_a, (\psi_{s + 1}(x))_{x \in C - \{x_0\}}), (\phi_{\{p_i(1)p_i(2) : i \in T\}; -1}(\mathfrak{G}))_{T \subseteq [s]})
\end{align*}
is R\'edei admissible. Once we have completed this task, then the desired conclusion follows at once from Theorem \ref{Redei reciprocity} combined with equation (\ref{eInfinityReflection}). 

Let us start by examining the coprimality conditions. The coprimality condition at the odd places is satisfied since one of the two pointers is $\chi_{-1}$. At $\infty$ we need to guarantee that $a$ is positive, which follows from Remark \ref{a must be positive}. At $(2)$ we need to guarantee that $\chi_{-1}$ and $\chi_a$ are orthogonal with respect to the local Hilbert pairing. Since $a$ is special, $\chi_a$ is contained in the span of $\{\chi_2, \chi_5\}$ locally at $2$. This last space is precisely the orthogonal complement of $\langle \chi_{-1} \rangle$ with respect to the local Hilbert pairing at $2$.

It remains to check the conditions in Definition \ref{def: Redei admissible}. In what follows we will denote, as in Section \ref{sRed}, by $\psi_1$ the first expansion map and by $\psi_2$ the second expansion map of the $4$-tuple. The first condition of Definition \ref{def: Redei admissible} is trivially satisfied. Let us now check the second condition. 

Observe that $\infty$ certainly ramifies in $\Q(i)/\Q$. Therefore we need to check that $\infty$ splits completely in $M(\psi_1)$. But the extension $L(\psi_s(x))/\Q$ is totally real for each $x \in C$ due to the equation $2\cdot \psi_{s + 1}(x) = \psi_s(x)$. Hence the desired conclusion follows immediately, since $M(\psi_1)$ is contained in the compositum of totally real fields. We now check the third condition of Definition \ref{def: Redei admissible}. Thanks to Theorem \ref{main theorem on profitable triples} and Proposition \ref{normalized expansions are unramified}, we only need to check the odd primes dividing $d_0$ and those of the form $p_i(h)$ for some $i \in [s]$ and some $h \in [2]$. 

Let $p$ be an odd prime dividing $d_0$. Then $p$ splits completely in $M(\psi_2)/\Q$ by the second condition of Definition \ref{def:-1-profitable}. Furthermore, $p$ is unramified in $L(\psi_2)/\Q$ by Proposition \ref{normalized expansions are unramified}. Let now $i \in [s]$ and $h \in [2]$. In this case we claim that every place of $\Q(\{\sqrt{p_j(1)p_j(2)} : j \in [s]\})$ above $p_i(h)$ stays unramified in $L(\psi_1)L(\psi_2)/\Q(\{\sqrt{p_j(1)p_j(2)} : j \in [s]\})$. To see this last claim, we use the last statement in Theorem \ref{main theorem on profitable triples} for $\psi_1$ and Proposition \ref{normalized expansions are unramified} for $\psi_2$.

It remains to check the fourth condition of Definition \ref{def: Redei admissible}. We have $\chi_{-1}(\sigma_2(2)) = 0$ by construction of $\sigma_2(2)$. We will now verify that $\phi_T(\psi_2) \circ i_2^\ast = 0$ for each $\varnothing \subsetneq T \subsetneq [s]$. First observe that $\phi_T(\psi_2) \circ i_2^\ast$ is a quadratic character. Since $\psi_2$ is normalized, it follows that $\phi_T(\psi_2) \circ i_2^\ast$ is in the span of the unramified quadratic character. But $(1 + i)$ splits completely in $M(\psi_2)/\Q(i)$ by the fourth condition of Definition \ref{def:-1-profitable}. This shows that $\phi_T(\psi_2) \circ i_2^\ast = 0$.

Finally, we need to check for each $T \subseteq [s]$ that the map $\phi_{T}(\psi_1)$ restricts to a quadratic character that is in the span of $\{\chi_{5}, \chi_{2}\}$, which is the orthogonal complement of $\langle \chi_{-1} \rangle$ with respect to the local Hilbert pairing at $(2)$. This follows from the last part of Theorem \ref{main theorem on profitable triples} keeping in mind that $x_0$ is a special integer, which forces the span of $\{\chi_{x_0}, \chi_5\}$ to be contained in the span of $\{\chi_{5}, \chi_2\}$ locally at $(2)$. This ends the proof that the $4$-tuple is R\'edei admissable, which, as explained above, concludes the proof of the theorem. 
\end{proof}

\subsection{Reflection principles for the self-pairing}
We begin with two auxiliary results, which are the two fundamental steps towards the main results of this subsection, Theorem \ref{main thm 1 on self-pairing} and Theorem \ref{main thm 2 on self-pairing}. Theorem \ref{auxiliary result for the self-pairing} is a result similar to Theorem \ref{thm: reflection principle for infinity}, if not a simpler one, in the sense that it does not invoke any further usage of reciprocity than already invoked in Theorem \ref{main theorem on profitable triples}. Instead Theorem \ref{involution spin are 0} is a genuinely novel result having no analogue in the previous reflection principles. It relies on Hilbert's reciprocity law in a critical manner.   

\begin{theorem} 
\label{auxiliary result for the self-pairing}
Let $(C, \chi_a, (\psi_{s + 1}(x))_{x \in C - \{x_0\}})$ be a profitable triple. Suppose that 
\[
\textup{Up}_{\Q(\sqrt{x})/\Q}(a) \in 2^s \cdot \textup{Cl}(\mathbb{Q}(\sqrt{x}))[2^{s + 1}]
\]
for each $x \in C$. Then $\textup{Up}_{\Q(\sqrt{a})/\Q}(p)$ splits completely in $M(\psi(C, \chi_a, (\psi_{s + 1}(x))_{x \in C - \{x_0\}}))$ for each $p \mid a$. Furthermore, we have that 
\[
\chi_a \in  2^s \cdot \textup{Cl}(\mathbb{Q}(\sqrt{x_0}))^{\vee}[2^{s + 1}]
\]
and
$$
\sum_{x \in C} \langle \textup{Up}_{\Q(\sqrt{x})/\Q}(a), \chi_a \rangle_{\textup{Art}_{s + 1}(x)} = \sum_{p \mid a} \psi_{[s]}(C, \chi_a, (\psi_{s + 1}(x))_{x \in C - \{x_0\}})(\textup{Frob}(\textup{Up}_{\Q(\sqrt{a})/\Q}(p))).
$$ 
\end{theorem} 

\begin{proof}
For each prime $p \mid a$ the map $i_p$ gives us a unique prime $\mathfrak{p}$ above $p$ in 
\[
K := \Q(\{\sqrt{p_i(1)p_i(2)} : i \in [s]\}, \sqrt{x_0}). 
\]
For each $x \in C$ the extension $\Q(\sqrt{x})$ is inside $K$. Therefore we conclude that
$$
KL(\psi_{s + 1}(x))/K
$$ 
is unramified at $\mathfrak{p}$ and
$$
\langle \textup{Up}_{\Q(\sqrt{x})/\Q}(a), \chi_a \rangle_{\textup{Art}_{s + 1}(x)} = \sum_{p \mid a} \psi_{s + 1}(x)(\textup{Frob}(\mathfrak{p})),
$$
since $p$ has residue field degree $1$ in $K$ by Proposition \ref{profitable cubes are quad consistent}. Summing up all the contributions, we get
$$
\sum_{p \mid a} \psi_{[s]}(C, \chi_a, (\psi_{s + 1}(x))_{x \in C - \{x_0\}})(\textup{Frob}(\mathfrak{p})).
$$
Lemma \ref{lX} implies that $\textup{Up}_{\Q(\sqrt{a})/\Q}(p)$ splits completely in $M(\psi(C, \chi_a, (\psi_{s + 1}(x))_{x \in C - \{x_0\}})) \cdot K$ for each $p \mid a$. Therefore we can rewrite 
$$
\psi_{[s]}(C, \chi_a, (\psi_{s + 1}(x))_{x \in C - \{x_0\}})(\textup{Frob}(\mathfrak{p}))
$$
as
$$
\psi_{[s]}(C, \chi_a, (\psi_{s + 1}(x))_{x \in C - \{x_0\}})(\textup{Frob}(\textup{Up}_{\Q(\sqrt{a})/\Q}(p))).
$$ 
We conclude that
$$
\sum_{x \in C} \langle \textup{Up}_{\Q(\sqrt{x})/\Q}(a), \chi_a \rangle_{\textup{Art}_{s + 1}(x)} = \sum_{p \mid a} \psi_{[s]}(C, \chi_a, (\psi_{s + 1}(x))_{x \in C - \{x_0\}})(\textup{Frob}(\textup{Up}_{\Q(\sqrt{a})/\Q}(p)))
$$
as desired. 
\end{proof}

We are now going to abstract the crucial features of the cochains appearing on the right hand side of Theorem \ref{auxiliary result for the self-pairing}. As an auxiliary piece of notation, we remind the standard notation that one attaches to a finite Galois extension $F/\mathbb{Q}$ and to a prime number $p$
$$
e_p(F/\mathbb{Q}) := \#(\text{proj}(G_\Q \to \text{Gal}(F/\mathbb{Q})) \circ i_p^{*})(I_p),
$$
and
$$
f_p(F/\mathbb{Q}) := \frac{\#(\text{proj}(G_\Q \to \text{Gal}(F/\mathbb{Q})) \circ i_p^{*})(G_{\Q_p})}{e_p(F/\Q)}.
$$

\begin{mydef} 
\label{profitable expansions}
Let $s \in \mathbb{Z}_{\geq 1}$. Let $\{a_1, \ldots, a_{s + 1}\}$ be pairwise coprime, special squarefree integers greater than $1$. An expansion map $(\phi_{\{a_i : i \in T\}; a_{s + 1}})_{T \subseteq [s]}$ with support set $\{a_1, \ldots, a_{s + 1}\}$ and pointer $a_{s + 1}$ is said to be \emph{Pellian} in case 
\begin{itemize}
\item for each finite prime $p$ ramifying in $L(\phi_{\{a_i : i \in [s]\}; a_{s + 1}})/\Q$, we have that 
\[
e_p(L(\phi_{\{a_i : i \in [s]\}; a_{s + 1}})/\Q) = e_p(M(\phi_{\{a_i : i \in [s]\}; a_{s + 1}})/\Q) = 2
\]
and $p$ is not $3$ modulo $4$;
\item we have
\[
f_p(M(\phi_{\{a_i : i \in [s]\}; a_{s + 1}})/\Q) = 1
\]
for each finite prime $p$ ramifying in $M(\phi_{\{a_i : i \in [s]\}; a_{s + 1}})/\Q$;
\item the character $\chi_{a_i}$ is locally trivial at $2$ for all $i \in [s]$;
\item $\infty$ splits completely in $M(\phi_{\{a_i : i \in [s]\}; a_{s + 1}})/\Q$;
\item $\phi_{\{a_i : i \in T\}; a_{s + 1}}(\sigma_2(1)) = 0$ for each $T \subsetneq [s]$;
\item all prime divisors of $a_i$ are unramified in 
$$
L(\phi_{\{a_j : j \in [s] - \{i\}\}; a_{s + 1}})/\Q
$$
for each $i \in [s]$.
\end{itemize}
\end{mydef}

Let $p$ be a prime divisor of $a_{s + 1}$. By definition of a Pellian expansion, we have that $\textup{Up}_{\Q(\sqrt{a_{s + 1}})/\Q}(p)$ is unramified in $L(\phi_{\{a_i : i \in [s]\}; a_{s + 1}})$ and splits completely in $M(\phi_{\{a_i : i \in [s]\}; a_{s + 1}})$. Therefore we have that $\text{Frob}(\textup{Up}_{\Q(\sqrt{a_{s + 1}})/\Q}(p))$ is well-defined and lands in 
\[
\text{Gal}(L(\phi_{\{a_i : i \in [s]\}; a_{s + 1}})/M(\phi_{\{a_i : i \in [s]\}; a_{s + 1}})) \cong \FF_2.
\]
We now give the second auxiliary result, which has some similarities with the computations in \cite[Section 12]{FIMR}.

\begin{theorem} 
\label{involution spin are 0} 
Let $s \in \mathbb{Z}_{\geq 2}$. Let $\{a_1, \ldots, a_{s + 1}\}$ be pairwise coprime, special squarefree integers. Let $(\phi_{\{a_i : i \in T\}; a_{s + 1}})_{T \subseteq [s]}$ be a Pellian expansion map with support $\{a_1, \ldots, a_{s + 1}\}$ and pointer $a_{s + 1}$. Then
$$
\sum_{p \mid a_{s + 1}} \phi_{\{a_i : i \in [s]\}; a_{s + 1}}(\textup{Frob}(\textup{Up}_{\Q(\sqrt{a_{s + 1}})/\Q}(p))) = 0. 
$$
\end{theorem}

\begin{proof}
Until specified, all intermediate assertions that we are going to make also apply to the case $s = 1$.
 
\subsubsection*{Construction of an auxiliary $1$-cochain}
Let us write $K := \Q(\{a_i : i \in [s] - \{1\}\})$. Observe that the restriction of $\phi_{\{a_i : i \in [s] - \{1\}\}; a_{s + 1}}$ to $G_K$ is a quadratic character, while $\phi_{\{a_i : i \in [s]\}; a_{s + 1}}$ restricted to $G_K$ is a $1$-cochain whose differential equals the following cup product
$$
(d\phi_{\{a_i : i \in [s]\}; a_{s + 1}})(\sigma, \tau) = \chi_{a_1}(\sigma) \cdot \phi_{\{a_i : i \in [s] - \{1\}\}; a_{s + 1}}(\tau)
$$
of quadratic characters. 

Thanks to the first point of Definition \ref{profitable expansions}, the quadratic character $\phi_{\{a_i\}_{i \in [s] - \{1\}}; a_{s + 1}}$ does not ramify at odd finite places of $K$ not possessing a primitive $4$-th root of unity. As such the $2$-cocycle
$$
\theta(\sigma, \tau) := \phi_{\{a_i : i \in [s] - \{1\}\}; a_{s + 1}}(\sigma) \cdot \phi_{\{a_i : i \in [s] - \{1\}\}; a_{s + 1}}(\tau)
$$
is locally trivial at all finite odd places of $K$. Furthermore, $\theta$ is trivial at all the $2^{s - 1}$ infinite places of $K$, in virtue of the fourth point of Definition \ref{profitable expansions}. We now examine the places above $(2)$. Observe that by the third point $(2)$ splits completely in $K$. Call $\mathfrak{t}$ the unique place of $K$ given by $i_2$. Then $K_{\mathfrak{t}} = \Q_2$.

Recall that the set of characters $\chi \in \text{Hom}_{\text{top.gr.}}(G_{\Q_2}, \mathbb{F}_2)$ with $\chi \cup \chi = 0$ in $H^2(G_{\Q_2}, \mathbb{F}_2)$ forms a subspace thanks to the antisymmetry of the Hilbert symbol. Furthermore, it coincides precisely with the space of characters vanishing at $\sigma_2(1)$. To see this, remark that $\chi_2 \cup \chi_2$ and $\chi_5 \cup \chi_5$ are trivial (since $2^2 = 2 + 2$ and $5^2 = 5 + 5 \cdot 2^2$). Then, since both spaces are $2$-dimensional, they must coincide.

Now, thanks to the fifth condition in Definition \ref{profitable expansions}, we see that 
$$
\text{inv}_{\mathfrak{t}}(\phi_{\{a_i : i \in [s] - \{1\}\}; a_{s + 1}}(\sigma) \cdot \phi_{\{a_i : i \in [s] - \{1\}\}; a_{s + 1}}(\tau)) = 0.
$$
We now show it for all the other places above $(2)$. Take any $\rho \in \text{Gal}(K/\Q)$. Observe that
\begin{multline*}
\text{inv}_{\rho(\mathfrak{t})}(\phi_{\{a_i : i \in [s] - \{1\}\}; a_{s + 1}}(\sigma) \cdot \phi_{\{a_i : i \in [s] - \{1\}\}; a_{s + 1}}(\tau)) = \\
\text{inv}_{\mathfrak{t}}(\phi_{\{a_i : i \in [s] - \{1\}\}; a_{s + 1}}(\rho\sigma \rho^{-1}) \cdot \phi_{\{a_i : i \in [s] - \{1\}\}; a_{s + 1}}(\rho\tau \rho^{-1})).
\end{multline*} 
Invoking Proposition \ref{norming stuff down}, we see that $\phi_{\{a_i : i \in [s] - \{1\}\}; a_{s + 1}}(\rho\sigma \rho^{-1})$ is in the span of
$$
\{\phi_{\{a_i : i \in T\}; a_{s + 1}} : T \subseteq [s] - \{1\}\}. 
$$
However, invoking again the fifth condition of Definition \ref{profitable expansions}, we see that each $\phi_{\{a_i : i \in T\}; a_{s + 1}}$ is in the subspace of $\chi \in \text{Hom}_{\text{top.gr.}}(G_{K_{\mathfrak{t}}}, \mathbb{F}_2)$ with 
\[
\text{inv}_{\mathfrak{t}}(\chi \cup \chi) = 0. 
\]
Therefore we conclude that 
\[
\text{inv}_{\mathfrak{t}}(\phi_{\{a_i : i \in [s] - \{1\}\}; a_{s + 1}}(\rho\sigma \rho^{-1}) \cdot \phi_{\{a_i : i \in [s] - \{1\}\}; a_{s + 1}}(\rho\tau \rho^{-1})) = 0
\]
and hence 
\[
\text{inv}_{\rho(\mathfrak{t})}(\phi_{\{a_i : i \in [s] - \{1\}\}; a_{s + 1}}(\sigma) \cdot \phi_{\{a_i : i \in [s] - \{1\}\}; a_{s + 1}}(\tau)) = 0.
\]
We have proved that $\theta$ vanishes locally at all places above $(2)$ as well. Hence we have completed the proof that the class of $\theta$ is locally trivial at all places of $K$. Therefore there exists a $1$-cochain $\widetilde{\phi}:G_K \to \mathbb{F}_2$ with 
$$
(d\widetilde{\phi})(\sigma, \tau) = \theta(\sigma, \tau) = \phi_{\{a_i : i \in [s] - \{1\}\}; a_{s + 1}}(\sigma) \cdot \phi_{\{a_i : i \in [s] - \{1\}\}; a_{s + 1}}(\tau).
$$
The possible choices of the cochain $\widetilde{\phi}$ form a coset under $\text{Hom}_{\text{top.gr.}}(G_K, \mathbb{F}_2)$. We are going to make a choice of $\widetilde{\phi}$ that will simplify the coming discussion. Since $\chi_{a_1}$ is a character that ramifies at some place, we can always choose a set $S_1(K)$ of odd prime ideals of $O_K$ spanning $\frac{\text{Cl}(K)}{2\text{Cl}(K)}$ such that $\chi_{a_1}$ is locally trivial at each of them. 

Let us call $S(K)$ the set of places of $K$ that are in $S_1(K)$ or lie above $(2)$ or $\infty$. We now claim that we can always make a choice of $\widetilde{\phi}$ such that $L(\widetilde{\phi})/K$ is unramified outside of $S(K)$ and the places where $L(\phi_{\{a_i : i \in [s] - \{1\}\}; a_{s + 1}})/K$ ramifies. Indeed, start with one cochain $\widetilde{\phi}_1: G_K \to \mathbb{F}_2$ such that 
\[
(d\widetilde{\phi}_1)(\sigma, \tau) = \theta(\sigma, \tau) = \phi_{\{a_i : i \in [s] - \{1\}\}; a_{s + 1}}(\sigma) \cdot \phi_{\{a_i : i \in [s] - \{1\}\}; a_{s + 1}}(\tau). 
\]
Suppose first that the field extension $L(\widetilde{\phi}_1)/K$ ramifies at some place $\mathfrak{p}$ outside of $S(K)$ and the set of places where $L(\phi_{\{a_i : i \in [s] - \{1\}\}; a_{s + 1}})/K$ ramifies. By definition of $S(K)$ we can find an ideal $J$ entirely supported in $S(K)$ such that $\mathfrak{p} J = (\gamma_{\mathfrak{p}}) \cdot I^2$ for some integral ideal $I$ of $\mathcal{O}_K$ and some non-zero element $\gamma_{\mathfrak{p}}$ of $K$. 

Observe that the quadratic character $\chi_{\gamma_{\mathfrak{p}}}$ ramifies at $\mathfrak{p}$ and its ramification locus is contained in $\{\mathfrak{p}\} \cup S(K)$. Furthermore, writing $\widetilde{\phi}_2 = \widetilde{\phi}_1 + \chi_{\gamma_{\mathfrak{p}}}$, we see that $L(\widetilde{\phi}_2)/K$ is unramified at $\mathfrak{p}$. Since $L(\widetilde{\phi}_1)/K$ ramifies at finitely many places, we get a $1$-cochain $\widetilde{\phi}$ with the claimed properties by iterating this procedure. In what follows we consider this choice of $\widetilde{\phi}$. 

\subsubsection*{Construction of a $2$-cocycle}
A quick calculation shows that
$$
\widetilde{\theta}(\sigma, \tau):=\phi_{\{a_i : i \in [s]\}; a_{s + 1}}(\sigma) \cdot \phi_{\{a_i : i \in [s] - \{1\}\}; a_{s + 1}}(\tau) + \chi_{a_1}(\sigma) \cdot \widetilde{\phi}(\tau)
$$
is a $2$-cocycle, giving a class in $H^2(G_K, \mathbb{F}_2)$. Hence Hilbert reciprocity yields
$$
\sum_{v \in \Omega_K} \text{inv}_v(\widetilde{\theta}) = 0. 
$$
We next calculate for each $v \in \Omega_K$ the value of $\text{inv}_v(\widetilde{\theta})$. 

\subsubsection*{The infinite places}
Observe that $\chi_{a_1}$ is locally trivial at any infinite place, since $a_1>0$. Furthermore, thanks to the fourth condition in Definition \ref{profitable expansions}, we see that $\phi_{\{a_i : i \in [s] - \{1\}\}; a_{s + 1}}$ is also locally trivial at each infinite place. We therefore conclude that the $2$-cocycle $\widetilde{\theta}$ becomes literally the zero map when restricted to each decomposition group of an archimedean place of $K$. Hence $\text{inv}_v(\widetilde{\theta}) = 0$ for each infinite place $v$ of $K$.

\subsubsection*{Places above $a_{s + 1}$}
We distinguish two cases. Let first $p$ be an odd prime factor of $a_{s + 1}$. Clearly, $p$ ramifies in $L(\phi_{\{a_i : i \in [s]\}; a_{s + 1}})/\Q$. It follows that $p$ splits completely in $K(\sqrt{a_1})$ thanks to the first and the second condition of Definition \ref{profitable expansions}. Let $\mathfrak{p}$ be one of the $2^{s - 1}$ places of $\Omega_K$ lying above $p$. Then $\chi_{a_1}$ is locally trivial at $\mathfrak{p}$, thanks to the just mentioned splitting of $p$ in $K(\sqrt{a_1})$. We conclude that $\phi_{\{a_i : i \in [s]\}; a_{s + 1}}$ and $\phi_{\{a_i : i \in [s] - \{1\}\}; a_{s + 1}}$ are both quadratic characters locally at $\mathfrak{p}$. Therefore, locally at $\mathfrak{p}$, the $2$-cocycle $\widetilde{\theta}$ becomes the cup product
$$
\phi_{\{a_i : i \in [s]\}; a_{s + 1}}(\sigma) \cdot \phi_{\{a_i : i \in [s] - \{1\}\}; a_{s + 1}}(\tau).
$$ 
This shows that
$$
\text{inv}_{\mathfrak{p}}(\widetilde{\theta}) = \text{inv}_{\mathfrak{p}}(\phi_{\{a_i : i \in [s]\}; a_{s + 1}}(\sigma) \cdot \phi_{\{a_i : i \in [s] - \{1\}\}; a_{s + 1}}(\tau)).
$$
It follows from the first and the second condition in Definition \ref{profitable expansions} that the set
$$
\{\phi_{\{a_i : i \in T\}; a_{s + 1}} : T \subsetneq [s]\}
$$  
spans the $1$-dimensional space generated by $\chi_{a_{s + 1}}$ locally at $\mathfrak{p}$. From this we deduce that
$$
\phi_{\{a_i : i \in [s]\}; a_{s + 1}}(\text{Frob}(\textup{Up}_{\Q(\sqrt{a_{s + 1}})/\Q}(p))) = 0
$$
if and only if $\phi_{\{a_i : i \in [s]\}; a_{s + 1}}$ lands in the space generated by $\chi_{a_{s + 1}}$. In case $\phi_{\{a_i : i \in [s] - \{1\}\}; a_{s + 1}}$ is locally trivial at $\mathfrak{p}$, we evidently get 
\[
\text{inv}_{\mathfrak{p}}(\phi_{\{a_i : i \in [s]\}; a_{s + 1}}(\sigma) \cup \phi_{\{a_i : i \in [s] - \{1\}\}; a_{s + 1}}(\tau)) = 0.
\]
Otherwise, $\phi_{\{a_i : i \in [s] - \{1\}\}; a_{s + 1}}$ is non-trivial locally at $\mathfrak{p}$ and hence equals $\chi_{a_{s + 1}}$ locally at $\mathfrak{p}$. Then we get
$$
\text{inv}_{\mathfrak{p}}(\phi_{\{a_i : i \in [s]\}; a_{s + 1}}(\sigma) \cdot \phi_{\{a_i : i \in [s] - \{1\}\}; a_{s + 1}}(\tau)) = \phi_{\{a_i : i \in [s]\}; a_{s + 1}}(\text{Frob}(\textup{Up}_{\Q(\sqrt{a_{s + 1}})/\Q}(p))).
$$
Indeed, the above Hilbert symbol vanishes if and only if $\phi_{\{a_i : i \in [s]\}; a_{s + 1}}$ is in the space generated by the non-trivial character $\chi_{a_{s + 1}}$, which, as we argued, is also decided by the Artin symbol. Finally, Corollary \ref{right parity} implies that the total contribution of the places above $p$ is precisely
$$
\phi_{\{a_i : i \in [s]\}; a_{s + 1}}(\text{Frob}(\textup{Up}_{\Q(\sqrt{a_{s + 1}})/\Q}(p))).
$$
Suppose now that $2$ divides $a_{s + 1}$. We explain why the above argument works in this case as well. Thanks to the first and the second condition in Definition \ref{profitable expansions}, we know that
$$
\{\phi_{\{a_i : i \in T\}; a_{s + 1}}|_{i_2^\ast(G_{\Q_2})} : T \subseteq [s]\}
$$
spans a space that is at most $2$-dimensional, containing $\chi_{a_{s + 1}}$ and reaching dimension $2$ if and only if $\chi_5$ is therein. Again the symbol $\phi_{\{a_i : i \in [s]\}; a_{s + 1}}(\text{Frob}(\textup{Up}_{\Q(\sqrt{a_{s + 1}})/\Q}(2)))$ detects whether the dimension is $2$ and likewise for the Hilbert symbol 
\[
\text{inv}_{\mathfrak{p}}(\phi_{\{a_i : i \in [s]\}; a_{s + 1}}(\sigma) \cdot \phi_{\{a_i : i \in [s] - \{1\}\}; a_{s + 1}}(\tau)) = \phi_{\{a_i : i \in [s]\}; a_{s + 1}}(\text{Frob}(\textup{Up}_{\Q(\sqrt{a_{s + 1}})/\Q}(2)))
\]
in case $\phi_{\{a_i : i \in [s] - \{1\}\}; a_{s + 1}}$ is non-trivial locally at $\mathfrak{p}$. With these small differences, the rest of the proof is identical and one reaches again the conclusion that the final contribution above $(2)$ is 
$$
\phi_{\{a_i : i \in [s]\}; a_{s + 1}}(\text{Frob}(\textup{Up}_{\Q(\sqrt{a_{s + 1}})/\Q}(2))).
$$
Hence we conclude that the total contribution from the places above the prime divisors of $a_{s + 1}$ is
$$
\sum_{p \mid a_{s + 1}} \phi_{\{a_i : i \in [s]\}; a_{s + 1}}(\textup{Frob}(\textup{Up}_{\Q(\sqrt{a_{s + 1}})/\Q}(p))).
$$

\subsubsection*{The remaining places not dividing $a_1$} 
Suppose that $v$ is a finite place of $K$ above a prime $q$ such that the residue field $\mathbb{F}_v$ is a non-trivial extension of $\mathbb{F}_q$. We claim that
\begin{align}
\label{eLocTrivialBigResidue}
\text{inv}_v(\widetilde{\theta}) = 0.
\end{align}
Observe that $q$ is unramified in $L(\phi_{\{a_i : i \in [s]\}; a_{s + 1}})/\Q$ by the first and second condition of Definition \ref{profitable expansions}. By our assumption on $v$, we have $\mathbb{F}_q^{*} \subseteq \mathbb{F}_v^{*2}$. Then the character $\chi_{a_i}$ is locally trivial at $v$ for all $i \in [s]$, and therefore $\phi_{\{a_i : i \in [s]\}; a_{s + 1}}(\sigma)$ and $\phi_{\{a_i : i \in [s] - \{1\}\}; a_{s + 1}}(\tau)$ are both quadratic characters locally at $v$. Hence $\widetilde{\theta}$ becomes the cup product 
\[
\phi_{\{a_i : i \in [s]\}; a_{s + 1}}(\sigma) \cdot \phi_{\{a_i : i \in [s] - \{1\}\}; a_{s + 1}}(\tau)
\]
of two unramified characters, which proves the claimed equation (\ref{eLocTrivialBigResidue}).

We are left with the finite places $v$ with residue field degree one in $K$ and not dividing $a_{s + 1}$. Write $q$ for the place of $\Q$ below $v$. Let us first consider $v$ such that $q$ is unramified in $L(\phi_{\{a_i : i \in [s]\}; a_{s + 1}})/\Q$. Then $q$ splits completely in $K$. 

First suppose that $v$ ramifies in $L(\widetilde{\phi})/K$. Then, by our choice of $S(K)$ if $q$ is odd, and by the third condition of Definition \ref{profitable expansions} in case $q = 2$, we have that $\chi_{a_1}$ is trivial locally at $q$. Since $q$ splits completely in $K$, we conclude that $\chi_{a_1}$ is also locally trivial at $v$. Therefore at all such places we are left with the cup product of two unramified characters and thus we get that the Hilbert symbol is trivial.

Next suppose that $v$ is unramified in $L(\widetilde{\phi})/K$. For each such $v$ we have that the restriction of $\widetilde{\theta}$ to $G_{K_v}$ is in the image of the inflation
$$
\text{inf}: H^2(\Gal(K_v^{\text{unr}}/K_v), \mathbb{F}_2) \to H^2(G_{K_v}, \mathbb{F}_2).
$$
But we know that $\Gal(K_v^{\text{unr}}/K_v) \cong \hat{\Z}$ and $H^2(\hat{\mathbb{Z}}, \mathbb{F}_2) = 0$, since any central extension of $\hat{\mathbb{Z}}$ is clearly trivial. We conclude that 
\[
\text{inv}_v(\widetilde{\theta}) = 0
\]
in this case.

Hence we have proved that the contribution from all the places above finite rational primes $q$ unramified in $L(\phi_{\{a_i : i \in [s]\}; a_{s + 1}})/\Q$ is pointwise $0$. We now turn to the ramified places. 

Among the remaining places, let us first examine those $v$ for which $q$ does not divide $a_1$. It follows from the second condition of Definition \ref{profitable expansions} that $\chi_{a_1}$ is locally trivial at such places. Hence we are left with
\[
\phi_{\{a_i : i \in [s]\}; a_{s + 1}}(\sigma) \cdot \phi_{\{a_i : i \in [s] - \{1\}\}; a_{s + 1}}(\tau)
\]
locally at $v$. Let us further suppose that $q$ does not ramify in $K/\Q$. Then it must split completely in $K/\Q$. Now for each of the $2^{s - 1}$ places above $q$ where the quadratic character $\phi_{\{a_i : i \in [s] - \{1\}\}; a_{s + 1}}$ of $G_K$ ramifies, the invariant map is zero if $f_q(L(\phi_{\{a_i : i \in [s]\}; a_{s + 1}})/\Q) = 1$ and non-zero if $f_q(L(\phi_{\{a_i : i \in [s]\}; a_{s + 1}})/\Q) = 2$. Since $q$ does not divide $a_{s + 1}$, there is an even number of places above $q$ ramifying in $\phi_{\{a_i : i \in [s] - \{1\}\}; a_{s + 1}}$ in virtue of Corollary \ref{right parity}. Therefore the total contribution from such places is trivial. 

In case $s \geq 2$, we still need to deal with the places $q$ not dividing $a_1$ but dividing $a_i$ for some $i \in [s] - \{1\}$. Take a place $v$ above $q$ and observe that both characters 
\[
\phi_{\{a_i : i \in [s]\}; a_{s + 1}}, \quad \phi_{\{a_i : i \in [s] - \{1\}\}; a_{s + 1}}
\]
have even valuation at $v$, thanks to the first condition in Definition \ref{profitable expansions}. Therefore the Hilbert symbol is trivial also in this case. 

\subsubsection*{Places dividing $a_1$} 
We are now left with the $v$ such that $q$ divides $a_1$. Indeed, we have shown that the total contribution from all the other places not dividing $a_{s + 1}$ is trivial. Take now such a prime $q$ dividing $a_1$. Observe that all the characters
\[
\phi_{\{a_i : i \in [s] - \{1\}\}; a_{s + 1}}
\]
are locally trivial at $v$, thanks to the second and fifth condition of Definition \ref{profitable expansions}. In particular $\widetilde{\phi}$ becomes a quadratic character locally at $v$ and $\widetilde{\theta}$ becomes the cup product
$$
\chi_{a_1}(\sigma) \cdot \widetilde{\phi}(\tau). 
$$
Furthermore, $v$ is unramified in the cyclic degree $4$ extension $L(\tilde{\phi})/K$ and splits in the unique quadratic subextension given by $\phi_{\{a_i : i \in [s] - \{1\}\}; a_{s + 1}}$. Let us call
$$
\widetilde{\psi}: G_K \twoheadrightarrow \mathbb{Z}/4\mathbb{Z}
$$
the continuous epimorphism given by 
\[
\sigma \mapsto (\widetilde{\phi}(\sigma), \phi_{\{a_i : i \in [s] - \{1\}\}; a_{s + 1}}(\sigma)),
\]
where $\mathbb{Z}/4\mathbb{Z}$ is represented as $\mathbb{F}_2 \times \mathbb{F}_2$ with the product law 
\[
(a_1, b_1) * (a_2, b_2) = (a_1 + a_2 + b_1b_2, b_1 + b_2). 
\]
We have that the total contribution from places above $q$ dividing $a_1$ is
$$
\sum_{\substack{v \in \Omega_K \\ v \mid q}} \widetilde{\psi}(\text{Frob}(v)).
$$
This is precisely the same as
$$
(N_{K/\mathbb{Q}}(\widetilde{\psi}))(\text{Frob}(q)). 
$$
Now $N_{K/\Q}(\widetilde{\psi})$ is a cyclic degree $4$ character of $G_{\mathbb{Q}}$ lifting the quadratic character $\chi_{a_{s + 1}}$, thanks to Proposition \ref{norming stuff down}. As such it can be given by
$$
\sigma \mapsto (\phi_{a_{s + 1}; a_{s + 1}}(\mathfrak{G})(\sigma) + \chi(\sigma), \chi_{a_{s + 1}}(\sigma)),
$$
thanks to Lemma \ref{lCyclicDegree4}. 

\subsubsection*{Study of the character $\chi$}
We claim that $\chi$ is a quadratic character that only ramifies at primes that split completely in $\Q(\sqrt{a_1})$. Indeed, observe that all the primes ramifying in $K$ split completely in $\Q(\sqrt{a_1})$. Furthermore, $L(\widetilde{\psi})/K$ ramifies only at places of $K$ that split completely in $\Q(\sqrt{a_1})$ by construction of $S(K)$, and the same is true for all the conjugates of $\widetilde{\psi}$. Therefore the field of definition $L(\phi_{a_{s + 1}; a_{s + 1}}(\mathfrak{G}) + \chi)/\Q$ can ramify only at places where $\chi_{a_1}$ is locally trivial. 

We will now show how this implies that $\chi$ has the desired property. By assumption the places $(2)$ and $\infty$ split completely in $\Q(\sqrt{a_1})/\Q$, so we can focus entirely on the odd primes. We show that if an odd prime $p$ ramifies in $\Q(\chi)/\Q$, then it must also ramify in $L(\phi_{a_{s + 1}; a_{s + 1}}(\mathfrak{G}) + \chi)/\Q$, and thus, as we argued above, $p$ must split completely in $\Q(\sqrt{a_1})/\Q$. Observe that, by definition, we have that
$$
\phi_{a_{s + 1}; a_{s + 1}}(\mathfrak{G})(\sigma_p) = 0,
$$
and, by Proposition \ref{ramification read off by inertia}, we have that $\chi(\sigma_p) = 1$. It follows that
$$
\sigma_p \mapsto (1, \chi_{a_{s + 1}}(\sigma_p)) \neq (0,0).
$$
Therefore by Proposition \ref{ramification read off by inertia} we conclude that $p$ ramifies in $L(\phi_{a_{s + 1}; a_{s + 1}}(\mathfrak{G}) + \chi)/\Q$, which establishes the claim.

\subsubsection*{End of proof}
What we have proven so far applies also to the Pellian expansion map $\phi_{a_1; a_{s + 1}}$, since we have not made use that $s \geq 2$ yet. But now we have
\begin{align*}
\sum_{p \mid a_{s + 1}} \phi_{\{a_i : i \in [s]\}; a_{s + 1}}(\textup{Frob}(\textup{Up}_{\Q(\sqrt{a_{s + 1}})/\Q}(p))) &= \sum_{q \mid a_1}(\phi_{a_{s + 1}; a_{s + 1}}(\mathfrak{G}) + \chi)(\text{Frob}(q)) \\
&= \sum_{p \mid a_{s + 1}} \phi_{a_1; a_{s + 1}}(\textup{Frob}(\textup{Up}_{\Q(\sqrt{a_{s + 1}})/\Q}(p))) = 0,
\end{align*}
where the first equality follows from the above analysis for $\phi_{\{a_i : i \in [s]\}; a_{s + 1}}$, the second equality follows from the above analysis for $\phi_{a_1; a_{s + 1}}$ (i.e. the case $s = 1$) and the last equality follows from $s \geq 2$ and the first two conditions of Definition \ref{profitable expansions}. 
\end{proof} 

\noindent The following results are the main theorems of this subsection.

\begin{theorem} 
\label{main thm 1 on self-pairing}
Let $s \in \mathbb{Z}_{\geq 2}$. Let $(C, \chi_a, (\psi_{s + 1}(x))_{x \in C - \{x_0\}})$ be a $-1$-profitable triple. Suppose that $\textup{Up}_{\Q(\sqrt{x})/\Q}(a) \in 2^s \cdot \textup{Cl}(\mathbb{Q}(\sqrt{x}))[2^{s + 1}]$ for each $x \in C$. Then we also have that 
\[
\chi_a \in 2^s \cdot \textup{Cl}(\mathbb{Q}(\sqrt{x}))^{\vee}[2^{s + 1}]
\] 
and
$$
\sum_{x \in C} \langle \textup{Up}_{\Q(\sqrt{x})/\Q}(a), \chi_a \rangle_{\textup{Art}_{s + 1}(x)} = 0. 
$$
\end{theorem}

\begin{proof}
This is now a straightforward combination of Theorem \ref{auxiliary result for the self-pairing} and Theorem \ref{involution spin are 0}.
\end{proof}

\begin{theorem} 
\label{main thm 2 on self-pairing}
Let $s \in \mathbb{Z}_{\geq 2}$. Let $(C, \chi_a, (\psi_{s + 1}(x))_{x \in C - \{x_0\}})$ be a $-1$-profitable triple. Suppose that $\textup{Up}_{\Q(\sqrt{x})/\Q}(a) \in 2^s \cdot \textup{Cl}(\mathbb{Q}(\sqrt{x}))[2^{s + 1}]$ for each $x \in C$. Then we also have that 
\[
\chi_a \in 2^s \cdot \textup{Cl}(\mathbb{Q}(\sqrt{x}))^{\vee}[2^{s + 1}]
\]
and
$$
\sum_{x \in C} \left\langle \frac{(\sqrt{x})}{\textup{Up}_{\Q(\sqrt{x})/\Q}(a)}, \chi_{\frac{x}{a}} \right\rangle_{\textup{Art}_{s + 1}(x)} = \sum_{p \mid a} \phi_{\{p_i(1)p_i(2) : i \in [s]\}; -1}(\mathfrak{G})(\textup{Frob}(p)). 
$$
\end{theorem}

\begin{proof}
Observe that 
\begin{align*}
\left\langle \frac{(\sqrt{x})}{\textup{Up}_{\Q(\sqrt{x})/\Q}(a)}, \chi_{\frac{x}{a}} \right\rangle_{\textup{Art}_{s + 1}(x)} 
&= \left\langle (\sqrt{x}) \cdot \textup{Up}_{\Q(\sqrt{x})/\Q}(a), \chi_a \right\rangle_{\textup{Art}_{s + 1}(x)} \\
&= \left\langle \textup{Up}_{\Q(\sqrt{x})/\Q}(a), \chi_a \right\rangle_{\textup{Art}_{s + 1}(x)} + \left\langle (\sqrt{x}), \chi_a \right\rangle_{\textup{Art}_{s + 1}(x)}.
\end{align*}
Hence the conclusion follows upon combining Theorem \ref{thm: reflection principle for infinity} and Theorem \ref{main thm 1 on self-pairing}.
\end{proof}

\subsection{Standard reflection principles}
We now establish the more classical reflection principles. Some of our arguments are different, but the material of this subsection is based on \cite{Smith}. We include it here with proofs in order to keep our work self-contained. 

Let $s$ be a positive integer and let
$$
C := \{p_1(1), p_1(2)\} \times \dots \times \{p_s(1), p_s(2)\} \times \{p_{s + 1}(1), p_{s + 1}(2)\} \times \{d_0\},
$$
where $\{p_1(1), p_1(2), \dots, p_{s + 1}(1), p_{s + 1}(2)\}$ is a set of $2s + 2$ distinct prime numbers, each of them coprime with the positive squarefree integer $d_0$. Furthermore, we demand that $p_i(1)p_i(2) \equiv 1 \bmod 8$ and that $p_i(1)p_i(2)$ is a square modulo every odd prime factor of $d_0$ for every $i \in [s]$. Let $a$ be a divisor of $d_0$. As at the beginning of Section \ref{profitability}, we identify points in $C$ with squarefree integers. We write
$$
x_0 := d_0 \cdot \prod_{i = 1}^{s + 1} p_i(1).
$$
Suppose that we have for each $x \in C - \{x_0\}$ a raw cocycle $\psi_{s + 1}(x)$ for $N(x)$ such that either 
\[
2^s \cdot \psi_{s + 1}(x) = \chi_a
\]
for each $x \in C - \{x_0\}$, or 
\[
2^s \cdot \psi_{s + 1}(x) = \chi_{\pi_{s + 1}(x) a}
\]
for each $x \in C - \{x_0\}$, where we recall that $\pi_{s + 1}$ is the projection on the $s + 1$-th coordinate of $C$. For convenience, the first case will be referred to as being \emph{of type 1}, while the second case is said to be \emph{of type 2}. We now make the two key definitions of this subsection. We begin with the notion of minimal triples. 

\begin{mydef} 
\label{def: minimal triples}
Suppose that $(C, (\psi_{s + 1}(x))_{x \in C - \{x_0\}}, \chi_a)$ is a triple of type $1$. We say that the triple is a \emph{minimal triple} in case for each non-empty subset $T \subseteq [s + 1]$ we have that
\begin{align}
\label{eMinimality}
\sum_{x \in C_T} \psi_{s - |T| + 1}(x) = 0.
\end{align}
\end{mydef}

\noindent We next turn to the notion of governing triples. 

\begin{mydef} 
\label{def: governing triples}
Suppose that  $(C, (\psi_{s + 1}(x))_{x \in C - \{x_0\}}, \chi_a)$ is a triple of type $2$. We say that the triple is a \emph{governing triple} in case 
\begin{itemize}
\item for each $T \subseteq [s]$ the map $\phi_{\{p_i(1)p_i(2) : i \in T\}; p_{s + 1}(1)p_{s + 1}(2)}(\mathfrak{G})$ exists;
\item for each non-empty $T \subseteq [s]$, we have 
$$
\sum_{x \in C_T} \psi_{s - |T| + 1}(x) = \phi_{\{p_i(1)p_i(2) : i \in [s] - T\}; p_{s + 1}(1)p_{s + 1}(2)}(\mathfrak{G});
$$
\item for each $T \subseteq [s + 1]$ with $s + 1 \in T$, we have
$$
\sum_{x \in C_T} \psi_{s - |T| + 1}(x) = 0.
$$
\end{itemize}
\end{mydef}

\noindent We can now give the two main results of this subsection. We start with the one on minimal triples. 

\begin{theorem} 
\label{main thm: minimal triples}
Let $(C, (\psi_{s + 1}(x))_{x \in C - \{x_0\}}, \chi_a)$ be a minimal triple. Then 
\[
\chi_a \in 2^{s} \cdot \textup{Cl}(\Q(\sqrt{x_0}))^{\vee}[2^{s + 1}]. 
\]
Moreover, we have the following results.
\begin{enumerate}
\item[(i)] Let $b \mid d_0$ (resp. $b \mid 2d_0$ in case $\Q(\sqrt{x_0})/\Q$ ramifies at $2$) be such that
$$
\textup{Up}_{\Q(\sqrt{x})/\Q}(b) \in 2^{s} \cdot \textup{Cl}(\Q(\sqrt{x}))[2^{s + 1}]
$$
for each $x \in C$. Then 
$$
\sum_{x \in C} \langle \textup{Up}_{\Q(\sqrt{x})/\Q}(b), \chi_a \rangle_{\textup{Art}_{s + 1}(x)} = 0.
$$
\item[(ii)] Instead suppose that
$$
\textup{Up}_{\Q(\sqrt{x})/\Q}(x) \in 2^{s} \cdot \textup{Cl}(\Q(\sqrt{x}))[2^{s + 1}]
$$
for each $x \in C$. Then
$$
\sum_{x \in C} \langle \textup{Up}_{\Q(\sqrt{x})/\Q}(x), \chi_a \rangle_{\textup{Art}_{s + 1}(x)} = 0.
$$
\end{enumerate}
\end{theorem}

\begin{proof}
Let us start by proving that $\chi_a \in 2^{s} \cdot \textup{Cl}(\Q(\sqrt{x_0}))^{\vee}[2^{s + 1}]$. To this end, consider the $1$-cochain $\psi_{s + 1}(x_0): G_\Q \to N[2^{s + 1}]$ defined by
$$
\psi_{s + 1}(x_0) := -\sum_{x \in C:x \neq x_0} \psi_{s + 1}(x).
$$
We will show that $\psi_{s + 1}(x_0)$ is a $1$-cocycle for $N(x_0)$ with $2^s \cdot \psi_{s + 1}(x_0) = \chi_a$. This last assertion follows immediately from
\[
2^s \cdot \psi_{s + 1}(x_0) = -(2^s - 1) \cdot \chi_a = \chi_a. 
\]
Let us next prove that $\psi_{s + 1}(x_0): G_\Q \to N(x_0)[2^{s + 1}]$ is in fact a $1$-cocycle. Thanks to Proposition \ref{key calculation of cocycles} we have that 
\begin{align*}
(d_{x_0}\psi_{s + 1}(x_0))(\sigma, \tau) &= \sum_{\varnothing \neq T \subseteq [s + 1]} \chi_{\{p_i(1)p_i(2) : i \in T\}}(\sigma)\cdot(-1)^{|T| + 1 + \chi_{x_0}(\sigma)} \left(\sum_{x \in C_T} \psi_{s - |T| + 1}(x)(\tau)\right) \\
&= 0,
\end{align*}
where the last equation follows from our assumption
\[
\sum_{x \in C_T} \psi_{s - |T| + 1}(x) = 0,
\]
see equation (\ref{eMinimality}).

In particular it follows that $\psi_s(x_0) := 2 \cdot \psi_{s + 1}(x_0)$ is a $1$-cocycle with values in $N(x_0)[2^s]$. We claim that $\psi_s(x_0)$ once restricted to $G_{\Q(\sqrt{x_0})}$ naturally lands in $\text{Cl}(\Q(\sqrt{x_0}))^{\vee}[2^s]$, i.e. that it yields a cyclic degree $2^s$ extension of $\Q(\sqrt{x_0})$ unramified at all finite places. Let us first consider a prime $p$ that does not ramify in $\Q(\sqrt{x})$ for any $x \in C$. Then we simply observe that 
\begin{align}
\label{eFoDtrick}
L(\psi_s(x_0)) \subseteq \prod_{x \in C - \{x_0\}} L(\psi_s(x)).
\end{align}
Lemma \ref{lX} and equation (\ref{eFoDtrick}) give the desired conclusion in this case.

We now consider the case of a prime $p$ ramifying in $\Q(\sqrt{x})$ for every $x \in C$ or of the form $p_i(1)$ for some $i \in [s + 1]$. Since $p$ has ramification index $2$ in $\prod_{x \in C - \{x_0\}} L(\psi_s(x))/\Q$ by Lemma \ref{lX} and since $p$ ramifies in $\Q(\sqrt{x_0})/\Q$, this case is also okay. We finally consider the case of an odd prime of the form $p_i(2)$ for some $i \in [s + 1]$. Then we deduce from equation (\ref{eMinimality}) that
$$
\sum_{x \in C_{p_i(2)}} \psi_s(x) = 0,
$$
which shows that
$$
\psi_s(x_0) = -\sum_{\substack{x \in C - \{x_0\} \\ \pi_i(x) = p_i(1)}} \psi_s(x).
$$
It follows that $L(\psi_s(x_0)) \cdot \Q(\sqrt{x_0})$ is contained in the compositum 
$$
\Q(\sqrt{x_0}) \cdot \left(\prod_{\substack{x \in C - \{x_0\} \\ \pi(x) = p_i(1)}} L(\psi_s(x))\right).
$$ 
The prime $p_i(2)$ visibly does not ramify in the above field, since $\Q(\sqrt{x})$ is unramified at $p_i(2)$ for each $x$ with $\pi_i(x) = p_i(1)$ and $L(\psi_s(x)) \cdot \Q(\sqrt{x})/\Q(\sqrt{x})$ is unramified at all finite places of $\Q(\sqrt{x})$. This establishes the claim, since every place of $\Q$ either equals $p_i(1)$ or $p_i(2)$ for some $i \in [s + 1]$, or ramifies in $\Q(\sqrt{x})$ for every $x \in C$, or is unramified in $\Q(\sqrt{x})$ for all $x \in C$.

We now prove that there exists a raw cocycle $\psi'_{s + 1}(x_0)$ with $2 \cdot \psi'_{s + 1}(x) = \psi_s(x_0)$. This is equivalent to showing that $\text{Up}_{\Q(\sqrt{x_0})/\Q}(p)$ has trivial Artin symbol in $\text{Gal}(L(\psi_s(x_0)) \cdot \Q(\sqrt{x_0})/\Q(\sqrt{x_0}))$ for each prime $p$ ramifying in $\Q(\sqrt{x_0})/\Q$.

Let us first prove this for an odd prime $p$. Then, necessarily, $p$ divides $x_0$. We are going to show that the mere existence of the lift $\psi_{s + 1}(x_0)$, which itself is not guaranteed to be unramified, guarantees the existence of an unramified lift of $\psi_s(x_0)$. Thanks to the shape of the tame local Galois group,
we see that $(\psi_{s + 1}(x_0), \chi_{x_0}) \circ i_p^{*}$ induces a homomorphism 
$$
\mathbb{Z}_2 \rtimes p^{\mathbb{Z}_2} \to N(x_0)[2^{s + 1}] \rtimes \mathbb{F}_2.
$$
This homomorphism factors through the subgroup $2 \cdot \mathbb{Z}_2 \rtimes p^{\mathbb{Z}_2}$, since the order of the image of $\sigma_p$ equals exactly $2$. Indeed, elements with non-trivial second coordinate in the dihedral group $N(x_0)[2^{s + 1}] \rtimes \mathbb{F}_2$ are all involutions and clearly $\chi_{x_0}(\sigma_p) = 1$. Hence $(\psi_{s + 1}(x_0), \chi_{x_0}) \circ i_p^{*}$ induces a homomorphism $\mathbb{F}_2 \times p^{\mathbb{Z}_2} \to N(x_0)[2^{s + 1}] \rtimes \mathbb{F}_2$, which surjects on the $\FF_2$ component. This forces the restriction of the homomorphism to $G_{\Q_p(\sqrt{x_0})}$ to land in $N(x_0)[2]$, since no other element of $N(x_0)$ commutes with an element of the shape $(n, 1)$. But hence, when we project to $N(x_0)[2^s] \rtimes \mathbb{F}_2$, we obtain that
$$
\psi_s(x_0)(\text{Frob}(\text{Up}_{\Q(\sqrt{x_0})/\Q}(p))) = 0
$$
as desired. 

It remains to prove that $\text{Up}_{\Q(\sqrt{x_0})/\Q}(2)$ splits completely in $L(\psi_s(x_0)) \cdot \Q(\sqrt{x_0})/\Q(\sqrt{x_0})$ in case $(2)$ ramifies in $\Q(\sqrt{x_0})/\Q$. Recalling that $p_i(1)p_i(2) \equiv 1 \bmod 8$ for each $i \in [s + 1]$ we find that $(2)$ ramifies in $\Q(\sqrt{x})/\Q$ for every $x \in C$. Hence, the equation $2 \cdot \psi_{s + 1}(x) = \psi_s(x)$ in $\text{Cl}(\Q(\sqrt{x}))^{\vee}$ implies that $\psi_s(x)$ pairs trivially with the $2$-torsion class $\text{Up}_{\Q(\sqrt{x})/\Q}(2)$ for each $x \in C - \{x_0\}$. Therefore, keeping in mind that $\Q_2(\sqrt{x})=\Q_2(\sqrt{x_0})$, we have that 
$$
i_2(L(\psi_s(x)) \cdot \Q(\sqrt{x})) \subseteq \Q_2(\sqrt{x_0}),
$$
for each $x$ in $C - \{x_0\}$, which implies
$$
i_2(L(\psi_s(x_0)) \cdot \Q(\sqrt{x_0})) \subseteq \Q_2(\sqrt{x_0})
$$
due to the inclusion (\ref{eFoDtrick}). This means exactly that $\text{Up}_{\Q(\sqrt{x_0})/\Q}(2)$ splits completely in $L(\psi_s(x_0)) \cdot \Q(\sqrt{x_0})/\Q(\sqrt{x_0})$. This concludes our examination at $(2)$. 

We deduce the existence of a raw cocycle $\psi'_{s + 1}(x)$ satisfying $2 \cdot \psi'_{s + 1}(x) = \psi_s(x)$. In particular, we have shown that $\chi_a \in 2^s \cdot \text{Cl}(\Q(\sqrt{x_0}))^{\vee}[2^{s + 1}]$, which is the first part of the theorem. We now proceed with the proof of $(i)$. The proof of $(ii)$ is similar and is omitted. First of all notice that 
$$
\chi := \psi'_{s + 1}(x_0) - \psi_{s + 1}(x_0)
$$ 
is a quadratic character. Indeed, $\chi$ is a $1$-cocycle valued in $N(x_0)[2] = \mathbb{F}_2$, equipped with the only possible $G_\Q$-action, namely the trivial one. As we shall see, this observation will allow us to replace $\psi'_{s + 1}(x_0)$ with $\psi_{s + 1}(x_0)$ for the purpose of evaluating the Artin pairing. To take advantage of this, let us rewrite each of the Artin pairings as an Artin symbol taken from the same field extension. 

To this end, fix $x \in C$ and a prime divisor $p$ of $b$. The map $i_p$ gives a unique place above $p$ in 
$$
K := \Q(\{\sqrt{p_i(1)p_i(2)} : i \in [s + 1]\}, \sqrt{x_0}).
$$
Let us denote this place of $K$ as $\mathfrak{p}$. Observe that the extension 
$$
L(\psi_{s + 1}(x)) \cdot K/K
$$
is abelian and unramified at $\mathfrak{p}$. Hence the Artin symbol of $\mathfrak{p}$ in this extension is a well-defined element of its Galois group. We claim that
$$
\langle \text{Up}_{\Q(\sqrt{x})/\Q}(b), \chi_a \rangle = \sum_{p \mid b} \psi_{s + 1}(x)|_{G_K}(\text{Frob}(\mathfrak{p})).
$$
Indeed, $p_i(1)p_i(2)$ is a square modulo each odd prime divisor $p$ of $b$ for all $i \in [s + 1]$. Hence the residue field of $p$ in $L(\psi_{s + 1}(x)) \cdot K$ coincides with the residue field of $p$ in $L(\psi_{s + 1}(x))$. This implies that 
\[
\psi_{s + 1}(x)|_{G_K}(\text{Frob}(\mathfrak{p})) = \psi_{s + 1}(x)|_{G_{\Q(\sqrt{x})}}(\text{Frob}(\text{Up}_{\Q(\sqrt{x})/\Q}(p))). 
\]
Hence our claim follows from the definition of the Artin pairing. Therefore we conclude that
$$
\sum_{x \in C} \langle \textup{Up}_{\Q(\sqrt{x})/\Q}(b), \chi_a \rangle_{\textup{Art}_{s + 1}(x)} = \sum_{p \mid b} \left(\psi'_{s + 1}(x_0)(\text{Frob}(\mathfrak{p})) + \sum_{\substack{x \in C \\ x \neq x_0}} \psi_{s + 1}(x)(\text{Frob}(\mathfrak{p}))\right).
$$
Recalling that $\chi = \psi'_{s + 1}(x_0) - \psi_{s + 1}(x_0)$ and recalling the definition of $\psi_{s + 1}(x_0)$ we conclude that
$$
\sum_{x \in C} \langle \textup{Up}_{\Q(\sqrt{x})/\Q}(b), \chi_a \rangle_{\textup{Art}_{s + 1}(x)} = \sum_{p \mid b} \chi|_{G_K}(\text{Frob}(\mathfrak{p})).
$$
Hence it suffices to show that
$$
\sum_{p \mid b} \chi|_{G_K}(\text{Frob}(\mathfrak{p})) = 0.
$$
To this end we will conduct a more careful study of the character $\chi$. First of all observe that 
$$i
_2(L(\psi_{s + 1}(x_0))\cdot L(\psi_{s + 1}'(x_0))) \subseteq \Q_2(\sqrt{5}, \sqrt{x_0}). 
$$
Therefore, since $\Q(\chi) \subseteq L(\psi_{s + 1}(x_0))\cdot L(\psi_{s + 1}'(x_0))$, we conclude that at least one of $\chi$ and $\chi + \chi_{x_0}$ is unramified at $(2)$. Let us pick one such character unramified at $(2)$ and call this choice $\chi'$. We now claim that
$$
\chi': = \sum_{\substack{i \in [s + 1] \\ h \in [2]}} \lambda_{i, h} \chi_{p_i(h)^\ast} + \sum_{\substack{l \mid d_0 \\ l \neq 2}} \lambda_l \chi_{l^\ast},
$$
where the various $\lambda_{i, h}$ and $\lambda_l$ are in $\mathbb{F}_2$ and $n^\ast$ is the unique integer satisfying $|n^\ast| = n$ and $n^\ast \equiv 1 \bmod 4$. Indeed, the character $\chi'$ is unramified at $(2)$ and $\chi'$ certainly does not ramify at any odd prime not dividing $p_1(1) p_1(2)  \dots p_{s + 1}(1) p_{s + 1}(2) \cdot d_0$, since 
\[
\Q(\chi') \subseteq L(\psi_{s + 1}(x_0)) \cdot L(\psi_{s + 1}'(x_0)) \cdot \Q(\sqrt{x_0}). 
\]
This establishes the claim. Hence it suffices to prove that
$$
\sum_{p \mid b} \chi_{x_0}|_{G_K}(\text{Frob}(\mathfrak{p})) = 0
$$
and 
$$
\sum_{p \mid b} \chi_{t^\ast}|_{G_K}(\text{Frob}(\mathfrak{p})) = 0
$$
for each odd prime $t$ dividing $p_1(1) p_1(2)  \dots p_{s + 1}(1) p_{s + 1}(2) \cdot d_0$. The first equation is trivial, since $\chi_{x_0}$ is identically $0$ when restricted to $G_K$. Let us show the second equation. Each $p$ dividing $b$ splits completely in $\Q(\sqrt{p_i(1)p_i(2)})$ for each $i \in [s + 1]$. It follows that
\[
\chi_{t^\ast}|_{G_K}(\text{Frob}(\mathfrak{p})) = \chi_{t^\ast}|_{G_{\Q(\sqrt{x})}}(\text{Frob}(\text{Up}_{\Q(\sqrt{x})/\Q}(p))),
\]
where $x$ is any element of $C$ with $t \mid x$. This yields
$$
\sum_{p \mid b} \chi_{t^\ast}|_{G_K}(\text{Frob}(\mathfrak{p})) = \langle \text{Up}_{\Q(\sqrt{x})/\Q}(b), \chi_{t^\ast} \rangle_1 = 0,
$$ 
where the last equality follows from our assumption $\text{Up}_{\Q(\sqrt{x})/\Q}(b) \in 2 \cdot \text{Cl}(\Q(\sqrt{x}))$.
\end{proof}

We now turn to the main result on governing triples. The proof is similar to the one given for Theorem \ref{main thm: minimal triples}, however there are a few technical differences. Below we report only the steps in the proof where the argument differs.

\begin{theorem} 
\label{main thm: governing triples}
Let $(C, (\psi_{s + 1}(x))_{x \in C - \{x_0\}}, \chi_a)$ be a governing triple. Then 
\[
\chi_{p_{s + 1}(1) \cdot a} \in 2^{s} \cdot \textup{Cl}(\Q(\sqrt{x_0}))^{\vee}[2^{s + 1}].
\]
Suppose that $b \mid d_0$ is such that
$$
\textup{Up}_{\Q(\sqrt{x})/\Q}(b) \in 2^s \cdot \textup{Cl}(\Q(\sqrt{x}))[2^{s + 1}]
$$
for each $x \in C$. Then every $p \mid b$ splits completely in $M(\phi_{\{p_i(1)p_i(2) : i \in [s]\}; p_{s + 1}(1)p_{s + 1}(2)}(\mathfrak{G}))/\Q$, is unramified in $L(\phi_{\{p_i(1)p_i(2) : i \in [s]\}; p_{s + 1}(1)p_{s + 1}(2)}(\mathfrak{G}))/\Q$ and 
$$
\sum_{x \in C} \langle \textup{Up}_{\Q(\sqrt{x})/\Q}(b), \chi_a \rangle_{\textup{Art}_{s + 1}(x)} = \sum_{p \mid b} \phi_{\{p_i(1)p_i(2) : i \in [s]\}; p_{s + 1}(1)p_{s + 1}(2)}(\mathfrak{G})(\textup{Frob}(p)).
$$
\end{theorem}

\begin{proof}
We start by showing that $\chi_{p_{s + 1}(1) \cdot a} \in 2^{s} \cdot \textup{Cl}(\Q(\sqrt{x_0}))^{\vee}[2^{s + 1}]$. Consider the $1$-cochain
$$
\psi_{s + 1}(x_0) := \phi_{\{p_i(1)p_i(2) : i \in [s]\}; p_{s + 1}(1)p_{s + 1}(2)}(\mathfrak{G}) - \sum_{\substack{x \in C \\ x \neq x_0}} \psi_{s + 1}(x).
$$
We have
\begin{align*}
2^s \cdot \psi_{s + 1}(x_0) &= 2^s \cdot \phi_{\{p_i(1)p_i(2) : i \in [s]\}; p_{s + 1}(1)p_{s + 1}(2)}(\mathfrak{G}) - \sum_{\substack{x \in C \\ x \neq x_0}} \chi_{\pi_{s + 1}(x) \cdot a} \\
&= 0 - \chi_{p_{s + 1}(1) \cdot a} =  \chi_{p_{s + 1}(1) \cdot a},
\end{align*}
since the $1$-cochain $\phi_{\{p_i(1)p_i(2) : i \in [s]\}; p_{s + 1}(1)p_{s + 1}(2)}(\mathfrak{G})$ is valued in $\mathbb{F}_2$ and $s \geq 1$. Let us next prove that $\psi_{s + 1}(x_0):G_\Q \to N(x_0)[2^{s + 1}]$ is in fact a $1$-cocycle for $N(x_0)$. Thanks to Proposition \ref{key calculation of cocycles} we have that 
\begin{multline*}
(d_{x_0}\psi_{s + 1}(x_0))(\sigma, \tau) = \sum_{\substack{\varnothing \neq T \subseteq [s + 1] \\ s + 1 \not \in T}} \chi_{\{p_i(1)p_i(2) : i \in T\}}(\sigma) \cdot \\
\left((-1)^{|T|+1+\chi_{x_0}(\sigma)}\left(\sum_{x \in C_T}\psi_{s - |T| + 1}(x)(\tau)\right) - \phi_{\{p_i(1)p_i(2) : i \in [s]\}; p_{s + 1}(1)p_{s + 1}(2)}(\mathfrak{G})(\tau)\right),
\end{multline*}
where we have used the third condition of Definition \ref{def: governing triples} to remove the terms with $s + 1 \in T$. It follows from the second condition of Definition \ref{def: governing triples} that the above expression is identically zero.

Certainly, $\psi_s(x_0) := 2 \cdot \psi_{s + 1}(x_0)$ is a $1$-cocycle with values in $N(x_0)[2^s]$. We claim that $\psi_s(x_0)$ once restricted to $G_{\Q(\sqrt{x_0})}$ naturally lands in $\text{Cl}(\Q(\sqrt{x_0}))^{\vee}[2^s]$ so that it yields a cyclic degree $2^s$ extension of $\Q(\sqrt{x_0})$ unramified at all finite places. Following the proof of Theorem \ref{main thm: minimal triples}, we obtain the claim for all places lying above $(2)$, all odd primes dividing $d_0$ and all odd primes not dividing any $x \in C$. 

It remains to treat the places lying above a prime of the form $p_j(h)$ with $j \in [s + 1]$ and $h \in [2]$. Let us begin with the case $j \in [s]$ and $h = 2$. Then we can rewrite $\psi_s(x_0)$ as
$$
-\sum_{x \in C_{p_j(2)}} \psi_s(x) - \sum_{x \in C_{p_j(1)} - \{x_0\}} \psi_s(x).
$$
We now invoke the second condition of Definition \ref{def: governing triples}, with $T := \{j\}$, and we get
$$
\psi_{s}(x_0) = \phi_{\{p_i(1)p_i(2) : i \in [s] - \{j\}\}; p_{s + 1}(1)p_{s + 1}(2)}(\mathfrak{G}) - \sum_{x \in C_{p_j(1)} - \{x_0\}}\psi_s(x).
$$
By construction of the $1$-cochain $\phi_{\{p_i(1)p_i(2) : i \in [s] - \{j\}\}; p_{s + 1}(1)p_{s + 1}(2)}(\mathfrak{G})$ and by our assumption $j \in [s]$ we deduce that 
$$
\phi_{\{p_i(1)p_i(2) : i \in [s] - \{j\}\}; p_{s + 1}(1)p_{s + 1}(2)}(\mathfrak{G})(\sigma_{p_j(2)}) = 0.
$$
Now notice that
$$
\sum_{x \in C_{p_j(1)} - \{x_0\}} \psi_s(x)(\sigma_{p_j(2)}) = 0,
$$
since every individual field $L(\psi_s(x))/\Q$ is unramified above $p_j(2)$ for each $x \in C_{p_j(1)}$. This demonstrates that the homomorphism $(\psi_s(x_0), \chi_{x_0})$ maps $\sigma_{p_j(2)}$ to the identity and therefore $p_j(2)$ is unramified in the extension $L(\psi_s(x_0))/\Q$ by Proposition \ref{ramification read off by inertia}.

We now consider $j = s + 1$ and $h = 2$. We still rewrite $\psi_s(x_0)$ as
$$
-\sum_{x \in C_{p_j(2)}} \psi_s(x) - \sum_{x \in C_{p_j(1)} - \{x_0\}} \psi_s(x),
$$
but this time we invoke the third condition of Definition \ref{def: governing triples} with $T = \{s + 1\}$ to obtain that
$$
\psi_s(x_0) = -\sum_{x \in C_{p_j(1)} - \{x_0\}} \psi_s(x).
$$
The field of definition of $\psi_s(x_0)$ is contained in 
$$
\prod_{x \in C_{p_{s + 1}(1)} - \{x_0\}} L(\psi_s(x)).
$$ 
The prime $p_{s + 1}(2)$ is unramified in the above field, since $p_{s + 1}(2)$ is unramified in $\Q(\sqrt{x})/\Q$ for each $x \in C_{p_{s + 1}(1)}$ and furthermore $L(\psi_s(x))\Q(\sqrt{x})/\Q(\sqrt{x})$ is unramified at all finite places. 

We finally consider the primes of the form $p_j(1)$ for some $j \in [s + 1]$. We observe that
$$
L(\psi_s(x_0)) \subseteq \prod_{x \in C - \{x_0\}}L(\psi_s(x))
$$
and that $\sigma_{p_j(1)}$ has order at most $2$ in $\text{Gal}(L(\psi_s(x))/\Q)$ for each $x$ in $C - \{x_0\}$. Hence it has order at most $2$ in $\text{Gal}(L(\psi_s(x_0))/\Q)$, which, by means of Proposition \ref{ramification read off by inertia}, implies that the ramification index of $p_j(1)$ in $L(\psi_s(x_0))\Q(\sqrt{x_0})/\Q$ is exactly $2$. But the prime $p_j(1)$ ramifies in the extension $\Q(\sqrt{x_0})/\Q$ and therefore the place $\text{Up}_{\Q(\sqrt{x_0})/\Q}(p_j(1))$ is unramified in $L(\psi_s(x_0))\Q(\sqrt{x_0})/\Q(\sqrt{x_0})$. We have now verified that $\psi_s(x_0)$ restricts to an unramified character of $G_{\Q(\sqrt{x_0})}$. 

At this point we follow the argument in the proof of Theorem \ref{main thm: minimal triples} to deduce that $\psi_s(x_0)$ is in $2 \cdot \text{Cl}(\Q(\sqrt{x_0}))^{\vee}$. Indeed, all that we used in that proof was the existence of a lift (by multiplication of $2$) of $\psi_s(x_0)$ in the space of cocycles from $G_\Q$ to $N(x_0)$, which, of course, is also available in this proof (namely $\psi_{s + 1}(x_0)$). Let $\psi'_{s + 1}(x_0)$ be a raw cocycle with $2 \cdot \psi'_{s + 1}(x_0) = \psi_s(x_0)$. The examination of the quadratic character
$$
\chi := \psi'_{s + 1}(x_0) - \psi_{s + 1}(x_0)
$$   
can be reproduced almost verbatim also in this argument with the only difference being that we invoke Proposition \ref{normalized expansions are unramified} to control the ramification coming from the $1$-cochain 
\[
\phi_{\{p_i(1)p_i(2) : i \in [s]\}; p_{s + 1}(1)p_{s + 1}(2)}(\mathfrak{G}).
\]
We now justify the claim that each prime $p$ dividing $b$ splits completely in the extension 
\[
M(\phi_{\{p_i(1)p_i(2) : i \in [s]\}; p_{s + 1}(1)p_{s + 1}(2)}(\mathfrak{G}))/\Q
\]
and is unramified in the extension
\[
L(\phi_{\{p_i(1)p_i(2) : i \in [s]\}; p_{s + 1}(1)p_{s + 1}(2)}(\mathfrak{G}))/\Q. 
\]
This last conclusion is just Proposition \ref{normalized expansions are unramified}. In order to show that $p$ splits completely in $M(\phi_{\{p_i(1)p_i(2) : i \in [s]\}; p_{s + 1}(1)p_{s + 1}(2)}(\mathfrak{G}))$, we invoke the second condition of Definition \ref{def: governing triples}. Then we see that it suffices to show that
\begin{align}
\label{eipL}
i_p(L(\psi_s(x))) \subseteq \Q_p(\sqrt{x_0})
\end{align}
for each $x$ in $C - \{x_0\}$. Indeed, this forces 
\[
i_p(M(\phi_{\{p_i(1)p_i(2) : i \in [s]\}; p_{s + 1}(1)p_{s + 1}(2)}(\mathfrak{G}))) \subseteq \Q_p^{\text{unr}} \cap \Q_p(\sqrt{x_0}) = \Q_p
\]
as we want. But equation (\ref{eipL}) follows from Lemma \ref{lX}. From here onwards, the rest of the proof is identical to the proof of Theorem \ref{main thm: minimal triples}.
\end{proof}

\section{Combinatorial results}
\label{sCombinatorics}
In this section we recall for the reader's convenience the main combinatorial results from Smith's work \cite[Section 3-4]{Smith}.

\subsection{Additive systems}
One of the key tools in Smith's work is the notion of an additive system. It provides a convenient abstract framework for our algebraic results.

\begin{mydef}
Let $X_1, \dots, X_r$ be non-empty finite sets, let $X = X_1 \times \dots \times X_r$ and let $S \subseteq [r]$. An additive system on $(X, S)$ is a tuple $(C_T, C_T^{\textup{acc}}, F_T, A_T)_{T \subseteq S}$, indexed by the subsets $T \subseteq S$, satisfying the following properties
\begin{itemize}
\item we have $C_T \subseteq C_T^{\textup{acc}} \subseteq \textup{Cube}(X, T)$, $F_T: C_T \rightarrow A_T$ is a function and $A_T$ is a finite $\FF_2$-vector space;
\item we have
\[
C_T^{\textup{acc}} = \{\bar{x} \in C_T : F_T(\bar{x}) = 0\}
\]
and furthermore
\[
C_T = \{\bar{x} \in \textup{Cube}(X, T) : \bar{x}(T - \{i\}) \subseteq C_{T - \{i\}}^{\textup{acc}} \textup{ for all } i \in T\}
\]
for all $\varnothing \neq T \subseteq S$;
\item suppose that $\bar{x}_1, \bar{x}_2, \bar{x}_3 \in C_T$ and suppose that there exists $i \in T$ and $p_1, p_2, p_3 \in X_i$ such that
\[
\pi_{[r] - \{i\}}(\bar{x}_1) = \pi_{[r] - \{i\}}(\bar{x}_2) = \pi_{[r] - \{i\}}(\bar{x}_3)
\]
and
\[
\pi_i(\bar{x}_1) = (p_1, p_2), \pi_i(\bar{x}_2) = (p_2, p_3), \pi_i(\bar{x}_3) = (p_1, p_3).
\]
Then we have
\begin{align}
\label{eAdditive}
F_T(\bar{x}_1) + F_T(\bar{x}_2) = F_T(\bar{x}_3).
\end{align}
\end{itemize}
\end{mydef}

The most important feature of additive systems is that we have good control over the density of $C_S^{\text{acc}}$.

\begin{prop}
\label{pAS}
Take non-empty finite sets $X_1, \dots, X_r$, take $X = X_1 \times \dots \times X_r$ and take a subset $S \subseteq [r]$. Let $(C_T, C_T^{\textup{acc}}, F_T, A_T)_{T \subseteq S}$ be an additive system on $(X, S)$ such that $|A_T| \leq a$ for all $T \subseteq S$, and write $\delta$ for the density of $C_\varnothing^{\textup{acc}}$ in $X$. Then we have
\[
\frac{|C_S^{\textup{acc}}|}{|\textup{Cube}(X, S)|} \geq \delta^{2^{|S|}} \cdot a^{-3^{|S|}}.
\]
\end{prop}

\begin{proof}
This is first proven in \cite[Proposition 3.2]{Smith} and reproven in \cite[Proposition 2.2]{KP4}.
\end{proof}

\subsection{Ramsey theory}
The following result allows us to find structured subsets of any $Y \subseteq X$ provided that $Y$ has sufficiently large density in $X$.

\begin{lemma}
\label{lFindBox}
Take non-empty finite sets $X_1, \dots, X_r$ and take $X = X_1 \times \dots \times X_r$. Suppose that $Y$ is a subset of $X$ of cardinality at least $\delta \cdot |X|$. We further assume that $2^{-r - 1} > \delta > 0$ and we let $b \geq 1$ be an integer satisfying
\[
b \leq \left(\frac{\log \min_{i \in [r]} |X_i|}{5 \log \delta^{-1}}\right)^{1/(r - 1)}.
\]
Then there exist subsets $Z_i \subseteq X_i$ with $|Z_i| = b$ for all $i \in [r]$ such that
\[
Z_1 \times \dots \times Z_r \subseteq Y.
\]
\end{lemma}

\begin{proof}
This is \cite[Lemma 4.1]{Smith}, where their $d$ is our $r$ and their $r$ is our $b$.
\end{proof}

\begin{remark}
As Smith \cite[p. 6]{Smith} already remarks, this lemma is sharp up to a change of constant. Indeed, take a random subset $Y$ of $[2^N]^r$. Then the probability that $Y$ contains a given product set $Z_1 \times \dots \times Z_r$, where $|Z_i| = b$ for each $i \in [r]$, equals $(1/2)^{b^r}$. There are at most $2^{N r b}$ such sets $Z_1 \times \dots \times Z_r$. Hence we see that, as soon as $b > (rN)^{1/(r -1)}$, a random subset $Y$ of $[2^N]^r$ fails the conclusion of the lemma with positive probability.
\end{remark}

We now state Smith's main combinatorial result. It provides the crucial connection between our equidistribution results based on the Chebotarev density theorem and the reflection principles. Let $X = X_1 \times \dots \times X_r$ and let $S \subseteq [r]$. Define
\[
V = \text{Map}(X, \FF_2), \quad W = \text{Map}(\text{Cube}(X, S), \FF_2).
\]
We say that $\bar{x} \in \text{Cube}(X, S)$ is degenerate if there exists $i \in S$ with
\[
\text{pr}_1(\pi_i(\bar{x})) = \text{pr}_2(\pi_i(\bar{x})).
\]
We define $\Sigma: V \rightarrow W$ to be the linear map given by
\[
\Sigma F(\bar{x}) =
\left\{
	\begin{array}{ll}
		\sum_{x \in \bar{x}(\varnothing)} F(x) & \mbox{if } \bar{x} \text{ is not degenerate} \\
		0 & \mbox{if } \bar{x} \text{ is degenerate.}
	\end{array}
\right.
\]
We denote the image of $\Sigma$ by $\text{Add}(X, S)$. More generally, we define for a subset $Y \subseteq X$ a map from $\text{Map}(Y, \FF_2)$ to $\text{Map}(\text{Cube}(X, S), \FF_2)$ by
\[
\Sigma F(\bar{x}) =
\left\{
	\begin{array}{ll}
		\sum_{x \in \bar{x}(\varnothing)} F(x) & \mbox{if } \bar{x} \text{ is not degenerate and } \bar{x}(\varnothing) \subseteq Y \\
		0 & \mbox{if } \bar{x} \text{ is degenerate.}
	\end{array}
\right.
\]
By abuse of notation, we will also denote this map by $\Sigma$.

\begin{lemma}
\label{lAdd}
Let $X_1, \dots, X_r$ be finite, non-empty sets, let $X = X_1 \times \dots \times X_r$ and let $S \subseteq [r]$. Then
\[
\dim_{\FF_2} \textup{Add}(X, S) = \prod_{i \in S} (|X_i| - 1) \cdot \prod_{i \in [r] - S} |X_i|.
\]
\end{lemma}

\begin{proof}
This follows from \cite[Proposition 9.3]{KP1} by taking $l = 2$.
\end{proof}

Given an additive system $\mathfrak{A} = (C_T, C_T^{\textup{acc}}, F_T, A_T)_{T \subseteq S}$ on $(X, S)$, we set
\[
C(\mathfrak{A}) := \bigcap_{i \in S} \left\{\bar{x} \in \text{Cube}(X, S) : \bar{x}(S - \{i\}) \cap C_{S - \{i\}}^{\text{acc}} \neq \varnothing\right\}.
\]
We say that $\mathfrak{A}$ is $(a, S)$-acceptable if $|A_T| \leq a$ for all $T \subseteq S$ and $\bar{x} \in C(\mathfrak{A})$ implies $\bar{x}(\varnothing) \subseteq C_\varnothing^{\text{acc}}$.

\begin{prop}
\label{pARinput}
There exists an absolute constant $A > 0$ such that the following holds. Take non-empty finite sets $X_1, \dots, X_r$, take $X = X_1 \times \dots \times X_r$ and take $S \subseteq [r]$. Let $a \geq 2$ and $\epsilon > 0$ be such that $\epsilon < a^{-1}$ and
\[
\log \min_{i \in S} |X_i| \geq A \cdot 6^r \cdot \log \epsilon^{-1}.
\]
Then there exists $g \in \textup{Add}(X, S)$ such that for all $(a, S)$-acceptable additive systems 
\[
\mathfrak{A} = (C_T, C_T^{\textup{acc}}, F_T, A_T)_{T \subseteq S}
\]
on $(X, S)$ and for all $F: C_\varnothing^{\textup{acc}} \rightarrow \FF_2$ satisfying $\Sigma F(\bar{x}) = g(\bar{x})$ for all $\bar{x} \in C(\mathfrak{A})$, we have
\[
\frac{|C_\varnothing^{\textup{acc}}|}{2} - \epsilon \cdot |X| \leq |F^{-1}(0)| \leq \frac{|C_\varnothing^{\textup{acc}}|}{2} + \epsilon \cdot |X|.
\]
\end{prop}

\begin{proof}
This is \cite[Proposition 4.4]{Smith}.
\end{proof}

\subsection{Galois groups}
In this subsection we describe the structure of some important Galois groups that we shall encounter later. To do so, we will introduce some notation. If $X_1, \dots, X_r$ are non-empty, finite sets of odd prime numbers, we write $K(X_1 \times \dots \times X_r)$ for the compositum of $\Q(\sqrt{p})$ with $p$ lying in some $X_j$. If $\bar{x} \in \text{Cube}(X, S)$ and $d \neq 1$ is an integer coprime to all elements of $X$, then we say that the expansion map
\[
\phi_{\bar{x}; d}
\]
exists if $\bar{x}$ is degenerate or if the expansion map 
\[
\phi_{\{\text{pr}_1(\pi_i(\bar{x})) \cdot \text{pr}_2(\pi_i(\bar{x})) : i \in S\}; d}(\mathfrak{G})
\]
exists. In case $\bar{x}$ is degenerate, we define the field of definition of $\phi_{\bar{x}; d}$ to be $\Q$. The coming results are very similar to \cite[Proposition 2.4]{Smith}. 

\begin{lemma}
\label{lTopdegree}
Let $r \in \Z_{\geq 1}$. Take non-empty, disjoint, finite sets of odd prime numbers $X_1, \dots, X_r$ with product $X$. Let $i \in [r]$ and suppose that the map $\phi_{\pi_{[r] - \{i\}}(\bar{x}); \textup{pr}_1(\pi_i(\bar{x})) \textup{pr}_2(\pi_i(\bar{x}))}$ exists for all $\bar{x} \in \textup{Cube}(X, [r])$. For $S \subseteq [r] - \{i\}$ write
\[
M_S := K(X) \prod_{\bar{x} \in \textup{Cube}(X, S \cup \{i\})} L(\phi_{\pi_S(\bar{x}); \textup{pr}_1(\pi_i(\bar{x})) \textup{pr}_2(\pi_i(\bar{x}))}).
\]
Then we have
\[
\log_2 \left[\prod_{\substack{S \subseteq [r] - \{i\} \\ |S| = j + 1}} M_S : \prod_{\substack{S' \subseteq [r] - \{i\} \\ |S'| = j}} M_{S'}\right] = \sum_{\substack{S \subseteq [r] - \{i\} \\ |S| = j + 1}} \prod_{k \in S \cup \{i\}} (|X_k| - 1).
\]
for all integers $0 \leq j \leq r - 2$.
\end{lemma}

\begin{proof}
Write
\[
M_j := \prod_{\substack{S' \subseteq [r] - \{i\} \\ |S'| = j}} M_{S'}.
\]
Let $S \subseteq [r] - \{i\}$ with $|S| = j + 1$. Note that 
\[
\phi_{\pi_S(\bar{x}); \textup{pr}_1(\pi_i(\bar{x})) \textup{pr}_2(\pi_i(\bar{x}))}
\] 
is a quadratic character when restricted to $M_j$ and gives a central $\FF_2$-extension of $\Gal(M_j/\Q)$ by Proposition \ref{prop: field of def of exp maps}. Enumerate $X_k$ as $\{p_{k, 1}, \dots, p_{k, |X_k|}\}$ and define
\[
T_{k, l} = \{p_{k, l}, p_{k, l + 1}\}
\]
for $l \in [|X_k| - 1]$. For $T \subseteq [r] - \{i\}$, let $\text{Basis}(X, T)$ be the set of $\bar{x} \in \text{Cube}(X, T \cup \{i\})$ such that for every $k \in T \cup \{i\}$ there is some $l$ satisfying
\[
\{\text{pr}_1(\pi_k(\bar{x})), \text{pr}_2(\pi_k(\bar{x}))\} = T_{k, l}.
\]
It follows from Proposition \ref{prop: field of def of exp maps} that
\[
\bigcup_{\substack{S \subseteq [r] - \{i\} \\ |S| = j + 1}} \bigcup_{\bar{x} \in \text{Basis}(X, S)} \{\phi_{\pi_S(\bar{x}); \textup{pr}_1(\pi_i(\bar{x})) \textup{pr}_2(\pi_i(\bar{x}))}\}
\]
generates the dual of $\Gal(M_{j + 1}/M_j)$. We claim that the above characters are also linearly independent, which clearly implies the lemma. Here we caution the reader that two distinct elements $\bar{x}_1, \bar{x}_2 \in \text{Basis}(X, S)$ can give rise to the same character. This happens precisely when
\[
\{\text{pr}_1(\pi_k(\bar{x}_1)), \text{pr}_2(\pi_k(\bar{x}_1))\} = \{\text{pr}_1(\pi_k(\bar{x}_2)), \text{pr}_2(\pi_k(\bar{x}_2))\}
\]
for all $k \in S \cup \{i\}$. This defines an equivalence relation on $\text{Basis}(X, S)$, which we denote by $\sim$. To prove the claim, let $c_{\bar{x}} \in \mathbb{F}_2$ be such that
\begin{align}
\label{eLinearPhi}
\sum_{\substack{S \subseteq [r] - \{i\} \\ |S| = j + 1}} \sum_{\bar{x} \in \text{Basis}(X, S)/\sim} c_{\bar{x}} \cdot \phi_{\pi_S(\bar{x}); \textup{pr}_1(\pi_i(\bar{x})) \textup{pr}_2(\pi_i(\bar{x}))}(\sigma) = 0
\end{align}
for all $\sigma \in G_{M_j}$. We define the operator $\beta_j$ by
\[
\beta_j \phi = \phi([\sigma_1, [\sigma_2, [\dots [\sigma_{j + 1}, \sigma_{j + 2}]]]])
\]
for $\phi \in \text{Map}(G_{M_j}, \FF_2)$ and $\sigma_1, \dots, \sigma_{j + 2} \in G_\Q$. This is well-defined, since the nilpotency degree of $\Gal(M_j/\Q)$ is $j + 1$. Writing $\phi$ for $\phi_{\pi_S(\bar{x}); \textup{pr}_1(\pi_i(\bar{x})) \textup{pr}_2(\pi_i(\bar{x}))}$, we compute
\begin{align*}
\phi([\sigma, \tau]) &= \phi(\sigma \tau) + \phi(\tau \sigma) + d\phi([\sigma, \tau], \tau \sigma) \\
&= \phi(\sigma \tau) + \phi(\tau \sigma) \\
&= d\phi(\sigma, \tau) + d\phi(\tau, \sigma),
\end{align*}
where the second equality uses equation (\ref{eSmith22}) and the fact that quadratic characters vanish on commutators. Utilizing equation (\ref{eSmith22}) again and inducting on $j$ yields
\[
\beta_j \phi_{\pi_S(\bar{x}); \textup{pr}_1(\pi_i(\bar{x})) \textup{pr}_2(\pi_i(\bar{x}))} = \sum_{\substack{f : [j + 2] \rightarrow S \cup \{i\} \\ f \text{ is a bijection} \\ f(j + 1) = i \text{ or } f(j + 2) = i}} \prod_{1 \leq k \leq j + 2} \chi_{\textup{pr}_1(\pi_{f(k)}(\bar{x})) \textup{pr}_2(\pi_{f(k)}(\bar{x}))}(\sigma_k),
\]
where we again use that $d\phi(\tau, \sigma) = 0$ if $\tau$ is a commutator. Hence $\beta_j$ induces a homomorphism from $\Gal(M_{j + 1}/M_j)^\vee$ to the set of bilinear maps from $\Gal(K(X)/\Q)^{j + 2}$ to $\FF_2$. Observe that the bilinear maps
\[
\left\{(\sigma_1, \dots, \sigma_{j + 2}) \mapsto \prod_{1 \leq m \leq j + 2} \chi_{p_{k_m, l_m} \cdot p_{k_m, l_m + 1}}(\sigma_m)\right\}_{k_1 \in [r], \dots, k_{j + 2} \in [r], l_1 \in [|X_{k_1}| - 1], \dots, l_{j + 2} \in [|X_{k_{j + 2}}] - 1]}
\]
are linearly independent. Therefore, applying $\beta_j$ to equation (\ref{eLinearPhi}), we see that $c_{\bar{x}} = 0$, which completes the proof.
\end{proof}

\begin{lemma}
\label{lMdegree}
Let $r \in \Z_{\geq 1}$ be given. Let $X_1, \dots, X_r$ be non-empty, disjoint, finite sets of odd prime numbers, let $X$ be their product and let $i \in [r]$. Assume that the map 
\[
\phi_{\pi_{[r] - \{i\}}(\bar{x}); \textup{pr}_1(\pi_i(\bar{x})) \textup{pr}_2(\pi_i(\bar{x}))}
\]
exists for all $\bar{x} \in \textup{Cube}(X, [r])$. Write
\[
M_\circ(X) := K(X) \prod_{S \subset [r] - \{i\}} \prod_{\bar{x} \in \textup{Cube}(X, S \cup \{i\})} L(\phi_{\pi_S(\bar{x}); \textup{pr}_1(\pi_i(\bar{x})) \textup{pr}_2(\pi_i(\bar{x}))})
\]
Then we have
\begin{align}
\label{eMdegree}
\left[M_\circ(X) : K(X)\right] = \prod_{\varnothing \subset S \subset [r] - \{i\}} 2^{\prod_{j \in S \cup \{i\}} (|X_j| - 1)}.
\end{align}
Also let $Y_1, \dots, Y_r$ be non-empty, disjoint, finite sets of odd prime numbers and let $Y$ be their product. Assume that the map $\phi_{\pi_{[r] - \{i\}}(\bar{x}); \textup{pr}_1(\pi_i(\bar{x})) \textup{pr}_2(\pi_i(\bar{x}))}$ exists for all $\bar{x} \in \textup{Cube}(Y, [r])$. Suppose that $|X_j \cap Y_j| = 1$ for all $j \in [r]$ and $X_j \cap Y_k = \varnothing$ for distinct $j$ and $k$. Then we have 
\[
M_\circ(X) \cap M_\circ(Y) = \Q(\sqrt{p_1}, \dots, \sqrt{p_r}),
\]
where $p_j$ is the unique element of $X_j \cap Y_j$.
\end{lemma}

\begin{proof}
Equation (\ref{eMdegree}) follows immediately from Lemma \ref{lTopdegree}. For the second part, $\supseteq$ is obvious, so it remains to prove $\subseteq$. But the arguments from Lemma \ref{lTopdegree} imply that $M_\circ(X)/K(X)K(Y)$ and $M_\circ(Y)/K(X)K(Y)$ are disjoint extensions. Then $M_\circ(X) \cap M_\circ(Y)$ is a subfield of $K(X) K(Y)$ only ramified at $2$ and $p_1, \dots, p_r$. This readily implies $\subseteq$.
\end{proof}

\begin{lemma}
\label{lPhiAdd}
Take an integer $r \in \Z_{\geq 2}$. Let $X_1, \dots, X_r$ be non-empty, disjoint, finite sets of odd prime numbers, let $X$ be their product and let $i \in [r]$. Suppose that the map 
\[
\phi_{\pi_{[r] - \{i\}}(\bar{x}); \textup{pr}_1(\pi_i(\bar{x})) \textup{pr}_2(\pi_i(\bar{x}))}
\]
exists for all $\bar{x} \in \textup{Cube}(X, [r])$. Write 
\[
M_\circ := K(X) \prod_{S \subset [r] - \{i\}} \prod_{\bar{x} \in \textup{Cube}(X, S \cup \{i\})} L(\phi_{\pi_S(\bar{x}); \textup{pr}_1(\pi_i(\bar{x})) \textup{pr}_2(\pi_i(\bar{x}))})
\]
and
\[
M := K(X) \prod_{\bar{x} \in \textup{Cube}(X, [r])} L(\phi_{\pi_{[r] - \{i\}}(\bar{x}); \textup{pr}_1(\pi_i(\bar{x})) \textup{pr}_2(\pi_i(\bar{x}))}).
\]
Then the map $f$, that sends $\sigma \in \Gal(M/M_\circ)$ to the function $g_\sigma$ in $\textup{Add}(X, [r])$ given by
\[
\bar{x} \mapsto \phi_{\pi_{[r] - \{i\}}(\bar{x}); \textup{pr}_1(\pi_i(\bar{x})) \textup{pr}_2(\pi_i(\bar{x}))}(\sigma),
\]
is a group isomorphism.
\end{lemma}

\begin{proof}
First of all, we have to check that $f$ is well-defined, that is $g_\sigma \in \text{Add}(X, [r])$. This follows from Proposition \ref{prop: field of def of exp maps}. Next recall that 
\[
\phi_{\pi_{[r] - \{i\}}(\bar{x}); \textup{pr}_1(\pi_i(\bar{x})) \textup{pr}_2(\pi_i(\bar{x}))}
\]
is a quadratic character when restricted to $G_{M_\circ}$. Since $\sigma$ fixes $M_\circ$, it follows that $f$ is a group homomorphism.

Furthermore, $\phi_{\pi_{[r] - \{i\}}(\bar{x}); \textup{pr}_1(\pi_i(\bar{x})) \textup{pr}_2(\pi_i(\bar{x}))}$ gives a central $\FF_2$-extension of $\Gal(M_\circ/\Q)$ by Proposition \ref{prop: field of def of exp maps}, and hence the restrictions of the $\phi_{\pi_{[r] - \{i\}}(\bar{x}); \textup{pr}_1(\pi_i(\bar{x})) \textup{pr}_2(\pi_i(\bar{x}))}$ to $G_{M_\circ}$ together generate the dual of $\Gal(M/M_\circ)$. This shows that $f$ is injective. Since $r \geq 2$, we deduce from Lemma \ref{lTopdegree} that
\[
|\Gal(M/M_\circ)| = 2^{\prod_{j \in [r]} (|X_j| - 1)}.
\]
Therefore our lemma follows from Lemma \ref{lAdd}.
\end{proof}

Finally, we will need a version of the above lemmata in case the pointer is an arbitrary fixed squarefree integer $d \neq 1$.

\begin{lemma}
\label{lConstantPointer}
Let $r \in \Z_{\geq 1}$ be given. Let $X_1, \dots, X_r$ be non-empty, disjoint, finite sets of odd prime numbers and write $X = X_1 \times \dots \times X_r$. Also suppose that $d \neq 1$ is a squarefree integer coprime with all the primes in $X$. Suppose that the map
\[
\phi_{\bar{x}; d}
\]
exists for all $\bar{x} \in \textup{Cube}(X, [r])$. Write
\[
M_\circ(X) := \prod_{i \in [r]} \prod_{\bar{x} \in \textup{Cube}(X, [r])} L(\phi_{\pi_{[r] - \{i\}}(\bar{x}); \textup{pr}_1(\pi_i(\bar{x})) \textup{pr}_2(\pi_i(\bar{x}))}) \times \prod_{T \subset [r]} \prod_{\bar{x} \in \textup{Cube}(X, T)} L(\phi_{\bar{x}; d})
\]
and
\[
M(X) := M_\circ(X) \prod_{\bar{x} \in \textup{Cube}(X, [r])} L(\phi_{\bar{x}; d}).
\]
Then we have
\[
[M(X) : M_\circ(X)] = 2^{\prod_{j \in [r]} (|X_j| - 1)}
\]
and
\[
[M_\circ(X) : K(X)] = \prod_{\varnothing \subset S \subset [r]} 2^{\prod_{j \in S} (|X_j| - 1)} \cdot \prod_{\substack{(i, S) \\ \{i\} \subset S \subseteq [r] \\ i \neq \max(S)}} 2^{\prod_{j \in S} (|X_j| - 1)}.
\]
Furthermore, the map $f$, that sends $\sigma \in \Gal(M(X)/M_\circ(X))$ to $g_\sigma \in \textup{Add}(X, [r])$ given by
\[
\bar{x} \mapsto \phi_{\bar{x}; d}(\sigma),
\]
is a group isomorphism between $\Gal(M(X)/M_\circ(X))$ and $\textup{Add}(X, [r])$. Finally, let $Y_1, \dots, Y_r$ satisfy the same conditions as $X_1, \dots, X_r$ and write $Y$ for their product. We also assume that $|X_j \cap Y_j| = 1$ for all $j \in [r]$ and $X_j \cap Y_k = \varnothing$ for distinct $j$ and $k$. Then
\[
M_\circ(X) \cap M_\circ(Y) = \Q(\sqrt{p_1}, \dots, \sqrt{p_r}, \sqrt{d}),
\]
where $p_j$ is the unique element of $X_j \cap Y_j$.
\end{lemma}

\begin{proof}
The proof proceeds among the same lines as for the other lemmata presented in this subsection. The condition $i \neq \max(S)$ is imposed because we have the relation
\[
\sum_{i \in S} \phi_{\pi_{S - \{i\}}(\bar{x}); \textup{pr}_1(\pi_i(\bar{x})) \textup{pr}_2(\pi_i(\bar{x}))} = \chi_{\{\textup{pr}_1(\pi_j(\bar{x})) \textup{pr}_2(\pi_j(\bar{x})) : j \in S\}}
\]
for all $|S| \geq 1$. Using the operator $\beta_{|S| - 2}$ one sees that there are no further relations, since for every $\varnothing \subsetneq T \subsetneq S$, there exists a bijection $f: [|S|] \rightarrow S$ such that
\[
|\{i \in T : f(|S| - 1) = i \text{ or } f(|S|) = i\}|
\]
is odd.
\end{proof}

\section{Equidistribution of the first Artin pairing}
In this section we will recall and reprove some important results from \cite[Section 4-5]{CKMP}, which are directly based on \cite[Section 5-6]{Smith} of Smith's work. Although most of the material in this section is rather similar to \cite[Section 4-5]{CKMP} and \cite[Section 5-6]{Smith}, we have simplified some of the proofs.

The main theorem of this section is Theorem \ref{t4rank}, which improves on earlier work of Fouvry and Kl\"uners \cite{FK1}. In \cite[Corollary 2]{FK1} one finds the limiting distribution\footnote{We remark that Fouvry and Kl\"uners proceed by computing the moments of the $4$-rank, while we do not. In particular there is no direct analogue of their \cite[Theorem 3]{FK1} in our work.} of the $4$-rank of special discriminants. Here we give an error term and prove that the implied constant is effectively computable. It is worth mentioning that the work in \cite[Theorem 5.15]{CKMP} also provides an error term, but there it was not shown that the implied constant is effectively computable.

From now on $X_1, \dots, X_r$ denote disjoint, non-empty sets of prime numbers (always equal to $1$ or $2$ modulo $4$) and we let $X = X_1 \times \dots \times X_r$ be the corresponding product space. There is an injection from $X$ to $\mathcal{D}$ by sending $x = (x_1, \dots, x_r) \in X$ to $x_1 \cdot \ldots \cdot x_r$. We shall often implicitly think of elements of $X$ as squarefree integers in this way.

\subsection{Prime divisors}
Let us first consider some desirable properties of elements in $\mathcal{D}$.
 
\begin{mydef}
\label{dNice}
Let $N \geq 10^{1000}$ and let $d \in \mathcal{D}$. Write $p_1, \dots, p_r$ for the prime divisors of $d$ ordered such that
\[
p_1 < \dots < p_r.
\]
Put
\[
D_1 := e^{(\log \log N)^{1/10}}, \quad C_0 := \sqrt{\log \log \log N}.
\]
We say that $d \in \mathcal{D}$ is $N$-nice if the following three properties are all satisfied
\begin{itemize}
\item $p_i > D_1$ implies $2p_i < p_{i + 1}$;
\item we have
\begin{align}
\label{eRegularSpace}
\left|\frac{1}{2} \log \log p_i - i\right| < C_0^{1/5} \max(i, C_0)^{4/5} \textup{ for all } 1 \leq i < \frac{r}{3};
\end{align}
\item there exists some $i$ satisfying $\frac{r^{1/2}}{2} < i < \frac{r}{2}$ and
\[
\log p_i \geq (\log \log p_i)^2 \cdot \log \log \log N \cdot \sum_{j = 1}^{i - 1} \log p_j.
\]
\end{itemize}
Define $\mathcal{D}_r(N)$ to be the subset of $d \in \mathcal{D}(N)$ with $\omega(d) = r$, where $\omega(d)$ equals the number of distinct prime divisors of $d$.
\end{mydef}

\begin{theorem}
\label{tSquarefree}
We have
\begin{align}
\label{eBadr}
\bigcup_{|r -  \frac{1}{2}\log \log N| \geq (\log \log N)^{2/3}} |\mathcal{D}_r(N)| \ll \frac{|\mathcal{D}(N)|}{(\log \log N)^{1/100}}.
\end{align}
Furthermore, for all $A > 0$, there exist $C_1, C_2, N_0 > 0$ such that for all $r \leq A \log \log N$ and all $N \geq N_0$
\begin{align}
\label{eDrNsize}
\frac{C_1N}{\log N} \cdot \frac{(\frac{1}{2} \log \log N)^{r - 1}}{(r - 1)!}\leq |\mathcal{D}_r(N)| \leq \frac{C_2N}{\log N} \cdot \frac{(\frac{1}{2} \log \log N)^{r - 1}}{(r - 1)!}.
\end{align}
Now suppose that $r$ satisfies
\begin{align}
\label{eErdosKac}
\left|r - \frac{1}{2} \log \log N\right| < (\log \log N)^{2/3}.
\end{align}
Then we also have
\[
\frac{|\{d \in \mathcal{D}_r(N) : d \textup{ is not } N\textup{-nice}\}|}{|\mathcal{D}_r(N)|} \ll \frac{1}{e^{(\log \log \log N)^{1/4}}}.
\]
\end{theorem}

\begin{proof}
Equation (\ref{eBadr}) follows from the Erd\H os--Kac theorem, while equation (\ref{eDrNsize}) can be found on \cite[p. 13]{CKMP}. The second part is \cite[Theorem 4.1]{CKMP}, which is based on \cite[Theorem 5.4]{Smith}.
\end{proof}

\subsection{Preboxes and Legendre symbols}
We will now introduce the notion of preboxes. Their usefulness lies in the simplicity of their definition, which makes preboxes suitable for inductive arguments. Later on we shall define the similar, but more stringent, notion of boxes, which will be the objects we naturally encounter in our final section.

\begin{mydef}
A prebox is a pair $(X, P)$, where $P, X_1, \dots, X_r$ are disjoint sets of primes all $1$ or $2$ modulo $4$ and $X = X_1 \times \dots \times X_r$. We assume that there exist real numbers
\[
2 < s_1 < t_1 < \dots < s_r < t_r
\]
such that $X_i \subseteq (s_i, t_i]$ for all $i \in [r]$ and $P \subseteq (1, s_1]$.
\end{mydef}

Recall that the first Artin pairing of $d \in \mathcal{D}$ is determined by the quadratic behavior between the prime divisors of $d$. This prompts the following definition. Write $\iota$ for the unique group isomorphism from $\FF_2$ to $\{\pm 1\}$.

\begin{mydef}
Let $(X, P)$ be a prebox. Define
\[
M_r = \{(i, j) : i, j \in [r], i < j\}, \quad M_{r, P} := [r] \times P.
\]
Let $\mathcal{M}_r \subseteq M_r$ and $\mathcal{M}_{r, P} \subseteq M_{r, P}$ and let $a : \mathcal{M}_r \sqcup \mathcal{M}_{r, P} \rightarrow \FF_2$, where $\sqcup$ denotes disjoint union. Then we define $X(a)$ to be the subset of $(x_1, \dots, x_r) \in X$ satisfying
\[
\left(\frac{x_i}{x_j}\right) = \iota(a(i, j)) \textup{ for all } (i, j) \in \mathcal{M}_r \textup{ and } \left(\frac{x_i}{p}\right) = \iota(a(i, p)) \textup{ for all } (i, p) \in \mathcal{M}_{r, P}.
\]
\end{mydef}

Our aim is to prove an equidistribution statement for $X(a)$ for certain desirable preboxes. One undesirable property is the presence of a Siegel zero, which we define now.

\begin{mydef}
\label{dSiegel}
For a real number $0.5 > c > 0$, we define $\mathcal{S}(c)$ to be the (conjecturally empty) set of squarefree integers $d$ such that $L(s, \chi_d)$ has a real zero in the region
\[
1 - \frac{c}{\log(|d| + 4)} \leq s \leq 1.
\]
List the elements of $\mathcal{S}(c)$ as $d_1, d_2, \dots$ in such a way that $|d_1| \leq |d_2| \leq \dots$. By Landau's theorem, there exists a sufficiently small $c_{\textup{Landau}} > 0$ such that $|d_i|^2 \leq |d_{i + 1}|$. We fix such a $c_{\textup{Landau}}$ for the remainder of the paper.

Let $(X, P)$ be a prebox. We say that $(X, P)$ is Siegel-free above a real number $t > 0$ if there does not exist some $|d_i| > t$ and some $x \in X$ such that $d_i$ divides $x \prod_{p \in P} p$.
\end{mydef}

Our next proposition is the first of a series of two on equidistribution of Legendre symbols. In our first result we shall avoid some delicate issues with small primes by making extra assumptions on the map $a : \mathcal{M}_r \sqcup \mathcal{M}_{r, P} \rightarrow \FF_2$ and the $X_i$, see equation (\ref{eSmallFix}). Later we shall see how to deal with the small primes. The next proposition is directly based on \cite[Proposition 6.3]{Smith}. Watkins \cite{Watkins} observed that $A$ can be effectively computed by using an effective lower bound for Siegel zeroes.

\begin{prop}
\label{pLegendre}
Let $c_1, c_2, c_3, c_4, c_5, c_6, c_7, c_8 > 0$ be real numbers such that
\[
c_3 > 1, \quad c_5 > 3, \quad c_6 > 1, \quad \frac{1}{8} > c_8 + \frac{c_7 \log 2}{2} + \frac{1}{c_1} + \frac{c_2c_4}{2}.
\]
There exists an effectively computable constant $A > 0$, depending only on the real numbers $c_1, c_2, c_3, c_4, c_5, c_6, c_7, c_8$, such that the following holds.

Let $(X, P)$ be a prebox with $s_1 > A$. Assume that $(X, P)$ is Siegel-free above some real number $t$ satisfying $s_1 > t > A$. Take subsets $\mathcal{M}_r \subseteq M_r$ and $\mathcal{M}_{r, P} \subseteq M_{r, P}$ and take $a : \mathcal{M}_r \sqcup \mathcal{M}_{r, P} \rightarrow \FF_2$. Suppose that $1 \leq k \leq r$ is an integer such that $(i, p) \in \mathcal{M}_{r, P}$ implies $i > k$. Furthermore, we assume that
\begin{align}
\label{eSmallFix}
X_i = \left\{p \in (s_i, t_i] : p \equiv 1 \bmod 4 \textup{ and } \left(\frac{p}{p'}\right) = \iota(a(i, p')) \textup{ for all } (i, p') \in \mathcal{M}_{r, P} \right\}
\end{align}
for all $i > k$. Moreover, we assume that
\begin{itemize}
\item[(i)] $|P| \leq \log t_i - i$ for all $i \in [r]$;
\item[(ii)] $t_1 > r^{c_1}$ and $t_k < e^{t_1^{c_2}}$;
\item[(iii)] we have
\[
|X_i| \geq \frac{2^{c_3 i} \cdot t_i}{(\log t_i)^{c_4}}
\]
for all $i \in [r]$;
\item[(iv)] if $k \neq r$, we assume that
\[
t_{k + 1} > \max\left(e^{(\log t_1)^{c_5}}, e^{t^{c_6}}\right);
\]
\item[(v)] we assume that $k \leq \lceil c_7 \log t_1 \rceil$. Furthermore, we assume that for all $i \leq r$ and all $i + c_7 \log t_i \leq j \leq r$ that
\[
t_j > e^{(\log t_i)^{c_5}}.
\]
\end{itemize}
Then we have
\[
\left|\left|X(a)\right| - \frac{|X|}{2^{|\mathcal{M}_r|}}\right| \leq \frac{r |X|}{t_1^{c_8 + 1/c_1} \cdot  2^{|\mathcal{M}_r|}} \leq \frac{|X|}{t_1^{c_8} \cdot  2^{|\mathcal{M}_r|}}.
\]
\end{prop}

\begin{proof}
The last inequality follows immediately from our assumption $(ii)$. We will prove the first inequality by induction on $r$. For $r = 1$ we have $k = 1$ and $\mathcal{M}_r = M_r = \mathcal{M}_{r, P} = M_{r, P} = \varnothing$. We conclude that $X(a) = X$, so the inequality holds. Now suppose that $r > 1$.

Let $x_1 \in X_1$ and let $2 \leq i \leq r$ be an integer. We define
\[
X_i(a, x_1) := \left\{p \in X_i : \left(\frac{x_1}{p}\right) = \iota(a(1, i))\right\}
\]
if $(1, i) \in \mathcal{M}_r$. It will be convenient to set $X_i(a, x_1) := X_i$ in case $(1, i) \not \in \mathcal{M}_r$. We claim that
\begin{align}
\label{eXx1}
\left||X_i(a, x_1)| - \frac{1}{2} |X_i|\right| \leq \frac{|X_i|}{t_1}
\end{align}
for all $i > k$ with $(1, i) \in \mathcal{M}_r$, provided that we pick $A$ sufficiently large. Define
\[
K := \Q(\sqrt{-1}, \sqrt{x_1}, \{\sqrt{p} : p \in P\}).
\]
Let $\chi: \Gal(K/\Q) \rightarrow \{\pm 1\}$ be a non-trivial character. Then there exists a squarefree integer $D \neq 1$ such that $\chi$ factors through $\Q(\sqrt{D})$. Therefore the Chebotarev density theorem gives the estimate
\begin{align}
\label{eUglyCheb}
\sum_{s_i < p < t_i} \chi(\text{Frob}_p) \ll t_i^{\beta} + t_i \cdot e^{\frac{-c \log t_i}{\sqrt{\log t_i} + \log |D|}} \cdot (\log t_i |D|)^4
\end{align}
for some absolute constant $c > 0$, see for example \cite[Theorem 5.13]{IK}. We have the effective lower bound
\[
1 - \beta \gg_\epsilon \frac{1}{|D|^{1/2 + \epsilon}}
\]
for every $\epsilon > 0$. Furthermore, by the Siegel-free assumption we have $|D| \leq t$ in this case. Then assumption $(iv)$ together with $c_6 > 1$ implies
\begin{align}
\label{eBetaBound}
1 - \beta \gg_\epsilon \frac{1}{|D|^{1/2 + \epsilon}} \geq \frac{1}{t^{1/2 + \epsilon}} \gg \frac{1}{(\log t_i)^{1/2}},
\end{align}
if we pick $\epsilon$ sufficiently small in terms of $c_6$. Now observe that
\begin{align}
\label{eDBound}
\log |D| \ll (|P| + 2) \log t_1 \ll (\log t_1)^2 \ll (\log t_i)^{2/c_5}
\end{align}
by assumption $(i)$ and $(iv)$. Plugging equations (\ref{eBetaBound}) and (\ref{eDBound}) in equation (\ref{eUglyCheb}) we obtain
\[
\sum_{s_i < p < t_i} \chi(\text{Frob}_p) \ll \frac{t_i}{e^{(\log t_i)^{1/3}}},
\]
where we used that $c_5 > 3$. We conclude that
\[
\left||X_i(a, x_1)| - \frac{1}{2} |X_i|\right| \ll \frac{[K: \Q] \cdot t_i}{e^{(\log t_i)^{1/3}}}.
\]
Since $[K : \Q] \ll t_1 < e^{(\log t_i)^{1/c_5}}$ and since $c_5 > 3$, we deduce from assumption $(iii)$ that
\[
\left||X_i(a, x_1)| - \frac{1}{2} |X_i|\right| \leq \frac{|X_i|}{t_1}
\]
upon taking $A$ sufficiently large. This establishes equation (\ref{eXx1}).

Next consider the case $1 < i \leq k$ and suppose that $(1, i) \in \mathcal{M}_r$. Choose real numbers $c_9$ and $c_{10}$ such that
\[
c_9 + c_{10} < \frac{1}{4} - c_2c_4, \quad c_9 > c_7 \log 2 + c_8 + \frac{1}{c_1}, \quad c_{10} > c_8 + \frac{1}{c_1}.
\]
Note that there exist such real numbers, since
\[
\frac{1}{8} > c_8 + \frac{c_7 \log 2}{2} + \frac{1}{c_1} + \frac{c_2c_4}{2}
\]
by assumption. Call $x_1 \in X_1$ bad if there exists $(1, i) \in \mathcal{M}_r$ with
\[
\left||X_i(a, x_1)| - \frac{1}{2} |X_i|\right| \geq \frac{|X_i|}{t_1^{c_{10}}},
\]
which implies $i \leq k$ by equation (\ref{eXx1}). Write $X_1^{\text{bad}}$ for the subset of $x_1 \in X_1$ that are bad. We claim that
\begin{align}
\label{eX1bad}
|X_1^{\text{bad}}| \leq \frac{k \cdot |X_1|}{t_1^{c_9}}.
\end{align}
The large sieve implies that for all $\epsilon > 0$
\[
\sum_{x_1 \in X_1} \left|\sum_{x_i \in X_i} \left(\frac{x_1}{x_i}\right)\right| \ll_\epsilon t_i \cdot t_1^{3/4 + \epsilon},
\]
see for example \cite[Proposition 6.6]{Smith}. Then we get
\begin{align}
\label{eLargeSieve}
\sum_{x_1 \in X_1} \left|\sum_{x_i \in X_i} \left(\frac{x_1}{x_i}\right)\right| \ll_\epsilon \frac{|X_1| \cdot |X_i|}{t_1^{1/4 - c_2c_4 - \epsilon}}
\end{align}
thanks to assumptions $(ii)$ and $(iii)$. Now fix $0 < \epsilon < 1/4 - c_2c_4 - c_9 - c_{10}$. Applying equation (\ref{eLargeSieve}) with $\epsilon/2$ and taking $A$ sufficiently large in terms of $\epsilon/2$, we get
\[
\sum_{x_1 \in X_1} \left|\sum_{x_i \in X_i} \left(\frac{x_1}{x_i}\right)\right| \leq \frac{|X_1| \cdot |X_i|}{t_1^{-1/4 + c_2c_4 + \epsilon}}.
\]
This quickly implies equation (\ref{eX1bad}), therefore establishing the claim. We now split $X(a)$ depending on $x_1$
\begin{align}
\label{eGoodBadX1}
\left|\left|X(a)\right| - \frac{|X|}{2^{|\mathcal{M}_r|}}\right| \leq \sum_{x_1 \in X_1^{\text{bad}}} \sum_{\substack{x \in X(a) \\ \pi_1(x) = x_1}} 1 + \left|\left(\sum_{x_1 \not \in X_1^{\text{bad}}} \sum_{\substack{x \in X(a) \\ \pi_1(x) = x_1}} 1\right) - \frac{|X|}{2^{|\mathcal{M}_r|}}\right|.
\end{align}
To deal with the first term, fix some $x_1 \in X_1^{\text{bad}}$. We move $x_1$ to $P$ and apply the induction hypothesis to the prebox
\[
(X_2 \times \dots \times X_k \times X_{k + 1}(a, x_1) \times \dots \times X_r(a, x_1), P \cup \{x_1\}).
\]
Write $k_{\text{old}}$ for the current value of $k$ and $k_{\text{new}}$ for the value of $k$ to which we apply the induction hypothesis. Then we take $k_{\text{new}} \geq k_{\text{old}} - 1$ to be the smallest integer such that
\[
t_{k_{\text{new}} + 2} > \max(e^{(\log t_2)^{c_5}}, e^{t^{c_6}}).
\]
We take $k_{\text{new}} = r - 1$ if no such integer exists. Assumptions $(i)$ and $(ii)$ are satisfied, since
\[
e^{t_2^{c_2}} > e^{(\log t_2)^{c_5}}
\]
for $A$ sufficiently large. Furthermore, assumption $(iii)$ follows from $c_3 > 1$ and our bounds on $X_i$ if we take $A$ sufficiently large. Meanwhile assumption $(iv)$ holds by construction of $k_{\text{new}}$. To check assumption $(v)$, we distinguish two cases. If $k_{\text{new}} = k_{\text{old}} - 1$, then assumption $(v)$ holds. Otherwise note that
\[
k_{\text{new}} + 2 \leq 2 + \lceil c_7 \log t_2 \rceil
\]
and hence $k_{\text{new}} \leq \lceil c_7 \log t_2 \rceil$ as desired. Having checked all the conditions, we apply the induction hypothesis, which leads to the bound
\[
\sum_{\substack{x \in X(a) \\ \pi_1(x) = x_1}} 1 \leq 2^{-|\mathcal{M}_r| + k + 1} \cdot \frac{|X|}{|X_1|}
\]
for $A$ sufficiently large, since $(1 + \frac{2}{t_1})^r - 1$ is small by assumption $(ii)$ and the inequality $c_1 > 8$. Therefore we deduce that
\begin{align}
\label{eBadX1}
\sum_{\substack{x \in X(a) \\ \pi_1(x) \in X_1^{\text{bad}}}} 1 \leq k \cdot 2^{-|\mathcal{M}_r| + k + 1} \cdot \frac{|X|}{t_1^{c_9}} \leq \frac{|X|}{t_1^{c_8 + 1/c_1} \cdot 2^{|\mathcal{M}_r| + 2}},
\end{align}
thanks to assumption $(v)$ and the choice of $c_9$. Next we turn our attention to bounding
\[
\left|\left(\sum_{x_1 \not \in X_1^{\text{bad}}} \sum_{\substack{x \in X(a) \\ \pi_1(x) = x_1}} 1\right) - \frac{|X|}{2^{|\mathcal{M}_r|}}\right|.
\]
We move $x_1$ to $P$ and apply the induction hypothesis to the prebox
\[
(X_2(a, x_1) \times \dots \times X_k(a, x_1) \times X_{k + 1}(a, x_1) \times \dots \times X_r(a, x_1), P \cup \{x_1\}).
\]
Since $x_1 \not \in X_1^{\text{bad}}$, we see that all the assumptions are satisfied. Let $M$ be the number of pairs of the shape $(1, j)$ in $\mathcal{M}_r$. Writing $X(x_1)$ for the above product space, the induction hypothesis shows that
\begin{align}
\label{eXx1aFirst}
\left||X(x_1)(a)| - \frac{|X(x_1)|}{2^{|\mathcal{M}_r| - M}}\right| \leq \frac{(r - 1) |X(x_1)|}{t_1^{c_8 + 1/c_1} \cdot 2^{|\mathcal{M}_r| - M}},
\end{align}
where we have implicitly restricted $a$ in the obvious way. Using that $x_1 \not \in X_1^{\text{bad}}$ once more, we see that
\begin{align}
\label{eXx1aSecond}
\left|2^M |X(x_1)| - \frac{|X|}{|X_1|}\right| &\leq \frac{|X|}{|X_1|} \cdot \left(\left(1 + \frac{2}{t_1^{c_{10}}}\right)^k \cdot \left(1 + \frac{2}{t_1}\right)^r - 1\right) \nonumber \\
&\leq \frac{|X|}{100 \cdot t_1^{c_8 + 1/c_1} \cdot |X_1|}
\end{align}
for $A$ sufficiently large, where we used that $r < t_1^{1/c_1}$, $k \leq \lceil c_7 \log t_1 \rceil$ and $c_{10} > c_8 + 1/c_1$. We obtain from equations (\ref{eXx1aFirst}) and (\ref{eXx1aSecond})
\[
\left||X(x_1)(a)| - \frac{|X|}{2^{|\mathcal{M}_r|} \cdot |X_1|}\right| \leq \frac{(r - 1/2) |X|}{t_1^{c_8 + 1/c_1} \cdot 2^{|\mathcal{M}_r|} \cdot |X_1|}.
\]
We conclude that
\begin{align}
\label{eGoodX1}
\left|\left(\sum_{x_1 \not \in X_1^{\text{bad}}} \sum_{\substack{x \in X(a) \\ \pi_1(x) = x_1}} 1\right) - \frac{|X|}{2^{|\mathcal{M}_r|}}\right| \leq \frac{(r - 1/4) |X|}{t_1^{c_8 + 1/c_1} \cdot 2^{|\mathcal{M}_r|}}.
\end{align}
Inserting equations (\ref{eBadX1}) and (\ref{eGoodX1}) in equation (\ref{eGoodBadX1}) finishes the proof of the induction step.
\end{proof}

Our next goal is to deal with the small primes. We shall only be able to achieve savings by introducing some extra averaging. Write $\mathcal{P}_r$ for the set of permutations of $[r]$ and write $\mathcal{P}_r(k)$ for the $\sigma \in \mathcal{P}_r$ such that $\sigma(i) = i$ for all $i > k$. Given $a: M_r \sqcup M_{r, P} \rightarrow \FF_2$, we will use the convention that $a(j, i)$ equals $a(i, j)$ for $j > i$. We then define $\sigma(a)$ to be
\[
\sigma(a)(i, j) := a(\sigma^{-1}(i), \sigma^{-1}(j)), \quad \sigma(a)(i, p) := a(\sigma^{-1}(i), p).
\]
Given integers $k_1 \geq k_0 \geq 0$, we define $X^{\text{trun}}_{k_0, k_1}(\sigma, a)$ to be the set of $x = (x_1, \dots, x_r) \in X$ such that
\[
\left(\frac{x_i}{x_j}\right) = \iota(a(\sigma^{-1}(i), \sigma^{-1}(j))) \text{ for all } i, j \in [k_1] \text{ with } \sigma^{-1}(i) < \sigma^{-1}(j) \text{ and } \text{min}(i, j) \leq k_0
\]
and furthermore
\[
\left(\frac{x_i}{p}\right) = \iota(a(\sigma^{-1}(i), p)) \text{ for all } (i, p) \in [k_1] \times P.
\]
We remark that the number of distinct elements $i, j \in [k_1]$ with $\sigma^{-1}(i) < \sigma^{-1}(j)$ and $\text{min}(i, j) \leq k_0$ equals
\[
C(k_0, k_1) := \frac{1}{2}(k_0^2 - k_0) + k_0(k_1 - k_0).
\]
In particular, this number does not depend on $\sigma$. Roughly speaking, $X^{\text{trun}}_{k_0, k_1}(\sigma, a)$ imposes the conditions of $\sigma(a)$ on the small primes. This makes it possible to apply Proposition \ref{pLegendre} to $X^{\text{trun}}_{k_0, k_1}(\sigma, a)$. The next entirely combinatorial lemma offers a convenient tool to reduce to $X^{\text{trun}}_{k_0, k_1}(\sigma, a)$, see also \cite[Proposition 6.7]{Smith} for a similar result.

\begin{lemma}
\label{lSmallPrimeCombinatorial}
Let $(X, P)$ be a prebox. Let $0 \leq k_0 \leq k_1 \leq k_2 \leq r$ be integers such that
\[
2^{k_0 + |P| + 1} k_1^2 < k_2.
\]
Then we have for all $x \in X$
\begin{multline*}
\sum_{a: M_r \sqcup M_{r, P} \rightarrow \FF_2} \left(\frac{k_2!}{2^{C(k_0, k_1) + k_1|P|}} - |\{\sigma \in \mathcal{P}_r(k_2) : x \in X^{\textup{trun}}_{k_0, k_1}(\sigma, a)\}|\right)^2 \leq \\
\frac{2^{k_0 + |P| + 1} \cdot k_1^2}{k_2} \cdot 2^{-2C(k_0, k_1) - 2k_1|P| + |M_r \sqcup M_{r, P}|} \cdot (k_2!)^2.
\end{multline*}
\end{lemma}

\begin{proof}
Write
\[
B(x, a) := \{\sigma \in \mathcal{P}_r(k_2) : x \in X^{\textup{trun}}_{k_0, k_1}(\sigma, a)\}.
\]
We start by computing
\begin{align}
\label{eFirstMomentSigma}
\sum_{a: M_r \sqcup M_{r, P} \rightarrow \FF_2} |B(x, a)| &= \sum_{\sigma \in \mathcal{P}_r(k_2)} |\{a: M_r \sqcup M_{r, P} \rightarrow \FF_2 : x \in X^{\textup{trun}}_{k_0, k_1}(\sigma, a)\}| \nonumber \\
&= \sum_{\sigma \in \mathcal{P}_r(k_2)} \frac{2^{|M_r \sqcup M_{r, P}|}}{2^{C(k_0, k_1) + k_1|P|}} = k_2! \cdot \frac{2^{|M_r \sqcup M_{r, P}|}}{2^{C(k_0, k_1) + k_1|P|}}.
\end{align}
Next, we compute
\[
\sum_{a: M_r \sqcup M_{r, P} \rightarrow \FF_2} |B(x, a)|^2 = \sum_{\sigma_1, \sigma_2 \in \mathcal{P}_r(k_2)} |\{a: M_r \sqcup M_{r, P} \rightarrow \FF_2 : x \in X^{\textup{trun}}_{k_0, k_1}(\sigma_1, a) \cap X^{\textup{trun}}_{k_0, k_1}(\sigma_2, a)\}|.
\]
For now we fix $\sigma_1, \sigma_2 \in \mathcal{P}_r(k_2)$ and we aim to bound
\[
N(\sigma_1, \sigma_2) := |\{a: M_r \sqcup M_{r, P} \rightarrow \FF_2 : x \in X^{\textup{trun}}_{k_0, k_1}(\sigma_1, a) \cap X^{\textup{trun}}_{k_0, k_1}(\sigma_2, a)\}|.
\]
We introduce the quantities
\[
I_{k_1, k_2}(\sigma_1, \sigma_2) := |\{i \in [k_2] : \sigma_1(i) \leq k_1, \sigma_2(i) \leq k_1\}|
\]
and
\[
S(\sigma_1, \sigma_2) := \{(i, j) \in M_r : (\sigma_1(i), \sigma_1(j)), (\sigma_2(i), \sigma_2(j)) \in [k_0] \times [k_1] \cup [k_1] \times [k_0]\}
\]
and
\[
T(\sigma_1, \sigma_2) := \{(i, p) \in M_{r, P} : \sigma_1(i) \in [k_1], \sigma_2(i) \in [k_1]\}.
\]
Then $x \in X^{\textup{trun}}_{k_0, k_1}(\sigma_1, a) \cap X^{\textup{trun}}_{k_0, k_1}(\sigma_2, a)$ gives at least
\[
2C(k_0, k_1) + 2k_1|P| - |S(\sigma_1, \sigma_2)| - |T(\sigma_1, \sigma_2)|
\]
conditions on $a$. But we have the upper bounds
\[
|S(\sigma_1, \sigma_2)| \leq I_{k_1, k_2}(\sigma_1, \sigma_2) \cdot k_0, \quad |T(\sigma_1, \sigma_2)| \leq I_{k_1, k_2}(\sigma_1, \sigma_2) \cdot |P|.
\]
The first upper bound comes from the injective map 
\[
S(\sigma_1, \sigma_2) \rightarrow \{i \in [k_2] : \sigma_1(i) \leq k_1, \sigma_2(i) \leq k_1\} \times [k_0]
\]
given by
\[
(i, j) \mapsto (\sigma_1^{-1}(\max(\sigma_1(i), \sigma_1(j))), \min(\sigma_1(i), \sigma_1(j))),
\]
while the second upper bound is straightforward. This yields the bound
\[
N(\sigma_1, \sigma_2) \leq \frac{2^{|M_r \sqcup M_{r, P}|}}{2^{2C(k_0, k_1) + 2k_1|P| - I_{k_1, k_2}(\sigma_1, \sigma_2) \cdot (k_0 + |P|)}}.
\]
Having bounded $N(\sigma_1, \sigma_2)$, we next fix some $I \leq k_1$. We will bound the number of pairs $(\sigma_1, \sigma_2) \in \mathcal{P}_r(k_2) \times \mathcal{P}_r(k_2)$ with
\[
I_{k_1, k_2}(\sigma_1, \sigma_2) = I.
\]
There are $\binom{k_2}{I}$ ways to pick the indices that map to $[k_1]$ under $\sigma_1$ and $\sigma_2$. Once we have picked these indices, we can extend them to a pair $(\sigma_1, \sigma_2) \in \mathcal{P}_r(k_2) \times \mathcal{P}_r(k_2)$ in at most
\[
\left(\frac{k_1!}{(k_1 - I)!} \cdot (k_2 - I)!\right)^2
\]
ways. Indeed, suppose that $B \subseteq [k_2]$ is the set of indices mapping to $[k_1]$ under $\sigma_1$ and $\sigma_2$, so that $|B| = I$. Then there are
\[
k_1! \cdot (k_1! - 1) \cdot \ldots \cdot (k_1! - I + 1)
\]
possibilities for $\sigma_1$ to send $B$ to $[k_1]$, and similarly for $\sigma_2$. Furthermore, there are at most $(k_2 - I)!$ choices where $\sigma_1$ can send the indices in $[k_2] \setminus B$, and similarly for $\sigma_2$. We conclude that the total number of pairs $(\sigma_1, \sigma_2)$ with given $I_{k_1, k_2}(\sigma_1, \sigma_2) = I$ is bounded by
\[
\binom{k_2}{I} \cdot \left(\frac{k_1!}{(k_1 - I)!} \cdot (k_2 - I)!\right)^2 \leq k_2!^2 \cdot \left(\frac{k_1!}{(k_1 - I)!}\right)^2 \cdot \frac{(k_2 - I)!}{k_2!} \leq k_2!^2 \cdot  \left(\frac{k_1^2}{k_2}\right)^I.
\]
We deduce that
\begin{align}
\label{eSecondMomentSigma}
\sum_{a: M_r \sqcup M_{r, P} \rightarrow \FF_2} |B(x, a)|^2 
&= \hspace{-0.2cm} \sum_{\sigma_1, \sigma_2 \in \mathcal{P}_r(k_2)} |\{a: M_r \sqcup M_{r, P} \rightarrow \FF_2 : x \in X^{\textup{trun}}_{k_0, k_1}(\sigma_1, a) \cap X^{\textup{trun}}_{k_0, k_1}(\sigma_2, a)\}| \nonumber \\
&= \hspace{-0.2cm} \sum_{\sigma_1, \sigma_2 \in \mathcal{P}_r(k_2)} N(\sigma_1, \sigma_2) \nonumber \\
&\leq k_2!^2 \cdot \sum_{I \geq 0} \left(\frac{k_1^2}{k_2}\right)^I \cdot \frac{2^{|M_r \sqcup M_{r, P}|}}{2^{2C(k_0, k_1) + 2k_1|P| - I \cdot (k_0 + |P|)}} \nonumber \\
&= (k_2!)^2 \cdot \frac{2^{|M_r \sqcup M_{r, P}|}}{2^{2C(k_0, k_1) + 2k_1|P|}} \cdot \frac{k_2}{k_2 - 2^{k_0 + |P|} k_1^2},
\end{align}
where the last equality uses that
\[
2^{k_0 + |P| + 1} k_1^2 < k_2,
\]
so the geometric series converges. The lemma follows upon combining equations (\ref{eFirstMomentSigma}) and (\ref{eSecondMomentSigma}) with the inequality $2^{k_0 + |P| + 1} k_1^2 < k_2$.
\end{proof}

We will now deal with the small primes. The coming proposition is directly based on \cite[Theorem 6.4]{Smith}.

\begin{prop}
\label{pLegendre2}
Take real numbers $c_1, c_2, c_3, c_4, c_5, c_6, c_7, c_8, c_9, c_{10}, c_{11}, c_{12}$ satisfying
\begin{align*}
c_3 > 1, \quad c_5& > 3, \quad c_6 > 1, \quad \frac{1}{8} > c_8 + \frac{c_7 \log 2}{2} + \frac{1}{c_1} + \frac{c_2c_4}{2}, \\
c_{10} &\log 2 + 2 c_{11} + 2c_{12} < 1, \quad c_{11} + c_{12} < c_9.
\end{align*}
There exists an effectively computable constant $A > 0$, depending only on the real numbers $c_1, c_2, c_3, c_4, c_5, c_6, c_7, c_8, c_9, c_{10}, c_{11}, c_{12}$, such that the following holds.

Let $(X, P)$ be a prebox such that $X_i$ equals the set of primes in the interval $(s_i, t_i]$ that are $1$ or $2$ modulo $4$. Let $0 \leq k_0 < k_1 < k_2 \leq r$ be integers such that $k_2 > A$. We assume that $(X, P)$ is Siegel-free above some real number $t$ satisfying $s_{k_0 + 1} > t > A$. Moreover, suppose that
\begin{itemize}
\item[(i)] $|P| \leq \log t_i - i$ for all $k_0 < i \leq r$;
\item[(ii)] $t_{k_0 + 1} > r^{c_1}$ and $t_{k_1} < e^{t_{k_0 + 1}^{c_2}}$;
\item[(iii)] we have
\[
|X_i| \geq \frac{2^{|P| + c_3 i} \cdot k_2^{c_9} \cdot t_i}{(\log t_i)^{c_4}}
\]
for all $k_0 < i \leq r$;
\item[(iv)] we assume that
\[
t_{k_1 + 1} > \max\left(e^{(\log t_{k_0 + 1})^{c_5}}, e^{t^{c_6}}\right);
\]
\item[(v)] we assume that $k_1 - k_0 \leq \lceil c_7 \log t_{k_0 + 1} \rceil$. Furthermore, we assume that for all $k_0 < i \leq r$ and all $i + c_7 \log t_i \leq j \leq r$ that
\[
t_j > e^{(\log t_i)^{c_5}};
\]
\item[(vi)] $k_0 + |P| < c_{10} \log k_2$ and $\log k_1 < c_{11} \log k_2$.
\end{itemize}
Then we have
\[
\sum_{a: M_r \sqcup M_{r, P} \rightarrow \FF_2} \left|\frac{k_2! \cdot |X|}{2^{|M_r \sqcup M_{r, P}|}} - \sum_{\sigma \in \mathcal{P}_r(k_2)} |X(\sigma(a))| \right| \leq k_2! \cdot |X| \cdot (k_2^{-c_{12}} + t_{k_0 + 1}^{-c_8}).
\]
\end{prop}

\begin{proof}
We readily reduce to the case where $|X_i| = 1$ for $i \in [k_0]$. Henceforth we shall assume that $X_i = \{x_i\}$ for $i \in [k_0]$. It follows from the triangle inequality that
\[
\sum_{a: M_r \sqcup M_{r, P} \rightarrow \FF_2} \left|\frac{k_2! \cdot |X|}{2^{|M_r \sqcup M_{r, P}|}} - \sum_{\sigma \in \mathcal{P}_r(k_2)} |X(\sigma(a))| \right|
\]
is bounded by
\begin{multline*}
\sum_{a: M_r \sqcup M_{r, P} \rightarrow \FF_2} \left|\frac{k_2! \cdot |X|}{2^{|M_r \sqcup M_{r, P}|}} - \sum_{\sigma \in \mathcal{P}_r(k_2)} \frac{|X_{k_0, k_1}^{\text{trun}}(\sigma, a)|}{2^{|M_r \sqcup M_{r, P}| - C(k_0, k_1) - k_1|P|}}\right| + \\
\sum_{\sigma \in \mathcal{P}_r(k_2)} \sum_{a: M_r \sqcup M_{r, P} \rightarrow \FF_2} \left|\frac{|X_{k_0, k_1}^{\text{trun}}(\sigma, a)|}{2^{|M_r \sqcup M_{r, P}| - C(k_0, k_1) - k_1|P|}} - |X(\sigma(a))|\right|.
\end{multline*}
By the triangle inequality the first sum is at most
\[
\frac{2^{C(k_0, k_1) + k_1|P|}}{{2^{|M_r \sqcup M_{r, P}|}}} \sum_{a: M_r \sqcup M_{r, P} \rightarrow \FF_2} \sum_{x \in X} \left|\frac{k_2!}{2^{C(k_0, k_1) + k_1|P|}} - |\{\sigma \in \mathcal{P}_r(k_2) : x \in X_{k_0, k_1}^{\text{trun}}(\sigma, a)\}|\right|.
\]
An application of Lemma \ref{lSmallPrimeCombinatorial} and the Cauchy--Schwarz inequality shows that the above is bounded by
\begin{align}
\label{eSmallBound}
\left(\frac{2^{k_0 + |P| + 1} \cdot k_1^2}{k_2}\right)^{1/2} \cdot k_2! \cdot |X| \leq \frac{k_2! \cdot |X|}{2 \cdot k_2^{c_{12}}},
\end{align}
where the last inequality follows for sufficient $A$ from assumption $(vi)$ and $c_{10} \log 2 + 2c_{11} + 2c_{12} < 1$. Let us now consider the second sum. Fix $\sigma \in \mathcal{P}_r(k_2)$ and $a: M_r \sqcup M_{r, P} \rightarrow \FF_2$. We aim to bound each
\[
\left|\frac{|X_{k_0, k_1}^{\text{trun}}(\sigma, a)|}{2^{|M_r \sqcup M_{r, P}| - C(k_0, k_1) - k_1|P|}} - |X(\sigma(a))|\right|
\]
individually. Without loss of generality we may assume that 
\[
\left(\frac{x_i}{x_j}\right) = \iota(a(\sigma^{-1}(i), \sigma^{-1}(j))) \text{ for all distinct } i, j \in [k_0],
\]
since otherwise $X(\sigma(a))$ and $X_{k_0, k_1}^{\text{trun}}(\sigma, a)$ are empty. If $i > k_0$, define $X_i^{\text{trim}}(\sigma, a)$ to be the subset of $x_i \in X_i$ such that
\[
\left(\frac{x_i}{p}\right) = \iota(a(\sigma^{-1}(i), p)) \text{ for all } p \in P, \quad \left(\frac{x_i}{x_j}\right) = \iota(a(\sigma^{-1}(i), \sigma^{-1}(j))) \text{ for all } j \in [k_0].
\]
We turn $X_{k_0, k_1}^{\text{trun}}(\sigma, a)$ in a prebox $(X', P')$ by taking $P'$ to be the union of the primes $x_1, \dots, x_{k_0}$ with $P$ and by taking $X' = X_1' \times \dots \times X_{r - k_0}'$, where 
\[
X_i' =
\left\{
	\begin{array}{ll}
		X_{i + k_0}^{\text{trim}}(\sigma, a)  & \mbox{if } i \in [k_1 - k_0] \\
		X_{i + k_0} & \mbox{if } i > k_1 - k_0.
	\end{array}
\right.
\]
It follows from the Chebotarev density theorem that
\begin{align}
\label{eTrimsize}
\left||X_i^{\text{trim}}(\sigma, a)| - \frac{|X_i|}{2^{k_0 + |P|}}\right| \leq \frac{|X_i|}{t_{k_0 + 1}}
\end{align}
for all $i > k_1$. Therefore, in case
\[
|X_i'| \geq \frac{|X_{i + k_0}|}{2^{k_0 + |P|} \cdot k_2^{c_9}} \text{ for all } i \in [r - k_0],
\]
we apply Proposition \ref{pLegendre} to $(X'', P'') := (X_{k_0 + 1}^{\text{trim}}(\sigma, a) \times \dots \times X_r^{\text{trim}}(\sigma, a), P')$ to obtain the bound
\[
\left|\frac{|X''|}{2^{\frac{(r - k_0)(r - k_0 - 1)}{2}}} - |X(\sigma(a))|\right| \leq \frac{|X''|}{2 \cdot t_{k_0 + 1}^{c_8} \cdot 2^{|M_r \sqcup M_{r, P}| - C(k_0, k_1) - k_1|P|}},
\]
where we save an extra factor $2$ by taking the $c_8$ of Proposition \ref{pLegendre} slightly smaller than our $c_8$. This implies
\begin{align}
\label{eBigBound1}
\left|\frac{|X_{k_0, k_1}^{\text{trun}}(\sigma, a)|}{2^{|M_r \sqcup M_{r, P}| - C(k_0, k_1) - k_1|P|}} - |X(\sigma(a))|\right| \leq \frac{|X_{k_0, k_1}^{\text{trun}}(\sigma, a)|}{t_{k_0 + 1}^{c_8} \cdot 2^{|M_r \sqcup M_{r, P}| - C(k_0, k_1) - k_1|P|}}
\end{align}
thanks to equation (\ref{eTrimsize}). Instead suppose that
\[
|X_i'| < \frac{|X_{i + k_0}|}{2^{k_0 + |P|} \cdot k_2^{c_9}}
\]
for some $i \in [r - k_0]$. Then certainly $i$ must be in $[k_1 - k_0]$, and we call the pair $(\sigma, a)$ bad at $i$. For every given $\sigma \in \mathcal{P}_r(k_2)$ and $j \in [k_1 - k_0]$, we define an equivalence relation on the set of maps $a: M_r \sqcup M_{r, P} \rightarrow \FF_2$ by declaring $a \sim_{\sigma, j} a'$ if
\[
a(\sigma^{-1}(j'), \sigma^{-1}(j + k_0)) = a'(\sigma^{-1}(j'), \sigma^{-1}(j + k_0)) \text{ for all } j' \in [k_0]
\]
and
\[
a(\sigma^{-1}(j + k_0), p) = a'(\sigma^{-1}(j + k_0), p) \text{ for all } p \in P.
\]
If the pair $(\sigma, a)$ is bad at $j$ and if $a' \sim_{\sigma, j} a$, then $(\sigma, a')$ is also bad at $j$. The trivial bound yields that in a given equivalence class represented by $a'$
\[
\sum_{\substack{a: M_r \sqcup M_{r, P} \rightarrow \FF_2 \\ (\sigma, a) \text{ bad} \\ a \sim_{\sigma, j} a'}} \left|\frac{|X_{k_0, k_1}^{\text{trun}}(\sigma, a)|}{2^{|M_r \sqcup M_{r, P}| - C(k_0, k_1) - k_1|P|}} - |X(\sigma(a))|\right| \leq \frac{|X|}{2^{k_0 + |P|} \cdot k_2^{c_9}}.
\]
Since there are at most $2^{k_0 + |P|}$ equivalence classes and at most $k_1 \cdot k_2!$ equivalence relations $\sim_{\sigma, j}$, we get after summing over all $\sigma$ and all $j$
\begin{align}
\label{eBigBound2}
\sum_{\sigma \in \mathcal{P}_r(k_2)} \sum_{\substack{a: M_r \sqcup M_{r, P} \rightarrow \FF_2 \\ (\sigma, a) \text{ bad}}} \left|\frac{|X_{k_0, k_1}^{\text{trun}}(\sigma, a)|}{2^{|M_r \sqcup M_{r, P}| - C(k_0, k_1) - k_1|P|}} - |X(\sigma(a))|\right| \leq k_1 \cdot k_2! \cdot \frac{|X|}{k_2^{c_9}} \leq \frac{k_2! \cdot |X|}{2 \cdot k_2^{c_{12}}}
\end{align}
thanks to assumption $(vi)$ and the inequality $c_{11} + c_{12} < c_9$. Furthermore, equation (\ref{eBigBound1}) implies that
\begin{align}
\label{eBigBound3}
\sum_{\sigma \in \mathcal{P}_r(k_2)} \sum_{\substack{a: M_r \sqcup M_{r, P} \rightarrow \FF_2 \\ (\sigma, a) \text{ not bad}}} \left|\frac{|X_{k_0, k_1}^{\text{trun}}(\sigma, a)|}{2^{|M_r \sqcup M_{r, P}| - C(k_0, k_1) - k_1|P|}} - |X(\sigma(a))|\right| \leq \frac{k_2! \cdot |X|}{t_{k_0 + 1}^{c_8}},
\end{align}
since for a given $\sigma$ the sets $X_{k_0, k_1}^{\text{trun}}(\sigma, a)$ partition $X$ and there are
\[
2^{|M_r \sqcup M_{r, P}| - C(k_0, k_1) - k_1|P|}
\]
maps $a: M_r \sqcup M_{r, P} \rightarrow \FF_2$ with the same $X_{k_0, k_1}^{\text{trun}}(\sigma, a)$. Adding the bounds from equations (\ref{eSmallBound}), (\ref{eBigBound2}) and (\ref{eBigBound3}), we get the proposition.
\end{proof}

\subsection{Boxes}
We are now ready to define boxes. We recall that
\[
D_1 = e^{(\log \log N)^{1/10}}.
\]
Also define
\[
D_{\text{Siegel}} := e^{(\log \log N)^{1/100}}.
\]

\begin{mydef}
Let $N \geq 10^{1000}$, let $r$ be a positive integer and let $0 \leq k \leq r$. Let $\mathbf{t} = (p_1, \dots, p_k, s_{k + 1}, \dots, s_r)$, where the $p_i$ are prime numbers congruent to $1$ or $2$ modulo $4$ and the $s_j$ are real numbers satisfying
\[
p_1 < \dots < p_k < D_1, \quad D_1 \leq s_{k + 1} < \dots < s_r, \quad s_i < \frac{2}{3} s_{i + 1} \textup{ for all } k < i < r.
\]
To $\mathbf{t}$ we associate a product space $X(\mathbf{t}) = X_1 \times \dots \times X_r$ as follows
\begin{itemize}
\item for $1 \leq i \leq k$ we have $X_i = \{p_i\}$;
\item for all $k < i \leq r$ we have that $X_i$ consists of the prime numbers ($1$ or $2$ modulo $4$) in the interval
\[
\left(s_i, s_i \cdot \left(1 + \frac{1}{C_{\textup{compr}}}\right)\right],
\]
\end{itemize}
where $100 \leq C_{\textup{compr}} \leq (\log N)^{100}$ is a real number. If $N$ is sufficiently large, then the sets $X_i$ are disjoint and non-empty, so that we may naturally view $X(\mathbf{t})$ as a subset of $\mathcal{D}$. We call a product space $X = X_1 \times \dots \times X_r$ a $(N, k, r)$-box if there is some $\mathbf{t}$ such that $X = X(\mathbf{t})$ and $X \subseteq \mathcal{D}(N)$. We say that $X$ is an excellent $(N, k, r)$-box if additionally there is $d \in X$ that is $N$-nice and 
\[
\left\{\alpha \prod_{i \in S} \pi_i(x) : x \in X, S \subseteq [r], \alpha \in \{1, -1\}, \left|\alpha \prod_{i \in S} \pi_i(x)\right| > D_{\textup{Siegel}}\right\} \cap \mathcal{S}(c_{\textup{Landau}}) = \varnothing.
\]
\end{mydef}

The next proposition is \cite[Proposition 6.9]{Smith}, although our proof is more similar to the one appearing in Watkins \cite{Watkins}.

\begin{theorem}
\label{tToBoxes}
Let $N \geq 10^{1000}$ be a real number and let $r$ be a positive integer satisfying equation (\ref{eErdosKac}). Let $V, W \subseteq \mathcal{D}_r(N)$ and let $\epsilon > 0$ be such that
\[
|W| > (1 - \epsilon) \cdot |\mathcal{D}_r(N)|.
\]
Suppose that
\[
\left||V \cap X| - \delta \cdot |X|\right| \leq \epsilon \cdot |X|
\]
for all integers $k$ and all $(N, k, r)$-boxes $X$ such that $X \cap W \neq \varnothing$. Then
\[
|V| = \delta \cdot |\mathcal{D}_r(N)| + O\left(\left(\epsilon + \frac{1}{e^{(\log \log \log N)^{1/4}}}\right) \cdot |\mathcal{D}_r(N)| + \left|\mathcal{D}_r(N) - \mathcal{D}_r\left(\frac{N}{\left(1 + \frac{1}{C_{\textup{compr}}}\right)^{r}}\right)\right|\right).
\]
\end{theorem}

\begin{proof}
Define
\[
e_i = D_1 \cdot \left(1 + \frac{1}{C_{\text{compr}}}\right)^i.
\]
Let $\mathcal{C}$ be the collection of product spaces $X = X_1 \times \dots \times X_r$ that are $(N, k, r)$-boxes for some integer $k$ such that
\[
s_i \in \{e_0, e_1, e_2, \dots\}
\]
for all $k + 1 \leq i \leq r$ and such that $X \cap W \neq \varnothing$. If we take distinct $X, Y \in \mathcal{C}$, then we see that $X \cap Y = \varnothing$. Furthermore, in case $d \in \mathcal{D}_r(N)$ is not in some product space $X \in \mathcal{C}$, then we see that $d$ fails the first condition of Definition \ref{dNice} (and hence $d$ is not $N$-nice) or $d \not \in W$ or
\[
d \geq N\left(1 + \frac{1}{C_{\text{compr}}}\right)^{-r}.
\]
Write $\mathcal{D}_r^{\text{bad}}(N)$ for this set of $d$. Then we have
\[
\left| |V| - \delta \cdot |\mathcal{D}_r(N)|\right| \leq 2 \cdot |\mathcal{D}_r^{\text{bad}}(N)| + \sum_{X \in \mathcal{C}} \left| |V \cap X| - \delta \cdot |X|\right|.
\]
The latter term is bounded by $\epsilon \cdot |\mathcal{D}_r(N)|$ thanks to our assumption
\[
\left||V \cap X| - \delta \cdot |X|\right| \leq \epsilon \cdot |X|.
\]
Furthermore, it follows from Theorem \ref{tSquarefree} and 
\[
|W| > (1 - \epsilon) \cdot |\mathcal{D}_r(N)|
\]
that
\[
|\mathcal{D}_r^{\text{bad}}(N)| \ll \left(\epsilon + \frac{1}{e^{(\log \log \log N)^{1/4}}}\right) \cdot |\mathcal{D}_r(N)| + \left|\mathcal{D}_r(N) - \mathcal{D}_r\left(N\left(1 + \frac{1}{C_{\text{compr}}}\right)^{-r}\right)\right|.
\]
This gives the theorem.
\end{proof}

\begin{remark}
\label{rBox}
If $(\log N)^{1 - \epsilon} \geq C_{\textup{compr}} \geq r^{1 + \epsilon}$, then one can show that
\[
\left|\mathcal{D}_r(N) - \mathcal{D}_r\left(N\left(1 + \frac{1}{C_{\textup{compr}}}\right)^{-r}\right)\right| \ll_{\epsilon} \frac{|\mathcal{D}_r(N)|}{r^\epsilon},
\]
and hence we get a genuine error term in Theorem \ref{tToBoxes}. However, since we have the luxury to also average over $r$ later on, we have opted to use the trivial bound
\[
\sum_r \left|\mathcal{D}_r(N) - \mathcal{D}_r\left(N\left(1 + \frac{1}{C_{\textup{compr}}}\right)^{-r}\right)\right| \leq |\mathcal{D}(N)| - \left|\mathcal{D}\left(N\left(1 + \frac{1}{C_{\textup{compr}}}\right)^{-r}\right)\right|
\]
and the classical asymptotic formula
\[
|\mathcal{D}(N)| = \frac{CN}{\sqrt{\log N}} \cdot \left(1 + O\left(\frac{1}{\log N}\right)\right)
\]
for some $C > 0$.
\end{remark}

From now on we shall take
\[
C_{\textup{compr}} := (\log \log N)^{50}.
\]
We remark that Theorem \ref{tToBoxes} works equally well for the set of odd or even radicands, or in fact any congruence conditions on our radicands. Keeping this in mind, we see that all our proofs also work if we order by the discriminant instead. We will now deal with Siegel zeroes, see also \cite[Proposition 6.10]{Smith}.

\begin{prop}
\label{pSiegel}
Let $N \geq t \geq 10^{1000}$ be real numbers and let $r$ be a positive integer satisfying equation (\ref{eErdosKac}). Let $f_1, f_2, f_3, \dots$ be a sequence of squarefree integers greater than $t$ such that $f_{i + 1} > f_i^2$. Define
\[
W(f_i) := \{x \in \mathcal{D}_r(N) : \textup{ there is a } (N, k, r)\textup{-box } X \textup{ and } x' \in X \textup{ such that } f_i \mid x' \textup{ and } x \in X\}.
\]
Suppose that $\log N \geq (\log t)^3$. Then
\[
\left|\bigcup_{i \geq 1} W(f_i)\right| \ll \frac{|\mathcal{D}_r(N)|}{\log t}.
\]
\end{prop}

\begin{proof}
We factor $f_i = p_1 \cdot \ldots \cdot p_m$. Suppose that $x \in W(f_i)$. Then there exist prime factors $q_1, \dots, q_m$ of $x$ such that
\[
p_i = q_i \text{ for } p_i < D_1, \quad \frac{p_i}{2} < q_i < 2p_i \text{ for } p_i \geq D_1.
\]
First suppose that $f_i < N^{2/3}$. Since $r$ satisfies equation (\ref{eErdosKac}), it follows from equation (\ref{eDrNsize}) that
\begin{align*}
|W(f_i)| &\leq \left|\mathcal{D}_{r - m}\left(\frac{2^mN}{f_i}\right)\right| \cdot \prod_{p_i \geq D_1} |\{q_i \text{ prime} : p_i/2 < q_i < 2p_i\}| \\
&\leq \frac{C^m \cdot |\mathcal{D}_r(N)|}{f_i} \cdot \prod_{p_i \geq D_1} \frac{p_i}{\log p_i} \ll \frac{|\mathcal{D}_r(N)|}{\log f_i}
\end{align*}
for some absolute constant $C > 0$. Since $f_i > t^{2^i}$, we deduce that
\[
\left|\bigcup_{\substack{i \geq 1 \\ f_i < N^{2/3}}} W(f_i)\right| \ll \frac{|\mathcal{D}_r(N)|}{\log t}.
\]
Next suppose that $f_i \geq N^{2/3}$. Since $f_{i + 1} \geq N^{4/3}$, it follows that $W(f_j)$ is empty for $j \geq i + 1$. Therefore it remains to bound $|W(f_i)|$. We have
\begin{align*}
|W(f_i)| &\leq \frac{N}{f_i} \cdot \prod_{p_i \geq D_1} |\{q_i \text{ prime} : p_i/2 < q_i < 2p_i\}| \\
&\leq \frac{N}{f_i} \cdot \prod_{p_i \geq D_1} \frac{2p_i}{\log p_i} \ll \frac{N}{\log f_i} \ll \frac{N}{\log N}
\end{align*}
for $N$ sufficiently large. A final appeal to equation (\ref{eDrNsize}) shows that this is within the error term thanks to our assumption $\log N \geq (\log t)^3$.
\end{proof}

We can now prove the main result of this section. Let $D_{2, n}$ be the set of $d \in \mathcal{D}$ such that $\text{rk}_4 \text{Cl}(\Q(\sqrt{d})) = n$. Write $P_{\text{Sym}}(r, n)$ for the probability that a symmetric $r \times r$ matrix with coefficients in $\FF_2$ has kernel of dimension $n$ with respect to the uniform probability measure.

\begin{theorem}
\label{t4rank}
Let $N \geq 10^{1000}$ be a real number and let $n \geq 0$ be an integer. Then we have
\[
\frac{|D_{2, n} \cap \mathcal{D}(N)|}{|\mathcal{D}(N)|} = \lim_{s \rightarrow \infty} P_{\textup{Sym}}(s, n) + O\left(\frac{1}{e^{(\log \log \log N)^{1/4}}}\right).
\]
The implied constant is effectively computable.
\end{theorem}

Better error terms are available in the literature, see \cite{CKMP, Smith, Watkins}, with the best result due to Watkins \cite{Watkins}. One can also find the above result in earlier work of Fouvry and Kl\"uners \cite{FK1}, albeit without an error term. We have not yet used (and will not use in this section) the third point in the definition of $N$-nice. Similarly, one can pick $C_0$ substantially larger than we did for the purposes of this section. The improvements available in the literature come from a better upper bound in Theorem \ref{tSquarefree} in this more relaxed setting.

\begin{proof}
By Remark \ref{rBox} and by equation (\ref{eBadr}) it suffices to show that
\begin{multline*}
\frac{|D_{2, n} \cap \mathcal{D}_r(N)|}{|\mathcal{D}_r(N)|} = \lim_{s \rightarrow \infty} P_{\textup{Sym}}(s, n) + O\left(\frac{1}{e^{(\log \log \log N)^{1/4}}}\right) + \\ O\left(\left|\mathcal{D}_r(N) - \mathcal{D}_r\left(N\left(1 + \frac{1}{C_{\text{compr}}}\right)^{-r}\right)\right|\right)
\end{multline*}
with $r$ satisfying equation (\ref{eErdosKac}). By the explicit formulas in \cite{Williams}, we readily bound the difference
\[
P_{\textup{Sym}}(r - 1, n) - \lim_{s \rightarrow \infty} P_{\textup{Sym}}(s, n).
\]
We apply Theorem \ref{tToBoxes} with $W$ equal to the largest subset of the $N$-nice elements in $\mathcal{D}_r(N)$ disjoint form the $W(d_i)$ with $|d_i| > D_{\text{Siegel}}$, see Definition \ref{dSiegel} for the definition of $d_i$. It follows from Theorem \ref{tSquarefree} and Proposition \ref{pSiegel} with $t = D_{\text{Siegel}}$ that
\[
|\mathcal{D}_r(N)| - |W| \ll \frac{|\mathcal{D}_r(N)|}{e^{(\log \log \log N)^{1/4}}}.
\]
Hence it suffices to show that
\begin{align}
\label{e4rankBox}
\frac{|D_{2, n} \cap X|}{|X|} = P_{\textup{Sym}}(r - 1, n) + O\left(\frac{1}{e^{(\log \log \log N)^{1/4}}}\right)
\end{align}
for every $(N, k, r)$-excellent box $X$. We apply Proposition \ref{pLegendre2} to $(X, \varnothing)$ with
\[
(c_1, c_2, c_3, c_4, c_5, c_6, c_7, c_8, c_9, c_{10}, c_{11}, c_{12}) = \left(100, \frac{1}{10^9}, 5, 10^6, 5, 5, \frac{1}{100}, \frac{1}{100}, 5, \frac{1}{5}, \frac{1}{5}, \frac{1}{5}\right),
\]
$k_0 = k$, $k_1$ the smallest integer such that 
\[
t_{k_1 + 1} > \max(e^{(\log t_{k + 1})^5}, e^{t^5})
\]
and $k_2 = r$. It follows from equations (\ref{eRegularSpace}) and (\ref{eErdosKac}) that all the conditions of Proposition \ref{pLegendre2} are satisfied for $N$ sufficiently large. We conclude that
\begin{align}
\label{e4rankPermuted}
\sum_{a: M_r \rightarrow \FF_2} \left|\frac{r! \cdot |X|}{2^{r(r - 1)/2}} - \sum_{\sigma \in \mathcal{P}_r} |X(\sigma(a))| \right| \leq r! \cdot |X| \cdot (r^{-1/5} + t_{k + 1}^{-1/100})
\end{align}
for $N$ sufficiently large. To $a: M_r \rightarrow \FF_2$ we associate a matrix $A(a)$ as follows: the $(i, j)$-th entry of $A(a)$ equals $a(i, j)$ if $i \neq j$ and 
\begin{align}
\label{eMatrixa}
a(i, i) = \sum_{k \neq i} a(i, k).
\end{align}
Then $A(a)$ is a symmetric $r \times r$ matrix with row sum zero. In this way we have created a bijection between maps $M_r \rightarrow \FF_2$ and symmetric $r \times r$ matrices with row sum zero. Write $A'(a)$ for the matrix obtained form $A(a)$ by dropping the last column and last row of $A(a)$. Then we get a bijection between maps $M_r \rightarrow \FF_2$ and symmetric $(r - 1) \times (r - 1)$ matrices. If $x \in X(a)$, then we have
\begin{align}
\label{eRedeiFormula}
\text{rk}_4 \text{Cl}(\Q(\sqrt{x})) = -1 + \text{dim}_{\FF_2} \ \text{ker}(A(a)) = \text{dim}_{\FF_2} \ \text{ker}(A'(a))
\end{align}
thanks to classical work of R\'edei. Since we have
\[
\text{dim}_{\FF_2} \ \text{ker}(A(a)) = \text{dim}_{\FF_2} \ \text{ker}(A(\sigma(a))),
\]
equation (\ref{e4rankBox}) follows from equations (\ref{e4rankPermuted}) and (\ref{eRedeiFormula}).
\end{proof}

\section{\texorpdfstring{Proof of Theorem \ref{tPell}}{Proof of Theorem 1.1}}
Let $D_{k, n}$ be the set of $d \in \mathcal{D}$ such that $\text{rk}_{2^k} \text{Cl}(\Q(\sqrt{d})) = n$ and furthermore $(\sqrt{d}) \in 2^{k - 1} \text{Cl}(\Q(\sqrt{d}))$. Define $D_{k, n}(N)$ for the subset of $d \in D_{k, n}$ with $d < N$. We write $P(m, n, j)$ for the probability that a $m \times n$ matrix, chosen uniformly at random with coefficients in $\FF_2$, has right kernel of rank $j$. Let $A_{\text{Comb}} > 0$ be a choice of real number such that Proposition \ref{pARinput} holds. Define
\[
c = \frac{1}{A_{\text{Comb}} \cdot 10^{10}}.
\]
In this section we shall establish the following theorem.

\begin{theorem}
\label{tRank}
There are real numbers $A, N_0 > 0$ such that for all reals $N > N_0$, all integers $m \geq 2$ and all sequences of integers $n_2, \dots, n_{m + 1} \geq 0$
\[
\left|\left|\bigcap_{i = 2}^{m + 1} D_{i, n_i}(N)\right| - \frac{P(n_m, n_m, n_{m + 1})}{2^{n_m}} \cdot \left|\bigcap_{i = 2}^m D_{i, n_i}(N) \right|\right| \leq \frac{A \cdot |\mathcal{D}(N)|}{(\log \log \log \log N)^{\frac{c}{m^26^m}}}.
\]
\end{theorem}

We remark that the $m = 2$ case of the above theorem was already established in \cite[Theorem 6.1]{CKMP}. Therefore it remains to prove Theorem \ref{tRank} for $m \geq 3$. It is perhaps worth emphasizing once more that from the $16$-rank onwards, the class group behaves rather differently than for the $8$-rank. This leads to substantial differences between the reflection principles presented here and those in \cite{CKMP}. 

The assumption $m \geq 3$ means then that we do not have to reprove some of the results from \cite{CKMP}. We will use this assumption in particular when applying Theorem \ref{main thm 1 on self-pairing} (we remark that $m = 3$ corresponds to $s = 2$). Note that Theorem \ref{main thm 1 on self-pairing} is not correct when $s = 1$, which shows the qualatitive difference between the $16$-rank and the $8$-rank. This is certainly related to the fact that there are no governing fields for the $16$-rank, while there are governing fields for the $8$-rank. For a precise statement, see \cite[Theorem 3]{KM}, which builds on earlier work in \cite{FIMR}.

The error term from \cite{CKMP} was later substantially improved by Watkins \cite{Watkins2}. Let us now show that Theorem \ref{tRank} implies Theorem \ref{tPell}. The argument is very similar to the material in \cite[Appendix A]{KP3}.

\begin{proof}[Proof that Theorem \ref{tRank} implies Theorem \ref{tPell}]
We will show that \cite[Conjecture 3.4(i)]{Stevenhagen} and \cite[Conjecture 3.4(ii)]{Stevenhagen} hold. Then, as already argued immediately after \cite[Conjecture 3.4]{Stevenhagen}, we get Theorem \ref{tPell}. We also remark that \cite[Conjecture 3.4(ii)]{Stevenhagen} follows from the results of Fouvry--Kl\"uners \cite[Corollary 2]{FK1} (or, alternatively, our Theorem \ref{t4rank}), so it remains to verify \cite[Conjecture 3.4(i)]{Stevenhagen}. Define for every integer $m \in \Z_{\geq 0}$ the quantities
\[
P_2^-(m) = \liminf_{X \rightarrow \infty} \frac{|\mathcal{D}^-(X) \cap D_{2, m}(X)|}{|D_{2, m}(X)|}, \quad P_2^+(m) = \limsup_{X \rightarrow \infty} \frac{|\mathcal{D}^-(X) \cap D_{2, m}(X)|}{|D_{2, m}(X)|}.
\]
The content of \cite[Conjecture 3.4(i)]{Stevenhagen} is that
\begin{align}
\label{eSteheuristic}
P_2^-(m) = P_2^+(m) = \frac{1}{2^{m + 1} - 1}.
\end{align}
It follows from Theorem \ref{tRank} that $P_2^-(m) = P_2^+(m)$ for all $m$, and hence
\[
P_2(m) = \lim_{X \rightarrow \infty} \frac{|\mathcal{D}^-(X) \cap D_{2, m}(X)|}{|D_{2, m}(X)|}
\]
exists. From the Markov chain behavior in Theorem \ref{tRank}, we similarly deduce that
\[
P_3(m, n) = \lim_{X \rightarrow \infty} \frac{|\mathcal{D}^-(X) \cap D_{2, m}(X) \cap D_{3, n}(X)|}{|D_{2, m}(X) \cap D_{3, n}(X)|}
\]
exists and equals $P_2(n)$ for all pairs of integers $m \geq n \geq 0$. Now consider the identity
\[
\frac{|\mathcal{D}^-(X) \cap D_{2, m}(X)|}{|D_{2, m}(X)|} = \sum_{n = 0}^m \frac{|\mathcal{D}^-(X) \cap D_{2, m}(X) \cap D_{3, n}(X)|}{|D_{2, m}(X) \cap D_{3, n}(X)|} \cdot \frac{|D_{2, m}(X) \cap D_{3, n}(X)|}{|D_{2, m}(X)|}.
\]
Taking $X \rightarrow \infty$, we obtain
\begin{align}
\label{eRecursionP2m}
P_2(m) = \sum_{n = 0}^m P_3(m, n) \cdot \frac{P(m, m, n)}{2^m} = \sum_{n = 0}^m P_2(n) \cdot \frac{P(m, m, n)}{2^m}.
\end{align}
We recall the identity \cite[eq. (A.2)]{KP3}
\[
\frac{1}{2^{m + 1} - 1} = \sum_{n = 0}^m \frac{1}{2^{n + 1} - 1} \cdot \frac{P(m, m, n)}{2^m}.
\]
Together with equation (\ref{eRecursionP2m}) and $P_2(0) = 1$, this implies equation (\ref{eSteheuristic}).
\end{proof}

\subsection{Boxes}
To prove Theorem \ref{tRank}, we are going to cover $\mathcal{D}(N)$ by boxes $X$ with some desirable properties.

\begin{theorem}
\label{tBox}
There are real numbers $A, N_0 > 0$ such that for all reals $N > N_0$, all integers $m \geq 3$, all sequences of integers $n_2, \dots, n_{m + 1} \geq 0$, all integers $r$ satisfying equation (\ref{eErdosKac}), all integers $0 \leq k \leq r$ and all excellent $(N, k, r)$-boxes $X$
\[
\left|\left|X \cap \bigcap_{i = 2}^{m + 1} D_{i, n_i}(N)\right| - \frac{P(n_m, n_m, n_{m + 1})}{2^{n_m}} \cdot \left|X \cap \bigcap_{i = 2}^m D_{i, n_i}(N) \right|\right| \leq \frac{A \cdot |X|}{(\log \log \log \log N)^{\frac{c}{m^26^m}}}.
\]
\end{theorem}

\begin{proof}[Proof that Theorem \ref{tBox} implies Theorem \ref{tRank}.]
By equation (\ref{eBadr}) of Theorem \ref{tSquarefree} we have
\[
\bigcup_{|r -  \frac{1}{2}\log \log N| \geq (\log \log N)^{2/3}} |\mathcal{D}_r(N)| \ll \frac{N}{(\log \log N)^{1/100}}.
\]
Hence, by Remark \ref{rBox}, it suffices to prove that there exist real numbers $A', N_0 > 0$ such that
\begin{multline*}
\left|\left|\mathcal{D}_r(N) \cap \bigcap_{i = 2}^{m + 1} D_{i, n_i}(N)\right| - \frac{P(n_m, n_m, n_{m + 1})}{2^{n_m}} \cdot \left|\mathcal{D}_r(N) \cap \bigcap_{i = 2}^m D_{i, n_i}(N) \right|\right| \\
\leq \frac{A' \cdot |\mathcal{D}_r(N)|}{(\log \log \log \log N)^{\frac{c}{m^26^m}}} + A' \cdot \left|\mathcal{D}_r(N) - \mathcal{D}_r\left(N\left(1 + \frac{1}{C_{\text{compr}}}\right)^{-r}\right)\right|
\end{multline*}
for all real numbers $N > N_0$ and all positive integers $r$ satisfying equation (\ref{eErdosKac}). Now apply Theorem \ref{tToBoxes} twice with
\[
V = \mathcal{D}_r(N) \cap \bigcap_{i = 2}^{m + 1} D_{i, n_i}(N) \quad \text{ and } \quad V' = \mathcal{D}_r(N) \cap \bigcap_{i = 2}^m D_{i, n_i}(N)
\]
and with $W$ the maximal subset of $\mathcal{D}_r(N)$ that intersects trivially with all the $W(d_i)$ for $d_i > D_{\text{Siegel}}$ and also intersects trivially with the set
\[
\{d \in \mathcal{D}_r(N) : d \text{ is not } N\text{-nice}\}.
\]
To apply Theorem \ref{tToBoxes}, we need a lower bound for $|W|$ and a good estimate for $|V \cap X|$ and $|V' \cap X|$ for boxes $X$ that intersect $W$ non-trivially (note that such boxes are $(N, k, r)$-excellent). From Theorem \ref{tSquarefree} and Proposition \ref{pSiegel} we get the required lower bound for $W$. Let $X$ be a box that intersects $W$ non-trivially. Repeated application of Theorem \ref{tBox} yields
\begin{align*}
\left||V \cap X| - \left(\prod_{i = 2}^m \frac{P(n_i, n_i, n_{i + 1})}{2^{n_i}}\right) \cdot |X \cap D_{2, n_2}(N)|\right| &\leq \sum_{i = 2}^m \frac{A \cdot |X|}{(\log \log \log \log N)^{\frac{c}{i^2 6^i}}} \\
&\leq \frac{2A \cdot |X|}{(\log \log \log \log N)^{\frac{c}{m^2 6^m}}}
\end{align*}
for $N$ sufficiently large, and similarly for $|V' \cap X|$. Here we used that
\[
(\log \log \log \log N)^{\frac{c}{m^2 6^m}} \geq 2,
\]
which we may always assume, since otherwise Theorem \ref{tRank} is trivially true. Inserting the estimate for $|X \cap D_{2, n_2}(N)|$ from equation (\ref{e4rankBox}) shows that
\[
\left||V \cap X| - \left(\prod_{i = 2}^m \frac{P(n_i, n_i, n_{i + 1})}{2^{n_i}}\right) \cdot P_{\text{Sym}}(r - 1, n_2) \cdot |X|\right| \leq \frac{3A \cdot |X|}{(\log \log \log \log N)^{\frac{c}{m^2 6^m}}}
\]
for $N$ sufficiently large, and likewise for $|V' \cap X|$. This finishes the proof of the reduction step.
\end{proof}

\subsection{Genericity}
We will split $X$ as the union over $X(a)$ and then prove equidistribution for most $a$. If $X$ is an excellent $(N, k, r$)-box, then there exists $x \in X$ and an integer $k_{\text{gap}}$ satisfying $\frac{r^{1/2}}{2} < k_{\text{gap}} + 1 < \frac{r}{2}$ and
\begin{align}
\label{ekgap}
\log x_{k_{\text{gap}} + 1} \geq (\log \log x_{k_{\text{gap}} + 1})^2 \cdot \log \log \log N \cdot \sum_{j = 1}^{k_{\text{gap}}} \log x_j,
\end{align}
where $x = (x_1, \dots, x_r)$. From now on fix such a choice of $k_{\text{gap}}$.

\begin{mydef}
Let $X = X_1 \times \dots \times X_r$ be a product space, let $S \subseteq [r]$ and let $Q \in \prod_{i \in S} X_i$. We say that $Q$ is $a$-consistent with $a: M_r \sqcup M_{r, \varnothing} \rightarrow \FF_2$ if
\[
\left(\frac{\pi_i(Q)}{\pi_j(Q)}\right) = \iota(a(i, j)) \textup{ for all distinct } i, j \in S.
\]
We define $X(a, Q)$ to be the subset of $x \in X(a)$ for which $\pi_S(x) = Q$. Finally, given $j \in [r] - S$, let $X_j(a, Q)$ be the subset of $p \in X_j$ with
\[
\left(\frac{\pi_i(Q)}{p}\right) = \iota(a(i, j))
\]
for all $i \in S$.
\end{mydef}

\begin{mydef}
\label{dGen}
Let $r$ be a positive integer and let $a: M_r \sqcup M_{r, \varnothing} \rightarrow \FF_2$. Recall that we have associated to $a: M_r \sqcup M_{r, \varnothing} \rightarrow \FF_2$ a $r \times r$ matrix that we call $A(a)$, see equation (\ref{eMatrixa}). The left and right kernel of $A(a)$ are naturally subspaces of $\FF_2^r$, and viewed as such they coincide and contain the element $R := (1, \dots, 1)$. Let us call this subspace $V_{a, 2}$. We say that $a: M_r \sqcup M_{r, \varnothing} \rightarrow \FF_2$ is $(N, m)$-generic for a box $X$ if
\begin{itemize}
\item putting
\[
n_{\textup{max}} := \left \lfloor \sqrt{\frac{10c}{m^2 6^m} \log \log \log \log \log N} \right \rfloor,
\]
we have 
\begin{align}
\label{eMax4rank}
\dim V_{a, 2} \leq n_{\textup{max}};
\end{align}
\item we have for all $j > k$
\begin{align}
\label{eXjaQ}
|X_j(a, Q)| \geq \frac{|X_j|}{(\log s_{k + 1})^{100}},
\end{align}
where $Q$ is the unique element of $X_1 \times \dots \times X_k$. Furthermore, $Q$ is $a$-consistent;
\item setting $\alpha_{\textup{pre}}$ to be the number of integers $i$ satisfying $k_{\textup{gap}}/2 \leq i \leq k_{\textup{gap}}$, we have for all elements $w \in V_{a, 2} \setminus \langle R \rangle$ and all $j \in \FF_2$
\begin{align}
\label{eGen1}
\left|\left|\left\{i \in [r] : \frac{k_{\textup{gap}}}{2} \leq i \leq k_{\textup{gap}} \textup{ and } \pi_i(w) = j\right\}\right| - \frac{\alpha_{\textup{pre}}}{2}\right| \leq \frac{k_{\textup{gap}}}{\log \log \log N}
\end{align}
and
\begin{align}
\label{eGen2}
\left|\left|\left\{i \in [r] : k_{\textup{gap}} < i \leq 2k_{\textup{gap}} \textup{ and } \pi_i(w) = j\right\}\right| - \frac{k_{\textup{gap}}}{2}\right| \leq \frac{k_{\textup{gap}}}{\log \log \log N}.
\end{align}
\end{itemize}
\end{mydef}

Equation (\ref{eGen1}) and equation (\ref{eGen2}) give us very fine control over $V_{a, 2}$ as the following lemma shows, see also \cite[Lemma 6.9]{CKMP}.

\begin{lemma}
\label{lVariable}
Suppose that $a: M_r \sqcup M_{r, \varnothing} \rightarrow \FF_2$ is $(N, m)$-generic for a box $X$. Suppose that $w_1, \dots, w_d, R \in V_{a, 2}$ are linearly independent. Then we have for all $\mathbf{v} \in \FF_2^d$
\[
\left|\left|\left\{i \in [r] : \frac{k_{\textup{gap}}}{2} \leq i \leq k_{\textup{gap}} \textup{ and } \pi_i(w_j) = \pi_j(\mathbf{v}) \textup{ for all } 1 \leq j \leq d\right\}\right| - \frac{\alpha_{\textup{pre}}}{2^d}\right| \leq \frac{3^d \cdot k_{\textup{gap}}}{\log \log \log N}
\]
and similarly
\[
\left|\left|\left\{i \in [r] : k_{\textup{gap}} < i \leq 2k_{\textup{gap}} \textup{ and } \pi_i(w_j) = \pi_j(\mathbf{v}) \textup{ for all } 1 \leq j \leq d\right\}\right| - \frac{k_{\textup{gap}}}{2^d}\right| \leq \frac{3^d \cdot k_{\textup{gap}}}{\log \log \log N}.
\]
\end{lemma}

\begin{proof}
Let us prove the first part by induction on $d$, the second part follows using the same argument. For $d = 0$ the statement follows immediately from the definition of $\alpha_{\text{pre}}$, while for $d = 1$ the statement falls as a consequence of equation (\ref{eGen1}). Now suppose that $d > 1$ and define for $\mathbf{w} \in \FF_2^d$
\[
f(\mathbf{w}) = \left|\left\{i \in [r] : \frac{k_{\textup{gap}}}{2} \leq i \leq k_{\textup{gap}} \textup{ and } \pi_i(w_j) = \pi_j(\mathbf{w}) \textup{ for all } 1 \leq j \leq d\right\}\right|.
\]
Let $\mathbf{v} \in \FF_2^d$ be given and define $\mathbf{v}_1, \mathbf{v}_2, \mathbf{v}_3$ as the unique unordered triple of distinct vectors that have the same projection as $\mathbf{v}$ on the first $d - 2$ coordinates but are not equal to $\mathbf{v}$. We have the inequalities
\begin{align*}
2\left|f(\mathbf{v}) - \frac{\alpha_{\textup{pre}}}{2^d}\right| &\leq \left|\left(\sum_{i = 1}^3 f(\mathbf{v}) + f(\mathbf{v}_i)\right) - \frac{3\alpha_{\textup{pre}}}{2^{d - 1}}\right| + \left|\frac{\alpha_{\textup{pre}}}{2^{d - 2}} - f(\mathbf{v}) - \sum_{i = 1}^3 f(\mathbf{v}_i)\right| \\
&\leq \left(\sum_{i = 1}^3 \left|f(\mathbf{v}) + f(\mathbf{v_i}) - \frac{\alpha_{\textup{pre}}}{2^{d - 1}}\right|\right) + \left|\frac{\alpha_{\textup{pre}}}{2^{d - 2}} - f(\mathbf{v}) - \sum_{i = 1}^3 f(\mathbf{v}_i)\right| \\
&\leq \frac{(3 \cdot 3^{d - 1} + 3^{d - 2})\cdot k_{\textup{gap}}}{\log \log \log N},
\end{align*}
where the last inequality follows from the induction hypothesis applied to $w_1, \dots, w_{d - 2}$ and $w_1, \dots, w_{d - 2}, w$ with $w \in \{w_{d - 1}, w_d, w_{d - 1} + w_d\}$.
\end{proof}

Instead of controlling the $2^k$-ranks of the narrow class group, we will control the full Artin pairing of the narrow class group, which is a finer invariant than just the sequence of $2^k$-ranks. We will make this precise in our next definition.

\begin{mydef}
Let $X$ be a box and let $a: M_r \sqcup M_{r, \varnothing} \rightarrow \FF_2$. A sequence of bilinear pairings $\{\textup{Art}_i: A_i \times B_i \rightarrow \FF_2\}_{2 \leq i < m}$, is called valid if
\begin{itemize}
\item $A_2 = B_2 = V_{a, 2}$;
\item the left kernel of $\textup{Art}_i$ is $A_{i + 1}$ and the right kernel of $\textup{Art}_i$ is $B_{i + 1}$ for $2 \leq i < m - 1$;
\item $R$ is in the right kernel of $\textup{Art}_i$ for $2 \leq i < m$.
\end{itemize}
The sequence is called Pellian if furthermore $R$ is in the left kernel of $\textup{Art}_i$ for all $2 \leq i < m$. For such a sequence of bilinear pairings, we define $A_m$ to be the left kernel of $\textup{Art}_{m - 1}$ and $B_m$ to be the right kernel of $\textup{Art}_{m - 1}$.

Meanwhile for $x \in X$, we get for every integer $k \geq 1$ a pairing
\[
\textup{Art}_{k, x}: 2^{k - 1} \textup{Cl}(\Q(\sqrt{x}))[2^k] \times 2^{k - 1} \textup{Cl}^\vee(\Q(\sqrt{x}))[2^k] \rightarrow \FF_2,
\]
see Subsection \ref{ssArtinpairing}. Genus theory gives natural surjections $\FF_2^r \rightarrow \textup{Cl}(\Q(\sqrt{x}))[2]$ and $\FF_2^r \rightarrow \textup{Cl}^\vee(\Q(\sqrt{x}))[2]$. More explicitly, the map $\FF_2^r \rightarrow \textup{Cl}(\Q(\sqrt{x}))[2]$ is given by
\[
(e_1, \dots, e_r) \mapsto \textup{Up}_{\Q(\sqrt{x})/\Q}\left(\prod_{i = 1}^r p_i^{e_i}\right)
\]
and the map $\FF_2^r \rightarrow \textup{Cl}^\vee(\Q(\sqrt{x}))[2]$ is given by
\[
(e_1, \dots, e_r) \mapsto \sum_{i = 1}^r e_i \chi_{p_i},
\]
where $x = p_1 \cdot \ldots \cdot p_r$ with $p_1 < \dots < p_r$. Pulling back, this induces a pairing
\[
\FF_2^r \times \FF_2^r \rightarrow \FF_2,
\]
and for $x \in X(a)$ the left and right kernel both equal $V_{a, 2}$. Then we get natural surjections $V_{a, 2} \rightarrow  2\textup{Cl}(\Q(\sqrt{x}))[4]$ and $V_{a, 2} \rightarrow 2\textup{Cl}^\vee(\Q(\sqrt{x}))[4]$ and hence a pairing
\[
V_{a, 2} \times V_{a, 2} \rightarrow \FF_2.
\]
Continuing this process gives for each $x \in X(a)$ a valid sequence of Artin pairings, which we will also call $\textup{Art}_{k, x}$. Furthermore, the negative Pell equation is soluble for $x$ if and only if its sequence of Artin pairings is Pellian.

We then define for a sequence of valid Artin pairings $\{\textup{Art}_i\}_{2 \leq i < m}$
\[
X(a, \{\textup{Art}_i\}_{2 \leq i < m}) := \left\{x \in X(a) : \textup{ the Artin pairing of } x \textup{ equals } \{\textup{Art}_i\}_{2 \leq i < m}\right\}.
\]
\end{mydef}

\noindent We will now reduce to the case where $a: M_r \sqcup M_{r, \varnothing} \rightarrow \FF_2$ is $(N, m)$-generic for $X$ and the first $m - 1$ Artin pairings are given. The next theorem essentially says that the Artin pairing is a random pairing on $A_m \times B_m$ except that $R$ must always be in the right kernel.

\begin{theorem}
\label{tGeneric}
There are real numbers $A, N_0 > 0$ such that for all reals $N > N_0$, all integers $m \geq 3$, all integers $r$ satisfying equation (\ref{eErdosKac}), all integers $0 \leq k \leq r$, all excellent $(N, k, r)$-boxes $X$, all $(N, m)$-generic $a: M_r \sqcup M_{r, \varnothing} \rightarrow \FF_2$ for $X$, all Pellian sequences of Artin pairings $\{\textup{Art}_k\}_{2 \leq k < m}$ and all Artin pairings $\textup{Art}_m: A_m \times B_m \rightarrow \FF_2$ with $R$ in the right kernel
\[
\left|\left|X(a, \{\textup{Art}_k\}_{2 \leq k \leq m})\right| - \frac{|X(a, \{\textup{Art}_k\}_{2 \leq k < m})|}{2^{\dim_{\FF_2} A_m \cdot (-1 + \dim_{\FF_2} A_m)}}\right| \leq \frac{A \cdot |X(a)|}{(\log \log \log \log N)^{\frac{100c}{m6^m}}}.
\]
\end{theorem}

\begin{proof}[Proof that Theorem \ref{tGeneric} implies Theorem \ref{tBox}.]
We claim that there exists an absolute constant $A' > 0$ such that
\begin{align}
\label{eNotGeneric}
\sum_{\substack{a: M_r \sqcup M_{r, \varnothing} \rightarrow \FF_2 \\ a \text{ not } (N, m)\text{-generic}}} |X(a)| \leq A' \cdot |X| \cdot 2^{\frac{-2c \log \log \log \log \log N}{m^2 6^m}}.
\end{align}
As a first step, we estimate
\begin{align}
\label{eBadQ}
\sum_{\substack{a: M_r \sqcup M_{r, \varnothing} \rightarrow \FF_2 \\ (\ref{eXjaQ}) \text{ fails for some } j > k}} |X(a)|.
\end{align}
For every $j > k$ we define an equivalence relation $\sim_j$ by setting $a \sim_j a'$ if and only if $a(i, j) = a'(i, j)$ for all $1 \leq i \leq k$. Note that if $a \sim_j a'$, then equation (\ref{eXjaQ}) fails for $a$ if and only if it fails for $a'$. By our choice of $C_0$, $D_1$ and equation (\ref{eRegularSpace}), we see that 
\[
k \leq \frac{1}{5} \log \log \log N
\]
for $N$ sufficiently large. Hence $\sim_j$ has at most
\[
2^k \leq 2^{\frac{1}{5} \log \log \log N}
\]
equivalence classes. Furthermore, in a given equivalence class we have the trivial bound
\[
\left|\bigcup_{\substack{a \sim_j a' \\ |X_j(a', Q)| \leq \frac{|X|}{(\log s_{k + 1})^{100}}}} X(a)\right| \leq \frac{|X|}{(\log s_{k + 1})^{100}}.
\]
Summing over all $j > k$ and summing over all equivalence classes gives the desired estimate for equation (\ref{eBadQ}).

We return to bounding equation (\ref{eNotGeneric}), and thanks to our upper bound for equation (\ref{eBadQ}), it suffices to bound the union of $|X(a)|$ for $a$ failing equation (\ref{eMax4rank}), equation (\ref{eGen1}) or equation (\ref{eGen2}). If $\sigma$ is a permutation of $[r]$, then we recall that
\[
\sigma(a)(i, j) = a(\sigma^{-1}(i), \sigma^{-1}(j)).
\]
Note that $a$ satisfies equation (\ref{eMax4rank}) if and only if $\sigma(a)$ does.

Take $k_2$ to be the largest integer smaller than $k_{\text{gap}}/2$ and suppose that $\sigma$ is a permutation of $[r]$ fixing all indices greater than $k_2$: recall that $\mathcal{P}_r(k_2)$ denotes the set of such permutations. For any such permutation $\sigma$, we have that $a$ satisfies equation (\ref{eGen1}) if and only if $\sigma(a)$ does, and similarly for equation (\ref{eGen2}). We apply Proposition \ref{pLegendre2} to obtain
\begin{align}
\label{eTheorem5.9}
\sum_{a: M_r \sqcup M_{r, \varnothing} \rightarrow \FF_2} \left|\frac{k_2! \cdot |X|}{2^{|M_r| + |M_{r, \varnothing}|}} - \sum_{\sigma \in \mathcal{P}_r(k_2)} |X(\sigma(a))| \right| \leq \frac{k_2! \cdot |X|}{k_2^{b_1}} \leq \frac{k_2! \cdot |X|}{(\log \log N)^{b_1b_2}}
\end{align}
for some absolute constants $b_1, b_2 > 0$. Here we used equation (\ref{eErdosKac}) to obtain the lower bound $k_2 \geq (\log \log N)^{b_2}$ for $N$ sufficiently large. We will now give an upper bound for 
\begin{align}
\label{eProportiona}
\frac{|\{a: M_r \sqcup M_{r, \varnothing} \rightarrow \FF_2 : a \text{ fails equation } (\ref{eMax4rank}) \text{ or } (\ref{eGen1}) \text{ or } (\ref{eGen2})\}|}{|\{a: M_r \sqcup M_{r, \varnothing} \rightarrow \FF_2\}|}. 
\end{align}
Note that this is an entirely combinatorial problem about certain symmetric $r \times r$ matrices. Once this is done, equation (\ref{eTheorem5.9}) will give the desired upper bound for equation (\ref{eNotGeneric}).

After dropping a row and a column from $A(a)$, we see that $A(a)$ is a random symmetric matrix. By the explicit formulas for the number of symmetric matrices with a given rank in \cite{Williams}, we see that the proportion of $a: M_r \sqcup M_{r, \varnothing} \rightarrow \FF_2$ with $\dim_{\FF_2} V_{a, 2} > n_{\text{max}}$ is bounded by
\[
O\left(2^{\frac{-2c \log \log \log \log \log N}{m^2 6^m}}\right),
\]
which disposes with the $a$ failing equation (\ref{eMax4rank}).

Next we bound the proportion of $a: M_r \sqcup M_{r, \varnothing} \rightarrow \FF_2$ failing equation (\ref{eGen1}), a similar argument works for equation (\ref{eGen2}). The proportion of $w \in \FF_2^r$ with
\begin{align}
\label{eBadw}
\left|\left|\left\{i \in [r] : \frac{k_{\textup{gap}}}{2} \leq i \leq k_{\textup{gap}} \textup{ and } \pi_i(w) = j\right\}\right| - \frac{\alpha_{\textup{pre}}}{2}\right| > \frac{k_{\textup{gap}}}{\log \log \log N}
\end{align}
is bounded by
\[
O\left(e^{-(\log \log \log N)^{-2} \cdot k_{\text{gap}}}\right)
\]
thanks to Hoeffding's inequality. For any given $w \in \FF_2^r \setminus \langle R \rangle$ we have that the proportion of $a: M_r \sqcup M_{r, \varnothing} \rightarrow \FF_2$ with $w \in V_{a, 2}$ is bounded by $O(2^{-r})$. Using this for all $w \in \FF_2^r \setminus \langle R \rangle$ satisfying equation (\ref{eBadw}) then shows that the proportion of $a: M_r \sqcup M_{r, \varnothing} \rightarrow \FF_2$ failing equation (\ref{eGen1}) is also bounded by
\[
O\left(e^{-(\log \log \log N)^{-2} \cdot k_{\text{gap}}}\right).
\]
Since $k_{\text{gap}} > r^{1/2}/2$, it follows from equation (\ref{eErdosKac}) that this fits in the error term. We conclude that equation (\ref{eProportiona}) is bounded by
\[
A_1 \cdot 2^{\frac{-2c \log \log \log \log \log N}{m^2 6^m}}
\]
for some absolute constant $A_1 > 0$. Then equation (\ref{eTheorem5.9}) implies that there exists an absolute constant $A_2 > 0$ such that
\[
\sum_{\substack{a: M_r \sqcup M_{r, \varnothing} \rightarrow \FF_2 \\ a \text{ not } (N, m)\text{-generic}}} |X(a)| \leq A_2 \cdot |X| \cdot 2^{\frac{-2c \log \log \log \log \log N}{m^2 6^m}}.
\]
Therefore we have established the claimed equation (\ref{eNotGeneric}). To complete the proof of the reduction step, we split $X$ as the union over $X(a)$, removing all $a$ that are not $(N, m)$-generic. For the $a$ that are $(N, m)$-generic, we split each $X(a)$ over all sequences of Artin pairings and use Theorem \ref{tGeneric}. Observe that there are at most
\[
2^{m (\dim_{\FF_2} V_{a, 2})^2} \leq 2^{m n_{\text{max}}^2}
\]
sequences of Artin pairings. By the choice of $n_{\text{max}}$, we complete the proof of the reduction step.
\end{proof}

The next reduction step is a straightforward application of orthogonality of characters.

\begin{theorem}
\label{tF}
There are real numbers $A, N_0 > 0$ such that for all reals $N > N_0$, all integers $m \geq 3$, all integers $r$ satisfying equation (\ref{eErdosKac}), all integers $0 \leq k \leq r$, all excellent $(N, k, r)$-boxes $X$, all $(N, m)$-generic $a: M_r \sqcup M_{r, \varnothing} \rightarrow \FF_2$ for $X$, all Pellian sequences of Artin pairings $\{\textup{Art}_k\}_{2 \leq k < m}$ and all non-trivial linear maps $F: \{A_m \times B_m \rightarrow \FF_2 \textup{ bilinear with } R \textup{ in the right kernel}\}\rightarrow \FF_2$
\[
\left|\sum_{x \in X(a, \{\textup{Art}_k\}_{2 \leq k < m})} \iota(F(\textup{Art}_{m, x}))\right| \leq \frac{A \cdot |X(a)|}{(\log \log \log \log N)^{\frac{100c}{m6^m}}}.
\]
\end{theorem}

The plan is to fix all but $m$ or $m + 1$ indices of the box $X$, and then apply the algebraic and combinatorial results from Sections \ref{sReflection} and \ref{sCombinatorics}. The indices of the box that we do not fix are called \emph{variable indices}. We now lay out exactly what properties we demand from these variable indices as a function of $F$.

\begin{mydef}
\label{dVariable}
Define for $2 \leq i \leq m$
\[
C_i := A_i \cap B_i, \quad d_i := -1 + \dim_{\FF_2} C_i, \quad n_i := -1 + \dim_{\FF_2} A_i.
\]
Pick a basis $v_1, \dots, v_{n_m}, R$ of $A_m$ and a basis $w_1, \dots, w_{n_m}, R$ of $B_m$ such that
\[
v_i = w_i \textup{ for all } 1 \leq i \leq d_m.
\]
Hence $v_1, \dots, v_{d_m}, R$ is a basis of $C_m$. We extend the basis $w_1, \dots, w_{n_m}, R$ of $B_m$ to a basis
\[
w_1, \dots, w_{n_2}, R
\]
of $V_{a, 2}$ such that $w_1, \dots, w_{n_i}, R$ is a basis of $B_i$ for all $2 \leq i \leq m$. Write $\textup{Mat}(a, b, K)$ for the set of $a \times b$ matrices over a field $K$. Using our basis, $\textup{Art}_{m, x}$ is naturally an element of $\textup{Mat}(n_m + 1, n_m + 1, \FF_2)$. More precisely, the entry $(i, j)$ of the associated matrix is
\[
\textup{Art}_{m, x}(v_i, w_j)
\]
for $i, j \in [n_m]$, and we use the last row and column for the pairing with $R$. Since $R$ is in the right kernel of $\textup{Art}_{m, x}$, we shall from now on implicitly view $\textup{Art}_{m, x}$ also as an element of $\textup{Mat}(n_m + 1, n_m, \FF_2)$.

Let $E_{j_3, j_4}$ be the matrix with a $1$ on the entry $(j_3, j_4)$ and $0$ everywhere else. Define $F_{j_1, j_2}: \textup{Mat}(n_m + 1, n_m, \FF_2) \rightarrow \FF_2$ be the unique linear map that sends the matrix $E_{j_3, j_4}$ to $1$ if and only if $j_1 = j_3$ and $j_2 = j_4$. Then we can write write any linear map $F: \textup{Mat}(n_m + 1, n_m, \FF_2) \rightarrow \FF_2$ as
\[
F = \sum_{\substack{1 \leq j_1 \leq n_m + 1 \\ 1 \leq j_2 \leq n_m}} c_{j_1, j_2} F_{j_1, j_2}, \quad c_{j_1, j_2} \in \FF_2.
\]
For every non-trivial map $F$, we pick a pair $(j_1, j_2)$ satisfying $c_{j_1, j_2} = 1$ according to the following rules. In case there exists $(j_1, j_2)$ such that $c_{j_1, j_2} = 1$, such that $j_1 \leq n_m$ and such that one of the following three conditions is satisfied
\begin{itemize}
\item $n_m \geq j_1 > d_m$ or
\item $n_m \geq j_2 > d_m$ or
\item $1 \leq j_1, j_2 \leq d_m$ and $c_{j_2, j_1} = 0$,
\end{itemize}
then we fix such a choice of $(j_1, j_2)$. We say that we are in case $I$. If there is no such pair $(j_1, j_2)$, but we have $c_{n_m + 1, j} = 1$ for some $j$, then we fix a choice of $(j_1, j_2) = (n_m + 1, j)$ and we say that we are in case $II$. Again if there is no such pair $(j_1, j_2)$, but we have $c_{j, j} = 1$ for some $j$, then we fix $(j_1, j_2) = (j, j)$ and we say that we are in case $III$. Finally, in the remaining cases we fix a pair $(j_1, j_2)$ with $j_1 \neq j_2$ and $c_{j_1, j_2} = c_{j_2, j_1} = 1$ and we say that we are in case $IV$. Observe that we can always find such a pair $(j_1, j_2)$ since $F$ is assumed to be non-trivial.

We say that a subset $S$ of the integers is a set of variable indices for a non-trivial linear map $F: \textup{Mat}(n_m + 1, n_m, \FF_2) \rightarrow \FF_2$ if the following properties are satisfied
\begin{itemize}
\item if we are in case $I$, then we demand that $S$ contains $m + 1$ elements of which $m - 1$ elements are contained in
\[
\{k_{\textup{gap}}/2 \leq i \leq k_{\textup{gap}}\} \cap \bigcap_{l = 1}^{n_2} \{i \in [r]: \pi_i(w_l) = 0\},
\]
one element is contained in the intersection of $\{k_{\textup{gap}}/2 \leq i \leq k_{\textup{gap}}\}$ with
\[
\bigcap_{\substack{1 \leq l \leq n_m \\ j_2 \leq d_m \Rightarrow l \neq j_2}} \{i \in [r]: \pi_i(v_l) = 0\} \cap \bigcap_{\substack{1 \leq l \leq n_m \\ l \neq j_2}} \{i \in [r]: \pi_i(w_l) = 0\} \cap \{i \in [r] : \pi_i(w_{j_2}) = 1\},
\]
and one element is contained in the intersection of $\{k_{\textup{gap}} < i \leq 2k_{\textup{gap}}\}$ with
\[
\bigcap_{\substack{1 \leq l \leq n_m \\ l \neq j_1}} \{i \in [r]: \pi_i(v_l) = 0\}  \cap \bigcap_{\substack{1 \leq l \leq n_m \\ j_1 \leq d_m \Rightarrow l \neq j_1}} \{i \in [r]: \pi_i(w_l) = 0\} \cap \{i \in [r] : \pi_i(v_{j_1}) = 1\};
\]
\item if we are in case $II$, then we pick $S$ such that $|S| = m$, of which $m - 1$ are contained in
\[
\{k_{\textup{gap}}/2 \leq i \leq k_{\textup{gap}}\} \cap \bigcap_{l = 1}^{n_2} \{i \in [r]: \pi_i(w_l) = 0\}
\]
and one element is contained in the intersection of $\{k_{\textup{gap}} < i \leq 2k_{\textup{gap}}\}$ with
\[
\bigcap_{\substack{1 \leq l \leq n_m \\ j_2 \leq d_m \Rightarrow l \neq j_2}} \{i \in [r]: \pi_i(v_l) = 0\}  \cap \bigcap_{\substack{1 \leq l \leq n_m \\ l \neq j_2}} \{i \in [r]: \pi_i(w_l) = 0\} \cap \{i \in [r] : \pi_i(w_{j_2}) = 1\};
\]
\item if we are in case $III$, recall that $j = j_1 = j_2$. We choose $S$ such that $|S| = m$, of which $m - 1$ are contained in
\[
\{k_{\textup{gap}}/2 \leq i \leq k_{\textup{gap}}\} \cap \bigcap_{\substack{l = 1 \\ l \neq j}}^{n_2} \{i \in [r]: \pi_i(w_l) = 0\} \cap \{i \in [r] : \pi_i(w_j) = 1\},
\]
and one element is contained in
\[
\{k_{\textup{gap}} < i \leq 2k_{\textup{gap}}\} \cap \bigcap_{l = 1}^{n_m} \{i \in [r]: \pi_i(v_l) = \pi_i(w_l) = 0\};
\]
\item if we are in case $IV$, then we pick $S$ such that $|S| = m$, of which $m - 1$ are contained in
\[
\{k_{\textup{gap}}/2 \leq i \leq k_{\textup{gap}}\} \cap \bigcap_{\substack{l = 1 \\ l \neq j_2}}^{n_2} \{i \in [r]: \pi_i(w_l) = 0\} \cap \{i \in [r] : \pi_i(w_{j_2}) = 1\},
\]
and one element is contained in the intersection of $\{k_{\textup{gap}} < i \leq 2k_{\textup{gap}}\}$ with
\[
\bigcap_{\substack{l = 1 \\ l \neq j_1}}^{n_m} \{i \in [r]: \pi_i(v_l) = \pi_i(w_l) = 0\} \cap \{i \in [r] : \pi_i(v_{j_1}) = 1\}.
\]
\end{itemize}
Define $i_{\textup{Cheb}}$ to be the unique element of $S \cap \{k_{\textup{gap}} < i \leq 2k_{\textup{gap}}\}$.
\end{mydef}

\noindent Note that we may always assume that
\begin{align}
\label{emSmall}
m < \log \log \log \log \log \log N,
\end{align}
since otherwise Theorem \ref{tRank} is trivially true. Then Lemma \ref{lVariable} shows that for every non-trivial bilinear map $F$ we can find a set of variable indices $S$ for $F$, provided that we take $N$ sufficiently large. For the final reduction step of this subsection, we fix all primes before $k_{\textup{gap}}$ except those in $S$. For a point $Q \in \prod_{i \in [k_{\text{gap}}] - S} X_i$, define $X(a, Q, \{\textup{Art}_k\}_{2 \leq k < m})$ to be the subset of $x \in X(a, \{\textup{Art}_k\}_{2 \leq k < m})$ for which $\pi_{[k_{\text{gap}}] - S}(x) = Q$.

\begin{theorem}
\label{tPrePoint}
There are real numbers $A, N_0 > 0$ such that for all reals $N > N_0$, all integers $m \geq 3$, all integers $r$ satisfying equation (\ref{eErdosKac}), all integers $0 \leq k \leq r$, all excellent $(N, k, r)$-boxes $X$, all $(N, m)$-generic $a: M_r \sqcup M_{r, \varnothing} \rightarrow \FF_2$ for $X$, all Pellian sequences of Artin pairings $\{\textup{Art}_k\}_{2 \leq k < m}$, all non-trivial linear maps $F: \textup{Mat}(n_m + 1, n_m, \FF_2) \rightarrow \FF_2$, all variable indices $S$ for $F$ and all $a$-consistent $Q \in \prod_{i \in [k_{\textup{gap}}] - S} X_i$ such that
\begin{align}
\label{eXjaQS}
|X_j(a, Q)| \geq 4^{-m \cdot k_{\textup{gap}}} \cdot |X_j|
\end{align}
for all $j \in S \cap [k_{\textup{gap}}]$, we have
\[
\left|\sum_{x \in X(a, Q, \{\textup{Art}_k\}_{2 \leq k < m})} \iota(F(\textup{Art}_{m, x}))\right| \leq \frac{A \cdot |X(a, Q)|}{(\log \log \log \log N)^{\frac{100c}{m6^m}}}.
\]\end{theorem}

\begin{proof}[Proof that Theorem \ref{tPrePoint} implies Theorem \ref{tF}.]
First, we use the triangle inequality
\[
\left|\sum_{x \in X(a, \{\textup{Art}_k\}_{2 \leq k < m})} \iota(F(\textup{Art}_{m, x}))\right| \leq \sum_{Q \in \prod_{i \in [k_{\textup{gap}}] - S} X_i} \left|\sum_{x \in X(a, Q, \{\textup{Art}_k\}_{2 \leq k < m})} \iota(F(\textup{Art}_{m, x}))\right|.
\]
In case $Q$ satisfies equation (\ref{eXjaQS}), the desired upper bound follows immediately from Theorem \ref{tPrePoint}. It remains to bound
\[
\sum_{\substack{Q \in \prod_{i \in [k_{\textup{gap}}] - S} X_i \\ Q \text{ fails equation } (\ref{eXjaQS})}} \left|\sum_{x \in X(a, Q, \{\textup{Art}_k\}_{2 \leq k < m})} \iota(F(\textup{Art}_{m, x}))\right| \leq \sum_{\substack{Q \in \prod_{i \in [k_{\textup{gap}}] - S} X_i \\ Q \text{ fails equation } (\ref{eXjaQS})}} |X(a, Q)|.
\]
Let $Q'$ be the unique element of $X_1 \times \dots \times X_k$ and let $X'$ be the product space
\[
\prod_{i \in [k_{\text{gap}}] - S - [k]} X_i(a, Q').
\]
We apply Proposition \ref{pLegendre} to $(X', Q')$ to deduce that the number of $a$-consistent elements $Q \in \prod_{i \in [k_{\textup{gap}}] - S} X_i(a, Q')$ is bounded by
\[
\frac{2 \cdot |X'|}{2^{\frac{(k_{\text{gap}} - |S| - k)(k_{\text{gap}} - |S| - k - 1)}{2}}}.
\]
Note that condition $(iii)$ of Proposition \ref{pLegendre} is satisfied thanks to equation (\ref{eXjaQ}). Let us now bound $|X(a, Q)|$ for each individual $Q$. Since $Q$ fails equation (\ref{eXjaQS}), we have
\[
\prod_{i \in S \cap [k_{\text{gap}}]} |X_i(a, Q)| \leq \frac{\prod_{i \in S \cap [k_{\text{gap}}]} |X_i|}{4^{m \cdot k_{\text{gap}}}} \leq \frac{\prod_{i \in S \cap [k_{\text{gap}}]} |X_i(a, Q')|}{4^{m \cdot k_{\text{gap}}}} \cdot (\log s_{k + 1})^{100m}.
\]
Therefore
\[
|X(a, Q)| \leq \frac{\prod_{i \in S} |X_i(a, Q')|}{4^{m \cdot k_{\text{gap}}}} \cdot (\log s_{k + 1})^{100m} \cdot \max_{P \in \prod_{i \in S} X_i} X(a, P \times Q).
\]
But for $i > k_{\text{gap}}$, we apply the Chebotarev density theorem to obtain
\[
|X_i(a, P \times Q)| = \frac{|X_i(a, Q')|}{2^{k_{\text{gap}} - k}} \cdot \left(1 + O\left(e^{-k_{\text{gap}}}\right)\right).
\]
We apply Proposition \ref{pLegendre} to
\[
(X, P) \leftarrow \left(\prod_{i \in [r] - [k_{\text{gap}}]} X_i(a, P \times Q), \varnothing\right), \quad k \leftarrow r - k_{\text{gap}}
\]
to deduce that
\[
\left||X(a, P \times Q)| - \frac{\prod_{i > k_{\text{gap}}} |X_i(a, P \times Q)|}{2^{\frac{(r - k_{\text{gap}})(r - k_{\text{gap}} - 1)}{2}}}\right| \leq \frac{\prod_{i > k_{\text{gap}}} |X_i(a, P \times Q)|}{\log N \cdot 2^{\frac{(r - k_{\text{gap}})(r - k_{\text{gap}} - 1)}{2}}}.
\]
We conclude that
\begin{multline}
\label{eUBXiaQ}
\sum_{\substack{Q \in \prod_{i \in [k_{\textup{gap}}] - S} X_i \\ Q \text{ fails equation } (\ref{eXjaQS})}} |X(a, Q)| \leq \\
\frac{4 \cdot (\log s_{k + 1})^{100m} \cdot \prod_{i \in [r] - [k]} |X_i(a, Q')|}{2^{\frac{(k_{\text{gap}} - k - |S|)(k_{\text{gap}} - k - |S| - 1)}{2}} \cdot 4^{m \cdot k_{\text{gap}}} \cdot 2^{(k_{\text{gap}} - k) \cdot (r - k_{\text{gap}}) + \frac{(r - k_{\text{gap}})(r - k_{\text{gap}} - 1)}{2}}}.
\end{multline}
A final application of Proposition \ref{pLegendre} to the prebox
\[
\left(\prod_{i \in [r] - [k]} X_i(a, Q'), Q'\right)
\]
shows that
\begin{align}
\label{eLBXiaQ}
2|X(a)| \geq \frac{\prod_{i \in [r] - [k]} |X_i(a, Q')|}{2^{(r - k)(r - k - 1)}}. 
\end{align}
The proposition follows from equations (\ref{eUBXiaQ}) and (\ref{eLBXiaQ}).
\end{proof}

\subsection{A second moment computation}
We have now arrived at a critical point of the proof. We have a rather unstructured set $X(a, Q, \{\textup{Art}_k\}_{2 \leq k < m})$. To make matters even worse, we do not know yet that the expansion maps exist to apply our algebraic theorems. We will cover $X(a, Q, \{\textup{Art}_k\}_{2 \leq k < m})$ by small product spaces and then use a second moment trick to reduce to such product spaces. We will then be able to apply the algebraic and combinatorial results from the previous sections.

\begin{mydef}
Let $N$, $m$, $k$, $r$, $X$, $a$, $Q$, $\{\textup{Art}_k\}_{2 \leq k < m}$, $F$ and $S$ be as above. Put
\[
M_{\textup{box}} := \left \lfloor (\log \log \log \log N)^{\frac{1}{5(m + 1)}} \right \rfloor.
\]
It follows from equation (\ref{emSmall}) that $M_{\textup{box}} \geq 2$ for sufficiently large $N$. Define
\[
S' := S \cap [k_{\textup{gap}}],
\]
so that $S = S' \cup \{i_{\textup{Cheb}}\}$. Let $Z_i \subseteq X_i$ for $i \in S'$ and put
\[
Z := \prod_{i \in S'} Z_i.
\]
We say that $Z$ is a great product space if
\begin{itemize}
\item $|Z_i| = M_{\textup{box}}$ for all $i \in S'$;
\item $Z_i \subseteq X_i(a, Q)$ for all $i \in S'$;
\item we have for all $z \in Z$ and all distinct $i, j \in S'$
\[
\left(\frac{\pi_i(z)}{\pi_j(z)}\right) = \iota(a(i, j));
\]
\item if we are in case $I$, let $i_{\textup{char}}$ be the unique index in $S'$ with $\pi_{i_{\textup{char}}}(w_{j_2}) = 1$. Then we want that for all subsets $T \subseteq S' - \{i_{\textup{char}}\}$, all $\bar{z} \in \textup{Cube}(Z, T \cup \{i_{\textup{char}}\})$, there exists an expansion map $\phi_{\pi_T(\bar{z}); \textup{pr}_1(\pi_{i_{\textup{char}}}(\bar{z})) \textup{pr}_2(\pi_{i_{\textup{char}}}(\bar{z}))}$, in which all primes in $Q$, $(2)$ and $\infty$ split completely in case $T \subset S' - \{i_{\textup{char}}\}$. We put
\[
M_\circ(Z) := \prod_{T \subset S' - \{i_{\textup{char}}\}} \prod_{\bar{z} \in \textup{Cube}(Z, T \cup i_{\textup{char}})} L(\phi_{\pi_T(\bar{z}); \textup{pr}_1(\pi_{i_{\textup{char}}}(\bar{z})) \textup{pr}_2(\pi_{i_{\textup{char}}}(\bar{z}))})
\]
and
\[
M(Z) := \prod_{\bar{z} \in \textup{Cube}(Z, S')} L(\phi_{\pi_{S' - \{i_{\textup{char}}\}}(\bar{z}); \textup{pr}_1(\pi_{i_{\textup{char}}}(\bar{z})) \textup{pr}_2(\pi_{i_{\textup{char}}}(\bar{z}))});
\]
\item if we are in case $II$, $III$ or $IV$, we demand that for all subsets $T \subseteq S'$, all $\bar{z} \in \textup{Cube}(Z, T)$ and all $i \in T$, there exists an expansion map $\phi_{\pi_{T - \{i\}}(\bar{z}); \textup{pr}_1(\pi_i(\bar{z})) \textup{pr}_2(\pi_i(\bar{z}))}$ in which all primes in $Q$, $(2)$ and $\infty$ split completely. Furthermore, for all subsets $T \subseteq S'$, all $\bar{z} \in \textup{Cube}(Z, T)$, there exists an expansion map $\phi_{\pi_T(\bar{z}); -1}$, in which all odd primes in $Q$ and $\textup{Up}_{\Q(i)/\Q}(2)$ split completely in case $T \subset S'$. Set
\[
M_\circ(Z) := \prod_{i \in S'} \prod_{\bar{z} \in \textup{Cube}(Z, S')} L(\phi_{\pi_{S' - \{i\}}(\bar{z}); \textup{pr}_1(\pi_i(\bar{z})) \textup{pr}_2(\pi_i(\bar{z}))}) \times \prod_{T \subset S'} \prod_{\bar{z} \in \textup{Cube}(Z, T)} L(\phi_{\pi_T(\bar{z}); -1})
\]
and
\[
M(Z) := \prod_{\bar{z} \in \textup{Cube}(Z, S')} L(\phi_{\pi_{S'}(\bar{z}); -1}).
\]
\end{itemize}
For $i > k_{\textup{gap}}$, define $X_i(a, Q, M_\circ(Z))$ to be those primes $p \in X_i(a, Q)$ that split completely in $M_\circ(Z)$ and
\[
\left(\frac{z}{p}\right) = \iota(a(j, i)) \textup{ for all } j \in S' \textup{ and all } z \in Z_j.
\]
This is equivalent to $\textup{Frob}_p$ landing in a given central element of the Galois group of the compositum of $M_\circ(Z)$ and $\Q(\sqrt{z})$ with $z$ equal to $-1$, a prime in $Q$ or a prime in $Z_j$ for some $j \in S'$. If $Z$ is a great product space, we define
\[
\widetilde{Z} := Q \times Z \times \prod_{i > k_{\textup{gap}}} X_i(a, Q, M_\circ(Z)).
\]
\end{mydef}

The following lemma constructs an additive system that will aid us in producing expansion maps. Define
\[
X_{\text{pre}} := \prod_{i \in S'} X_i(a, Q).
\]

\begin{lemma}
\label{lASphi}
Let $W \subseteq X_{\textup{pre}}$. Then there exists an additive system $(C_T, C_T^{\textup{acc}}, F_T, A_T)_{T \subseteq S'}$ on $(X_{\textup{pre}}, S')$ with the following properties
\begin{itemize}
\item $C_\varnothing^{\textup{acc}} = W$;
\item $|A_T| \leq 2^{(|S'| + 100) \cdot (k_{\textup{gap}} + |S'| + 100)}$ for all $T \subseteq S'$;
\item suppose that we are in case $I$. If $\bar{x} \in C_{S'}^{\textup{acc}}$, then there exists an expansion map 
\[
\phi_{\pi_{S' - \{i_{\textup{char}}\}}(\bar{x}); \textup{pr}_1(\pi_{i_{\textup{char}}}(\bar{x})) \textup{pr}_2(\pi_{i_{\textup{char}}}(\bar{x}))}.
\]
Furthermore, for $i_{\textup{char}} \in T \subset S'$ and $\bar{x} \in C_T^{\textup{acc}}$, there exists an expansion map
\[
\phi_{\pi_{T - \{i_{\textup{char}}\}}(\bar{x}); \textup{pr}_1(\pi_{i_{\textup{char}}}(\bar{x})) \textup{pr}_2(\pi_{i_{\textup{char}}}(\bar{x}))}
\]
in which all primes in $Q$, all primes in $\pi_{S' - T}(\bar{x})$, $(2)$ and $\infty$ split completely;
\item suppose that we are in case $II$, $III$ or $IV$. If $\bar{x} \in C_{S'}^{\textup{acc}}$, then $\phi_{\pi_{S'}(\bar{x}); -1}$ exists. Furthermore, all odd primes in $Q$, all primes in $\pi_{S' - T}(\bar{x})$ and $\textup{Up}_{\Q(i)/\Q}(2)$ split completely in $\phi_{\pi_T(\bar{x}); -1}$ for all $T \subset S'$ and $\bar{x} \in C_T^{\textup{acc}}$. Finally, we demand that for all $T \subseteq S'$, all $\bar{x} \in C_T^{\textup{acc}}$ and all $i \in T$, there exists an expansion map $\phi_{\pi_{T - \{i\}}(\bar{x}); \textup{pr}_1(\pi_i(\bar{x})) \textup{pr}_2(\pi_i(\bar{x}))}$ in which all primes in $Q$, all primes in $\pi_{S' - T}(\bar{x})$, $(2)$ and $\infty$ split completely.
\end{itemize}
\end{lemma}

\begin{proof}
Let us deal with case $I$, the other cases being analogous. Take $C_\varnothing = W$ and take $F_\varnothing$ to be the zero map. We will now inductively construct the maps $F_T$ for $T$ a non-empty subset of $S'$, which determine the additive system. If $T = \{i\}$, we let $F_{\{i\}}$ be the map that sends $\bar{x} \in C_{\{i\}}$ to
\[
\iota^{-1}\left(\frac{\text{pr}_1(\pi_i(\bar{x})) \text{pr}_2(\pi_i(\bar{x}))}{q}\right),
\]
where $q$ runs over all prime divisors of $Q$, all primes in $\pi_j(\bar{x})$ for $j \neq i$ and the prime $(2)$. Now suppose that $|T| > 1$. If $i_{\text{char}} \not \in T$, we let $F_T$ be the zero map. So suppose that $i_{\text{char}} \in T$ and take $\bar{x} \in C_T$. From the induction hypothesis it follows that for all $T' \subset T$ there exists an expansion map
\[
\phi_{\pi_{T' - \{i_{\textup{char}}\}}(\bar{x}); \textup{pr}_1(\pi_{i_{\textup{char}}}(\bar{x})) \textup{pr}_2(\pi_{i_{\textup{char}}}(\bar{x}))}
\]
in which all primes in $Q$, all primes in $\pi_{S' - T'}(\bar{x})$, $(2)$ and $\infty$ split completely. In particular, it follows from Proposition \ref{pCreatePhi} that
\[
\phi_{\pi_{T - \{i_{\textup{char}}\}}(\bar{x}); \textup{pr}_1(\pi_{i_{\textup{char}}}(\bar{x})) \textup{pr}_2(\pi_{i_{\textup{char}}}(\bar{x}))}
\]
exists. Furthermore, we know that the Frobenius symbol of a prime in $Q$, a prime in $\pi_{S' - T}(\bar{x})$, $(2)$ or $\infty$ lands in
\[
Z(\Gal(L(\phi_{\pi_{T - \{i_{\textup{char}}\}}(\bar{x}); \textup{pr}_1(\pi_{i_{\textup{char}}}(\bar{x})) \textup{pr}_2(\pi_{i_{\textup{char}}}(\bar{x}))})/\Q)) \cong \FF_2.
\]
Indeed, recall that $L(\phi_{\pi_{T - \{i_{\textup{char}}\}}(\bar{x}); \textup{pr}_1(\pi_{i_{\textup{char}}}(\bar{x})) \textup{pr}_2(\pi_{i_{\textup{char}}}(\bar{x}))})$ is a central $\FF_2$-extension of
\[
\Q\left(\left\{\sqrt{\text{pr}_1(\pi_i(\bar{x})) \text{pr}_2(\pi_i(\bar{x}))} : i \in T\right\}\right) \cdot \prod_{T' \subset T} L(\phi_{\pi_{T' - \{i_{\textup{char}}\}}(\bar{x}); \textup{pr}_1(\pi_{i_{\textup{char}}}(\bar{x})) \textup{pr}_2(\pi_{i_{\textup{char}}}(\bar{x}))})
\]
by Proposition \ref{prop: field of def of exp maps}. Now let $F_T$ be the map that sends $\bar{x} \in C_T$ to
\[
\phi_{\pi_{T - \{i_{\textup{char}}\}}(\bar{x}); \textup{pr}_1(\pi_{i_{\textup{char}}}(\bar{x})) \textup{pr}_2(\pi_{i_{\textup{char}}}(\bar{x}))}(\text{Frob}(q)),
\]
where $q$ is a prime divisor of $Q$, a prime in $\pi_{S' - T}(\bar{x})$, $(2)$ or $\infty$. Since $\text{Frob}(q)$ lands in the center, it follows from Lemma \ref{lAdditivity} that $F_T$ satisfies equation (\ref{eAdditive}).
\end{proof}

Our next theorem is our final reduction step. We will reduce to spaces of the shape $X(a, Q, \{\textup{Art}_k\}_{2 \leq k < m}) \cap \widetilde{Z}$ with $Z$ a great product space.

\begin{theorem}
\label{tBoxify}
There are real numbers $A, N_0 > 0$ such that for all reals $N > N_0$, all integers $m \geq 3$, all integers $r$ satisfying equation (\ref{eErdosKac}), all integers $0 \leq k \leq r$, all excellent $(N, k, r)$-boxes $X$, all $(N, m)$-generic $a: M_r \sqcup M_{r, \varnothing} \rightarrow \FF_2$ for $X$, all Pellian sequences of Artin pairings $\{\textup{Art}_k\}_{2 \leq k < m}$, all non-trivial linear maps $F: \textup{Mat}(n_m + 1, n_m, \FF_2) \rightarrow \FF_2$, all variable indices $S$ for $F$, all $a$-consistent $Q \in \prod_{i \in [k_{\textup{gap}}] - S} X_i$ and all great product spaces $Z$
\[
\left|\sum_{x \in X(a, Q, \{\textup{Art}_k\}_{2 \leq k < m}) \cap \widetilde{Z}} \iota(F(\textup{Art}_{m, x}))\right| \leq \frac{A \cdot |X(a, Q) \cap \widetilde{Z}|}{(\log \log \log \log N)^{\frac{100c}{m6^m}}}.
\]
\end{theorem}

\begin{proof}[Proof that Theorem \ref{tBoxify} implies Theorem \ref{tPrePoint}.]
We set
\[
R := \left \lfloor e^{e^{\frac{k_{\text{gap}}}{5}}} \right \rfloor
\]	
and we let $Z_{\text{pre}}$ be those $z \in X_{\text{pre}}$ satisfying
\[
\left(\frac{\pi_i(z)}{\pi_j(z)}\right) = \iota(a(i, j)) \text{ for all distinct } i, j \in S'.
\]
Choose a sequence $Z_1, \dots, Z_t$ of maximal length satisfying the following properties
\begin{itemize}
\item each $Z_i \subseteq Z_{\text{pre}}$ is a great product space;
\item we have that $|Z_i \cap Z_j| \leq 1$ for all distinct $i, j \in [t]$;
\item every $z \in Z_{\text{pre}}$ is in at most $R$ of the $Z_i$.
\end{itemize}
Define $Z_{\text{pre}, \text{bad}}$ to be the subset of $z \in Z_{\text{pre}}$ that are in less than $R$ of the $Z_i$ and write $\delta$ for the density of $Z_{\text{pre}, \text{bad}}$ in $X_{\text{pre}}$.

As a first step we aim to upper bound $\delta$. By a straightforward greedy algorithm we can find a subset $W$ of $Z_{\text{pre}, \text{bad}}$ of density at least $\delta/RM_{\text{box}}^m$ such that $|W \cap Z_i| \leq 1$ for all $i \in [t]$. We apply Proposition \ref{pAS} with the additive system from Lemma \ref{lASphi} with $W$ equal to $C_\varnothing^{\textup{acc}}$. This shows that the density of $C_{S'}^{\textup{acc}}$ in $\text{Cube}(X_{\text{pre}}, S')$ is at least
\[
\delta' := \left(\frac{\delta}{2^{(|S'| + 100) \cdot (k_{\textup{gap}} + |S'| + 100)} R M_{\text{box}}^m}\right)^{3^m}.
\]
Now given $x_0 \in W$, we define
\[
W(x_0) = \{x \in W: c(x_0, x) \in C_{S'}^{\textup{acc}}\},
\]
where $c(x_0, x)$ is the unique $\bar{x} \in \text{Cube}(X_{\text{pre}}, S')$ such that
\[
\text{pr}_1(\pi_i(\bar{x})) = \pi_i(x_0), \quad \text{pr}_2(\pi_i(\bar{x})) = \pi_i(x).
\]
There is a natural injective map from $C_{S'}^{\textup{acc}}$ to the disjoint union of the $W(x_0)$ as $x_0$ runs through $W$. This map sends $\bar{x}$ to the unique pair $(w, x_0) \in W \times W$ satisfying
\[
\pi_i(w) = \text{pr}_1(\pi_i(\bar{x})) \text{ for all } i \in S', \quad \pi_i(x_0) = \text{pr}_2(\pi_i(\bar{x})) \text{ for all } i \in S'.
\]
Hence we deduce that
\[
|C_{S'}^{\textup{acc}}| \leq |X_{\text{pre}}| \cdot \max_{x_0 \in W} |W(x_0)|,
\]
so that there exists $x_0$ such that the density of $W(x_0)$ in $X_{\text{pre}}$ is at least $\delta'$. Fix such a $x_0$. Now if we were to find subsets $W_i \subseteq X_i(a, Q)$ such that $|W_i| = M_{\text{box}}$ and
\[
\prod_{i \in S'} W_i \subseteq W(x_0),
\]
we could extend the sequence $Z_1, \dots, Z_t$ to a longer sequence. Since this is impossible by our choice of $Z_1, \dots, Z_t$, we apply the contrapositive of Lemma \ref{lFindBox} to deduce that
\[
5 M_{\text{box}}^m > \frac{\log \text{min}_{i \in S'} |X_i(a, Q)|}{\log \delta'^{-1}}.
\]
It follows from equation (\ref{eXjaQS}) and equation (\ref{eRegularSpace}) that
\[
|X_i(a, Q)| \geq e^{e^{\frac{3}{10} k_{\text{gap}}}}
\]
for $N$ sufficiently large. We conclude that
\begin{align}
\label{eUpperdelta}
\delta \leq e^{-\frac{1}{4}e^{k_{\text{gap}}}}
\end{align}
for $N$ sufficiently large. Now we split
\begin{align}
\label{eSplitProj}
\left|\sum_{x \in X(a, Q, \{\textup{Art}_k\}_{2 \leq k < m})} \hspace{-0.5cm} \iota(F(\textup{Art}_{m, x}))\right| \leq \left|\sum_{\substack{x \in X(a, Q, \{\textup{Art}_k\}_{2 \leq k < m}) \\ \pi_{S'}(x) \not \in Z_{\text{pre}, \text{bad}}}} \hspace{-0.5cm} \iota(F(\textup{Art}_{m, x}))\right| + \left|\sum_{\substack{x \in X(a, Q) \\ \pi_{S'}(x) \in Z_{\text{pre}, \text{bad}}}} 1\right|
\end{align}
with the latter sum fitting in the error term by equation (\ref{eUpperdelta}) and two applications of Proposition \ref{pLegendre} (namely to $X(a, Q)$ and to the subset of $X(a, Q)$ projecting to a given element in $Z_{\text{pre, bad}}$). So it remains to deal with the former sum.

Now for $x \in X(a, Q)$ with $\pi_{S'}(x) \not \in Z_{\text{pre}, \text{bad}}$, we write $\Lambda(x)$ for the number of $i \in [t]$ for which $x \in \widetilde{Z_i}$. We will compute the first and second moment of $\Lambda(x)$, and then use this to bring Theorem \ref{tBoxify} into play. We have
\begin{align}
\label{eFirstMoment}
\sum_{\substack{x \in X(a, Q) \\ \pi_{S'}(x) \not \in Z_{\text{pre}, \text{bad}}}} \Lambda(x) 
&= \sum_{\substack{z \in Z_{\text{pre}} \\ z \not \in Z_{\text{pre}, \text{bad}}}} \sum_{\substack{x \in X(a, Q) \\ \pi_{S'}(x) = z}} \sum_{i \in [t]} \mathbf{1}_{x \in \widetilde{Z_i}} \nonumber \\
&= \sum_{\substack{z \in Z_{\text{pre}}\\ z \not \in Z_{\text{pre}, \text{bad}}}} \sum_{i \in [t]} |X(a, Q) \cap \widetilde{Z_i} \cap \pi_{S'}^{-1}(z)|.
\end{align}
Fix some $z \in Z_{\text{pre}} \setminus Z_{\text{pre}, \text{bad}}$. Define
\[
L(Q) := \Q(\{\sqrt{-1}\} \cup \{\sqrt{\pi_j(z)} : j \in S'\} \cup \{\sqrt{\pi_j(Q)} : j \in [k_{\text{gap}}] - S'\}).
\]
Write $d(M_{\text{box}}, m)$ for the degree of $L(Q) M_\circ(\widetilde{Z_i})$ over $L(Q)$. Crucially, $d(M_{\text{box}}, m)$ does not depend on $i$ by Lemma \ref{lMdegree} and Lemma \ref{lConstantPointer}. Since our box $X$ is Siegel-less, the Chebotarev density theorem \cite[Proposition 6.5]{Smith} yields for $i > k_{\text{gap}}$
\[
|X_i(a, Q, M_\circ(Z))| = \frac{|X_i(a, Q \times \{z\})|}{d(M_{\text{box}}, m)}\left(1 + O\left(e^{-2k_{\text{gap}}}\right)\right).
\]
Here we use equation (\ref{ekgap}) and a well-known result of Heilbronn \cite{Heilbronn} to deal with potential Siegel zeroes. It follows from two applications of Proposition \ref{pLegendre} to the preboxes
\[
\left(\prod_{i \in [r] - [k_{\text{gap}}]} X_i(a, Q \times \{z\}), \varnothing\right), \quad \left(\prod_{i \in [r] - [k_{\text{gap}}]} X_i(a, Q, M_\circ(Z)), \varnothing\right)
\]
that
\[
|X(a, Q) \cap \widetilde{Z_i} \cap \pi_{S'}^{-1}(z)| = \mathbf{1}_{z \in Z_i} \cdot \frac{|X(a, Q) \cap \pi_{S'}^{-1}(z)|}{d(M_{\text{box}}, m)^{r - k_{\text{gap}}}}\left(1 + O\left(e^{-k_{\text{gap}}}\right)\right).
\]
Continuing the computation in equation (\ref{eFirstMoment}) gives
\begin{align}
\label{eLambdaFM}
\sum_{\substack{x \in X(a, Q) \\ \pi_{S'}(x) \not \in Z_{\text{pre}, \text{bad}}}} \Lambda(x) = \frac{R |X(a, Q)|}{d(M_{\text{box}}, m)^{r - k_{\text{gap}}}}\left(1 + O\left(e^{-k_{\text{gap}}}\right)\right),
\end{align}
since every $z \in Z_{\text{pre}} \setminus Z_{\text{pre}, \text{bad}}$ is in precisely $R$ of the $Z_i$. Having computed the first moment, we now consider the second moment
\[
\sum_{\substack{x \in X(a, Q) \\ \pi_{S'}(x) \not \in Z_{\text{pre}, \text{bad}}}} \Lambda(x)^2.
\]
For distinct $i$ and $j$, we see that the assumption $|Z_i \cap Z_j| \leq 1$ gives that $L(Q) M_\circ(\widetilde{Z_i}) M_\circ(\widetilde{Z_j})$ has degree $d(M_{\text{box}}, m)^2$ over $L(Q)$ by Lemma \ref{lMdegree} and Lemma \ref{lConstantPointer}. Then, following the first moment computation, we obtain
\begin{align}
\label{eLambdaSM}
\sum_{\substack{x \in X(a, Q) \\ \pi_{S'}(x) \not \in Z_{\text{pre}, \text{bad}}}} \Lambda(x)^2 = \frac{R^2 |X(a, Q)|}{d(M_{\text{box}}, m)^{2(r - k_{\text{gap}})}}\left(1 + O\left(e^{-k_{\text{gap}}}\right)\right).
\end{align}
It follows from equation (\ref{eLambdaFM}), equation (\ref{eLambdaSM}) and Chebyshev's inequality that outside a set of density $O(e^{-k_\text{gap}/2})$ in the set of $x \in X(a, Q)$ satisfying $\pi_{S'}(x) \not \in Z_{\text{pre}, \text{bad}}$, we have that
\begin{align}
\label{eLambda}
\left|\Lambda(x) - \frac{R}{d(M_{\text{box}}, m)^{r - k_{\text{gap}}}}\right| \leq \frac{Re^{\frac{-k_{\text{gap}}}{4}}}{d(M_{\text{box}}, m)^{r - k_{\text{gap}}}}.
\end{align}
Recalling equation (\ref{eSplitProj}), it suffices to bound
\[
\left|\sum_{\substack{x \in X(a, Q, \{\textup{Art}_k\}_{2 \leq k < m}) \\ \pi_{S'}(x) \not \in Z_{\text{pre}, \text{bad}}}} \hspace{-0.5cm} \iota(F(\textup{Art}_{m, x}))\right|.
\]
By the triangle inequality this is bounded by
\begin{multline}
\label{eReductionBoxify}
\frac{d(M_{\text{box}}, m)^{r - k_{\text{gap}}}}{R} \cdot \left|\sum_{i \in [t]}  \sum_{x \in X(a, Q, \{\textup{Art}_k\}_{2 \leq k < m}) \cap \widetilde{Z_i}} \iota(F(\textup{Art}_{m, x}))\right| + \\
\left|\sum_{\substack{x \in X(a, Q, \{\textup{Art}_k\}_{2 \leq k < m}) \\ \pi_{S'}(x) \not \in Z_{\text{pre}, \text{bad}}}} \hspace{-0.5cm} \iota(F(\textup{Art}_{m, x})) - \frac{d(M_{\text{box}}, m)^{r - k_{\text{gap}}}}{R} \sum_{i \in [t]} \sum_{x \in X(a, Q, \{\textup{Art}_k\}_{2 \leq k < m}) \cap \widetilde{Z_i}} \hspace{-0.5cm} \iota(F(\textup{Art}_{m, x}))\right|.
\end{multline}
It follows from Theorem \ref{tBoxify} that the former sum is bounded by
\[
A \cdot \frac{d(M_{\text{box}}, m)^{r - k_{\text{gap}}}}{R \cdot (\log \log \log \log N)^{\frac{100c}{m 6^m}}} \cdot \sum_{i \in [t]} |X(a, Q) \cap \widetilde{Z_i}|.
\]
Now observe that
\[
\sum_{i \in [t]} |X(a, Q) \cap \widetilde{Z_i}| \leq \sum_{\substack{x \in X(a, Q) \\ \pi_{S'}(x) \not \in Z_{\text{pre}, \text{bad}}}} \Lambda(x) + \sum_{\substack{x \in X(a, Q) \\ \pi_{S'}(x) \in Z_{\text{pre}, \text{bad}}}} R.
\]
To bound the first term above, we use equation (\ref{eLambdaFM}). For the second term, we use equation (\ref{eUpperdelta}) and the argument immediately following equation (\ref{eSplitProj}).

Finally, to deal with the latter sum in equation (\ref{eReductionBoxify}), we first take care of the points projecting to an element in $Z_{\text{pre}, \text{bad}}$. Then it remains to bound
\[
\left|\sum_{\substack{x \in X(a, Q, \{\textup{Art}_k\}_{2 \leq k < m}) \\ \pi_{S'}(x) \not \in Z_{\text{pre}, \text{bad}}}} \hspace{-0.5cm} \iota(F(\textup{Art}_{m, x})) - \frac{d(M_{\text{box}}, m)^{r - k_{\text{gap}}}}{R} \sum_{i \in [t]} \sum_{\substack{x \in X(a, Q, \{\textup{Art}_k\}_{2 \leq k < m}) \cap \widetilde{Z_i} \\ \pi_{S'}(x) \not \in Z_{\text{pre}, \text{bad}}}} \hspace{-0.5cm} \iota(F(\textup{Art}_{m, x}))\right|,
\]
which is at most
\[
\sum_{\substack{x \in X(a, Q, \{\textup{Art}_k\}_{2 \leq k < m}) \\ \pi_{S'}(x) \not \in Z_{\text{pre}, \text{bad}}}} \left|1 - \frac{d(M_{\text{box}}, m)^{r - k_{\text{gap}}}}{R} \cdot \Lambda(x)\right|.
\]
We split the above sum depending on whether $x$ satisfies equation (\ref{eLambda}). To deal with the $x$ failing equation (\ref{eLambda}), we use the trivial bound $\Lambda(x) \leq R$. This completes the proof.
\end{proof}

\subsection{Finishing the proof}
It remains to prove Theorem \ref{tBoxify}. The main technical input that we need is an additive system as in Proposition \ref{pARinput}, which we will construct in the first part of the proof.

\begin{proof}[Proof of Theorem \ref{tBoxify}.]
For every non-zero $v \in V_{a, 2}$ and for every $x \in X(a, Q)$, we choose a raw cocycle $(\psi_i(x, v))_{0 \leq i \leq t(x, v)}$ such that
\[
\psi_1(x, v) = \sum_{\substack{1 \leq j \leq r \\ \pi_j(v) = 1}} \chi_{\pi_j(x)}
\]
with $t(x, v)$ maximal among the set of such raw cocycles. Recall that we fixed a basis
\[
w_1, \dots, w_{n_2}, R
\]
of $V_{a, 2}$ such that
\[
w_1, \dots, w_{n_i}, R
\]
is a basis of $B_i$ for all $2 \leq i \leq m$, where we recall that $n_i := -1 + \dim_{\FF_2} B_i$ (we also set $n_1 := n_2$). We also remind the reader that we fixed a pair of integers $(j_1, j_2)$ associated to the linear map $F: \textup{Mat}(n_m + 1, n_m, \FF_2) \rightarrow \FF_2$. 

\subsubsection*{Desired properties of an additive system}
We start the proof by constructing an additive system $\mathfrak{A} = (C_T, C_T^{\textup{acc}}, F_T, A_T)_{T \subseteq S}$ on $(\widetilde{Z}, S)$ such that
\begin{enumerate}
\item[(a)] $C_\varnothing^{\text{acc}} = X(a, Q, \{\text{Art}_k\}_{2 \leq k < m}) \cap \widetilde{Z}$;
\item[(b)] $\mathfrak{A}$ is $(2^{(n_{\text{max}} + 10)(n_{\text{max}} + 2m + 10)}, S)$-acceptable;
\item[(c)] suppose that we are in case $I$. Then we have for all $\bar{x} \in C(\mathfrak{A})$
\begin{multline*}
(\Sigma F')(\bar{x}) = \phi_{\pi_{S' - \{i_{\text{char}}\}}(\bar{x}); \textup{pr}_1(\pi_{i_{\textup{char}}}(\bar{x})) \textup{pr}_2(\pi_{i_{\textup{char}}}(\bar{x}))}(\text{Frob}(\text{pr}_1(\pi_{i_{\text{Cheb}}}(\bar{x})))) + \\
\phi_{\pi_{S' - \{i_{\text{char}}\}}(\bar{x}); \textup{pr}_1(\pi_{i_{\textup{char}}}(\bar{x})) \textup{pr}_2(\pi_{i_{\textup{char}}}(\bar{x}))}(\text{Frob}(\text{pr}_2(\pi_{i_{\text{Cheb}}}(\bar{x})))),
\end{multline*}
where $F': X(a, Q, \{\text{Art}_k\}_{2 \leq k < m}) \cap \widetilde{Z} \rightarrow \FF_2$ is the function that sends $x$ to $F(\text{Art}_{m, x})$. If we are in case $II$, $III$ or $IV$, then we have
\[
(\Sigma F')(\bar{x}) = \phi_{\pi_{S'}(\bar{x}); - 1}(\text{Frob}(\text{pr}_1(\pi_{i_{\text{Cheb}}}(\bar{x})))) + \phi_{\pi_{S'}(\bar{x}); - 1}(\text{Frob}(\text{pr}_2(\pi_{i_{\text{Cheb}}}(\bar{x}))));
\]
\item[(d)] suppose that we are in case $I$. Let $|T| < m$ and let $1 \leq j \leq n_{|T| + 1}$. Denote by $\mathcal{W}$ the set $\{j_1, j_2\}$ if $j_1 \leq d_m$ and $\{j_2\}$ otherwise. Then we further demand that $\bar{x} \in C_T^{\text{acc}}$ implies that
\begin{itemize}
\item if $T \subseteq S' - \{i_{\text{char}}\}$ or $j \leq n_m$, we have
\begin{align}
\label{eCaseIc1}
\sum_{x \in \bar{x}(\varnothing)} \psi_{|T|}(x, w_j) = 0
\end{align}
for $j \not \in \mathcal{W}$;
\item if $i_{\text{char}} \in T$, we have
\begin{align}
\label{eCaseIc2}
\sum_{x \in \bar{x}(\varnothing)} \psi_{|T|}(x, w_{j_2}) = \phi_{\pi_{T - \{i_{\text{char}}\}}(\bar{x}); \textup{pr}_1(\pi_{i_{\textup{char}}}(\bar{x})) \textup{pr}_2(\pi_{i_{\textup{char}}}(\bar{x}))};
\end{align}
\item if $i_{\text{char}} \not \in T$, we have
\begin{align}
\label{eCaseIc3}
\sum_{x \in \bar{x}(\varnothing)} \psi_{|T|}(x, w_{j_2}) = 0;
\end{align}
\item if $i_{\text{Cheb}} \not \in T$ and $j_1 \in \mathcal{W}$, we have
\begin{align}
\label{eCaseIc5}
\sum_{x \in \bar{x}(\varnothing)} \psi_{|T|}(x, w_{j_1}) = 0;
\end{align}
\item if $T \neq \varnothing$, we have
\begin{align}
\label{eCaseIc4}
\sum_{x \in \bar{x}(\varnothing)} \psi_{|T| + 1}(x, w_j)(\sigma_{\pi_i(\bar{x})}) = 0
\end{align}
for all $i \in S - T$.
\end{itemize}
If we are in case $II$, $III$ or $IV$, take $\varnothing \subseteq T \subset S$ and $1 \leq j \leq n_{|T| + 1}$. Now suppose that we are in case $II$. Then $\bar{x} \in C_T^{\text{acc}}$ implies that
\begin{itemize}
\item if $T \subseteq S'$ or $j \leq n_m$, we have
\[
\sum_{x \in \bar{x}(\varnothing)} \psi_{|T|}(x, w_j) = 0
\]
for $j \neq j_2$;
\item if $T \neq \varnothing$, we have
\[
\sum_{x \in \bar{x}(\varnothing)} \psi_{|T| + 1}(x, w_j)(\sigma_{\pi_i(\bar{x})}) = 0
\]
for $j \neq j_2$ and all $i \in S - T$;
\item if $i_{\text{Cheb}} \not \in T$, we have
\[
\sum_{x \in \bar{x}(\varnothing)} \psi_{|T|}(x, w_{j_2}) = 0;
\]
\item if $i_{\text{Cheb}} \not \in T$ and $T \neq \varnothing$, we have
\[
\sum_{x \in \bar{x}(\varnothing)} \psi_{|T| + 1}(x, w_{j_2})(\sigma_{\pi_i(\bar{x})}) = 0
\]
for all $i \in S - T$.
\end{itemize}
Next we deal with case $III$. In this case we demand that $\bar{x} \in C_T^{\text{acc}}$ implies that
\begin{itemize}
\item if $T \subseteq S'$ or $j \leq n_m$, we have
\[
\sum_{x \in \bar{x}(\varnothing)} \psi_{|T|}(x, w_j) = 0
\]
for $j \neq j_2$;
\item if $T \neq \varnothing$, we have
\[
\sum_{x \in \bar{x}(\varnothing)} \psi_{|T| + 1}(x, w_j)(\sigma_{\pi_i(\bar{x})}) = 0
\]
for $j \neq j_2$ and all $i \in S - T$;
\item if $i_{\text{Cheb}} \not \in T$, we have
\[
\sum_{x \in \bar{x}(\varnothing)} \psi_{|T|}(x, w_{j_2} + R) = 0;
\]
\item if $i_{\text{Cheb}} \not \in T$ and $T \neq \varnothing$, we have
\[
\sum_{x \in \bar{x}(\varnothing)} \psi_{|T| + 1}(x, w_{j_2} + R)(\sigma_{\pi_i(\bar{x})}) = 0
\]
for all $i \in S - T$.
\end{itemize}
Finally, suppose that we are in case $IV$. Then we will ensure that $\bar{x} \in C_T^{\text{acc}}$ implies that
\begin{itemize}
\item if $T \subseteq S'$ or $j \leq n_m$, we have
\[
\sum_{x \in \bar{x}(\varnothing)} \psi_{|T|}(x, w_j) = 0
\]
for $j \not \in \{j_1, j_2\}$. Furthermore, we have
\[
\sum_{x \in \bar{x}(\varnothing)} \psi_{|T|}(x, w_{j_1} + w_{j_2} + R) = 0;
\]
\item if $T \neq \varnothing$, we have
\[
\sum_{x \in \bar{x}(\varnothing)} \psi_{|T| + 1}(x, w_j)(\sigma_{\pi_i(\bar{x})}) = 0
\]
for $j \not \in \{j_1, j_2\}$ and all $i \in S - T$, and furthermore
\[
\sum_{x \in \bar{x}(\varnothing)} \psi_{|T| + 1}(x, w_{j_1} + w_{j_2} + R)(\sigma_{\pi_i(\bar{x})}) = 0
\]
for all $i \in S - T$;
\item if $i_{\text{Cheb}} \not \in T$, we have
\[
\sum_{x \in \bar{x}(\varnothing)} \psi_{|T|}(x, w_{j_1}) = 0
\]
and
\[
\sum_{x \in \bar{x}(\varnothing)} \psi_{|T|}(x, w_{j_2} + R) = 0;
\]
\item if $i_{\text{Cheb}} \not \in T$ and $T \neq \varnothing$, we have
\[
\sum_{x \in \bar{x}(\varnothing)} \psi_{|T| + 1}(x, w_{j_1})(\sigma_{\pi_i(\bar{x})}) = 0
\]
for all $i \in S - T$ and
\[
\sum_{x \in \bar{x}(\varnothing)} \psi_{|T| + 1}(x, w_{j_2} + R)(\sigma_{\pi_i(\bar{x})}) = 0
\]
for all $i \in S - T$.
\end{itemize}
\end{enumerate}

To achieve this, we will first put
\[
C_\varnothing^{\text{acc}} := X(a, Q, \{\text{Art}_k\}_{2 \leq k < m}) \cap \widetilde{Z},
\]
so that $\mathfrak{A}$ indeed satisfies property $(a)$. Observe that an additive system $\mathfrak{A}$ is completely determined by $C_\varnothing^{\text{acc}}$ and the maps $F_T$. Our goal is now to construct the maps $F_T: C_T \rightarrow A_T$ (with $|A_T| \leq 2^{(n_{\text{max}} + 10)(n_{\text{max}} + 2m + 10)}$) such that property $(d)$ holds. We will then show that $\mathfrak{A}$ also satisfies properties $(b)$ and $(c)$.

\subsubsection*{Construction of the additive system}
In order to construct $F_T$, we will suppose that we are in case $I$ with the other cases being similar and we proceed by induction on $|T|$. If $T = \varnothing$ or if $|T| \geq m$, we let $F_T$ be the zero map. Let $0 < |T| < m$ and let $1 \leq j \leq n_{|T| + 1}$. We assume that $T \subseteq S' - \{i_{\text{char}}\}$ if $j > n_m$. Now take $\bar{x} \in C_T$ and define
\[
\psi(\bar{x}, j) := 
\left\{
	\begin{array}{ll}
		\sum\limits_{x \in \bar{x}(\varnothing)} \psi_{|T|}(x, w_j) & \hspace{-0.14cm} \mbox{if } j \neq j_2 \text{ or } i_{\text{char}} \not \in T \\
		\phi_{\pi_{T - \{i_{\text{char}}\}}(\bar{x}); \textup{pr}_1(\pi_{i_{\textup{char}}}(\bar{x})) \textup{pr}_2(\pi_{i_{\textup{char}}}(\bar{x}))} + \sum\limits_{x \in \bar{x}(\varnothing)} \psi_{|T|}(x, w_j) & \hspace{-0.14cm} \mbox{if } j = j_2 \text{ and } i_{\text{char}} \in T
	\end{array}
\right.
\]
for $j \neq j_1$. Further define
\[
\psi(\bar{x}, j_1) := 
\left\{
	\begin{array}{ll}
		\sum\limits_{x \in \bar{x}(\varnothing)} \psi_{|T|}(x, w_{j_1}) & \mbox{if } i_{\text{Cheb}} \not \in T \text{ or } j_1 \not \in \mathcal{W} \\
		0 & \mbox{if } i_{\text{Cheb}} \in T \text{ and } j_1 \in \mathcal{W}.
	\end{array}
\right.
\]
A priori $\psi(\bar{x}, j)$ is a $1$-cochain from $G_\Q$ with values in $N$. But since $\bar{x} \in C_T$, it follows from Proposition \ref{key calculation of cocycles} that
\[
d\psi(\bar{x}, j) = 0,
\]
so $\psi(\bar{x}, j)$ is a quadratic character. Furthermore, using once more that $\bar{x} \in C_T$, we find that
\begin{align}
\label{eVanishSigma}
\psi(\bar{x}, j)(\sigma_p) = 0
\end{align}
for all $p \in \{\text{pr}_1(\pi_i(\bar{x})), \text{pr}_2(\pi_i(\bar{x}))\}$ and all $i \in T$ provided that $|T| \geq 2$. In case $|T| = 1$, we may directly verify equation (\ref{eVanishSigma}). It follows that $\psi(\bar{x}, j)$ is an unramified quadratic character of $\Q(\sqrt{x})$ for all $x \in \bar{x}(\varnothing)$. Take a prime $p = \pi_j(\bar{x})$ with $j \in [r] - T$. We claim that the place $\text{Up}_{\Q(\sqrt{x})/\Q}(p)$ splits completely in the extension $L(\psi(\bar{x}, j))\Q(\sqrt{x})/\Q(\sqrt{x})$. Since $j \leq n_{|T| + 1}$, $\text{Up}_{\Q(\sqrt{x})/\Q}(p)$ certainly splits completely in $L(\psi_{|T|}(x, w_j)) \Q(\sqrt{x})/\Q(\sqrt{x})$ for every $x \in \bar{x}(\varnothing)$, and $p$ splits completely in
\[
L(\phi_{\pi_{T - \{i_{\text{char}}\}}(\bar{x}); \textup{pr}_1(\pi_{i_{\textup{char}}}(\bar{x})) \textup{pr}_2(\pi_{i_{\textup{char}}}(\bar{x}))})/\Q
\]
by assumption. This clearly implies the claim.

If $N$ is sufficiently large, it follows from equation (\ref{eGen2}) that there exists an injective map $f: [n_2] \rightarrow [r] - S - \{1\}$ such that
\[
\pi_{f(i)}(w_j) = 1 \Longleftrightarrow i = j
\]
for all $i, j \in [n_2]$. Similarly, there exists $i_0 \in [r] - S - \{1\}$ such that
\[
\pi_{i_0}(w_j) = 0 \text{ for all } j \in [n_2].
\]
Choose a point $x_0 \in \bar{x}(\varnothing)$. Define $F_{T, 1}$ to be the map that sends $\bar{x}$ to the tuple
\[
\left(\psi(\bar{x}, j)(\text{Frob}(\text{Up}_{\Q(\sqrt{x_0})/\Q}(\pi_i(x_0))))\right)_{1 \leq j \leq n_{|T| + 1}, i \in T},
\]
define $F_{T, 2}$ to be the map that sends $\bar{x}$ to the tuple
\[
\left(\left(\psi(\bar{x}, j)(\sigma_{\pi_{i_0}(x_0)})\right)_{1 \leq j \leq n_{|T| + 1}}, \left(\psi(\bar{x}, j)(\sigma_{\pi_{f(i)}(x_0)})\right)_{1 \leq j \leq n_{|T| + 1}, i \in [n_2]}\right)
\]
and finally define $F_{T, 3}$ to be the map that sends $\bar{x}$ to the tuple
\[
\left(\sum_{x \in \bar{x}(\varnothing)} \psi_{|T| + 1}(x, w_j)(\sigma_{\pi_i(x_0)})\right)_{1 \leq j \leq n_{|T| + 1}, i \in S - T}.
\]
Since $\pi_{f(i)}(x_0)$, $\pi_{i_0}(x_0)$ and $\pi_i(x_0)$ (for $i \in S - T$) do not depend on the choice of $x_0 \in \bar{x}(\varnothing)$, one readily verifies that $F_{T, 2}$ and $F_{T, 3}$ satisfy equation (\ref{eAdditive}). We will now argue that $F_{T, 1}$ also satisfies equation (\ref{eAdditive}). To this end, observe that $\bar{x}(\varnothing) \subseteq X(a)$ by our choice of $C_\varnothing^{\text{acc}}$. From this we deduce that $F_{T, 1}$ also does not depend on the choice of $x_0$. Then $F_{T, 1}$ also satisfies equation (\ref{eAdditive}). Finally, we define $F_T$ to be the map 
\[
(F_{T, 1}, F_{T, 2}, F_{T, 3})
\]
in case $T \subseteq S' - \{i_{\text{char}}\}$. In case $T$ is not a subset of $S' - \{i_{\text{char}}\}$, we define $F_T$ in the same way, except that $j$ runs up to $n_m$ instead of $n_{|T| + 1}$ in the definitions of $F_{T, 1}$ and $F_{T, 2}$. Having constructed the map $F_T$, we will now verify that $C_T^{\text{acc}}$ satisfies property $(d)$.

\subsubsection*{Verification of property $(d)$}
Let $\bar{x} \in C_T^{\text{acc}}$, so that $\bar{x} \in C_T$ and $F_T(\bar{x}) = 0$. In particular $F_{T, 3}(\bar{x}) = 0$, so equation (\ref{eCaseIc4}) holds. Therefore it remains to establish that $\bar{x}$ satisfies equations (\ref{eCaseIc1}), (\ref{eCaseIc2}), (\ref{eCaseIc3}) and (\ref{eCaseIc5}), which is equivalent to showing that $\psi(\bar{x}, j) = 0$ for all $1 \leq j \leq n_{|T| + 1}$. In case $|T| = 1$, this follows from our choice of variable indices. So henceforth we will assume that $|T| > 1$.

In order to show that $\psi(\bar{x}, j) = 0$, we first observe that $L(\psi(\bar{x}, j)) \Q(\sqrt{x})/\Q(\sqrt{x})$ is an unramified extension for every $x \in \bar{x}(\varnothing)$. Indeed, this is a consequence of the vanishing of $F_{T', 3}(\bar{y})$ for all $\varnothing \subset T' \subset T$ and all $\bar{y} \in \bar{x}(T')$: it is at this point that we make essential use of the assumption $|T| > 1$.

Next we claim that $\psi(\bar{x}, j) \in V_{a, 2}$. Let $x \in \bar{x}(\varnothing)$ and let $p = \pi_i(x)$ for some $i \in [r]$. We have to show that $\text{Up}_{\Q(\sqrt{x})/\Q}(p)$ splits completely in $L(\psi(\bar{x}, j)) \Q(\sqrt{x})/\Q(\sqrt{x})$. In case $i \in [r] -  T$, we have already established this, while for $i \in T$ this follows from the vanishing of $F_{T, 1}(\bar{x})$. This establishes the claim.

We are now ready to prove that $\psi(\bar{x}, j) = 0$. Indeed, since $\psi(\bar{x}, j) \in V_{a, 2}$, we can write $\psi(\bar{x}, j)$ as a linear combination of $R$ and the elements $w_k$. But $F_{T, 2}(\bar{x}) = 0$ implies that $\psi(\bar{x}, j)$ vanishes on $\sigma_{\pi_{i_0}(x)}$ for all $x \in \bar{x}(\varnothing)$, so $\psi(\bar{x}, j)$ must in fact be a linear combination of the elements $w_k$. Using once more that $F_{T, 2}(\bar{x}) = 0$ and the definition of the injection $f: [n_2] \rightarrow [r] - S - \{1\}$, we conclude that $\psi(\bar{x}, j) = 0$.

So far we have constructed an additive system $\mathfrak{A}$ satisfying properties $(a)$ and $(d)$ such that $|A_T| \leq 2^{(n_{\text{max}} + 10)(n_{\text{max}} + 2m + 10)}$. We will now show that this implies properties $(b)$ and $(c)$, and we will do so case by case.

\subsubsection*{Verification of property $(b)$ and $(c)$ in case $I$}
We start with case $I$. Take $\bar{x} \in C(\mathfrak{A})$. If $\bar{x} \in C(\mathfrak{A})$ is degenerate, then property $(b)$ and $(c)$ hold. So suppose that $\bar{x}$ is not degenerate. By definition of $C(\mathfrak{A})$, it follows that
\[
\bar{x}(S - \{i\}) \cap C_{S - \{i\}}^{\text{acc}} \neq \varnothing.
\]
Fix for every $i \in S$ an element $\bar{z}_i$ in the intersection $\bar{x}(S - \{i\}) \cap C_{S - \{i\}}^{\text{acc}}$. Then there exists a unique $x_0 \in \bar{x}(\varnothing)$ such that $x_0 \not \in \bar{z}_i(\varnothing)$ for all $i$. For all $x_1 \in \bar{x}(\varnothing)$ with $x_1 \neq x_0$, we see that there exists $i \in S$ such that $x_1 \in \bar{z}_i(\varnothing)$. In particular, we deduce that $x_1 \in C_\varnothing^{\text{acc}}$. Therefore it is enough to show that $x_0 \in C_\varnothing^{\text{acc}}$ to conclude that $\mathfrak{A}$ is $(2^{(n_{\text{max}} + 10)(n_{\text{max}} + 2m + 10)}, S)$-acceptable, i.e. property $(b)$ holds.

To start, we see that $x_0 \in X(a)$. We will show that
\[
\text{Art}_{i, x_0} = \text{Art}_i
\]
for all $2 \leq i < m$ by induction on $i$. Since $i < m$ and $|S| = m + 1$, we can pick a subset $T$ of $S$ such that $|T| = i$ and $T$ does not contain $i_{\text{Cheb}}$ or $i_{\text{char}}$. We apply Theorem \ref{main thm: minimal triples} to any element $\bar{y}$ in $\bar{x}(T)$ containing $x_0$. Since $\bar{z}_j \in C_{S - \{j\}}^{\text{acc}}$ for all $j \in T$, it follows that
\[
(\bar{y}(\varnothing), (\psi_{|T|}(x, w_k))_{x \in \bar{y}(\varnothing) - \{x_0\}}, \chi(w_k))
\]
is a minimal triple for all $1 \leq k \leq n_{|T|}$, where $\chi(w_k) = \psi_1(x, w_k)$ for any choice of $x \in \bar{y}(\varnothing)$. We emphasize that $\psi_1(x, w_k)$ does not depend on the choice of $x \in \bar{y}(\varnothing)$ by our choice of variable indices and our choice of $T$. Now take an element $b \in A_{|T|}$. If there does not exist $l \in T$ with $\pi_l(b) = 1$, then we deduce from Theorem \ref{main thm: minimal triples} that
\[
\text{Art}_{i, x_0}(b, w_k) = \text{Art}_i(b, w_k).
\]
Now suppose that there exists $l \in T$ with $\pi_l(b) = 1$. By our choice of variable indices, we see that there does not exist $l \in T$ with $\pi_l(b + R) = 1$. Therefore it follows from the previous case that
\[
\text{Art}_{i, x_0}(b + R, w_k) = \text{Art}_i(b + R, w_k).
\]
But Theorem \ref{main thm: minimal triples} yields
\[
\text{Art}_{i, x_0}(R, w_k) = \text{Art}_i(R, w_k),
\]
so $\mathfrak{A}$ is $(2^{(n_{\text{max}} + 10)(n_{\text{max}} + 2m + 10)}, S)$-acceptable in case $I$. 

We will now show that property $(c)$ holds in case $I$. Let $\bar{x} \in C(\mathfrak{A})$ and let $\bar{y} \in \bar{x}(S - \{i_{\text{Cheb}}\})$. It follows from Theorem \ref{main thm: minimal triples} and Theorem \ref{main thm: governing triples} that
\[
\sum_{y \in \bar{y}(\varnothing)} \text{Art}_{m, y}(v_l, w_k) = 
\left\{
	\begin{array}{ll}
		0  & \hspace{-0.17cm} \mbox{if } k \neq j_2 \\
		\sum\limits_{\substack{1 \leq i \leq r \\ \pi_i(v_l) = 1}} \phi_{\pi_{S' - \{i_{\text{char}}\}}(\bar{x}); \textup{pr}_1(\pi_{i_{\textup{char}}}(\bar{x})) \textup{pr}_2(\pi_{i_{\textup{char}}}(\bar{x}))}(\text{Frob}(\pi_i(\bar{y}))) & \hspace{-0.17cm} \mbox{if } k = j_2
	\end{array}
\right.
\]
for all $1 \leq k, l \leq n_m$, where we impose the additional condition $l \neq j_2$ if $j_2 \leq d_m$. Now suppose that $l = j_2$ with $j_2 \leq d_m$. Then Theorem \ref{main thm: minimal triples}, applied to the cube $\bar{y} \in \bar{x}(S - \{i_{\text{char}}\})$, gives
\[
\sum_{y \in \bar{y}(\varnothing)} \text{Art}_{m, y}(v_{j_2}, w_k) = 0
\]
for all $1 \leq k \leq n_m$, where we impose that $k \neq j_1$ if $j_1 \leq d_m$. Using the above equations for both choices of $\bar{y}$, we obtain the identity
\begin{align}
\label{eF1}
\sum_{x \in \bar{x}(\varnothing)} \text{Art}_{m, x}(v_l, w_k) = 0
\end{align}
for all pairs $(l, k)$ satisfying $(l, k) \not \in \{(j_1, j_2), (j_2, j_1)\}$ if $j_1, j_2 \leq d_m$ and all pairs $(l, k) \neq (j_1, j_2)$ otherwise. We also get the identity
\begin{multline}
\label{eF2}
\sum_{x \in \bar{x}(\varnothing)} \text{Art}_{m, x}(v_{j_1}, w_{j_2}) = \phi_{\pi_{S' - \{i_{\text{char}}\}}(\bar{x}); \textup{pr}_1(\pi_{i_{\textup{char}}}(\bar{x})) \textup{pr}_2(\pi_{i_{\textup{char}}}(\bar{x}))}(\text{Frob}(\text{pr}_1(\pi_{i_{\text{Cheb}}}(\bar{x})))) + \\
\phi_{\pi_{S' - \{i_{\text{char}}\}}(\bar{x}); \textup{pr}_1(\pi_{i_{\textup{char}}}(\bar{x})) \textup{pr}_2(\pi_{i_{\textup{char}}}(\bar{x}))}(\text{Frob}(\text{pr}_2(\pi_{i_{\text{Cheb}}}(\bar{x})))).
\end{multline}
We now apply part $(ii)$ of Theorem \ref{main thm: minimal triples} to a cube $\bar{y} \in \bar{x}(S - \{i_{\text{Cheb}}\})$. This yields
\begin{align}
\label{eySlice}
\sum_{y \in \bar{y}(\varnothing)} \text{Art}_{m, y}(R, w_k) = 0
\end{align}
if $k \neq j_2$. Similarly, part $(ii)$ of Theorem \ref{main thm: minimal triples} gives for $\bar{y} \in \bar{x}(S - \{i_{\text{char}}\})$
\begin{align}
\label{eySlice2}
\sum_{y \in \bar{y}(\varnothing)} \text{Art}_{m, y}(R, w_{j_2}) = 0.
\end{align}
Adding up equation (\ref{eySlice}) and equation (\ref{eySlice2}) for the two choices of $\bar{y}$, we conclude that
\begin{align}
\label{eF3}
\sum_{x \in \bar{x}(\varnothing)} \text{Art}_{m, x}(R, w_k) = 0
\end{align}
for all $1 \leq k \leq n_m$. Equation (\ref{eF1}, (\ref{eF2}) and (\ref{eF3}) and the shape of $F$ imply that
\begin{multline*}
(\Sigma F')(\bar{x}) = \phi_{\pi_{S' - \{i_{\text{char}}\}}(\bar{x}); \textup{pr}_1(\pi_{i_{\textup{char}}}(\bar{x})) \textup{pr}_2(\pi_{i_{\textup{char}}}(\bar{x}))}(\text{Frob}(\text{pr}_1(\pi_{i_{\text{Cheb}}}(\bar{x})))) + \\
\phi_{\pi_{S' - \{i_{\text{char}}\}}(\bar{x}); \textup{pr}_1(\pi_{i_{\textup{char}}}(\bar{x})) \textup{pr}_2(\pi_{i_{\textup{char}}}(\bar{x}))}(\text{Frob}(\text{pr}_2(\pi_{i_{\text{Cheb}}}(\bar{x}))))
\end{multline*}
as desired.

\subsubsection*{Verification of property $(b)$ and $(c)$ in case $II$}
We will now deal with case $II$. We will start with property $(b)$ and adapt the notation from case $I$. So let $\bar{x} \in C(\mathfrak{A})$ and let $x_0$ be the element of $\bar{x}(\varnothing)$ not in any of the $\bar{z}_i(\varnothing)$ (for the definition of $\bar{z}_i$, see case $I$). We have to show that
\[
\text{Art}_{i, x_0} = \text{Art}_i
\]
for all $2 \leq i < m$. Since $|S| = m$ in this case, we can find a subset $T$ of $S$ with $|T| = i$ and $i_{\text{Cheb}} \not \in T$. Take $\bar{y} \in \bar{x}(T)$. Then
\[
(\bar{y}(\varnothing), (\psi_{|T|}(x, w_k))_{x \in \bar{y}(\varnothing) - \{x_0\}}, \chi(w_k))
\]
is a minimal triple for all $1 \leq k \leq n_{|T|}$. Proceeding as in case $I$ gives
\[
\text{Art}_{i, x_0} = \text{Art}_i.
\]
Our next task is to verify property $(c)$ in case $II$. Take $\bar{x} \in C(\mathfrak{A})$. By construction of the variable indices, we know that
\[
F =  \sum_{\substack{1 \leq j_3 \leq n_m + 1 \\ 1 \leq j_4 \leq n_m}} c_{j_3, j_4} F_{j_3, j_4}, \quad c_{j_3, j_4} \in \FF_2,
\]
where $c_{j_3, j_4}$ satisfies
\begin{itemize}
\item $c_{j_3, j_4} = c_{j_4, j_3}$ for all $1 \leq j_3, j_4 \leq n_m$;
\item $c_{j_3, j_4} = 0$ if $n_m \geq j_3 > d_m$ or $n_m \geq j_4 > d_m$;
\item $c_{n_m +1, j_2} = 1$.
\end{itemize}
We deduce from Theorem \ref{main thm: minimal triples} that
\[
\sum_{x \in \bar{x}(\varnothing)} \text{Art}_{m, x}(v_l, w_k) = 
\left\{
	\begin{array}{ll}
		0  & \mbox{for all } l \neq j_2 \text{ and all } k \neq j_2 \mbox{ if } j_2 \leq d_m \\
		0 & \mbox{for all } k \neq j_2 \mbox { if } j_2 > d_m.
	\end{array}
\right.
\]
Let $\bar{y} \in \bar{x}(S - \{i_{\text{Cheb}}\})$. Theorem \ref{thm: reflection principle for infinity} implies that
\[
\sum_{y \in \bar{y}(\varnothing)} \text{Art}_{m, y}(R, w_k) = \sum_{\substack{1 \leq i \leq r \\ \pi_i(w_k) = 1}} \phi_{\pi_{S'}(\bar{x}); -1}(\text{Frob}(\pi_i(\bar{y}))).
\]
Therefore we get
\begin{align}
\label{eFII1}
\sum_{x \in \bar{x}(\varnothing)} \text{Art}_{m, x}(R, w_k) = 0
\end{align}
if $k \neq j_2$ and
\begin{align}
\label{eFII11}
\sum_{x \in \bar{x}(\varnothing)} \text{Art}_{m, x}(R, w_k) = \phi_{\pi_{S'}(\bar{x}); - 1}(\text{Frob}(\text{pr}_1(\pi_{i_{\text{Cheb}}}(\bar{x})))) + \phi_{\pi_{S'}(\bar{x}); - 1}(\text{Frob}(\text{pr}_2(\pi_{i_{\text{Cheb}}}(\bar{x}))))
\end{align}
if $k = j_2$. Now suppose that $j_2 \leq d_m$. Then, since $m \geq 3$, we apply Theorem \ref{main thm 1 on self-pairing} twice to obtain
\begin{align}
\label{eFII2}
\sum_{x \in \bar{x}(\varnothing)} \text{Art}_{m, x}(w_{j_2}, w_{j_2}) = 0.
\end{align}
Since we are working with pairings valued in $\FF_2$, we have the identity
\[
\text{Art}_{m, x}(v_l, w_{j_2}) + \text{Art}_{m, x}(w_{j_2}, v_l) = \text{Art}_{m, x}(v_l, v_l) + \text{Art}_{m, x}(w_{j_2}, w_{j_2}) + \text{Art}_{m, x}(v_l + w_{j_2}, v_l + w_{j_2})
\]
for all $l \leq d_m$. Six applications of Theorem \ref{main thm 1 on self-pairing} show that
\begin{align}
\label{eFII3}
\sum_{x \in \bar{x}(\varnothing)} \text{Art}_{m, x}(v_l, w_{j_2}) + \text{Art}_{m, x}(w_{j_2}, v_l) = 0
\end{align}
for all $l \leq d_m$. It follows from equations (\ref{eFII1}), (\ref{eFII11}), (\ref{eFII2}), (\ref{eFII3}) and the properties of $c_{j_3, j_4}$ that
\[
(\Sigma F')(\bar{x}) = \phi_{\pi_{S'}(\bar{x}); -1}(\text{Frob}(\text{pr}_1(\pi_{i_{\text{Cheb}}}(\bar{x})))) + \phi_{\pi_{S'}(\bar{x}); -1}(\text{Frob}(\text{pr}_2(\pi_{i_{\text{Cheb}}}(\bar{x}))))
\]
as desired.

\subsubsection*{Verification of property $(b)$ and $(c)$ in case $III$}
We will now treat case $III$. Let us verify property $(b)$ first. In case $III$ we can expand $F$ as
\[
F =  \sum_{\substack{1 \leq j_3 \leq n_m + 1 \\ 1 \leq j_4 \leq n_m}} c_{j_3, j_4} F_{j_3, j_4}, \quad c_{j_3, j_4} \in \FF_2,
\]
where $c_{j_3, j_4}$ satisfies
\begin{itemize}
\item $c_{j_3, j_4} = c_{j_4, j_3}$ for all $1 \leq j_3, j_4 \leq n_m$;
\item $c_{j_3, j_4} = 0$ if $j_3 > d_m$ or $j_4 > d_m$;
\item $c_{n_m +1, k} = 0$ for all $1 \leq k \leq n_m$;
\item $c_{j_2, j_2} = 1$.
\end{itemize}
Take $\bar{x} \in C(\mathfrak{A})$ and define $x_0$ to be the unique element of $\bar{x}(\varnothing)$ outside of all the $\bar{z}_i(\varnothing)$. Again we have to show that
\[
\text{Art}_{i, x_0} = \text{Art}_i
\]
for all $2 \leq i < m$. Because $|S| = m$, there exists a subset $T$ of $S$ with $|T| = i$ and $i_{\text{Cheb}} \not \in T$. Fix a choice of $\bar{y} \in \bar{x}(T)$ containing $x_0$. We now apply Theorem \ref{main thm: minimal triples} to the minimal triple
\[
(\bar{y}(\varnothing), (\psi_{|T|}(x, w_k))_{x \in \bar{y}(\varnothing) - \{x_0\}}, \chi(w_k))
\]
for $k \neq j_2$ and the minimal triple
\[
(\bar{y}(\varnothing), (\psi_{|T|}(x, w_{j_2} + R))_{x \in \bar{y}(\varnothing) - \{x_0\}}, \chi(w_{j_2} + R)),
\]
where $\chi(w_{j_2} + R) = \psi_1(y, w_{j_2} + R)$ for any choice of $y \in \bar{y}(\varnothing)$. Note that this does not depend on the choice of $y$ precisely by our choice of variable indices. This gives property $(b)$.

As for property $(c)$, take $\bar{x} \in C(\mathfrak{A})$ and take $\bar{y} \in \bar{x}(S - \{i_{\text{Cheb}}\})$. Theorem \ref{main thm 2 on self-pairing} yields
\[
\sum_{y \in \bar{y}(\varnothing)} \text{Art}_{m, y}(w_{j_2}, w_{j_2}) = \sum_{\substack{1 \leq i \leq r \\ \pi_i(w_{j_2} + R) = 1}} \phi_{\pi_{S'}(\bar{x}); -1}(\text{Frob}(\pi_i(\bar{y}))).
\]
We conclude that
\begin{align}
\label{eFIII1}
\sum_{x \in \bar{x}(\varnothing)} \text{Art}_{m, x}(w_{j_2}, w_{j_2}) = \phi_{\pi_{S'}(\bar{x}); -1}(\text{Frob}(\text{pr}_1(\pi_{i_{\text{Cheb}}}(\bar{x})))) + \phi_{\pi_{S'}(\bar{x}); -1}(\text{Frob}(\text{pr}_2(\pi_{i_{\text{Cheb}}}(\bar{x})))).
\end{align}
Take some $l \leq d_m$ satisfying $l \neq j_2$. Using Theorem \ref{main thm 1 on self-pairing} twice, we obtain
\begin{align}
\label{eFIIIs1}
\sum_{x \in \bar{x}(\varnothing)} \text{Art}_{m, x}(v_l, v_l) = 0.
\end{align}
Following the argument that led to equation (\ref{eFIII1}), we show in a completely analogous fashion that
\begin{multline}
\label{eFIIIs2}
\sum_{x \in \bar{x}(\varnothing)} \text{Art}_{m, x}(v_l + w_{j_2}, v_l + w_{j_2}) = \\
\phi_{\pi_{S'}(\bar{x}); -1}(\text{Frob}(\text{pr}_1(\pi_{i_{\text{Cheb}}}(\bar{x})))) + \phi_{\pi_{S'}(\bar{x}); -1}(\text{Frob}(\text{pr}_2(\pi_{i_{\text{Cheb}}}(\bar{x}))))
\end{multline}
for all $l \leq d_m$ with $l \neq j_2$. Recalling the identity
\[
\text{Art}_{m, x}(v_l, w_{j_2}) + \text{Art}_{m, x}(w_{j_2}, v_l) = \text{Art}_{m, x}(v_l, v_l) + \text{Art}_{m, x}(w_{j_2}, w_{j_2}) + \text{Art}_{m, x}(v_l + w_{j_2}, v_l + w_{j_2}),
\]
we get from equations (\ref{eFIII1}), (\ref{eFIIIs1}) and (\ref{eFIIIs2})
\begin{align}
\label{eFIII2}
\sum_{x \in \bar{x}(\varnothing)} \text{Art}_{m, x}(v_l, w_{j_2}) + \text{Art}_{m, x}(w_{j_2}, v_l) = 0.
\end{align}
If we take $k$ and $l$ both not equal to $j_2$, then we have the equality
\begin{align}
\label{eFIII3}
\sum_{x \in \bar{x}(\varnothing)} \text{Art}_{m, x}(v_l, w_k) + \text{Art}_{m, x}(w_k, v_l) = 0
\end{align}
as a consequence of Theorem \ref{main thm: minimal triples}. From the shape of $F$ and equations (\ref{eFIII1}), (\ref{eFIII2}) and (\ref{eFIII3}), we conclude that 
\[
(\Sigma F')(\bar{x}) = \phi_{\pi_{S'}(\bar{x}); -1}(\text{Frob}(\text{pr}_1(\pi_{i_{\text{Cheb}}}(\bar{x})))) + \phi_{\pi_{S'}(\bar{x}); -1}(\text{Frob}(\text{pr}_2(\pi_{i_{\text{Cheb}}}(\bar{x})))),
\]
which finishes the proof of property $(c)$ in case $III$.

\subsubsection*{Verification of property $(b)$ and $(c)$ in case $IV$}
Finally, we deal with case $IV$. In case $IV$ we can expand $F$ as
\[
F =  \sum_{\substack{1 \leq j_3 \leq n_m + 1 \\ 1 \leq j_4 \leq n_m}} c_{j_3, j_4} F_{j_3, j_4}, \quad c_{j_3, j_4} \in \FF_2,
\]
where $c_{j_3, j_4}$ satisfies
\begin{itemize}
\item $c_{j_3, j_4} = c_{j_4, j_3}$ for all $1 \leq j_3, j_4 \leq n_m$;
\item $c_{j_3, j_4} = 0$ if $j_3 > d_m$ or $j_4 > d_m$;
\item $c_{n_m +1, k} = 0$ for all $1 \leq k \leq n_m$;
\item $c_{k, k} = 0$ for all $1 \leq k \leq d_m$;
\item $c_{j_1, j_2} = 1$.
\end{itemize}
Suppose that we are given $\bar{x} \in C(\mathfrak{A})$. Then let $x_0$ be the unique element of $\bar{x}(\varnothing)$ outside of all the $\bar{z}_i(\varnothing)$. For property $(b)$, we need to establish that
\[
\text{Art}_{i, x_0} = \text{Art}_i
\]
for all $2 \leq i < m$. We will proceed as in case $III$. Because $|S| = m$, we may and will choose a subset $T$ of $S$ with $|T| = i$ and $i_{\text{Cheb}} \not \in T$. Take $\bar{y} \in \bar{x}(T)$ containing $x_0$. Now we are in the position to apply Theorem \ref{main thm: minimal triples} to the minimal triple
\[
(\bar{y}(\varnothing), (\psi_{|T|}(x, w_k))_{x \in \bar{y}(\varnothing) - \{x_0\}}, \chi(w_k))
\]
for $k \neq j_2$ and the minimal triple
\[
(\bar{y}(\varnothing), (\psi_{|T|}(x, w_{j_2} + R))_{x \in \bar{y}(\varnothing) - \{x_0\}}, \chi(w_{j_2} + R)),
\]
where $\chi(w_{j_2} + R) = \psi_1(y, w_{j_2} + R)$ for any choice of $y \in \bar{y}(\varnothing)$. This yields property $(b)$ in case $IV$.

It remains to deal with property $(c)$. Let $l$ and $k$ be distinct and bounded by $d_m$. If $\{l, k\} = \{j_1, j_2\}$, then Theorem \ref{main thm: minimal triples} yields (applied to the character $\chi(w_{j_2} + v_{j_1} + R)$)
\[
\sum_{x \in \bar{x}(\varnothing)} \text{Art}_{m, x}(w_{j_2} + v_{j_1}, w_{j_2} + v_{j_1}) = 0.
\]
Now suppose that $\{l, k\} \neq \{j_1, j_2\}$. It follows from Theorem \ref{main thm 1 on self-pairing} and Theorem \ref{main thm 2 on self-pairing} that
\[
\sum_{x \in \bar{x}(\varnothing)} \text{Art}_{m, x}(w_k + v_l, w_k + v_l) = 0
\]
if $j_2 \not \in \{k, l\}$ and
\begin{multline*}
\sum_{x \in \bar{x}(\varnothing)} \text{Art}_{m, x}(w_k + v_l, w_k + v_l) = \\
\phi_{\pi_{S'}(\bar{x}); -1}(\text{Frob}(\text{pr}_1(\pi_{i_{\text{Cheb}}}(\bar{x})))) + \phi_{\pi_{S'}(\bar{x}); -1}(\text{Frob}(\text{pr}_2(\pi_{i_{\text{Cheb}}}(\bar{x}))))
\end{multline*}
otherwise. Furthermore, we deduce from Theorem \ref{main thm 1 on self-pairing} and Theorem \ref{main thm 2 on self-pairing} that
\[
\sum_{x \in \bar{x}(\varnothing)} \text{Art}_{m, x}(w_k, w_k) = 0
\]
if $k \neq j_2$ and
\[
\sum_{x \in \bar{x}(\varnothing)} \text{Art}_{m, x}(w_{j_2}, w_{j_2}) = \phi_{\pi_{S'}(\bar{x}); -1}(\text{Frob}(\text{pr}_1(\pi_{i_{\text{Cheb}}}(\bar{x})))) + \phi_{\pi_{S'}(\bar{x}); -1}(\text{Frob}(\text{pr}_2(\pi_{i_{\text{Cheb}}}(\bar{x})))).
\]
Using the identity
\[
\text{Art}_{m, x}(v_l, w_k) + \text{Art}_{m, x}(w_k, v_l) = \text{Art}_{m, x}(v_l, v_l) + \text{Art}_{m, x}(w_k, w_k) + \text{Art}_{m, x}(v_l + w_k, v_l + w_k)
\]
once more, this shows that
\[
(\Sigma F')(\bar{x}) = \phi_{\pi_{S'}(\bar{x}); -1}(\text{Frob}(\text{pr}_1(\pi_{i_{\text{Cheb}}}(\bar{x})))) + \phi_{\pi_{S'}(\bar{x}); -1}(\text{Frob}(\text{pr}_2(\pi_{i_{\text{Cheb}}}(\bar{x}))))
\]
as desired.

\subsubsection*{End of proof}
The final step of the proof is very similar to the proof of Proposition 8.5 in Smith \cite{Smith}. Let $\sigma \in \Gal(M_\circ(Z)M(Z)/M_\circ(Z))$. Define $X_i(a, Q, M_\circ(Z), \sigma)$ to be the subset of $p \in X_i(a, Q, M_\circ(Z))$ such that $\text{Frob}(p) = \sigma$. This is well-defined, since $\Gal(M_\circ(Z)M(Z)/M_\circ(Z))$ is a central subgroup of $\Gal(M_\circ(Z)M(Z)/\Q)$. We apply the Chebotarev density theorem \cite[Proposition 6.5]{Smith} to the extension $L(Q)M_\circ(Z)M(Z)/\Q$ to deduce that
\begin{align}
\label{eChebTop}
|X_i(a, Q, M_\circ(Z), \sigma)| = \frac{|X_i(a, Q, M_\circ(Z))|}{2^{(M_{\text{box}} - 1)^{|S'|}}} \cdot \left(1 + O(e^{-k_{\text{gap}}})\right)
\end{align}
for $i > k_{\text{gap}}$, where we deal with potential Siegel zeroes by appealing to \cite{Heilbronn} and where we used Lemma \ref{lTopdegree} and Lemma \ref{lConstantPointer} to compute the size of $\Gal(M_\circ(Z)M(Z)/M_\circ(Z))$.

Given an element
\[
Q_{\text{post}} \in \prod_{\substack{i \in [r] - [k_{\text{gap}}] \\ i \neq i_{\text{Cheb}}}} X_i(a, Q, M_\circ(Z)),
\]
we define $\widetilde{Z}(Q_{\text{post}})$ to be the subset of $x \in \widetilde{Z}$ such that $\pi_{[r] - [k_{\text{gap}}] - \{i_{\text{Cheb}}\}}(x) = Q_{\text{post}}$. We call $Q_{\text{post}}$ unevenly distributed over the Chebotarev classes if there exists an element $\sigma \in \Gal(M_\circ(Z)M(Z)/M_\circ(Z))$ such that
\begin{multline*}
\left||X_{i_{\text{Cheb}}}(a, Q \times Q_{\text{post}}, M_\circ(Z), \sigma)| - \frac{|X_{i_{\text{Cheb}}}(a, Q \times Q_{\text{post}}, M_\circ(Z))|}{2^{(M_{\text{box}} - 1)^{|S'|}}}\right| \geq \\
\frac{|X_{i_{\text{Cheb}}}(a, Q \times Q_{\text{post}}, M_\circ(Z))|}{e^{0.5k_{\text{gap}}} \cdot 2^{(M_{\text{box}} - 1)^{|S'|}}},
\end{multline*}
and we call $Q_{\text{post}}$ evenly distributed otherwise. The sets $\widetilde{Z}(Q_{\text{post}})$ cover $\widetilde{Z}$, and we will make one more reduction step by reducing to the sets $\widetilde{Z}(Q_{\text{post}})$ with $Q_{\text{post}}$ evenly distributed over the Chebotarev classes. To complete this reduction step, we need to bound
\begin{multline}
\label{eFinalReduction}
\left|\sum_{Q_{\text{post}} \text{ uneven}} \sum_{x \in X(a, Q, \{\textup{Art}_k\}_{2 \leq k < m}) \cap \widetilde{Z}(Q_{\text{post}})} \iota(F(\textup{Art}_{m, x}))\right| \leq \\
|Z| \cdot |X_{i_{\text{Cheb}}}(a, Q, M_\circ(Z))| \cdot \sum_{Q_{\text{post}} \text{ uneven}} 1.
\end{multline}
We say that an element
\[
Q_{j, \text{post}} \in \prod_{\substack{i \in [j] - [k_{\text{gap}}] \\ i \neq i_{\text{Cheb}}}} X_i(a, Q, M_\circ(Z))
\]
is an evenly distributed candidate if the following two conditions are satisfied
\begin{itemize}
\item $j -1 > k_{\text{gap}}$ implies that $\pi_{[j - 1] - [k_{\text{gap}}]}(Q_{j, \text{post}})$ is an evenly distributed candidate;
\item $j \neq i_{\text{Cheb}}$ implies that
\[
\left|\sum_{p \in X_{i_{\text{Cheb}}}(a, Q \times Q_{j - 1, \text{post}}, M_\circ(Z), \sigma)} \left(\frac{p}{\pi_j(Q_{j, \text{post}})}\right)\right| \leq \frac{|X_{i_{\text{Cheb}}}(a, Q \times Q_{j - 1, \text{post}}, M_\circ(Z), \sigma)|}{|X_{k_{\text{gap}} + 1}(a, Q, M_\circ(Z))|^{1/100}}
\]
for all $\sigma \in \Gal(M_\circ(Z)M(Z)/M_\circ(Z))$.
\end{itemize}
If $Q_{\text{post}} = Q_{r, \text{post}}$ is an evenly distributed candidate, then equation (\ref{eChebTop}) and a direct computation show that $Q_{\text{post}}$ is evenly distributed over the Chebotarev classes provided that we take $N_0$ sufficiently large. We will now repeatedly apply the large sieve, as for example presented in \cite[Proposition 6.6]{Smith}. This yields that the number of $Q_{\text{post}}$, that are not evenly distributed candidates, is bounded by
\[
\frac{\prod_{\substack{i \in [r] - [k_{\text{gap}}] \\ i \neq i_{\text{Cheb}}}} |X_i(a, Q, M_\circ(Z))|}{|X_{k_{\text{gap}} + 1}(a, Q, M_\circ(Z))|^{1/100}}
\]
if we take $N_0$ sufficiently large. In particular, we conclude that equation (\ref{eFinalReduction}) is upper bounded by
\[
\frac{|Z| \cdot \prod_{i \in [r] - [k_{\text{gap}}]} |X_i(a, Q, M_\circ(Z))|}{|X_{k_{\text{gap}} + 1}(a, Q, M_\circ(Z))|^{1/100}}.
\]
Note that Proposition \ref{pLegendre} allows us to calculate the size of $X(a, Q) \cap \widetilde{Z}$. Then we see that
\[
\frac{|Z| \cdot \prod_{i \in [r] - [k_{\text{gap}}]} |X_i(a, Q, M_\circ(Z))|}{|X_{k_{\text{gap}} + 1}(a, Q, M_\circ(Z))|^{1/100}} \leq \frac{|X(a, Q) \cap \widetilde{Z}|}{|X_{k_{\text{gap}} + 1}(a, Q, M_\circ(Z))|^{1/1000}},
\]
which fits in the error term. Therefore it suffices to show that there exists $A > 0$ such that
\begin{align}
\label{eQpost}
\left|\sum_{x \in X(a, Q, \{\textup{Art}_k\}_{2 \leq k < m}) \cap \widetilde{Z}(Q_{\text{post}})} \iota(F(\textup{Art}_{m, x}))\right| \leq \frac{A \cdot |X(a, Q) \cap \widetilde{Z}(Q_{\text{post}})|}{(\log \log \log \log N)^{\frac{100c}{m6^m}}}
\end{align}
for every $Q_{\text{post}}$ evenly distributed over the Chebotarev classes.

We are now ready to use our combinatorial results from Section \ref{sCombinatorics}. We remark that the space
\[
X(a, Q) \cap \widetilde{Z}(Q_{\text{post}})
\]
is just the product
\[
Z \times X_{i_{\text{Cheb}}}(a, Q \times Q_{\text{post}}, M_\circ(Z)).
\]
We now formally apply Proposition \ref{pARinput} to the product set $Z \times [M_{\text{box}}]$ and
\[
a = 2^{(n_{\text{max}} + 10)(n_{\text{max}} + 2m + 10)}, \quad \epsilon = \frac{1}{(\log \log \log \log N)^{\frac{100c}{m6^m}}}.
\]
We observe that $\epsilon < a^{-1}$ and $\log M_{\text{box}} \geq A_{\text{Comb}} \cdot 6^{m + 1} \cdot \log \epsilon^{-1}$ by definition of $c$.

Denote by $g_0 \in \text{Add}(Z \times [M_{\text{box}}], S)$ the function guaranteed by Proposition \ref{pARinput}. Here we implicitly index the space $Z \times [M_{\text{box}}]$ by $S$ in the obvious way. To primes $p_1, \dots, p_{M_{\text{box}}}$ we associate the element $g_{p_1, \dots, p_{M_{\text{box}}}} \in \text{Add}(Z \times [M_{\text{box}}], S)$ by mapping a cube $(\bar{z}, (i, j)) \in \text{Cube}(Z \times [M_{\text{box}}], S)$ to
\[
\phi_{\pi_{S' - \{i_{\text{char}}\}}(\bar{z}); \textup{pr}_1(\pi_{i_{\textup{char}}}(\bar{z})) \textup{pr}_2(\pi_{i_{\textup{char}}}(\bar{z}))}(\text{Frob}(p_i)) + 
\phi_{\pi_{S' - \{i_{\text{char}}\}}(\bar{z}); \textup{pr}_1(\pi_{i_{\textup{char}}}(\bar{z})) \textup{pr}_2(\pi_{i_{\textup{char}}}(\bar{z}))}(\text{Frob}(p_j))
\]
in case $I$ and
\[
\phi_{\bar{z}; -1}(\text{Frob}(p_i)) + \phi_{\bar{z}; -1}(\text{Frob}(p_j))
\]
in case $II$, $III$ or $IV$. Our goal is now to partition $X_{i_{\text{Cheb}}}(a, Q \times Q_{\text{post}}, M_\circ(Z))$ as
\begin{align}
\label{ePartitionCheb}
X_{i_{\text{Cheb}}}(a, Q \times Q_{\text{post}}, M_\circ(Z)) = L \cup \bigcup_{i = 1}^t A_i,
\end{align}
where each $A_i = \{p_{i, 1}, \dots, p_{i, M_{\text{box}}}\}$ is an ordered set (so it comes with a map $[M_{\text{box}}] \rightarrow A_i$) of size $M_{\text{box}}$ satisfying
\[
g_{p_{i, 1}, \dots, p_{i, M_{\text{box}}}} = g_0
\]
and where $|L|$ is small.

To this end pick any element $p_1 \in X_{i_{\text{Cheb}}}(a, Q \times Q_{\text{post}}, M_\circ(Z))$. We claim that we can find $p_2, \dots, p_{M_{\text{box}}}$ with
\[
g_{p_1, \dots, p_{M_{\text{box}}}} = g_0.
\]
So fix $2 \leq j \leq M_{\text{box}}$ and consider the additive function
\[
g_0(\bar{z}, (1, j)) - \phi_{\pi_{S' - \{i_{\text{char}}\}}(\bar{z}); \textup{pr}_1(\pi_{i_{\textup{char}}}(\bar{z})) \textup{pr}_2(\pi_{i_{\textup{char}}}(\bar{z}))}(\text{Frob}(p_1)) \in \text{Add}(Z, S')
\]
in case $I$ and consider the additive function
\[
g_0(\bar{z}, (1, j)) - \phi_{\bar{z}; -1}(\text{Frob}(p_1)) \in \text{Add}(Z, S')
\]
in case $II$, $III$ or $IV$. By Lemma \ref{lPhiAdd} and Lemma \ref{lConstantPointer}, there exists a unique element $\sigma \in \Gal(M_\circ(Z) M(Z)/M_\circ(Z))$ such that
\begin{multline*}
g_0(\bar{z}, (1, j)) - \phi_{\pi_{S' - \{i_{\text{char}}\}}(\bar{z}); \textup{pr}_1(\pi_{i_{\textup{char}}}(\bar{z})) \textup{pr}_2(\pi_{i_{\textup{char}}}(\bar{z}))}(\text{Frob}(p_1)) = \\
\phi_{\pi_{S' - \{i_{\text{char}}\}}(\bar{z}); \textup{pr}_1(\pi_{i_{\textup{char}}}(\bar{z})) \textup{pr}_2(\pi_{i_{\textup{char}}}(\bar{z}))}(\sigma)
\end{multline*}
in case $I$ and
\[
g_0(\bar{z}, (1, j)) - \phi_{\bar{z}; -1}(\text{Frob}(p_1)) = \phi_{\bar{z}; -1}(\sigma)
\]
in case $II$, $III$ or $IV$. Now pick $p_j$ such that $\text{Frob}(p_j) = \sigma$. Note that $\sigma$ implicitly depends on our choice of $p_1$. In fact, $\sigma$ is a fixed linear shift of $\text{Frob}(p_1)$. Using additivity in the last coordinate, we directly check that
\[
g_{p_1, \dots, p_{M_{\text{box}}}} = g_0.
\]
Since $Q_{\text{post}}$ is evenly distributed in the Chebotarev classes, we can repeatedly use the above procedure to get a partition as in equation (\ref{ePartitionCheb}) with
\[
|L| \leq \frac{|X_{i_{\text{Cheb}}}(a, Q \times Q_{\text{post}}, M_\circ(Z))|}{e^{0.25k_{\text{gap}}}}
\]
provided that we take $N_0$ sufficiently large. From this, we obtain that
\begin{align}
\label{eLbound}
\left|\sum_{x \in X(a, Q, \{\textup{Art}_k\}_{2 \leq k < m}) \cap \widetilde{Z}(Q_{\text{post}}) \cap (Z \times L)} \iota(F(\textup{Art}_{m, x}))\right| \leq \frac{|Z| \cdot |X_{i_{\text{Cheb}}}(a, Q \times Q_{\text{post}}, M_\circ(Z))|}{e^{0.25k_{\text{gap}}}}.
\end{align}
Let $\mathfrak{A}_i$ be the additive system on $Z \times [M_{\text{box}}]$ obtained by restricting $\mathfrak{A}$ to $Z \times A_i$ and then using the map $[M_{\text{box}}] \rightarrow A_i$ (coming from the fact that $A_i$ was an ordered set). By properties $(b)$ and $(c)$ of $\mathfrak{A}$, it follows that $\mathfrak{A}_i$ is $(a, S)$-acceptable and that
\[
(\Sigma F')(\bar{x}) = g_0(\bar{x})
\]
for all $\bar{x} \in C(\mathfrak{A}_i)$. By the defining property of $g_0$, see Proposition \ref{pARinput}, this implies that 
\begin{align}
\label{eARbound}
\left|\sum_{x \in X(a, Q, \{\textup{Art}_k\}_{2 \leq k < m}) \cap \widetilde{Z}(Q_{\text{post}}) \cap (Z \times A_i)} \iota(F(\textup{Art}_{m, x}))\right| \leq \epsilon \cdot |Z \times A_i|.
\end{align}
Equation (\ref{eQpost}) now falls as a consequence of equations (\ref{ePartitionCheb}), (\ref{eLbound}) and (\ref{eARbound}), which completes the proof.
\end{proof}

\end{document}